\documentclass[11pt, reqno]{amsart}
\usepackage{amsmath, amsthm, amscd, amsfonts, amssymb, graphicx, color, mathtools,mathrsfs}
\usepackage[bookmarksnumbered, colorlinks, plainpages]{hyperref}
\usepackage{enumerate}
\usepackage{setspace}
\usepackage{multicol}
\usepackage[margin=2cm]{geometry}
\usepackage{comment}
\usepackage{graphicx,psfrag,caption,subcaption,color}
\usepackage{setspace}
\usepackage{multicol}
\usepackage{pgfplots}
\usepackage{booktabs}

\theoremstyle{definition}
\newtheorem{theorem}{Theorem}[section]
\newtheorem{lemma}[theorem]{Lemma}
\newtheorem{proposition}[theorem]{Proposition}

\theoremstyle{remark}
\newtheorem{remark}[theorem]{Remark}
\numberwithin{equation}{section}

\newcommand{\cal}{\mathcal}
\newcommand{\bff}{\boldsymbol}
\newcommand{\bb}{\mathbb}

\newcommand{\dt}{\mathrm{d}t}
\newcommand{\ddt}{\frac{\mathrm{d}}{\mathrm{d}t}}
\newcommand{\dx}{\mathrm{d}x}
\newcommand{\ds}{\mathrm{d}s}

\newcommand{\norm}[2]{\left\|{#1}\right\|_{#2}}
\newcommand{\inpro}[2]{\left\langle#1,#2\right\rangle}
\newcommand{\abs}[1]{\left|{#1}\right|}

\allowdisplaybreaks

\begin{document}
\setcounter{page}{1}

\title[Mixed FEM for the LL-Bar and the regularised LL-Bloch equations in micromagnetics]{Mixed finite element methods for the Landau--Lifshitz--Baryakhtar and the regularised Landau--Lifshitz--Bloch equations in micromagnetics}

\author[Agus L. Soenjaya]{Agus L. Soenjaya}
\address{School of Mathematics and Statistics, The University of New South Wales, Sydney 2052, Australia}
\email{\textcolor[rgb]{0.00,0.00,0.84}{a.soenjaya@unsw.edu.au}}

\date{June 10, 2024}


\begin{abstract}
The Landau--Lifshitz--Baryakhtar (LLBar) and the Landau--Lifshitz--Bloch (LLBloch) equations are nonlinear vector-valued PDEs which arise in the theory of micromagnetics to describe the dynamics of magnetic spin field in a ferromagnet at elevated temperatures. We consider the LLBar and the regularised LLBloch equations in a unified manner, thus allowing us to treat the numerical approximations for both problems at once. In this paper, we propose a semi-discrete mixed finite element scheme and two fully discrete mixed finite element schemes based on a semi-implicit Euler method and a semi-implicit Crank--Nicolson method to solve the problems. These numerical schemes provide accurate approximations to both the magnetisation vector and the effective magnetic field. Moreover, they are proven to be unconditionally energy-stable and preserve energy dissipativity of the system at the discrete level. Error analysis is performed which shows optimal rates of convergence in $\bb{L}^2$, $\bb{L}^\infty$, and $\bb{H}^1$ norms. These theoretical results are further corroborated by several numerical experiments.
\end{abstract}
\maketitle

\section{Introduction}\label{sec:intro}
Micromagnetics is a field of physics that deals with magnetic behaviours at sub-micrometre length scales. The length scales considered are large enough for small-scale atomic structure to be ignored, but still fine enough to resolve some local behaviour. The standard model to describe time evolution of the magnetic configuration of a ferromagnet was proposed by Landau and Lifshitz~\cite{LL35}, and is commonly known as the Landau--Lifshitz (LL) equation. According to this model, the dynamics of the magnetic spin field is influenced by a precessional term (which tends to cause precession of the magnetisation vector) and a damping term (which tends to dissipate energy and align the magnetisation vector with the effective magnetic field). The damping term is chosen purely for phenomenological reasons, namely that the damping process must lead to the state with minimum energy, and that the magnetisation magnitude must remain constant.

The Landau--Lifshitz equation has been remarkably successful in describing the magnetisation dynamics at low temperature, and was widely analysed in the physics and mathematics literature, see \cite{AloSoy92, CarFab01, Cim08, GuoDing08, Lak11, Vis85} and references therein. However, it is \emph{not} sufficient to describe many experimental observations in modern physics, especially at high temperature~\cite{Bar84, ChuNowChaGar06, DvoVanVan13, Gar91, Gar97, WanDvo15}, where the magnitude of the magnetisation vector is known to be varying in time. Various magnetic recording devices, including the heat-assisted magnetic recording (HAMR) and the thermally-assisted magnetic random access memory (TA-MRAM), operate at high temperatures~\cite{MeoPan20, NieChu16, WanDvo15}, thus an accurate model and simulation of magnetisation dynamics at elevated temperatures is necessary. Indeed, the Landau--Lifshitz equation is essentially a zero-temperature equation, since it cuts off all contributions from high-frequency spin waves responsible for longitudinal magnetisation dynamics~\cite{ChuNie20}.

To rectify these problems, several approaches are proposed in the physics literature, most prominently due to Baryakhtar~\cite{BarBar98, Bar84} and Garanin~\cite{Gar91, Gar97}. Baryakhtar formulates his model based on the Onsager principles and the laws of thermodynamics, resulting in the Landau--Lifshitz--Baryakhtar (LLBar) equation, which also takes into account long-range interactions. Garanin based his model on the thermal averaging of many exchange-coupled atomistic spins, resulting in the Landau--Lifshitz--Bloch (LLBloch) equation, which can be seen as an interpolation between the standard Landau--Lifshitz equation and the Ginzburg--Landau system. In both models, the dynamics of a magnetisation vector in an effective field is influenced by three factors: precession, damping, and torque. The net effect on the magnetic spin field is a competition between these factors.

\par \medskip

Mathematically, the dynamics of the magnetic spin field $\bff{u}(t):\mathscr{D}\to\bb{R}^3$ in a bounded magnetic domain $\mathscr{D}$ can be described as
\begin{align*}
	\partial_t \bff{u} = \mathcal{P}(\bff{u}) + \mathcal{D}(\bff{u}) + \mathcal{S}(\bff{u}),
\end{align*}
where $\mathcal{P}(\bff{u})$ is the precessional term, $\mathcal{D}(\bff{u})$ is the damping term, and $\mathcal{S}(\bff{u})$ is the torque term. Quantum mechanics dictates that the precessional term has the form $\mathcal{P}(\bff{u})= -\gamma \bff{u}\times \bff{H}$, as in the usual Landau--Lifshitz equation, where $\bff{H}(t):\mathscr{D}\to \bb{R}^3$ is the effective field. Baryakhtar proposes that the damping term should be proportional to the effective field and its second spatial derivatives~\cite{BarBar98, Bar84}. In most ferromagnetic materials, the damping term for the LLBar equation~\cite{WanDvo15} can be written as
\begin{equation}\label{equ:damping}
	\mathcal{D}_{\text{Bar}}(\bff{u}) = \lambda_r \bff{H} - \lambda_e \Delta\bff{H},
\end{equation}
where $\bff{H}$ itself depends on $\bff{u}$. The positive constants $\lambda_r$ and $\lambda_e$ are the phenomenological relativistic damping constant and the exchange damping constant, respectively. These constants can also be replaced by positive definite symmetric tensors without changing the arguments here substantially. The second term in \eqref{equ:damping} is responsible for the longitudinal damping and long range interaction between magnetic spins~\cite{WanDvo15}.  
For the LLBloch equation above $T_\mathrm{c}$, the damping term is simply~\eqref{equ:damping} with $\lambda_e=0$.

The torque term has various forms depending on the physical situations considered. In the presence of applied current, the flow of electrons move charges and spins across space, affecting the magnetic properties of the material. A commonly used torque term is given by Zhang and Li~\cite{ZhaLi04}:
\begin{equation}\label{equ:spin}
	\mathcal{S}(\bff{u})= \Lambda_1 \bff{u}\times (\bff{u}\times (\bff{j}\cdot \nabla)\bff{u}) + \Lambda_2 \bff{u}\times (\bff{j}\cdot \nabla)\bff{u},
\end{equation}
where $\Lambda_1$ and $\Lambda_2$ are constants (which can be positive or negative), and $\bff{j}=\bff{j}(t,\bff{x})$ is the current density.

The effective magnetic field $\bff{H}$ consists of the usual contributions from the theory of micromagnetism, namely the exchange field, the external (Zeeman) field, and the anisotropy field. To account for elevated temperatures and phase transition, the Ginzburg--Landau internal exchange field is added~\cite{BarBar98, Gar97}, giving
\begin{equation}\label{equ:eff field}
	\bff{H}(t,\bff{x}) 
	= \underbrace{\alpha \Delta \bff{u}}_{\text{exchange}} + \underbrace{\kappa \mu \bff{u}- \kappa |\bff{u}|^2 \bff{u}}_{\text{internal exchange}} + \underbrace{\bff{B}(t)}_{\text{applied}}
	- \underbrace{\beta \bff{e}(\bff{e}\cdot \bff{u})}_{\text{anisotropy}}.
\end{equation}
Here, $\alpha$ is a positive constant depending on the material structure. The unit vector $\bff{e}$ is related to the axis of anisotropy of the material, while the vector $\bff{B}(t)$ describes an applied magnetic field. Physical considerations dictate that $\kappa$ is positive, while $\mu>0$ for temperatures above the Curie temperature $T_\mathrm{c}$ and $\mu<0$ below $T_\mathrm{c}$. The constant $\beta$ corresponds to the uniaxial anisotropy constant of the object, which can be positive or negative. To simplify presentation, we assume $\beta<0$ in~\eqref{equ:eff field}, which corresponds to the presence of an easy plane. Higher order anisotropy field of the form $\beta_1 \bff{e}(\bff{e}\cdot\bff{u})- \beta_2 \bff{e}(\bff{e}\cdot \bff{u})^3$, where $\beta_1,\beta_2>0$, could also be treated in a similar manner as the internal exchange field.

Subsequently, we take $\bff{B}(t)=\mathcal{S}(\bff{u})\equiv \bff{0}$ and set~$\alpha=1$ for simplicity. The presence of applied field and the Zhang--Li torque term~\eqref{equ:spin} could still be handled by the methods presented here with minor modifications. However, we choose to focus on the simplified version since in this case the system dissipates energy, and we want to highlight this energy dissipation property in our numerical schemes. 

\par \medskip

With these in mind, we can state the LLBar and the LLBloch equations in a unified manner. Given a magnetic body $\mathscr{D}\subset \bb{R}^d$, the magnetisation vector $\bff{u}(t,\bff{x})$ for any time $t\geq 0$ and at any point $\bff{x}\in \mathscr{D}$ evolves according to a nonlinear vector-valued PDE, which can be written in the following mixed form:
\begin{equation}\label{equ:llbar}
	\begin{cases}
		\displaystyle 
		\partial_t \bff{u}
		= 
		\lambda_r \bff{H} 
		- \lambda_e \Delta \bff{H} 
		- \gamma \bff{u} \times \bff{H},
		\quad & \text{for $(t,\bff{x})\in(0,T)\times\mathscr{D}$,}
		\\[1mm]
		\bff{H}
		= 
		\Delta \bff{u} 
		+ \kappa \mu \bff{u}
		- \kappa |\bff{u}|^2 \bff{u}
		- \beta \bff{e}(\bff{e}\cdot \bff{u}),
		\quad & \text{for $(t,\bff{x})\in(0,T)\times\mathscr{D}$,}
		\\[1mm]
		\bff{u}(0,\bff{x})= \bff{u_0}(\bff{x}) 
		\quad & \text{for } \bff{x}\in \mathscr{D},
		\\[1mm]
		\displaystyle{
			\frac{\partial \bff{u}}{\partial \bff{n}}= \bff{0}}, 
		\;\displaystyle{\frac{\partial \bff{H}}{\partial \bff{n}}= \bff{0}} 
		\quad & \text{for } (t,\bff{x})\in (0,T) \times \partial \mathscr{D},
	\end{cases}
\end{equation}
where $\partial\mathscr{D}$ is the boundary of $\mathscr{D}$, with exterior unit normal vector denoted by $\bff{n}$. Depending on the signs of some parameters, problem \eqref{equ:llbar} describes several models:
\begin{enumerate}[(i)]
	\item If $\lambda_e>0$ and $\mu>0$, this is the LLBar equation below $T_\mathrm{c}$.
	\item If $\lambda_e=0$ and $\mu<0$, this is the LLBloch equation above $T_\mathrm{c}$.
	\item If $\lambda_e>0$ and $\mu<0$, this is the LLBar equation above $T_\mathrm{c}$, which can also be considered as the \emph{regularised} LLBloch equation above $T_\mathrm{c}$ (cf. section~\ref{sec:bloch}).
\end{enumerate}
We remark that the LLBloch equation below $T_\mathrm{c}$ cannot be written in the above form and will be analysed separately in an upcoming paper, although for temperatures not far from (but below) $T_\mathrm{c}$, the LLBar equation with $\lambda_e>0$ and $\mu>0$ can be considered as an approximation to the LLBloch equation below $T_\mathrm{c}$~\cite{WanDvo15}.

The \emph{energy} $\mathcal{E}(\bff{u})$ of the system described by \eqref{equ:llbar} is 
\begin{equation}\label{equ:energy}
	\mathcal{E}(\bff{u})
	:=
	\begin{cases}
	\displaystyle
	\frac{1}{2} \norm{\nabla \bff{u}}{\bb{L}^2}^2
	+
	\frac{\kappa}{4} \norm{|\bff{u}|^2-\mu}{\bb{L}^2}^2
	+
	\frac{\beta}{2} \int_\mathscr{D} (\bff{e}\cdot \bff{u})^2\, \dx,
	\quad & \text{if } \mu>0,
	\\[3mm]
	\displaystyle
	\frac{1}{2} \norm{\nabla \bff{u}}{\bb{L}^2}^2
	+
	\frac{\kappa}{4} \norm{\bff{u}}{\bb{L}^4}^4
	-
	\frac{\kappa\mu}{2} \norm{\bff{u}}{\bb{L}^2}^2
	+
	\frac{\beta}{2} \int_\mathscr{D} (\bff{e}\cdot \bff{u})^2\, \dx,
	\quad & \text{if } \mu< 0,
	\end{cases}
\end{equation}
where $\kappa, \mu, \beta$ are constants appearing in \eqref{equ:llbar}. Note that the two expressions in~\eqref{equ:energy} differ only by a constant $\frac{1}{4} \kappa \mu^2 |\mathscr{D}|$, and in principle one could take the first expression as the energy for any values of $\mu$. However, it is known~\cite{LeSoeTra24} that $\bff{u}(t)\to \bff{0}$ as $t\to\infty$ for $\mu< 0$, and so the energy in this case is modified to ensure $\mathcal{E}(\bff{0})= 0$. In the absence of the spin-torque term, this system dissipates energy.

\par \medskip

Global well-posedness of the LLBar equation (assuming exchange-dominated field) is shown in~\cite{SoeTra23}, while the existence and uniqueness of strong solution to the LLBloch equation are shown in~\cite{Le16, LeSoeTra24} (see also~\cite{BrzGolLe20} and~\cite{GolSoeTra24b} for the stochastic equations). On the numerical aspect, some linear $C^1$-conforming finite element methods based on a semi-implicit Euler and a BDF schemes are proposed to solve the LLBar equation~\cite{SoeTra23b}. They are proven to be stable and convergent to the actual solution at an optimal rate, without assuming quasi-uniformity of the triangulation in most cases. However, the implementation is computationally costly, since the method requires $C^1$-continuity across element boundaries. A linear conforming finite element scheme based on the linearised Euler method has been proposed for the regularised LLBloch equation with a different regularisation term, where an optimal convergence rate in $\bb{H}^1$ is shown~\cite{LeSoeTra24}. Other regularisation term is used in~\cite{GolJiaLe22}, where a $C^1$-conforming method is proposed for the stochastic regularised LLBloch equation, but a strong convergence rate is only provided in $\bb{L}^2$ and the results are limited to $d=1$ or $2$. In all the above, energy dissipativity of the schemes is not addressed.

\par \medskip

Here, we continue and build on the studies done in~\cite{SoeTra23b} and~\cite{LeSoeTra24} by proposing numerical schemes which can be used to solve the LLBar and the LLBloch equations for $d=1,2,3$ in a unified manner. More precisely, we propose some mixed finite element methods to solve the LLBar and the regularised LLBloch equations, including a semi-discrete (in space) finite element scheme and two fully-discrete finite element schemes based on the Euler and the Crank--Nicolson methods. These schemes are proven to be uniquely solvable and unconditionally energy-stable, and this $\bb{H}^1$-stability is robust with respect to the parameter $\lambda_e$. Moreover, the energy dissipativity of the system is preserved at the discrete level. Since we are using the mixed formulation~\eqref{equ:llbar}, these methods provide an accurate approximation to both the magnetisation vector~$\bff{u}$ and the effective magnetic field~$\bff{H}$, both of which are physically significant quantities.

\par \medskip

To summarise, the main results of this paper include proving the following:
\begin{enumerate}
	\item convergence of a semi-discrete conforming finite element scheme (Theorem~\ref{the:semidisc rate}),
	\item convergence of a fully discrete conforming finite element scheme based on the semi-implicit Euler method (Theorem~\ref{the:euler rate}),
	\item convergence of a fully discrete conforming finite element scheme based on the Crank--Nicolson method (Theorem~\ref{the:crank rate}),
	\item convergence of the strong solution of the LLBar equation to that of the LLBloch equation at a certain rate (in~$L^\infty(0,T;\bb{H}^1(\mathscr{D})$), thus implying the above schemes can also be used to solve the LLBloch equation. Convergence of the corresponding effective field (in~$L^2(0,T;\bb{L}^2(\mathscr{D})$) is also shown (Theorem~\ref{the:bloch to bar}).
\end{enumerate}
In all cases, we obtain the expected rates of convergence in $\bb{L}^2$, $\bb{L}^\infty$, and $\bb{H}^1$ norms. We remark that systems of vector-valued PDEs similar to~\eqref{equ:llbar} also appear in chemistry and biology to model long-range diffusion~\cite[Chapter 11]{Mur02}, anomalous bi-flux diffusion~\cite{BevGalSimRio13, JiaBevZhu20}, and population dynamics~\cite{CocDi23, CohMur81, Och84}. As such, numerical schemes proposed in this paper would also apply to certain cases of these models.

\par \medskip

This paper is organised as follows. Notations and various assumptions on the exact solution and the finite element space are outlined in Section~\ref{sec:formula}. A semi-discrete finite element approximation is described in Section~\ref{sec:semidiscrete}. Two fully discrete mixed finite element schemes based on the Euler and the Crank--Nicolson methods are proposed in Section~\ref{sec:Euler} and Section~\ref{sec:Crank}, respectively. The regularised LLBloch equation and its approximation are discussed in Section~\ref{sec:bloch}. Finally, some numerical simulations which support the theoretical results are described in Section~\ref{sec:simulation}.

\section{Preliminaries}\label{sec:formula}

\subsection{Notations}

We begin by defining some notations used in this paper. The function space $\bb{L}^p := \bb{L}^p(\mathscr{D}; \bb{R}^3)$ denotes the usual space of $p$-th integrable functions taking values in $\bb{R}^3$ and $\bb{W}^{k,p} := \bb{W}^{k,p}(\mathscr{D}; \bb{R}^3)$ denotes the usual Sobolev space of 
functions on $\mathscr{D} \subset \bb{R}^d$, for $d=1,2,3$, taking values in $\bb{R}^3$. Also, we write $\bb{H}^k := \bb{W}^{k,2}$. Let $\Delta$ be the Neumann Laplacian operator acting on $\bb{R}^3$-valued functions with domain $\mathrm{D}(\Delta)$ given by
\begin{equation*}
	\mathrm{D}(\Delta):= \left\{\bff{v}\in \bb{H}^2 : \frac{\partial \bff{v}}{\partial \bff{n}} = \bff{0} \text{ on } \partial\mathscr{D} \right\}.
\end{equation*}

If $X$ is a Banach space, the spaces $L^p(0,T;X)$ and $W^{m,r}(0,T;X)$ denote respectively the Lebesgue and Sobolev spaces of functions on $(0,T)$ taking values in $X$, where $T$ can be a finite number or $\infty$. The space $C([0,T];X)$ denotes the space of continuous function on $[0,T]$ taking values in $X$. For simplicity, we will write $L^p(\bb{W}^{m,r}) := L^p(0,T; \bb{W}^{m,r})$ and $L^p(\bb{L}^q) := L^p(0,T; \bb{L}^q)$. 

Throughout this paper, we denote the scalar product in a Hilbert space $H$ by $\inpro{ \cdot}{ \cdot}_H$ and its corresponding norm by $\norm{\cdot}{H}$. We will not distinguish between the scalar product of $\bb{L}^2$ vector-valued functions taking values in $\bb{R}^3$ and the scalar product of $\bb{L}^2$ matrix-valued functions taking values in $\bb{R}^{3\times 3}$, and still denote them by $\inpro{\cdot}{\cdot}$.

Here, $\mathscr{D}$ is assumed to be a bounded domain such that the $\bb{H}^2$-regularity result holds, namely:
\begin{align*}
	\norm{\bff{v}}{\bb{H}^2}^2
	&\lesssim 
	\norm{\bff{v}}{\bb{L}^2}^2 + \norm{\Delta \bff{v}}{\bb{L}^2}^2
\end{align*}
for all $\bff{v}\in \mathrm{D}(\Delta)$. The above is true for any domain with $C^2$-smooth boundary. For domain with less regular boundary, it is known, for instance, that the $\bb{H}^2$-regularity result is guaranteed to hold for any convex Lipschitz domains \cite{Gri11}. Henceforth, we will assume that $\mathscr{D}$ is a smooth, or convex polygonal or polyhedral domain. In this domain, the Gagliardo--Nirenberg inequalities and the Sobolev embedding theorems also hold.

Finally, throughout this paper, the constant $C$ in the estimate denotes a
generic constant which takes different values at different occurrences. If
the dependence of $C$ on some variable, e.g.~$T$, is highlighted, we often write
$C(T)$. The notation $A \lesssim B$ means $A \le C B$ where the specific form of
the constant $C$ is not important to clarify.

\subsection{Assumptions}\label{sec:exact sol}
Given $T>0$ and $\bff{u}_0\in\bb{H}^1(\mathscr{D})$, the pair of functions $(\bff{u}, \bff{H}):[0,T]\to \bb{H}^1 \times \bb{H}^1$ is a weak solution to the problem
\eqref{equ:llbar} if $(\bff{u}, \bff{H}) $ satisfies
\begin{align}\label{equ:weakform}
	\left\{
	\begin{alignedat}{1}
		\inpro{ \partial_t \bff{u}}{\bff{\chi}}
		&=
		\lambda_r \inpro{\bff{H}}{\bff{\chi}}
		+
		\lambda_e \inpro{\nabla \bff{H}}{\nabla \bff{\chi}}
		-
		\gamma \inpro{\bff{u} \times \bff{H}}{\bff{\chi}}
		\\
		\inpro{\bff{H}}{\bff{\phi}} 
		&=
		-
		\inpro{\nabla \bff{u}}{\nabla \bff{\phi}}
		+
		\kappa \mu \inpro{\bff{u}}{\bff{\phi}}
		-
		\kappa \inpro{|\bff{u}|^2 \bff{u}}{\bff{\phi}}
		-
		\beta \inpro{\bff{e}(\bff{e}\cdot \bff{u})}{\bff{\phi}},
	\end{alignedat}
	\right.
\end{align}
for all $\bff{\chi}, \bff{\phi}\in \bb{H}^1$ and $t\in [0,T]$, with $\bff{u}(0)=\bff{u}_0$.

%The symmetric tensors $\bb{R}=\bb{R}(\bff{x})$, $\bb{E}=\bb{E}(\bff{x})$, and $\bb{A}=\bb{A}(\bff{x})$ are assumed to satisfy boundedness assumption:
%\begin{align}
%	\norm{\bb{A}\bff{v}}{\bb{L}^2}
%	\leq
%	C \norm{\bff{v}}{\bb{L}^2},
%	\quad
%	\norm{\bb{E}\bff{v}}{\bb{L}^2}
%	\leq
%	C \norm{\bff{v}}{\bb{L}^2},
%	\quad
%	\norm{\bb{R}\bff{v}}{\bb{L}^2}
%	\leq
%	C \norm{\bff{v}}{\bb{L}^2},
%\end{align}
%and coercivity assumption:
%\begin{align}
%	\inpro{\bb{R}\bff{v}}{\bff{v}} 
%	\geq
%	\lambda_R \norm{\bff{v}}{\bb{L}^2}^2,
%	\quad
%	\inpro{\bb{E}\bff{v}}{\bff{v}} 
%	\geq
%	\lambda_E \norm{\bff{v}}{\bb{L}^2}^2,
%	\quad
%	\inpro{\bb{A}\bff{v}}{\bff{v}} 
%	\geq
%	\lambda_A \norm{\bff{v}}{\bb{L}^2}^2,
%\end{align}
%for any $\bff{v}\in \bb{L}^2$, where $C, \lambda_R, \lambda_E$, and $\lambda_A$ are positive constants.

Throughout this paper, we assume that problem~\eqref{equ:weakform} possesses
solution~$(\bff{u}, \bff{H})$ which satisfies
\begin{equation}\label{equ:ass 2 u}
\begin{aligned}
	\norm{\bff{u}}{L^\infty(\bb{H}^{r+1})}
	+ \norm{\partial_t \bff{u}}{L^\infty(\bb{H}^{r+1})}
	+ \norm{\partial_t^2 \bff{u}}{L^\infty(\bb{H}^2)}
	+ \norm{\partial_t^3 \bff{u}}{L^\infty(\bb{L}^2)}
	&\le K_0,
	\\
	\norm{\bff{H}}{L^\infty(\bb{H}^{r+1})}
	+ \norm{\partial_t \bff{H}}{L^\infty(\bb{H}^{r+1})}
	+ \norm{\partial_t^2 \bff{H}}{L^\infty(\bb{H}^2)}
	+ \norm{\partial_t^3 \bff{H}}{L^\infty(\bb{L}^2)}
	&\le K_0,
\end{aligned}
\end{equation}
where $r$ is the degree of piecewise continuous polynomials used as the finite element space and~$K_0>0$ depends on~$\bff{u}_0$. The existence of an arbitrarily smooth solution to the LLBar equation on $(0,T)\times\mathscr{D}$ is guaranteed for any initial data $\bff{u_0}\in\bb{H}^1$ (cf.~\cite{GolSoeTra24a}).

For simplicity of presentation, throughout this paper we assume that $\mu>0$, except in Section~\ref{sec:bloch} where we specifically discuss the regularised LLBloch equation and outline the modifications needed for the case $\mu<0$, and in Section~\ref{sec:simulation} where some numerical simulations are performed.

\subsection{Finite Element Approximation}\label{sec:finite element}
Let $\mathcal{T}_h$ be a shape-regular triangulation of $\mathscr{D}\subset \bb{R}^d$ with maximal mesh-size $h$. To discretise the LLBar equation \eqref{equ:llbar}, we introduce the finite element space $\bb{V}_h \subset \bb{H}^1$, which is the space of all piecewise continuous polynomials on $\mathcal{T}_h$ of degree at most $r$. By the Bramble--Hilbert lemma, there exists a constant $C$ independent of $h$ such that for any $\bff{v} \in \bb{H}^{r+1}$,
\begin{align*}
	\inf_{\chi \in {\bb{V}}_h} \left\{ \norm{\bff{v} - \bff{\chi}}{\bb{L}^2} 
	+ 
	h \norm{\nabla (\bff{v}-\bff{\chi})}{\bb{L}^2} 
	\right\} 
	\leq 
	C h^{r+1} \norm{\bff{v}}{\bb{H}^{r+1}}.
\end{align*}
We shall use several operators in the analysis. Firstly, define the $\bb{L}^2$-projection operator~$\Pi_h: \bb{L}^2 \to \bb{V}_h$ such that
\begin{align}\label{equ:orth proj}
	\inpro{\Pi_h \bff{v}-\bff{v}}{\bff{\chi}}=0,
	\quad
	\forall \bff{\chi}\in \bb{V}_h.
\end{align}
It is well known that if $\bff{v}\in \bb{H}^{r+1}$, then
\begin{align}\label{equ:proj approx}
	\norm{\bff{v}- \Pi_h\bff{v}}{\bb{L}^2}
	+
	h \norm{\nabla( \bff{v}-\Pi_h\bff{v})}{\bb{L}^2}
	\leq
	C h^{r+1} \norm{\bff{v}}{\bb{H}^{r+1}}.
\end{align}
Moreover, if the triangulation is globally quasi-uniform, then we have the $\bb{H}^1$-stability of the $\bb{L}^2$-projection operator~\cite{BraPasSte02}, namely
\begin{align}\label{equ:H1 stab proj}
	\norm{\nabla \Pi_h \bff{v}}{\bb{L}^2} \leq C \norm{\nabla \bff{v}}{\bb{L}^2}, \quad \forall \bff{v}\in \bb{H}^1.
\end{align}
Next, we introduce the discrete Laplacian operator $\Delta_h: \bb{V}_h \to \bb{V}_h$ defined by
\begin{align}\label{equ:disc laplacian}
	\inpro{\Delta_h \bff{v}_h}{\bff{\chi}}
	=
	- \inpro{\nabla \bff{v}_h}{\nabla \bff{\chi}},
	\quad 
	\forall \bff{v}_h, \bff{\chi} \in \bb{V}_h,
\end{align}
and the Ritz projection operator $R_h: \bb{H}^1 \to \bb{V}_h$ by
\begin{align}\label{equ:Ritz}
	\inpro{\nabla R_h \bff{v}- \nabla \bff{v}}{\nabla \bff{\chi}}=0
	\;\text{ such that }
	\inpro{R_h\bff{v}-\bff{v}}{\bff{1}}=0,
	\quad
	\forall \bff{\chi}\in \bb{V}_h,
\end{align}
For any $\bff{v}\in \bb{H}^{r+1}$, let $\bff{\omega}(t):= \bff{v}(t)-R_h\bff{v}(t)$. The following estimate is well known \cite{Tho06} for $s=0$ or $1$:
\begin{align}\label{equ:Ritz ineq}
	\norm{\bff{\omega}(t)}{\bb{H}^s} + \norm{\partial_t \bff{\omega}(t)}{\bb{H}^s}
	\leq
	C h^{r+1-s} \norm{\bff{v}(t)}{\bb{H}^{r+1}}.
\end{align}
Moreover, if the triangulation is globally quasi-uniform, then by~\cite{Sco76},
\begin{align}\label{equ:Ritz ineq L infty}
	\norm{\bff{\omega}(t)}{\bb{L}^\infty} \leq Ch^{r+1} \abs{\ln h} \norm{\bff{v}(t)}{\bb{W}^{r+1,\infty}}.
\end{align}
It is known that the term $\abs{\ln h}$ can be removed in case $r\geq 2$ and $\mathscr{D}$ is a polygonal domain.
Throughout, we shall assume sufficient conditions on the regularity of the domain and the geometry of the mesh so that the maximum-norm stability of the Ritz projection holds, namely
\begin{equation}\label{equ:Ritz stab infty}
	\norm{R_h \bff{v}}{\bb{X}}
	\leq 
	C \norm{\bff{v}}{\bb{X}},
	\quad \text{where either } \bb{X}:= \bb{L}^\infty \text{ or } \bb{W}^{1,\infty}.
\end{equation}
This holds for a globally quasi-uniform triangulation~\cite{DemLeySchWah12}. However, \eqref{equ:Ritz stab infty} also holds under more general conditions, for instance in a convex polygonal or polyhedral domain with mildly graded mesh satisfying certain assumptions~\cite{DemLeySchWah12, Li17}, or in a non-convex polygonal domain with locally refined mesh~\cite{Li22}.

\subsection{Auxiliary Results}\label{sec:basic results}

In the analysis, we use the following vector identities: for any vectors~$\bff{a}$, $\bff{b}\in \bb{R}^3$,
\begin{align}\label{equ:a dot ab}
	2\bff{a}\cdot (\bff{a}-\bff{b}) 
	&= 
	\abs{\bff{a}}^2 - \abs{\bff{b}}^2 + \abs{\bff{a}-\bff{b}}^2,
	\\
	\label{equ:a2a dot ab}
	4\abs{\bff{a}}^2 \bff{a} \cdot (\bff{a}-\bff{b})
	&=
	\abs{\bff{a}}^4 - \abs{\bff{b}}^4 + \left(\abs{\bff{a}}^2 - \abs{\bff{b}}^2\right)^2 
	+ 2 \abs{\bff{a}}^2 \abs{\bff{a}-\bff{b}}^2,
	\\
	\label{equ:a2a b2b dot ab}
	2(\abs{\bff{a}}^2 \bff{a} - \abs{\bff{b}}^2 \bff{b}) \cdot (\bff{a}-\bff{b})
	&=
	\left(\abs{\bff{a}}^2 - \abs{\bff{b}}^2 \right)^2 + \left(\abs{\bff{a}}^2+\abs{\bff{b}}^2\right) \abs{\bff{a}-\bff{b}}^2.
\end{align}

The following inequalities will also be used frequently.

\begin{lemma}
	Let $\epsilon>0$ be given. Then there exists a positive constant $C$ depending only on $\mathscr{D}$ such that the following inequalities hold.
	\begin{enumerate}
		\renewcommand{\labelenumi}{\theenumi}
		\renewcommand{\theenumi}{{\rm (\roman{enumi})}}
		\item For any $\bff{v} \in \bb{H}^1(\mathscr{D})$,
		\begin{align}
			\label{equ:L4 gal nir}
			\norm{\bff{v}}{\bb{L}^4}^2 
			&\leq
			C \norm{\bff{v}}{\bb{L}^2}^{2-d/2}
			\norm{\bff{v}}{\bb{H}^1}^{d/2},
			\\
			\label{equ:L4 young}
			\norm{\bff{v}}{\bb{L}^4}^2
			&\leq
			C \norm{\bff{v}}{\bb{L}^2}^2
			+
			\epsilon \norm{\nabla \bff{v}}{\bb{L}^2}^2.
		\end{align}
		
		\item For any $\bff{v} \in \mathrm{D}(\Delta)$, 
		\begin{align}
			\label{equ:equivnorm-nabL2 young}
			\norm{\nabla \bff{v}}{\bb{L}^2}^2 
			&\leq  
			\frac{1}{4\epsilon} \norm{\bff{v}}{\bb{L}^2}^2
			+ 
			\epsilon \norm{\Delta \bff{v}}{\bb{L}^2}^2.
		\end{align}
		
		\item For any $\bff{v}_h\in \bb{V}_h$,
		\begin{align}
			\label{equ:disc lapl L2}
			\norm{\nabla \bff{v}_h}{\bb{L}^2}^2
			\leq
			\norm{\bff{v}_h}{\bb{L}^2} \norm{\Delta_h \bff{v}_h}{\bb{L}^2}.
		\end{align}
		
		\item Let $\mathscr{D}$ be a convex polygonal or polyhedral domain with globally quasi-uniform triangulation. For any $\bff{v}_h\in \bb{V}_h$,
		\begin{align}
			\label{equ:disc lapl L infty}
			\norm{\bff{v}_h}{\bb{L}^\infty}
			&\leq
			C \norm{\bff{v}_h}{\bb{L}^2}^{1-\frac{d}{4}} \left(\norm{\bff{v}_h}{\bb{L}^2}^\frac{d}{4} + \norm{\Delta_h \bff{v}_h}{\bb{L}^2}^\frac{d}{4} \right),
			\\
			\label{equ:disc lapl L6}
			\norm{\nabla \bff{v}_h}{\bb{L}^6}
			&\leq
			C \norm{\Delta_h \bff{v}_h}{\bb{L}^2}.
		\end{align}
	\end{enumerate}
\end{lemma}

\begin{proof}
	Inequality \eqref{equ:L4 gal nir} follows from the Gagliardo--Nirenberg inequality. Moreover, applying Young's inequality to \eqref{equ:L4 gal nir} gives
	\begin{align*}
		\norm{\bff{v}}{\bb{L}^4}^2
		&\leq
		C \norm{\bff{v}}{\bb{L}^2}^{2-d/2}
		\big( \norm{\bff{v}}{\bb{L}^2}^{d/2}
		+
		\norm{\nabla \bff{v}}{\bb{L}^2}^{d/2} \big)
		\leq 
		C \norm{\bff{v}}{\bb{L}^2}^2
		+
		\epsilon \norm{\nabla \bff{v}}{\bb{L}^2}^2.
	\end{align*}
	Inequality~\eqref{equ:equivnorm-nabL2 young} follows from integration by parts and Young's inequality, while \eqref{equ:disc lapl L2} follows by taking $\bff{\chi}=\bff{v}_h$ in \eqref{equ:disc laplacian} and applying H\"older's inequality.
	
	Finally, \eqref{equ:disc lapl L infty} and \eqref{equ:disc lapl L6} are proven in \cite[Appendix A]{GuiLiWan22}). This completes the proof of the lemma.
\end{proof}

\section{A Semi-discrete Galerkin Approximation}\label{sec:semidiscrete}

A semi-discrete Galerkin approximation to the problem~\eqref{equ:llbar} is $(\bff{u}_h, \bff{H}_h): [0,T] \to \bb{V}_h \times \bb{V}_h$ such that for all $t\in [0,T]$ and $\bff{\chi}, \bff{\phi} \in \bb{V}_h$,
\begin{align}\label{equ:weaksemidisc}
	\left\{
	\begin{alignedat}{1}
	\inpro{ \partial_t \bff{u}_h}{\bff{\chi}}
	&=
	\lambda_r \inpro{\bff{H}_h}{\bff{\chi}}
	+
	\lambda_e \inpro{\nabla \bff{H}_h}{\nabla \bff{\chi}}
	-
	\gamma \inpro{\bff{u}_h \times \bff{H}_h}{\bff{\chi}},
	\\
	\inpro{\bff{H}_h}{\bff{\phi}} 
	&=
	-
	\inpro{\nabla \bff{u}_h}{\nabla \bff{\phi}}
	+
	\kappa \mu \inpro{\bff{u}_h}{\bff{\phi}}
	-
	\kappa \inpro{|\bff{u}_h|^2 \bff{u}_h}{\bff{\phi}}
	-
	\beta \inpro{\bff{e}(\bff{e}\cdot \bff{u}_h)}{\bff{\phi}},
	\end{alignedat}
	\right.
\end{align}
with $\bff{u}_h(0)= \bff{u}_{0,h}$, an approximation of $\bff{u_0}$ in $\bb{V}_h$.
Note that \eqref{equ:weaksemidisc} can be written as
\begin{align}\label{equ:weakproj}
	\left\{
	\begin{alignedat}{1}
		\partial_t \bff{u}_h
		&=
		\lambda_r \bff{H}_h
		-
		\lambda_e \Delta_h \bff{H}_h
		-
		\gamma \Pi_h (\bff{u}_h \times \bff{H}_h),
		\\
		\bff{H}_h 
		&= 
		\Delta_h \bff{u}_h 
		+ 
		\kappa \mu \bff{u}_h 
		- 
		\kappa \Pi_h (|\bff{u}_h|^2 \bff{u}_h)
		-
		\beta \Pi_h \big(\bff{e}(\bff{e}\cdot \bff{u}_h)\big).
	\end{alignedat}
	\right.
\end{align}
Substituting the second equation into the first gives
\begin{align}\label{equ:ode sub}
	\nonumber
	\partial_t \bff{u}_h 
	&=
	(\lambda_r-\kappa\lambda_e \mu) \Delta_h \bff{u}_h
	-
	\lambda_e \Delta_h^2 \bff{u}_h
	+
	\kappa \mu \lambda_r \bff{u}_h
	-
	\kappa\lambda_r \Pi_h (|\bff{u}_h|^2 \bff{u}_h)
	+
	\kappa\lambda_e \Delta_h \Pi_h (|\bff{u}_h|^2 \bff{u}_h)
	\\
	&\quad 
	-
	\gamma \Pi_h (\bff{u}_h \times \Delta_h \bff{u}_h)
	+
	\kappa \gamma \Pi_h \big( \bff{u}_h \times \Pi_h (|\bff{u}_h|^2 \bff{u}_h) \big)
	+
	\beta\gamma \Pi_h \big(\bff{u}_h\times \Pi_h \big(\bff{e}(\bff{e}\cdot \bff{u}_h)\big)\big)
	\nonumber \\
	&\quad
	-
	\beta \lambda_r \Pi_h \big(\bff{e}(\bff{e}\cdot \bff{u}_h)\big)
	+
	\beta\lambda_e \Delta_h \Pi_h \big(\bff{e}(\bff{e}\cdot \bff{u}_h)\big).
\end{align}
Noting that all norms on the finite-dimensional space $\bb{V}_h$ are equivalent and each term in \eqref{equ:ode sub} is locally Lipschitz, we obtain a unique solution $\bff{u}_h\in \bb{V}_h$ for \eqref{equ:ode sub} defined on the interval $[0, t_h]\subseteq [0,T]$ by the standard theory for ordinary differential equation (thus also giving a unique solution $\bff{H}_h \in \bb{V}_h$ by substituting $\bff{u}_h$ back to \eqref{equ:weakproj}). We will prove several stability results, which will be used to ensure the semi-discrete solution $(\bff{u}_h, \bff{H}_h)$ can be continued globally to $[0,\infty)$ for any initial data $\bff{u}_{0,h}\in \bb{V}_h$. Note that these estimates hold uniformly in time $t\in (0,\infty)$.

%\begin{lemma}\label{lem:der u2u}
%	For any vector-valued function $\bff{v}:\mathscr{D}\to\bb{R}^3$, we have
%	\begin{align}
%		\label{equ:nab un2}
%		\nabla (|\bff{v}|^2 \bff{v}) 
%		&= 
%		2 \bff{v} \ (\bff{v}\cdot \nabla\bff{v}) 
%		+ |\bff{v}|^2 \nabla \bff{v},
%		\\
%		\label{equ:nor der v2v}
%		\frac{\partial\big(|\bff{v}|^2\bff{v}\big)}{\partial\bff{n}}
%		&=
%		2 \bff{v} \Big(\bff{v}\cdot \frac{\partial\bff{v}}{\partial\bff{n}}\Big)
%		+
%		|\bff{v}|^2 \frac{\partial\bff{v}}{\partial\bff{n}},
%	\end{align}
%	provided that the partial derivatives are well defined.
%\end{lemma}
%\begin{proof}
%	See \cite[Lemma~3.1]{SoeTra23}.
%\end{proof}

\begin{proposition}\label{pro:semidisc est1}
Let $h>0$ and initial data $\bff{u}_0$ be given. For any $s$, $t\in (0,\infty)$ such that $s\leq t$,
\begin{align}
	\label{equ:semidisc-est ener}
	\cal{E}(\bff{u}_h(t))
	\leq
	\cal{E}(\bff{u}_h(s)),
\end{align}
where $\cal{E}$ is the energy defined in \eqref{equ:energy}, and
\begin{align}
	\label{equ:semidisc-est1}
	\norm{\bff{u}_h(t)}{\bb{L}^4}^4 
	+  
	\norm{\bff{u}_h(t)}{\bb{H}^1}^2 
	+
	\lambda_r \int_0^t \norm{\bff{H}_h(s)}{\bb{L}^2}^2 \ds
	&+
	\lambda_e \int_0^t \norm{\nabla \bff{H}_h(s)}{\bb{L}^2}^2 \ds
\leq
	C_1 \left(\norm{\bff{u}_0}{\bb{H}^1}^4
	+
	|\mathscr{D}| \right),
\end{align}
where $C_1$ depends only on $\kappa$, $\mu$, and $\beta$.
\end{proposition}

\begin{proof}
Taking $\bff{\chi} = \bff{H}_h$ and $\bff{\phi}= \partial_t \bff{u}_h$ in \eqref{equ:weaksemidisc}, we obtain
\begin{align*}
	\inpro{\partial_t \bff{u}_h}{\bff{H}_h}
	&=
	\lambda_r \norm{\bff{H}_h}{\bb{L}^2}^2
	+
	\lambda_e \norm{\nabla \bff{H}_h}{\bb{L}^2}^2,
	\\
	\inpro{\bff{H}_h}{\partial_t \bff{u}_h}
	&=
	-
	\frac{1}{2} \ddt \norm{\nabla \bff{u}_h}{\bb{L}^2}^2
	+
	\frac{\kappa \mu}{2} \ddt \norm{\bff{u}_h}{\bb{L}^2}^2
	-
	\frac{\kappa}{4} \ddt \norm{\bff{u}_h}{\bb{L}^4}^4
	-
	\frac{\beta}{2} \ddt (\bff{e}\cdot \bff{u}_h)^2,
\end{align*}
These imply
\begin{align*}
	\frac{1}{2} \ddt \norm{\nabla \bff{u}_h}{\bb{L}^2}^2
	+
	\frac{\kappa}{4} \ddt \norm{\bff{u}_h}{\bb{L}^4}^4
	&=
	\frac{\kappa\mu}{2} \ddt \norm{\bff{u}_h}{\bb{L}^2}^2
	-
	\lambda_r \norm{\bff{H}_h}{\bb{L}^2}^2
	-
	\lambda_e \norm{\nabla \bff{H}_h}{\bb{L}^2}^2,
\end{align*}
or equivalently the energy identity:
\begin{equation*}
	\ddt\, \cal{E}(\bff{u}_h(t)) + \lambda_r \norm{\bff{H}_h(t)}{\bb{L}^2}^2 + \lambda_e \norm{\nabla \bff{H}_h(t)}{\bb{L}^2}^2 = 0,
\end{equation*}
thus proving \eqref{equ:semidisc-est ener}. Next, integrating this with respect to $t$ and rearranging, we obtain
\begin{align}\label{equ:nab uh eq}
	\frac{1}{2} \norm{\nabla \bff{u}_h}{\bb{L}^2}^2
	&-
	\frac{1}{2}\norm{\nabla \bff{u}_h(0)}{\bb{L}^2}^2
	+
	\frac{\kappa}{4} \norm{\bff{u}_h}{\bb{L}^4}^4
	-
	\frac{\kappa}{4} \norm{\bff{u}_h(0)}{\bb{L}^4}^4
	+
	\frac{\beta}{2} \norm{\bff{e}\cdot \bff{u}_h}{\bb{L}^2}^2
	-
	\frac{\beta}{2} \norm{\bff{e}\cdot \bff{u}_h(0)}{\bb{L}^2}^2
	\nonumber \\
	&+ \lambda_r \int_0^t \norm{\bff{H}_h(s)}{\bb{L}^2}^2 \ds
	+ \lambda_e \int_0^t \norm{\nabla \bff{H}_h(s)}{\bb{L}^2}^2 \ds
	\leq
	\frac{\kappa\mu}{2} \norm{\bff{u}_h}{\bb{L}^2}^2
	-
	\frac{\kappa\mu}{2} \norm{\bff{u}_h(0)}{\bb{L}^2}^2.
\end{align}
Let $A:=\kappa\mu+\beta$. After rearranging some terms, we have
\begin{align}\label{equ:norm uh L4}
	\nonumber
	&\norm{\nabla \bff{u}_h}{\bb{L}^2}^2
	+
	\frac{A}{2} \norm{\bff{u}_h}{\bb{L}^2}^2
	+
	\frac{\kappa}{4} \norm{\bff{u}_h}{\bb{L}^4}^4
	+ \lambda_r \int_0^t \norm{\bff{H}_h(s)}{\bb{L}^2}^2 \ds
	+ \lambda_e \int_0^t \norm{\nabla \bff{H}_h(s)}{\bb{L}^2}^2 \ds
	\\
	\nonumber
	&\leq
	\kappa \norm{\bff{u}_h(0)}{\bb{L}^4}^4
	+
	C \norm{\bff{u}_h(0)}{\bb{H}^1}^2
	+
	\frac{3A}{2} \norm{\bff{u}_h}{\bb{L}^2}^2
	-
	\frac{\kappa}{4} \norm{\bff{u}_h}{\bb{L}^4}^4
	\\
	\nonumber
	&=
	\kappa \norm{\bff{u}_h(0)}{\bb{L}^4}^4
	+
	C \norm{\bff{u}_h(0)}{\bb{H}^1}^2
	+
	\frac{\kappa}{4} \int_\mathscr{D} \left( \frac{6A}{\kappa} |\bff{u}_h(t)|^2 - |\bff{u}_h(t)|^4 \right) \dx
	\\
	&\leq
	\kappa \norm{\bff{u}_h(0)}{\bb{L}^4}^4
	+
	C \norm{\bff{u}_h(0)}{\bb{H}^1}^2
	+
	C |\mathscr{D}|
	\leq
	C \left( \norm{\bff{u}_h(0)}{\bb{H}^1}^4 + |\mathscr{D}| \right),
\end{align}
where in the last step we used \eqref{equ:L4 young}, and $C$ depends only on $\kappa, \mu, \beta$.
This shows~\eqref{equ:semidisc-est1}, thus completing the proof of the proposition.
\end{proof}

\begin{proposition}\label{pro:semidisc est2}
Let $h>0$ and initial data $\bff{u}_0$ be given. For all $t\in (0,\infty)$,
\begin{align}\label{equ:semidisc-est2}
	\lambda_e \norm{\bff{H}_h(t)}{\bb{L}^2}^2
	+ 
	\int_0^t \norm{\partial_t \bff{u}_h(s)}{\bb{L}^2}^2 \ds
	\leq 
	C\norm{\bff{u}_0}{\bb{H}^2}^2
	+
	C_1 \alpha_1 \left( \norm{\bff{u}_0}{\bb{H}^1}^2 + |\mathscr{D}| \right),
\end{align}
where $\alpha_1:= \lambda_r^{-1} \big((\lambda_r+\kappa \mu\lambda_e+\beta\lambda_e)^2 + C\gamma^2 + C(\kappa \lambda_e)^2\big)$, and the constant $C$ depends only on $\kappa$, $\mu$, $\beta$, and $\mathscr{D}$ (but is independent of $t$, $h$, and $\lambda_e$).
\end{proposition}

\begin{proof}
Taking $\bff{\chi}= \partial_t \bff{u}_h$ in \eqref{equ:weaksemidisc} yields
\begin{align}\label{equ:norm dt uh L2}
	\norm{\partial_t \bff{u}_h}{\bb{L}^2}^2 
	=
	\lambda_r \inpro{\bff{H}_h}{\partial_t \bff{u}_h}
	+
	\lambda_e \inpro{\nabla \bff{H}_h}{\nabla \partial_t \bff{u}_h}
	-
	\gamma \inpro{\bff{u}_h \times \bff{H}_h}{\partial_t \bff{u}_h}.
\end{align}
Differentiating the second equation in \eqref{equ:weaksemidisc} with respect to $t$, then taking $\bff{\phi}=\lambda_e \bff{H}_h$ yields
\begin{align*}
	\frac{\lambda_e}{2} \ddt \norm{\bff{H}_h}{\bb{L}^2}^2
	&=
	- \lambda_e
	\inpro{\nabla \partial_t \bff{u}_h}{\nabla \bff{H}_h}
	+
	\kappa\mu \lambda_e \inpro{\partial_t \bff{u}_h}{\bff{H}_h}
	-
	\kappa \lambda_e \inpro{\partial_t (|\bff{u}_h|^2 \bff{u}_h)}{\bff{H}_h}
	\\
	&\quad
	-
	\beta\lambda_e \inpro{\bff{e}(\bff{e}\cdot \partial_t \bff{u}_h)}{\bff{H}_h}.
\end{align*}
Adding this to \eqref{equ:norm dt uh L2}, then applying H\"{o}lder's inequality gives
\begin{align*} 
	&\frac{\lambda_e}{2} \ddt \norm{\bff{H}_h}{\bb{L}^2}^2
	+
	\norm{\partial_t \bff{u}_h}{\bb{L}^2}^2
	\\
	&=
	(\lambda_r + \kappa\mu \lambda_e) \inpro{\partial_t \bff{u}_h}{\bff{H}_h}
	-
	\gamma \inpro{\bff{u}_h \times \bff{H}_h}{\partial_t \bff{u}_h}
	-
	\kappa \lambda_e \inpro{\partial_t (|\bff{u}_h|^2 \bff{u}_h)}{\bff{H}_h}
	-
	\beta\lambda_e \inpro{\bff{e}(\bff{e}\cdot \partial_t \bff{u}_h)}{\bff{H}_h}
	\\
	&\leq
	B \norm{\partial_t \bff{u}_h}{\bb{L}^2}
	\norm{\bff{H}_h}{\bb{L}^2}
	+
	\gamma \norm{\bff{u}_h}{\bb{L}^4}
	\norm{\bff{H}_h}{\bb{L}^4}
	\norm{\partial_t \bff{u}_h}{\bb{L}^2}
	+
	\kappa\lambda_e \norm{\partial_t \bff{u}_h}{\bb{L}^2}
	\norm{\bff{u}_h}{\bb{L}^6}^2
	\norm{\bff{H}_h}{\bb{L}^6}
	\\
	&\leq
	\frac{1}{4} \norm{\partial_t \bff{u}_h}{\bb{L}^2}^2
	+
	B^2 \norm{\bff{H}_h}{\bb{L}^2}^2
	+
	\frac{1}{4} \norm{\partial_t \bff{u}_h}{\bb{L}^2}^2
	+
	C\gamma^2 \norm{\bff{H}_h}{\bb{H}^1}^2
	+
	\frac{1}{4} \norm{\partial_t \bff{u}_h}{\bb{L}^2}^2
	+
	C (\kappa \lambda_e)^2 \norm{\bff{H}_h}{\bb{H}^1}^2,
\end{align*}
where $B:=\lambda_r+\kappa\mu\lambda_e+\beta\lambda_e$, and in the last line we used Young's inequality, Sobolev embedding $\bb{H}^1\hookrightarrow \bb{L}^6$ and \eqref{equ:semidisc-est1}.
Rearranging the inequality and integrating with respect to $t$ (and noting \eqref{equ:semidisc-est1} again), we obtain
\begin{align*}
	\lambda_e \norm{\bff{H}_h(t)}{\bb{L}^2}^2
	+
	\int_0^t \norm{\partial_s \bff{u}_h(s)}{\bb{L}^2}^2 \ds
	&\leq
	\norm{\bff{H}_h(0)}{\bb{L}^2}^2
	+
	\alpha_1 \int_0^t \norm{\bff{H}_h(s)}{\bb{H}^1}^2 \ds
	\\
	&\leq
	C\norm{\bff{u}_0}{\bb{H}^2}^2
	+
	C_1 \alpha_1 \left( \norm{\bff{u}_0}{\bb{H}^1}^2 + |\mathscr{D}| \right),
\end{align*}
where the constant $C$ depends only on $\kappa$, $\mu$, and $\mathscr{D}$, as required.
\end{proof}

\begin{proposition}\label{pro:semidisc est3}
Let $h>0$ and initial data $\bff{u}_0$ be given. For all $t\in (0,\infty)$,
\begin{align}\label{equ:semidisc-est lapl}
	\lambda_e \norm{\Delta_h \bff{u}_h(t)}{\bb{L}^2}^2
	\leq 
	C(1+\lambda_e).
\end{align}
Moreover, if the triangulation $\cal{T}_h$ is globally quasi-uniform, then
\begin{align}\label{equ:semidisc-est L infty}
	\lambda_e \norm{\bff{u}_h(t)}{\bb{L}^\infty}^2
	\leq 
	C(1+\lambda_e).
\end{align}
Here, the constant $C$ depends only on $\kappa$, $\mu$, $\beta$, $\gamma$, $\lambda_r$, $\mathscr{D}$, and $K_0$ (but is independent of $t$, $h$, and $\lambda_e$).
\end{proposition}

\begin{proof}
Taking $\bff{\phi}= \lambda_e \Delta_h \bff{u}_h$, we obtain
\begin{align*}
	\lambda_e \inpro{\bff{H}_h}{\Delta_h \bff{u}_h}
	=
	\lambda_e \norm{\Delta_h \bff{u}_h}{\bb{L}^2}^2
	-
	\kappa\mu \lambda_e \norm{\nabla \bff{u}_h}{\bb{L}^2}^2
	-
	\kappa \lambda_e \inpro{|\bff{u}_h|^2 \bff{u}_h}{\Delta_h \bff{u}_h}
	-
	\beta \lambda_e \inpro{\bff{e}(\bff{e}\cdot \bff{u}_h)}{\Delta_h \bff{u}_h}.
\end{align*}
Therefore, after rearranging, we have
\begin{align*}
	\lambda_e \norm{\Delta_h \bff{u}_h}{\bb{L}^2}^2
	&=
	\kappa\mu \lambda_e \norm{\nabla \bff{u}_h}{\bb{L}^2}^2
	+
	\lambda_e \inpro{\bff{H}_h}{\Delta_h \bff{u}_h}
	+
	\kappa \lambda_e \inpro{|\bff{u}_h|^2 \bff{u}_h}{\Delta_h \bff{u}_h}
	+
	\beta \lambda_e \inpro{\bff{e}(\bff{e}\cdot \bff{u}_h)}{\Delta_h \bff{u}_h}
	\\
	&\leq
	\kappa \mu \lambda_e \norm{\nabla \bff{u}_h}{\bb{L}^2}^2
	+
	\lambda_e \norm{\bff{H}_h}{\bb{L}^2}^2
	+
	\frac{\lambda_e}{2} \norm{\Delta_h \bff{u}_h}{\bb{L}^2}^2
	+
	\beta^2 \lambda_e  \norm{\bff{u}_h}{\bb{L}^2}^2
	\\
	&\leq
	C \lambda_e \norm{\bff{u}_h}{\bb{H}^1}^2 
	+ 
	\lambda_e \norm{\bff{H}_h}{\bb{L}^2}^2
	+
	\frac{\lambda_e}{2}  \norm{\Delta_h \bff{u}_h}{\bb{L}^2}^2
	\leq 
	C\lambda_e + C + \frac{\lambda_e}{2}  \norm{\Delta_h \bff{u}_h}{\bb{L}^2}^2
\end{align*}
where we used Young's inequality, \eqref{equ:semidisc-est1}, and the Sobolev embedding $\bb{H}^1\hookrightarrow \bb{L}^6$. This proves inequality \eqref{equ:semidisc-est lapl}.
Finally, \eqref{equ:semidisc-est L infty} follows from \eqref{equ:disc lapl L infty}, \eqref{equ:semidisc-est lapl}, and \eqref{equ:semidisc-est1}.
\end{proof}

The following proposition shows stability of $\bff{H}_h$ in $L^\infty(\bb{H}^1)$ norm under some assumptions.

\begin{proposition}\label{pro:dt uh L2 assum}
Let $h>0$ and initial data $\bff{u}_0$ be given. Suppose that one of the following holds:
\begin{enumerate}
	\item $d=1$ or $2$,
	\item $d=3$ and $\kappa^2 \lambda_e \left(\norm{\bff{u}_0}{\bb{H}^1}^4 + |\mathscr{D}|\right)\lesssim \mu$,
	\item $d=3$ and the triangulation is globally quasi-uniform.
\end{enumerate}
For all $t\in [0,\infty)$,
\begin{align}\label{equ:semidisc-est3}
	\norm{\partial_t \bff{u}_h(t)}{\bb{L}^2}^2
	+ \int_0^t \norm{\partial_t \bff{H}_h(s)}{\bb{L}^2}^2 \ds
	+ \int_0^t \norm{\nabla \partial_t \bff{u}_h(s)}{\bb{L}^2}^2 \ds
	\leq
	C,
\end{align}
where the constant $C$ depends only on the coefficients of the equation, $\mathscr{D}$, and $K_0$.
\end{proposition}

\begin{proof}
Differentiating the first equation in \eqref{equ:weaksemidisc} with respect to $t$, then taking $\bff{\chi}= \partial_t \bff{u}_h$ gives
\begin{align}\label{equ:ddt dt uh}
	\frac{1}{2} \ddt \norm{\partial_t \bff{u}_h}{\bb{L}^2}^2
	=
	\lambda_r \inpro{\partial_t \bff{H}_h}{\partial_t \bff{u}_h}
	+
	\lambda_e \inpro{\nabla \partial_t \bff{H}_h}{\nabla \partial_t \bff{u}_h}
	-
	\gamma \inpro{\bff{u}_h \times \partial_t \bff{H}_h}{\partial_t \bff{u}_h}.
\end{align}
Differentiating the second equation in \eqref{equ:weaksemidisc} with respect to $t$, then taking $\bff{\phi}= \lambda_e \partial_t \bff{H}_h$ gives
\begin{align}\label{equ:lambda e dt Hh L2}
	\lambda_e \norm{\partial_t \bff{H}_h}{\bb{L}^2}^2
	&=
	-\lambda_e
	\inpro{\nabla \partial_t \bff{u}_h}{\nabla \partial_t \bff{H}_h}
	+
	\kappa\mu \lambda_e
	\inpro{\partial_t \bff{u}_h}{\partial_t \bff{H}_h}
	-
	\kappa \lambda_e
	\inpro{\partial_t (|\bff{u}_h|^2 \bff{u}_h)}{\partial_t \bff{H}_h}
	\nonumber \\
	&\quad
	-
	\beta\lambda_e \inpro{\bff{e}(\bff{e}\cdot\partial_t \bff{u}_h}{\partial_t \bff{H}_h}.
\end{align}
while taking $\bff{\phi}= C_r \partial_t \bff{u}_h$, where $C_r:=\lambda_r+\kappa\mu\lambda_e$ gives
\begin{align}\label{equ:dt Hh dt uh L2 sub}
	C_r \inpro{\partial_t \bff{H}_h}{\partial_t \bff{u}_h}
	&=
	C_r \left(
	\kappa \mu \norm{\partial_t \bff{u}_h}{\bb{L}^2}^2
	- 
	\norm{\nabla \partial_t \bff{u}_h}{\bb{L}^2}^2
	-
	2\kappa \norm{\bff{u}_h \cdot \partial_t \bff{u}_h}{\bb{L}^2}^2
	-
	\kappa \norm{|\bff{u}_h| |\partial_t \bff{u}_h|}{\bb{L}^2}^2\right)
	\nonumber \\
	&\quad
	+
	C_r \beta \norm{\bff{e}\cdot \partial_t \bff{u}_h}{L^2}^2.
\end{align}
Adding \eqref{equ:ddt dt uh}, \eqref{equ:lambda e dt Hh L2}, and \eqref{equ:dt Hh dt uh L2 sub} yields
\begin{align*}
	&\frac{1}{2} \ddt \norm{\partial_t \bff{u}_h}{\bb{L}^2}^2
	+
	\lambda_e \norm{\partial_t \bff{H}_h}{\bb{L}^2}^2
	+
	C_r \big( \norm{\nabla \partial_t \bff{u}_h}{\bb{L}^2}^2
	+
	2\kappa \norm{\bff{u}_h \cdot \partial_t \bff{u}_h}{\bb{L}^2}^2
	+
	\kappa \norm{|\bff{u}_h| |\partial_t \bff{u}_h|}{\bb{L}^2}^2 \big)
	\\
	&=
	C_r \kappa \mu \norm{\partial_t \bff{u}_h}{\bb{L}^2}^2
	+
	C_r \beta \norm{\bff{e}\cdot \partial_t \bff{u}_h}{L^2}^2
	-
	\kappa \lambda_e \inpro{\partial_t (|\bff{u}_h|^2 \bff{u}_h)}{\partial_t \bff{H}_h}
	-
	\gamma \inpro{\bff{u}_h \times \partial_t \bff{H}_h}{\partial_t \bff{u}_h},
\end{align*}
%where $\mu:= \lambda_r + \kappa \lambda_e$.
We estimate each term appearing on the right-hand side as follows.

\medskip
\noindent
\underline{Case 1: $d=1$ or $2$}. In this case, H\"{o}lder's inequality implies
\begin{align*}
	&\frac{1}{2} \ddt \norm{\partial_t \bff{u}_h}{\bb{L}^2}^2
	+
	\lambda_e \norm{\partial_t \bff{H}_h}{\bb{L}^2}^2
	+
	C_r \norm{\nabla \partial_t \bff{u}_h}{\bb{L}^2}^2
	+
	C_r\kappa \norm{|\bff{u}_h| |\partial_t \bff{u}_h|}{\bb{L}^2}^2
	\\
	&\leq
	C_r \big(\kappa\mu+\beta\big) \norm{\partial_t \bff{u}_h}{\bb{L}^2}^2
	+
	\kappa \lambda_e \norm{\bff{u}_h}{\bb{L}^8}^2
	\norm{\partial_t \bff{u}_h}{\bb{L}^4}
	\norm{\partial_t \bff{H}_h}{\bb{L}^2}
	+
	\gamma \norm{\bff{u}_h}{\bb{L}^4}
	\norm{\partial_t \bff{H}_h}{\bb{L}^2}
	\norm{\partial_t \bff{u}_h}{\bb{L}^4}
	\\
	&\leq
	C \norm{\partial_t \bff{u}_h}{\bb{L}^2}^2
	+
	\frac{C_r}{2} \norm{\nabla \partial_t \bff{u}_h}{\bb{L}^2}^2
	+
	\frac{\lambda_e}{2} \norm{\partial_t \bff{H}_h}{\bb{L}^2}^2,
\end{align*}
where in the last step we used the Sobolev embedding $\bb{H}^1 \hookrightarrow \bb{L}^8$, Young's inequality, \eqref{equ:L4 young} and~\eqref{equ:semidisc-est1}. The required inequality then follows by integrating both sides with respect to $t$ and noting~\eqref{equ:semidisc-est2}.

\medskip
\noindent
\underline{Case 2: $d=3$ and $\kappa^2 \lambda_e \big(\norm{\bff{u}_0}{\bb{H}^1}^4 + |\mathscr{D}|\big)\lesssim \mu$}. Similarly, by H\"{o}lder's and Young's inequalities, we have
\begin{align*}
	&\frac{1}{2} \ddt \norm{\partial_t \bff{u}_h}{\bb{L}^2}^2
	+
	\lambda_e \norm{\partial_t \bff{H}_h}{\bb{L}^2}^2
	+
	C_r \norm{\nabla \partial_t \bff{u}_h}{\bb{L}^2}^2
	+
	C_r\kappa \norm{|\bff{u}_h| |\partial_t \bff{u}_h|}{\bb{L}^2}^2
	\\
	&\leq
	C_r \big(\kappa\mu+\beta\big) \norm{\partial_t \bff{u}_h}{\bb{L}^2}^2
	+
	\kappa \lambda_e \norm{\bff{u}_h}{\bb{L}^6}^2
	\norm{\partial_t \bff{u}_h}{\bb{L}^6}
	\norm{\partial_t \bff{H}_h}{\bb{L}^2}
	+
	\gamma \norm{\bff{u}_h}{\bb{L}^4}
	\norm{\partial_t \bff{H}_h}{\bb{L}^2}
	\norm{\partial_t \bff{u}_h}{\bb{L}^4}
	\\
	&\leq
	C \norm{\partial_t \bff{u}_h}{\bb{L}^2}^2
	+
	\frac{C_r}{2} \norm{\nabla \partial_t \bff{u}_h}{\bb{L}^2}^2
	+
	\frac{\lambda_e}{2} \norm{\partial_t \bff{H}_h}{\bb{L}^2}^2
	+
	C_{\mathrm{S}} \kappa^2 \lambda_e \norm{\bff{u}_h}{\bb{H}^1}^4 \norm{\partial_t \bff{u}_h}{\bb{H}^1}^2,
\end{align*}
where $C_{\mathrm{S}}$ is the constant associated with the Sobolev embedding $\bb{H}^1\hookrightarrow \bb{L}^6$.
If $\norm{\bff{u}_0}{\bb{H}^1}^4 + |\mathscr{D}|$ is sufficiently small, or more precisely (with $C_1$ as given in \eqref{equ:semidisc-est1})
\[
C_{\mathrm{S}} C_1^4 \kappa^2 \lambda_e \left(\norm{\bff{u}_0}{\bb{H}^1}^4 + |\mathscr{D}|\right)\leq \mu,
\]
then noting~\eqref{equ:semidisc-est1}, we can absorb the term~$C_{\mathrm{S}} \kappa^2 \lambda_e \norm{\bff{u}_h}{\bb{H}^1}^4 \norm{\partial_t \bff{u}_h}{\bb{H}^1}^2$ to the left-hand side. The required inequality then follows by integrating both sides with respect to $t$ and using~\eqref{equ:semidisc-est2}.

\medskip
\noindent
\underline{Case 3: $d=3$ and the triangulation is quasi-uniform}. In this case, by H\"older's inequality and~\eqref{equ:semidisc-est L infty},
\begin{align*}
	&\frac{1}{2} \ddt \norm{\partial_t \bff{u}_h}{\bb{L}^2}^2
	+
	\lambda_e \norm{\partial_t \bff{H}_h}{\bb{L}^2}^2
	+
	C_r \norm{\nabla \partial_t \bff{u}_h}{\bb{L}^2}^2
	+
	C_r\kappa \norm{|\bff{u}_h| |\partial_t \bff{u}_h|}{\bb{L}^2}^2
	\\
	&\leq
	C_r \big(\kappa\mu+\beta\big) \norm{\partial_t \bff{u}_h}{\bb{L}^2}^2
	+
	\kappa \lambda_e \norm{\bff{u}_h}{\bb{L}^\infty}^2
	\norm{\partial_t \bff{u}_h}{\bb{L}^2}
	\norm{\partial_t \bff{H}_h}{\bb{L}^2}
	+
	\gamma \norm{\bff{u}_h}{\bb{L}^\infty}
	\norm{\partial_t \bff{H}_h}{\bb{L}^2}
	\norm{\partial_t \bff{u}_h}{\bb{L}^2}
	\\
	&\leq
	C \norm{\partial_t \bff{u}_h}{\bb{L}^2}^2
	+
	\frac{C_r}{2} \norm{\nabla \partial_t \bff{u}_h}{\bb{L}^2}^2
	+
	\frac{\lambda_e}{2} \norm{\partial_t \bff{H}_h}{\bb{L}^2}^2.
\end{align*}
Integrating both sides with respect to $t$ then yields the inequality in this case.
This completes the proof of the proposition.
\end{proof}

\begin{proposition}
Let $h>0$ and initial data $\bff{u}_0$ be given. Under the same assumptions as Proposition~\ref{pro:dt uh L2 assum}, for all $t\in [0,\infty)$,
\begin{align*}
	\norm{\nabla \bff{H}_h(t)}{\bb{L}^2}
	\leq
	C,
\end{align*}
where the constant $C$ depends only on the coefficients of the equation, $\mathscr{D}$, and $K_0$.
\end{proposition}

\begin{proof}
Setting $\bff{\chi}=\bff{H}_h$ in \eqref{equ:weaksemidisc} gives
\begin{align*}
	\lambda_e \norm{\nabla \bff{H}_h}{\bb{L}^2}^2
	&=
	\inpro{\partial_t \bff{u}_h}{\bff{H}_h}
	-
	\lambda_r \norm{\bff{H}_h}{\bb{L}^2}^2
	\lesssim
	\norm{\partial_t \bff{u}_h}{\bb{L}^2}^2
	+
	\norm{\bff{H}_h}{\bb{L}^2}^2.
\end{align*}
The required result then follows from \eqref{equ:semidisc-est1} and \eqref{equ:semidisc-est3}.
\end{proof}

To estimate the error in the semi-discrete approximation, we write it as a sum of two terms:
\begin{align}\label{equ:theta rho}
	\bff{u}_h(t) - \bff{u}(t) 
	= 
	(\bff{u}_h(t)- R_h \bff{u}(t)) + (R_h \bff{u}(t) - \bff{u}(t)) 
	&=: \bff{\theta}(t) + \bff{\rho}(t),
	\\
	\label{equ:xi eta}
	\bff{H}_h(t) - \bff{H}(t) 
	= 
	(\bff{H}_h(t)- R_h \bff{H}(t)) + (R_h \bff{H}(t) - \bff{H}(t)) 
	&=: \bff{\xi}(t) + \bff{\eta}(t).
\end{align}
where $R_h$ is the Ritz projection operator (see \eqref{equ:Ritz}).

In the analysis, to bound the nonlinear terms, we will often write (by adding and substracting $R_h \bff{u}$)
\begin{align}\label{equ:nonlinear Rh u}
	\nonumber
	|\bff{u}_h|^2 \bff{u}_h - |\bff{u}|^2 \bff{u}
	&=
	(\bff{u}_h- R_h \bff{u}) \cdot (\bff{u}_h+ R_h \bff{u}) \bff{u}_h
	+
	|R_h \bff{u}|^2 (\bff{u}_h-\bff{u})
	+
	(R_h \bff{u} - \bff{u})\cdot (R_h \bff{u}+ \bff{u}) \bff{u}
	\\
	&=
	\big(\bff{\theta} \cdot (\bff{\theta}+ 2 R_h \bff{u}) \big) \bff{u}_h
	+
	\abs{R_h \bff{u}}^2 \big(\bff{\theta}+\bff{\rho}\big)
	+
	\big( \bff{\rho} \cdot (R_h \bff{u} + \bff{u}) \big) \bff{u}
\end{align}
and
\begin{align}\label{equ:nonlinear cross}
	\bff{u}_h \times \bff{H}_h - \bff{u} \times \bff{H}
	\nonumber
	&=
	(\bff{u}_h - R_h \bff{u})\times \bff{H}_h
	+
	R_h \bff{u} \times (\bff{H}_h-\bff{H})
	+
	(R_h \bff{u}-\bff{u}) \times \bff{H}
	\\
	&= 
	\bff{\theta} \times \bff{H}_h
	+
	R_h\bff{u} \times (\bff{\xi}+\bff{\eta})
	+
	\bff{\rho} \times \bff{H}.
\end{align}
We derive some estimates for the nonlinear terms in the following lemmas.

%\begin{lemma}\label{lem:Rh bounded}
%Let $R_h \bff{u}$ and $R_h \bff{H}$ be as defined in \eqref{equ:Ritz u} and \eqref{equ:Ritz H}, and $r\geq 2$. Then
%\begin{align*}
%	\norm{R_h \bff{u}(t)}{\bb{W}^{1,\infty}}
%	+ \norm{R_h \bff{H}(t)}{\bb{W}^{1,\infty}}
%	\leq
%	C \quad \text{for all } t\in [0,T],
%\end{align*}
%for some constant $C=C(T, \norm{\bff{u}}{\bb{H}^3}, \norm{\bff{H}}{\bb{H}^3})$.
%\end{lemma}
%
%\begin{proof}
%For any $\bff{\chi}\in \bb{V}_h$, the inverse estimate in $\mathscr{D} \subset \bb{R}^d$ (Theorem \ref{the:inverse}) gives
%\begin{align*}
%	\norm{\bff{\chi}}{\bb{W}^{1,\infty}} \leq Ch^{-d/2} \norm{\bff{\chi}}{\bb{H}^1}.
%\end{align*}
%Therefore, by this inverse estimate, the definition of $\bff{\rho}$ and the triangle inequality, we have
%\begin{align*}
%	\norm{R_h \bff{u}(t)}{\bb{W}^{1,\infty}}
%	&\leq
%	\norm{R_h \bff{u}(t) - \mathcal{I}_h(\bff{u}(t))}{\bb{W}^{1,\infty}} 
%	+
%	\norm{\mathcal{I}_h(\bff{u}(t))}{\bb{W}^{1,\infty}}
%	\\
%	&\leq 
%	Ch^{-d/2} \norm{R_h \bff{u}(t) - \mathcal{I}_h(\bff{u}(t))}{\bb{H}^1}
%	+
%	\norm{\mathcal{I}_h(\bff{u}(t))}{\bb{H}^3}
%	\\
%	&\leq Ch^{-d/2} 
%	\big(\norm{\bff{\rho}(t)}{\bb{H}^1} + \norm{ \mathcal{I}_h(\bff{u}(t))-\bff{u}(t) }{\bb{H}^1} \big)
%	+
%	\norm{\mathcal{I}_h(\bff{u}(t))}{\bb{H}^3}
%	\\
%	&\leq Ch^{-d/2} \big(C h^r \norm{\bff{u}(t)}{\bb{H}^{r+1}}  \big)
%	+
%	C \norm{\bff{u}(t)}{\bb{H}^3},
%\end{align*}
%where we also used the Sobolev embedding $\bb{H}^3 \subset \bb{W}^{1,\infty}$, and the boundedness of $\cal{I}_h$ on $\bb{H}^3$.
%\\
%Similar estimate holds for $R_h \bff{H}(t)$. This completes the proof of the lemma.
%\end{proof}

\begin{lemma}\label{lem:cross est}
Let $\epsilon>0$ and initial data $\bff{u}_0$ be given. Let $(\bff{u},\bff{H})$ be the solution of \eqref{equ:weakform} satisfying \eqref{equ:ass 2 u}. Then we have
\begin{align}
	\label{equ:uh cross Hh theta}
	\big| \inpro{\bff{u}_h \times \bff{H}_h- \bff{u} \times \bff{H}}{\bff{\theta}} \big| 
	&\lesssim
	h^{2(r+1)} 
	+ 
	\norm{\bff{\theta}}{\bb{L}^2}^2 
	+ 
	\epsilon \norm{\nabla \bff{\theta}}{\bb{L}^2}^2
	+
	\epsilon \norm{\bff{\xi}}{\bb{L}^2}^2,
	\\
	\label{equ:uh cross Hh xi}
	\big| \inpro{\bff{u}_h \times \bff{H}_h- \bff{u} \times \bff{H}}{\bff{\xi}} \big| 
	&\lesssim
	h^{2(r+1)} 
	+ 
	\norm{\bff{\xi}}{\bb{L}^2}^2 
	+ 
	\epsilon \norm{\nabla \bff{\xi}}{\bb{L}^2}^2
	+
	\epsilon \norm{\bff{\theta}}{\bb{L}^2}^2,
	\\
	\label{equ:uh cross Hh dt theta}
	\big| \inpro{\bff{u}_h \times \bff{H}_h- \bff{u} \times \bff{H}}{\partial_t \bff{\theta}} \big| 
	&\lesssim
	h^{2(r+1)}
	+
	\norm{\bff{\xi}}{\bb{L}^2}^2
	+
	\big( \norm{\bff{\theta}}{\bb{L}^2}^2 + \epsilon \norm{\nabla \bff{\theta}}{\bb{L}^2}^2 \big) \norm{\bff{H}_h}{\bb{H}^1}^2
	+
	\epsilon \norm{\partial_t \bff{\theta}}{\bb{L}^2}^2.
\end{align}
Furthermore, if the triangulation $\mathcal{T}_h$ is globally quasi-uniform, then for any $\bff{\zeta}\in \bb{V}_h$,
\begin{align}
	\label{equ:uh cross Hh Delta}
	\big| \inpro{\bff{u}_h \times \bff{H}_h- \bff{u} \times \bff{H}}{\bff{\zeta}} \big| 
	&\lesssim
	h^{2(r+1)} 
	+ 
	\norm{\bff{\theta}}{\bb{L}^2}^2 
	+
	\norm{\bff{\xi}}{\bb{L}^2}^2
	+ 
	\epsilon \norm{\bff{\zeta}}{\bb{L}^2}^2,
	\\
	\label{equ:nab uh cross Hh}
	\big|\inpro{\nabla\Pi_h\big(\bff{u}_h \times \bff{H}_h- \bff{u} \times \bff{H}\big)}{\bff{\zeta}} \big|
	&\lesssim
	h^{2r}
	+
	\norm{\bff{\theta}}{\bb{H}^1}^2
	+
	\norm{\bff{\xi}}{\bb{H}^1}^2
	+
	\epsilon \norm{\bff{\zeta}}{\bb{L}^2}^2.
\end{align}
Here, all constants are independent of $h$.
\end{lemma}

\begin{proof}
First, we will prove \eqref{equ:uh cross Hh theta} by writing
\begin{align}\label{equ:cross split}
	\bff{u}_h \times \bff{H}_h - \bff{u} \times \bff{H}
	=
	(\bff{\theta}+\bff{\rho}) \times \bff{H}
	+
	\bff{u}_h \times (\bff{\xi}+\bff{\eta}).
\end{align}
Therefore, by H\"older's and Young's inequalities, we have
\begin{align*}
	\big| \inpro{\bff{u}_h \times \bff{H}_h- \bff{u} \times \bff{H}}{\bff{\theta}} \big| 
	&\leq
	\norm{\bff{\rho}}{\bb{L}^2} \norm{\bff{H}}{\bb{L}^\infty} \norm{\bff{\theta}}{\bb{L}^2}
	+
	\norm{\bff{u}_h}{\bb{L}^4} \norm{\bff{\xi}+\bff{\eta}}{\bb{L}^2} \norm{\bff{\theta}}{\bb{L}^4}
	\\
	&\lesssim
	\norm{\bff{\rho}}{\bb{L}^2}^2
	+
	\norm{\bff{\theta}}{\bb{L}^2}^2
	+
	\norm{\bff{\theta}}{\bb{L}^4}^2
	+
	\epsilon \norm{\bff{\xi}+\bff{\eta}}{\bb{L}^2}^2
	\\
	&\lesssim
	h^{2(r+1)} 
	+
	\norm{\bff{\theta}}{\bb{L}^2}^2
	+
	\epsilon \norm{\nabla \bff{\theta}}{\bb{L}^2}^2
	+
	\epsilon \norm{\bff{\xi}}{\bb{L}^2}^2,
\end{align*}
where in the last step we used \eqref{equ:Ritz ineq} and \eqref{equ:L4 young}, thus proving \eqref{equ:uh cross Hh theta}.
Similarly,
\begin{align*}
	\big| \inpro{\bff{u}_h \times \bff{H}_h- \bff{u} \times \bff{H}}{\bff{\xi}} \big| 
	&\leq
	\norm{\bff{\theta}+\bff{\rho}}{\bb{L}^2} \norm{\bff{H}}{\bb{L}^\infty} \norm{\bff{\xi}}{\bb{L}^2}
	+
	\norm{\bff{u}_h}{\bb{L}^4} \norm{\bff{\eta}}{\bb{L}^2} \norm{\bff{\xi}}{\bb{L}^4}
	\\
	&\lesssim
	\epsilon \norm{\bff{\theta}}{\bb{L}^2}^2
	+
	\norm{\bff{\rho}}{\bb{L}^2}^2
	+
	\norm{\bff{\xi}}{\bb{L}^2}^2
	+
	\norm{\bff{\xi}}{\bb{L}^4}^2
	+
	\epsilon \norm{\bff{\eta}}{\bb{L}^2}^2
	\\
	&\lesssim
	h^{2(r+1)}
	+
	\norm{\bff{\xi}}{\bb{L}^2}^2
	+
	\epsilon \norm{\nabla \bff{\xi}}{\bb{L}^2}^2
	+
	\epsilon \norm{\bff{\theta}}{\bb{L}^2}^2,
\end{align*}
proving \eqref{equ:uh cross Hh xi}.
Next, using \eqref{equ:nonlinear cross}, H\"older's and Young's inequalities, we obtain
\begin{align*}
	&\big| \inpro{\bff{u}_h \times \bff{H}_h- \bff{u} \times \bff{H}}{\partial_t \bff{\theta}} \big| 
	\\
	&\leq
	\norm{\bff{\theta}}{\bb{L}^4}
	\norm{\bff{H}_h}{\bb{L}^4}
	\norm{\partial_t \bff{\theta}}{\bb{L}^2}
	+
	\norm{R_h \bff{u}}{\bb{L}^\infty}
	\norm{\bff{\xi}+\bff{\eta}}{\bb{L}^2}
	\norm{\partial_t \bff{\theta}}{\bb{L}^2}
	+
	\norm{\bff{\rho}}{\bb{L}^2}
	\norm{\bff{H}}{\bb{L}^\infty}
	\norm{\partial_t \bff{\theta}}{\bb{L}^2}
	\\
	&\lesssim
	\norm{\bff{\theta}}{\bb{L}^4}^2 \norm{\bff{H}_h}{\bb{L}^4}^2
	+
	\epsilon \norm{\partial_t \bff{\theta}}{\bb{L}^2}^2
	+
	\norm{\bff{\xi}}{\bb{L}^2}^2
	+
	h^{2(r+1)}
	+
	\epsilon \norm{\partial_t \bff{\theta}}{\bb{L}^2}^2
	\\
	&\lesssim
	\big(\norm{\bff{\theta}}{\bb{L}^2}^2 + \epsilon \norm{\nabla \bff{\theta}}{\bb{L}^2}^2 \big) 
	\norm{\bff{H}_h}{\bb{H}^1}^2
	+
	\norm{\bff{\xi}}{\bb{L}^2}^2
	+
	h^{2(r+1)}
	+
	\epsilon \norm{\partial_t \bff{\theta}}{\bb{L}^2}^2,
\end{align*}
where in the last step we used Sobolev embedding and \eqref{equ:L4 young}. This proves \eqref{equ:uh cross Hh dt theta}.

Finally, if the triangulation is quasi-uniform, then \eqref{equ:semidisc-est L infty} holds. Using \eqref{equ:cross split}, H\"older's and Young's inequalities, we obtain
\begin{align*}
	\big| \inpro{\bff{u}_h \times \bff{H}_h- \bff{u} \times \bff{H}}{\bff{\zeta}} \big| 
	&\leq
	\left( \norm{\bff{\theta}+\bff{\rho}}{\bb{L}^2} \norm{\bff{H}}{\bb{L}^\infty}
	+ \norm{\bff{u}_h}{\bb{L}^\infty} \norm{\bff{\xi}+\bff{\eta}}{\bb{L}^2} \right) \norm{\bff{\zeta}}{\bb{L}^2}
	\\
	&\lesssim
	h^{2(r+1)} 
	+ 
	\norm{\bff{\theta}}{\bb{L}^2}^2 
	+
	\norm{\bff{\xi}}{\bb{L}^2}^2
	+ 
	\epsilon \norm{\bff{\zeta}}{\bb{L}^2}^2,
\end{align*}
thus proving \eqref{equ:uh cross Hh Delta}. Similarly, noting \eqref{equ:H1 stab proj}, we obtain
\begin{align*}
	\big|\inpro{\nabla\Pi_h\big(\bff{u}_h \times \bff{H}_h- \bff{u} \times \bff{H}\big)}{\bff{\zeta}} \big|
	&\leq
	\big( \norm{\nabla \bff{\theta}+\nabla \bff{\rho}}{\bb{L}^2} \norm{\bff{H}}{\bb{L}^\infty}
	+ \norm{\bff{\theta}+\bff{\rho}}{\bb{L}^4} \norm{\nabla \bff{H}}{\bb{L}^4}
	\\
	&\quad
	+ \norm{\nabla \bff{u}_h}{\bb{L}^4} \norm{\bff{\xi}+\bff{\eta}}{\bb{L}^4}
	+ \norm{\bff{u}_h}{\bb{L}^\infty} \norm{\nabla \bff{\xi}+\nabla \bff{\eta}}{\bb{L}^2} \big) \norm{\bff{\zeta}}{\bb{L}^2}
	\\
	&\lesssim
	h^{2r}
	+
	\norm{\bff{\theta}}{\bb{H}^1}^2
	+
	\norm{\bff{\xi}}{\bb{H}^1}^2
	+
	\epsilon \norm{\bff{\zeta}}{\bb{L}^2}^2,
\end{align*}
where in the last step we used \eqref{equ:disc lapl L6}, Proposition~\ref{pro:semidisc est3}, and \eqref{equ:Ritz ineq}, completing the proof of the lemma.
\end{proof}

\begin{lemma}\label{lem:nonlinear est}
Let $\epsilon>0$ and initial data $\bff{u}_0$ be given. Let $(\bff{u},\bff{H})$ be the solution of \eqref{equ:weakform} satisfying \eqref{equ:ass 2 u}. Then we have
	\begin{align}
		\label{equ:uh2uh theta}
		\big| \inpro{|\bff{u}_h|^2 \bff{u}_h - |\bff{u}|^2 \bff{u}}{\bff{\theta}} \big|
		&\lesssim
		h^{2(r+1)}
		+
		\norm{\bff{\theta}}{\bb{L}^2}^2 
		+ 
		\epsilon \norm{\nabla \bff{\theta}}{\bb{L}^2}^2,
		\\
		\label{equ:uh2uh xi}
		\big| \inpro{|\bff{u}_h|^2 \bff{u}_h - |\bff{u}|^2 \bff{u}}{\bff{\xi}} \big|
		&\lesssim
		h^{2(r+1)}
		+
		\big( 1+\norm{\bff{\xi}}{\bb{H}^1}^2 \big) \norm{\bff{\theta}}{\bb{L}^2}^2 
		+ 
		\epsilon \norm{\nabla \bff{\theta}}{\bb{L}^2}^2
		+ 
		\epsilon \norm{\bff{\xi}}{\bb{L}^2}^2,
		\\
		\label{equ:uh2uh dt theta}
		\big| \inpro{|\bff{u}_h|^2 \bff{u}_h - |\bff{u}|^2 \bff{u}}{\partial_t \bff{\theta}} \big|
		&\lesssim
		h^{2(r+1)}
		+
		\big( 1+\norm{\partial_t \bff{\theta}}{\bb{L}^2}^2 \big) \norm{\bff{\theta}}{\bb{H}^1}^2  
		+ 
		\epsilon \norm{\partial_t \bff{\theta}}{\bb{L}^2}^2,
		\\
		\label{equ:dt uh2uh xi}
		\big| \inpro{\partial_t(|\bff{u}_h|^2 \bff{u}_h - |\bff{u}|^2 \bff{u})}{\bff{\xi}} \big|
		&\lesssim
		h^{2(r+1)}
		+
		\big( 1 + \norm{\bff{\xi}}{\bb{H}^1}^2 + \norm{\partial_t \bff{u}_h}{\bb{L}^2}^2 \big) \norm{\bff{\theta}}{\bb{H}^1}^2 
		+ 
		\norm{\bff{\xi}}{\bb{L}^2}^2
		+
		\epsilon \norm{\partial_t \bff{\theta}}{\bb{L}^2}^2.
	\end{align}
Furthermore, if the triangulation $\mathcal{T}_h$ is globally quasi-uniform, then for any $\bff{\zeta}\in \bb{V}_h$,
	\begin{align}
		\label{equ:uh2uh Delta theta}
		\big| \inpro{|\bff{u}_h|^2 \bff{u}_h - |\bff{u}|^2 \bff{u}}{\bff{\zeta}} \big|
		&\lesssim
		h^{2(r+1)}
		+
		\norm{\bff{\theta}}{\bb{L}^2}^2 
		+ 
		\epsilon \norm{\bff{\zeta}}{\bb{L}^2}^2,
		\\
		\label{equ:dt uh2uh}
		\big| \inpro{\partial_t \left(|\bff{u}_h|^2 \bff{u}_h - |\bff{u}|^2 \bff{u}\right)}{\bff{\zeta}} \big|
		&\lesssim
		h^{2(r+1)} + \norm{\bff{\theta}}{\bb{L}^2}^2
		+ \norm{\partial_t \bff{\theta}}{\bb{L}^2}^2
		+
		\epsilon \norm{\bff{\zeta}}{\bb{H}^1}^2.
	\end{align}
In the above estimates, all constants are independent of $h$.
\end{lemma}

\begin{proof}
First, we will prove \eqref{equ:uh2uh theta}. Note that
	\begin{align}\label{equ:simple uhu}
		|\bff{u}_h|^2 \bff{u}_h - |\bff{u}|^2 \bff{u}
		=
		|\bff{u}_h|^2 (\bff{\theta}+\bff{\rho}) - \big( (\bff{\theta}+\bff{\rho})\cdot (\bff{u}_h+\bff{u}) \big) \bff{u}.
	\end{align}
Therefore, noting \eqref{equ:L4 young} and \eqref{equ:semidisc-est1}, by H\"{o}lder's inequality and Sobolev embedding $\bb{H}^1 \hookrightarrow \bb{L}^6$ we have
	\begin{align*}
		\big| \inpro{|\bff{u}_h|^2 \bff{u}_h - |\bff{u}|^2 \bff{u}}{\bff{\theta}} \big|
		&\leq
		\norm{\bff{u}_h}{\bb{L}^4}^2 \norm{\bff{\theta}}{\bb{L}^4}^2
		+
		\norm{\bff{u}_h}{\bb{L}^6}^2 \norm{\bff{\rho}}{\bb{L}^2} \norm{\bff{\theta}}{\bb{L}^6}
		+
		\norm{\bff{u}}{\bb{L}^\infty}^2
		\big(
		\norm{\bff{\theta}}{\bb{L}^2}^2
		+
		\norm{\bff{\theta}}{\bb{L}^2}
		\norm{\bff{\rho}}{\bb{L}^2}
		\big)
		\\
		&\quad
		+
		\norm{\bff{u}_h}{\bb{L}^4} 
		\norm{\bff{u}}{\bb{L}^\infty}
		\big(
		\norm{\bff{\theta}}{\bb{L}^4}
		\norm{\bff{\theta}}{\bb{L}^2}
		+
		\norm{\bff{\theta}}{\bb{L}^4}
		\norm{\bff{\rho}}{\bb{L}^2}
		\big)
		\\
		&\lesssim
		h^{2(r+1)}
		+
	 	\norm{\bff{\theta}}{\bb{L}^2}^2 
		+ 
		\epsilon \norm{\nabla \bff{\theta}}{\bb{L}^2}^2
	\end{align*}
for any $\epsilon>0$, where we used Young's inequality and \eqref{equ:Ritz ineq} in the last step. Next, we will prove \eqref{equ:uh2uh xi}. In this case, using \eqref{equ:nonlinear Rh u}, H\"{o}lder's and Young's inequality, and \eqref{equ:L4 young}, we have
	\begin{align*}
		\nonumber
		\big| \inpro{|\bff{u}_h|^2 \bff{u}_h - |\bff{u}|^2 \bff{u}}{\bff{\xi}} \big|
		&\lesssim
		\norm{\bff{\theta}}{\bb{L}^2}
		\norm{\bff{\theta}}{\bb{L}^6}
		\norm{\bff{u}_h}{\bb{L}^6}
		\norm{\bff{\xi}}{\bb{L}^6}
		+
		\norm{\bff{\theta}}{\bb{L}^4}
		\norm{R_h \bff{u}}{\bb{L}^\infty}
		\norm{\bff{u}_h}{\bb{L}^4}
		\norm{\bff{\xi}}{\bb{L}^2}
		\\
		\nonumber
		&\quad
		+
		\norm{R_h \bff{u}}{\bb{L}^\infty}^2
		\norm{\bff{\theta}+\bff{\rho}}{\bb{L}^2}
		\norm{\bff{\xi}}{\bb{L}^2}
		+
		\norm{\bff{\rho}}{\bb{L}^2}
		\norm{R_h \bff{u}+ \bff{u}}{\bb{L}^\infty}
		\norm{\bff{u}}{\bb{L}^\infty}
		\norm{\bff{\xi}}{\bb{L}^2}
		\\
		\nonumber
		&\lesssim
		\norm{\bff{\xi}}{\bb{L}^6}^2 \norm{\bff{\theta}}{\bb{L}^2}^2 
		+
		\epsilon \norm{\bff{\theta}}{\bb{L}^6}^2
		+
		\norm{\bff{\theta}}{\bb{L}^2}^2
		+
		\epsilon \norm{\nabla \bff{\theta}}{\bb{L}^2}^2
		+
		\epsilon \norm{\bff{\xi}}{\bb{L}^2}^2
		+
		\norm{\bff{\rho}}{\bb{L}^2}^2
		+
		\epsilon \norm{\bff{\xi}}{\bb{L}^2}^2
		\\
		&\lesssim
		h^{2(r+1)}
		+
		\big( 1+\norm{\bff{\xi}}{\bb{H}^1}^2 \big) \norm{\bff{\theta}}{\bb{L}^2}^2 
		+ 
		\epsilon \norm{\nabla \bff{\theta}}{\bb{L}^2}^2
		+ 
		\epsilon \norm{\bff{\xi}}{\bb{L}^2}^2
	\end{align*}
for any $\epsilon>0$, where in the last step we used Sobolev embedding $\bb{H}^1 \hookrightarrow \bb{L}^6$ and \eqref{equ:Ritz ineq}.
Similarly, to prove \eqref{equ:uh2uh dt theta}, we use \eqref{equ:nonlinear Rh u}, H\"{o}lder's and Young's inequality to obtain
	\begin{align*}
		\big| \inpro{|\bff{u}_h|^2 \bff{u}_h - |\bff{u}|^2 \bff{u}}{\partial_t \bff{\theta}} \big|
		&\lesssim
		\norm{\bff{\theta}}{\bb{L}^6}^2
		\norm{\bff{u}_h}{\bb{L}^6}
		\norm{\partial_t \bff{\theta}}{\bb{L}^2}
		+
		\norm{\bff{\theta}}{\bb{L}^6}
		\norm{R_h \bff{u}}{\bb{L}^6}
		\norm{\bff{u}_h}{\bb{L}^6}
		\norm{\partial_t \bff{\theta}}{\bb{L}^2}
		\\
		&\quad
		+
		\norm{R_h \bff{u}}{\bb{L}^\infty}^2
		\norm{\bff{\theta}+\bff{\rho}}{\bb{L}^2}
		\norm{\partial_t \bff{\theta}}{\bb{L}^2}
		+
		\norm{\bff{\rho}}{\bb{L}^2}
		\norm{R_h \bff{u}+ \bff{u}}{\bb{L}^\infty}
		\norm{\bff{u}}{\bb{L}^\infty}
		\norm{\partial_t \bff{\theta}}{\bb{L}^2}
		\\
		&\lesssim
		\norm{\bff{\theta}}{\bb{H}^1}^2 \norm{\partial_t \bff{\theta}}{\bb{L}^2}^2
		+
		\norm{\bff{\theta}}{\bb{H}^1}^2
		+
		\epsilon \norm{\partial_t \bff{\theta}}{\bb{L}^2}
		+
		\norm{\bff{\rho}}{\bb{L}^2}^2
		+
		\epsilon \norm{\partial_t \bff{\theta}}{\bb{L}^2}
		\\
		&\lesssim
		h^{2(r+1)}
		+
		\big( 1+\norm{\partial_t \bff{\theta}}{\bb{L}^2}^2 \big) \norm{\bff{\theta}}{\bb{H}^1}^2  
		+ 
		\epsilon \norm{\partial_t \bff{\theta}}{\bb{L}^2}^2
	\end{align*}
for any $\epsilon>0$, as required. Finally, differentiating \eqref{equ:nonlinear Rh u} with respect to $t$ yields
\begin{align}\label{equ:dt uh2uh T1 to 8}
	\nonumber
	&\big| \inpro{\partial_t(|\bff{u}_h|^2 \bff{u}_h - |\bff{u}|^2 \bff{u})}{\bff{\xi}} \big|
	\\
	\nonumber
	&\lesssim
	\big| \inpro{\big( \partial_t\bff{\theta} \cdot (\bff{\theta}+ 2R_h \bff{u}) \big) \bff{u}_h}{\bff{\xi}} \big|
	+
	\big| \inpro{\big( \bff{\theta}\cdot (\partial_t \bff{\theta} + 2 R_h \partial_t \bff{u}) \big) \bff{u}_h}{\bff{\xi}} \big|
	+
	\big| \inpro{\big( \bff{\theta} \cdot (\bff{\theta} + R_h \bff{u}) \big) \partial_t \bff{u}_h}{\bff{\xi}} \big|
	\\
	\nonumber
	&\quad
	+
	\big| \inpro{2(R_h \bff{u} \cdot R_h \partial_t \bff{u}) (\bff{\theta}+\bff{\rho})}{\bff{\xi}} \big|
	+
	\big| \inpro{|R_h \bff{u}|^2 (\partial_t \bff{\theta}+\partial_t \bff{\rho})}{\bff{\xi}} \big|
	\\
	\nonumber
	&\quad
	+
	\big| \inpro{\big( \partial_t \bff{\rho}\cdot (R_h \bff{u}+\bff{u}) \big) \bff{u}}{\bff{\xi}} \big|
	+
	\big| \inpro{\big(\bff{\rho} \cdot (R_h \partial_t \bff{u} + \partial_t \bff{u}) \big) \bff{u}}{\bff{\xi}} \big|
	+
	\big| \inpro{\big(\bff{\rho} \cdot (R_h \bff{u} + \bff{u}) \big) \partial_t \bff{u}}{\bff{\xi}} \big|
	\\
	&=:
	|T_1| + |T_2| + \ldots + |T_8|.
\end{align}
We will now estimate each term in the last step. For the term $|T_1|$, noting $\bff{u}_h= \bff{\theta}+R_h \bff{u}$ and applying H\"older's inequality, we have
\begin{align*}
	|T_1| 
	&\lesssim
	\big| \inpro{( \partial_t\bff{\theta} \cdot \bff{\theta} ) \bff{u}_h}{\bff{\xi}} \big|
	+
	\big| \inpro{(\partial_t \bff{\theta} \cdot R_h \bff{u}) (\bff{\theta}+R_h \bff{u})}{\bff{\xi}} \big|
	\\
	&\leq
	\norm{\partial_t \bff{\theta}}{\bb{L}^2} \norm{\bff{\theta}}{\bb{L}^6} \norm{\bff{u}_h}{\bb{L}^6} \norm{\bff{\xi}}{\bb{L}^6}
	+
	\norm{\partial_t \bff{\theta}}{\bb{L}^2} \norm{R_h \bff{u}}{\bb{L}^\infty} \norm{\bff{\theta}}{\bb{L}^4} \norm{\bff{\xi}}{\bb{L}^4}
	+
	\norm{\partial_t \bff{\theta}}{\bb{L}^2} \norm{R_h \bff{u}}{\bb{L}^\infty}^2 \norm{\bff{\xi}}{\bb{L}^2}
	\\
	&\lesssim
	\norm{\bff{\xi}}{\bb{H}^1}^2 \norm{\bff{\theta}}{\bb{H}^1}^2
	+
	\norm{\bff{\xi}}{\bb{L}^2}^2
	+
	\epsilon \norm{\partial_t \bff{\theta}}{\bb{L}^2}^2
\end{align*}
for any $\epsilon>0$, where we used Young's inequality and Sobolev embedding in the last step. For the term $|T_2|$, by H\"older's and Young's inequality,
\begin{align*}
	|T_2|
	&\leq
	\norm{\bff{\theta}}{\bb{L}^6} \norm{\partial_t \bff{\theta}}{\bb{L}^2} \norm{\bff{u}_h}{\bb{L}^6} \norm{\bff{\xi}}{\bb{L}^6}
	+
	\norm{\bff{\theta}}{\bb{L}^4} \norm{R_h \partial_t \bff{u}}{\bb{L}^\infty} \norm{\bff{u}_h}{\bb{L}^4} \norm{\bff{\xi}}{\bb{L}^2}
	\\
	&\lesssim
	\norm{\bff{\xi}}{\bb{H}^1}^2 \norm{\bff{\theta}}{\bb{H}^1}^2 
	+ 
	\epsilon \norm{\partial_t \bff{\theta}}{\bb{L}^2}^2
	+
	\norm{\bff{\theta}}{\bb{H}^1}^2
	+
	\epsilon \norm{\bff{\xi}}{\bb{L}^2}^2,
\end{align*}	
where we used Sobolev embedding in the last step. For the term $|T_3|$, noting $\bff{u}_h= \bff{\theta}+R_h \bff{u}$ as done before, we obtain
\begin{align*}
	|T_3|
	&\leq
	\big| \inpro{( \bff{\theta} \cdot \bff{\theta}) \partial_t \bff{u}_h}{\bff{\xi}} \big|
	+
	\big| \inpro{( \bff{\theta} \cdot  R_h \bff{u} ) (\partial_t \bff{\theta}+ R_h \partial_t \bff{u})}{\bff{\xi}} \big|
	\\
	&\leq
	\norm{\bff{\theta}}{\bb{L}^6}^2 \norm{\partial_t \bff{u}_h}{\bb{L}^2} \norm{\bff{\xi}}{\bb{L}^6}
	+
	\norm{\bff{\theta}}{\bb{L}^4} \norm{R_h \bff{u}}{\bb{L}^\infty} \norm{\partial_t \bff{\theta}}{\bb{L}^2} \norm{\bff{\xi}}{\bb{L}^4}
	+
	\norm{\bff{\theta}}{\bb{L}^2} \norm{R_h \bff{u}}{\bb{L}^\infty} \norm{R_h \partial_t \bff{u}}{\bb{L}^\infty} \norm{\bff{\xi}}{\bb{L}^2}
	\\
	&\lesssim
	\norm{\partial_t \bff{u}_h}{\bb{L}^2}^2 \norm{\bff{\theta}}{\bb{H}^1}^2
	+
	\norm{\bff{\xi}}{\bb{H}^1}^2 \norm{\bff{\theta}}{\bb{H}^1}^2
	+
	\epsilon \norm{\partial_t \bff{\theta}}{\bb{L}^2}^2
	+
	\norm{\bff{\theta}}{\bb{L}^2}^2
	+
	\epsilon \norm{\bff{\xi}}{\bb{L}^2}^2.
\end{align*}
For the term $|T_4|$, by H\"older's and Young's inequalities, we have
\begin{align*}
	|T_4|
	&\leq
	\norm{R_h\bff{u}}{\bb{L}^\infty} \norm{R_h \partial_t \bff{u}}{\bb{L}^\infty} \norm{\bff{\theta}+\bff{\rho}}{\bb{L}^2} \norm{\bff{\xi}}{\bb{L}^2}
	\lesssim
	\norm{\bff{\theta}}{\bb{L}^2}^2
	+
	\norm{\bff{\rho}}{\bb{L}^2}^2
	+
	\epsilon \norm{\bff{\xi}}{\bb{L}^2}^2.
\end{align*}
Similarly for the next term,
\begin{align*}
	|T_5|
	&\leq
	\norm{R_h \bff{u}}{\bb{L}^\infty}^2 \norm{\partial_t \bff{\theta}+\partial_t \bff{\rho}}{\bb{L}^2} \norm{\bff{\xi}}{\bb{L}^2}
	\lesssim
	\norm{\partial_t \bff{\rho}}{\bb{L}^2}^2
	+
	\norm{\bff{\xi}}{\bb{L}^2}^2
	+
	\epsilon \norm{\partial_t \bff{\theta}}{\bb{L}^2}^2
\end{align*}
The terms $|T_6|, |T_7|$ and $|T_8|$ can be bounded in a similar way leading to
\begin{align*}
	|T_6|+|T_7|+|T_8|
	\leq
	\norm{\bff{\rho}}{\bb{L}^2}^2
	+
	\norm{\partial_t \bff{\rho}}{\bb{L}^2}^2
	+
	\epsilon \norm{\bff{\xi}}{\bb{L}^2}^2.
\end{align*}
Altogether, noting \eqref{equ:Ritz ineq}, we conclude the estimate \eqref{equ:dt uh2uh xi} from \eqref{equ:dt uh2uh T1 to 8}.

Now, suppose the triangulation is globally quasi-uniform. In this case, \eqref{equ:stab uh L infty} holds. Using \eqref{equ:semidisc-est L infty}, \eqref{equ:simple uhu}, and H\"older's and Young's inequality, we obtain
\begin{align*}
	\big| \inpro{|\bff{u}_h|^2 \bff{u}_h - |\bff{u}|^2 \bff{u}}{\Delta_h \bff{\theta}} \big|
	&\leq
	\left( \norm{\bff{u}_h}{\bb{L}^\infty}^2 \norm{\bff{\theta}+\bff{\rho}}{\bb{L}^2} 
	+
	\norm{\bff{\theta}+\bff{\rho}}{\bb{L}^2} \norm{\bff{u}_h+\bff{u}}{\bb{L}^\infty} \norm{\bff{u}}{\bb{L}^\infty} \right) \norm{\bff{\zeta}}{\bb{L}^2}
	\\
	&\lesssim
	h^{2(r+1)} + \norm{\bff{\theta}}{\bb{L}^2}^2 + \epsilon \norm{\bff{\zeta}}{\bb{L}^2}^2,
\end{align*}
proving \eqref{equ:uh2uh Delta theta}. Finally, writing
\begin{align*}
	\partial_t \left(|\bff{u}_h|^2 \bff{u}_h - |\bff{u}|^2 \bff{u}\right)
	&=
	\abs{\bff{u}_h}^2 (\partial_t \bff{u}_h-\partial_t \bff{u})
	+
	(\abs{\bff{u}_h}^2-\abs{\bff{u}}^2) \partial_t \bff{u}
	+
	2\left[\bff{u}_h \cdot (\partial_t \bff{u}_h-\partial_t \bff{u})\right] \bff{u}_h
	\\
	&\quad
	+
	2 \left[(\bff{u}_h-\bff{u})\cdot \partial_t \bff{u}\right] \bff{u}_h
	+
	2 \left(\bff{u}\cdot \partial_t \bff{u}\right) (\bff{u}_h-\bff{u}),
\end{align*}
then applying H\"older's and Young's inequality, we have
\begin{align*}
	&\big| \inpro{\partial_t \left(|\bff{u}_h|^2 \bff{u}_h - |\bff{u}|^2 \bff{u}\right)}{\bff{\zeta}} \big|
	\\
	&\leq
	\norm{\bff{u}_h}{\bb{L}^6}^2 \norm{\partial_t \bff{\theta}+\partial_t \bff{\rho}}{\bb{L}^2} \norm{\bff{\zeta}}{\bb{L}^6}
	+
	\norm{\bff{\theta}+\bff{\rho}}{\bb{L}^2} \norm{\bff{u}_h+\bff{u}}{\bb{L}^6} \norm{\partial_t \bff{u}}{\bb{L}^6} \norm{\bff{\zeta}}{\bb{L}^6}
	\\
	&\quad
	+
	2 \norm{\bff{u}_h}{\bb{L}^6} \norm{\partial_t \bff{\theta}+\partial_t \bff{\rho}}{\bb{L}^2} \norm{\bff{u}_h}{\bb{L}^6} \norm{\bff{\zeta}}{\bb{L}^6}
	+
	2 \norm{\bff{\theta}+\bff{\rho}}{\bb{L}^2} \norm{\partial_t \bff{u}}{\bb{L}^6} \norm{\bff{u}_h}{\bb{L}^6} \norm{\bff{\zeta}}{\bb{L}^6}
	\\
	&\quad
	+
	2 \norm{\bff{u}}{\bb{L}^6} \norm{\partial_t \bff{u}}{\bb{L}^6} \norm{\bff{\theta}+\bff{\rho}}{\bb{L}^2} \norm{\bff{\zeta}}{\bb{L}^6}
	\\
	&\lesssim
	h^{2(r+1)} + \norm{\bff{\theta}}{\bb{L}^2}^2
	+ \norm{\partial_t \bff{\theta}}{\bb{L}^2}^2
	+
	\epsilon \norm{\bff{\zeta}}{\bb{H}^1}^2,
\end{align*}
thus proving \eqref{equ:dt uh2uh}. This completes the proof of the lemma.
\end{proof}

We will prove that the semi-discrete scheme converges at an optimal rate. For simplicity of presentation, we will assume that $\bff{u}_{0,h}=R_h\bff{u}(0)$ is the approximation in $\bb{V}_h$ of the initial data, so that $\bff{\theta}(0)=0$. In this case, subtracting the second equation in \eqref{equ:weakform} from that in \eqref{equ:weaksemidisc} yields
\begin{align*}
	\inpro{\bff{\xi}+\bff{\eta}}{\bff{\phi}}
	=
	-\inpro{\nabla \bff{\theta}}{\nabla \bff{\phi}}
	+
	\kappa\mu \inpro{\bff{\theta}+\bff{\rho}}{\bff{\phi}}
	-
	\kappa \inpro{|\bff{u}_h|^2 \bff{u}_h - |\bff{u}|^2 \bff{u}}{\bff{\phi}}.
\end{align*}
Therefore, at $t=0$, noting \eqref{equ:simple uhu} and taking $\bff{\phi}=\bff{\xi}(0)$, we have
\begin{align*}
	\norm{\bff{\xi}(0)}{\bb{L}^2}^2
	&=
	- 
	\inpro{\bff{\eta}(0)}{\bff{\xi}(0)}
	+
	\kappa \mu \inpro{\bff{\rho}(0)}{\bff{\xi}(0)}
	-
	\kappa \inpro{|R_h\bff{u}(0)|^2 \bff{\rho}(0) - \big(\bff{\rho}(0)\cdot (R_h\bff{u}(0)+\bff{u}(0)) \big) \bff{u}(0)}{\bff{\xi}(0)}
	\\
	&\lesssim
	\norm{\bff{\eta}(0)}{\bb{L}^2}^2 + \norm{\bff{\rho}(0)}{\bb{L}^2}^2 
	+ 
	\epsilon \norm{\bff{\xi}(0)}{\bb{L}^2}^2
\end{align*}
for any $\epsilon>0$ by Young's inequality. Rearranging the above, then applying \eqref{equ:Ritz ineq} and \eqref{equ:Ritz ineq}, we obtain
\begin{align}\label{equ:assum xi}
	\norm{\bff{\xi}(0)}{\bb{L}^2}
	\lesssim
	h^{r+1}. 
\end{align}

Next, we prove some bounds for $\bff{\theta}(t)$ and $\bff{\xi}(t)$ for $t\in (0,T)$. In particular, \eqref{equ:semidisc theta xi} shows a superconvergence estimate for $\bff{\theta}$, which implies the corresponding estimate~\eqref{equ:semidisc theta xi infty} in the $\bb{L}^\infty$ norm under an additional assumption of global quasi-uniformity of the triangulation.

\begin{proposition}\label{pro:semidisc theta xi}
	For any $t\in (0,T)$,
\begin{align}\label{equ:semidisc theta xi}
	\norm{\bff{\theta}(t)}{\bb{L}^2}^2
	+
	\norm{\nabla \bff{\theta}(t)}{\bb{L}^2}^2
	+
	\norm{\bff{\xi}(t)}{\bb{L}^2}^2
	+
	\int_0^t \norm{\nabla \bff{\xi}(s)}{\bb{L}^2}^2 \ds
	+
	\int_0^t \norm{\partial_t \bff{\theta}(s)}{\bb{L}^2}^2 \ds
	&\leq Ch^{2(r+1)}.
\end{align}
Moreover, if the triangulation $\mathcal{T}_h$ is globally quasi-uniform then
\begin{align}\label{equ:semidisc theta xi infty}
	\norm{\bff{\theta}(t)}{\bb{L}^\infty}^2
	+
	\norm{\Delta_h \bff{\theta}(t)}{\bb{L}^2}^2
	+
	\int_0^t \norm{\bff{\xi}(s)}{\bb{L}^\infty}^2 \ds
	+
	\int_0^t \norm{\Delta_h \bff{\xi}(s)}{\bb{L}^2}^2 \ds
	&\leq 
	Ch^{2(r+1)},
	\\
	\label{equ:semidisc xi L2}
	\norm{\nabla \bff{\xi}(t)}{\bb{L}^2}^2
	+
	\int_0^t \norm{\nabla \Delta_h \bff{\xi}(s)}{\bb{L}^2}^2 \ds 
	&\leq
	Ch^{2r}.
\end{align}
The constant $C$ depends on the coefficients of the equation, $|\mathscr{D}|$, $T$, and $K_0$ (as defined in \eqref{equ:ass 2 u}), but is independent of $h$.
%\begin{align}\label{equ:semidisc theta xi infty}
%	\norm{\bff{\theta}(t)}{\bb{L}^\infty}^2
%	+
%	\int_0^t \norm{\bff{\xi}(t)}{\bb{L}^\infty}^2
%	\leq 
%	\begin{cases}
%		Ch^{2(r+1)}
%		\quad &\text{if $d=1$,}
%		\\[1ex]
%		Ch^{2(r+1)} \abs{\ln h} \quad &\text{if $d=2$,}
%		\\[1ex]
%		Ch^{2r+1} \quad &\text{if $d=3$.}
%	\end{cases}
%\end{align}
\end{proposition}

\begin{proof}
%In view of \eqref{equ:Ritz ineq} and \eqref{equ:Ritz ineq}, it suffices to prove the estimates for $\bff{\theta}(t)$ and $\bff{\xi}(t)$.
%\\
%First, we will prove \eqref{equ:semidisc uh-u}.
%To this end, 
Subtracting \eqref{equ:weakform} from \eqref{equ:weaksemidisc}, using \eqref{equ:theta rho}, \eqref{equ:xi eta} (and noting the definition of Ritz projection), we obtain for all $\bff{\chi}$, $\bff{\phi}\in \bb{V}_h$,
\begin{align}\label{equ:dt theta + dt rho}
	\inpro{\partial_t \bff{\theta}+ \partial_t \bff{\rho}}{\bff{\chi}}
	&=
	\lambda_r \inpro{\bff{\xi}+\bff{\eta}}{\bff{\chi}}
	+
	\lambda_e \inpro{\nabla \bff{\xi}}{\nabla \bff{\chi}}
	-
	\gamma \inpro{\bff{u}_h \times \bff{H}_h - \bff{u} \times \bff{H}}{\bff{\chi}}
\end{align}
and
\begin{align}
	\label{equ:xi + eta}
	\inpro{\bff{\xi}+\bff{\eta}}{\bff{\phi}}
	=
	-\inpro{\nabla \bff{\theta}}{\nabla \bff{\phi}}
	+
	\kappa\mu \inpro{\bff{\theta}+\bff{\rho}}{\bff{\phi}}
	-
	\kappa \inpro{|\bff{u}_h|^2 \bff{u}_h - |\bff{u}|^2 \bff{u}}{\bff{\phi}}
	-
	\beta \inpro{\bff{e}(\bff{e}\cdot \bff{\theta})}{\bff{\phi}}
\end{align}
Taking $\bff{\chi}= \bff{\theta}$ in \eqref{equ:dt theta + dt rho}, we have
\begin{align}\label{equ:dt theta chi theta}
	\frac{1}{2} \ddt \norm{\bff{\theta}}{\bb{L}^2}^2
	+
	\inpro{\partial_t \bff{\rho}}{\bff{\theta}}
	&=
	\lambda_r \inpro{\bff{\xi}+\bff{\eta}}{\bff{\theta}}
	+
	\lambda_e \inpro{\nabla \bff{\xi}}{\nabla \bff{\theta}}
	-
	\gamma \inpro{\bff{u}_h \times \bff{H}_h - \bff{u} \times \bff{H}}{\bff{\theta}}.
\end{align}
Taking $\bff{\phi}= \lambda_r \bff{\theta}$ in \eqref{equ:xi + eta}, we obtain
\begin{align}\label{equ:xi+eta theta}
	\lambda_r \inpro{\bff{\xi}+\bff{\eta}}{\bff{\theta}}
	=
	&-\lambda_r \norm{\nabla \bff{\theta}}{\bb{L}^2}^2
	+
	\kappa\mu \lambda_r \norm{\bff{\theta}}{\bb{L}^2}^2
	-
	\beta \lambda_r \norm{\bff{e}\cdot\bff{\theta}}{L^2}^2
	+
	\kappa \mu\lambda_r \inpro{\bff{\rho}}{\bff{\theta}}
	\nonumber \\
	&\quad
	-
	\kappa \lambda_r \inpro{|\bff{u}_h|^2 \bff{u}_h - |\bff{u}|^2 \bff{u}}{\bff{\theta}}.
\end{align}
Next, taking $\bff{\phi}= \lambda_e \bff{\xi}$, we have
\begin{align}\label{equ:xi phi=xi}
	\lambda_e \norm{\bff{\xi}}{\bb{L}^2}^2
	+
	\lambda_e \inpro{\bff{\eta}}{\bff{\xi}}
	&=
	-
	\lambda_e \inpro{\nabla \bff{\theta}}{\nabla \bff{\xi}}
	+
	\kappa\mu \lambda_e \inpro{\bff{\theta}+\bff{\rho}}{\bff{\xi}}
	-
	\kappa \lambda_e \inpro{|\bff{u}_h|^2 \bff{u}_h - |\bff{u}|^2 \bff{u}}{\bff{\xi}}
	\nonumber \\
	&\quad
	-
	\beta\lambda_e \inpro{\bff{e}(\bff{e}\cdot\bff{\theta})}{\bff{\xi}}.
\end{align}
Substituting \eqref{equ:xi+eta theta} into \eqref{equ:dt theta chi theta}, then adding the result to \eqref{equ:xi phi=xi} gives
\begin{align*}
	&\frac{1}{2} \ddt \norm{\bff{\theta}}{\bb{L}^2}^2 
	+
	\lambda_r \norm{\nabla \bff{\theta}}{\bb{L}^2}^2 
	+
	\lambda_e \norm{\bff{\xi}}{\bb{L}^2}^2
	+
	\beta\lambda_r \norm{\bff{e}\cdot\bff{\theta}}{L^2}^2
	\\
	&=
	-
	\inpro{\partial_t \bff{\rho}}{\bff{\theta}}
	-
	\lambda_e \inpro{\bff{\eta}}{\bff{\xi}}
	+
	\kappa\mu \lambda_e \inpro{\bff{\theta}+\bff{\rho}}{\bff{\xi}}
	+
	\kappa \mu\lambda_r \norm{\bff{\theta}}{\bb{L}^2}^2
	+
	\kappa \mu\lambda_r \inpro{\bff{\rho}}{\bff{\theta}}
	-
	\beta\lambda_e \inpro{\bff{e}(\bff{e}\cdot\bff{\theta})}{\bff{\xi}}
	\\
	&\quad
	-
	\gamma \inpro{\bff{u}_h \times \bff{H}_h - \bff{u} \times \bff{H}}{\bff{\theta}}
	-
	\kappa \lambda_e \inpro{|\bff{u}_h|^2 \bff{u}_h - |\bff{u}|^2 \bff{u}}{\bff{\xi}}
	-
	\kappa \lambda_r \inpro{|\bff{u}_h|^2 \bff{u}_h - |\bff{u}|^2 \bff{u}}{\bff{\theta}}.
\end{align*}
Applying H\"{o}lder's and Young's inequalities, \eqref{equ:uh cross Hh theta}, \eqref{equ:uh2uh theta}, and \eqref{equ:uh2uh xi} to the above equation yields
\begin{align}\label{equ:theta nab theta xi}
	\nonumber
	&\frac{1}{2} \ddt \norm{\bff{\theta}}{\bb{L}^2}^2 
	+
	\lambda_r \norm{\nabla \bff{\theta}}{\bb{L}^2}^2 
	+
	\lambda_e \norm{\bff{\xi}}{\bb{L}^2}^2
	\\
	\nonumber
	&\leq
	\norm{\partial_t \bff{\rho}}{\bb{L}^2}
	\norm{\bff{\theta}}{\bb{L}^2}
	+
	\lambda_e \norm{\bff{\eta}}{\bb{L}^2}
	\norm{\bff{\xi}}{\bb{L}^2}
	+
	\kappa\mu\lambda_e \left( \norm{\bff{\theta}}{\bb{L}^2}
	\norm{\bff{\xi}}{\bb{L}^2}
	+
	\norm{\bff{\rho}}{\bb{L}^2}
	\norm{\bff{\xi}}{\bb{L}^2} \right)
	+
	\kappa\mu\lambda_r \left( \norm{\bff{\theta}}{\bb{L}^2}^2
	+
	\norm{\bff{\rho}}{\bb{L}^2}
	\norm{\bff{\theta}}{\bb{L}^2} \right)
	\\
	\nonumber
	&\quad
	+
	\beta\lambda_e \norm{\bff{\theta}}{\bb{L}^2} \norm{\bff{\xi}}{\bb{L}^2}
	+
	\gamma \big| \inpro{\bff{u}_h \times \bff{H}_h - \bff{u} \times \bff{H}}{\bff{\theta}} \big| 
	+
	\kappa\lambda_e \big| \inpro{|\bff{u}_h|^2 \bff{u}_h - |\bff{u}|^2 \bff{u}}{\bff{\xi}} \big|
	\\
	\nonumber
	&\quad
	+
	\kappa\lambda_r \big| \inpro{|\bff{u}_h|^2 \bff{u}_h - |\bff{u}|^2 \bff{u}}{\bff{\theta}} \big| 
	\\
	\nonumber
	&\lesssim
	\norm{\partial_t \bff{\rho}}{\bb{L}^2}^2
	+
	\norm{\bff{\theta}}{\bb{L}^2}^2
	+
	\norm{\bff{\eta}}{\bb{L}^2}^2
	+
	\norm{\bff{\rho}}{\bb{L}^2}^2
	+
	\epsilon \norm{\bff{\xi}}{\bb{L}^2}^2
	\\
	\nonumber
	&\quad
	+
	h^{2(r+1)} 
	+ 
	\norm{\bff{\theta}}{\bb{L}^2}^2 
	+ 
	\epsilon \norm{\nabla \bff{\theta}}{\bb{L}^2}^2
	+
	\epsilon \norm{\bff{\xi}}{\bb{L}^2}^2
	+
	h^{2(r+1)}
	+
	\norm{\bff{\theta}}{\bb{L}^2}^2 
	+ 
	\epsilon \norm{\nabla \bff{\theta}}{\bb{L}^2}^2
	\\
	\nonumber
	&\quad
	+
	h^{2(r+1)}
	+
	\big( 1+\norm{\bff{\xi}}{\bb{H}^1}^2 \big) \norm{\bff{\theta}}{\bb{L}^2}^2 
	+ 
	\epsilon \norm{\nabla \bff{\theta}}{\bb{L}^2}^2
	+ 
	\epsilon \norm{\bff{\xi}}{\bb{L}^2}^2
	\\
	& 
	\lesssim
	h^{2(r+1)}
	+
	\big( 1+\norm{\bff{\xi}}{\bb{H}^1}^2 \big) \norm{\bff{\theta}}{\bb{L}^2}^2 
	+ 
	\epsilon \norm{\nabla \bff{\theta}}{\bb{L}^2}^2
	+ 
	\epsilon \norm{\bff{\xi}}{\bb{L}^2}^2
\end{align}
for any $\epsilon>0$, where the constant is independent of $h$, $t$, and $T$. Choosing $\epsilon>0$ sufficiently small and integrating over $(0,t)$, we obtain
\begin{align*}
	&\norm{\bff{\theta}}{\bb{L}^2}^2 
	+
	\int_0^t \norm{\nabla \bff{\theta}}{\bb{L}^2}^2 \ds 
	+
	\int_0^t \norm{\bff{\xi}}{\bb{L}^2}^2 \ds
	\lesssim
	\norm{\bff{\theta}(0)}{\bb{L}^2}^2
	+
	h^{2(r+1)}
	+
	\int_0^t \big( 1+\norm{\bff{\xi}}{\bb{H}^1}^2 \big) \norm{\bff{\theta}}{\bb{L}^2}^2 \ds.
\end{align*}
Note that by Proposition \ref{pro:semidisc est1} and \eqref{equ:Ritz stab infty},
\[
	\int_0^t 1+\norm{\bff{\xi}(s)}{\bb{H}^1}^2 \ds
	\lesssim
	1 + \int_0^t \big( \norm{\bff{H}_h(s)}{\bb{H}^1} + \norm{R_h \bff{H}_h(s)}{\bb{H}^1} \big) \ds
	\lesssim
	1,
\]
and so Gronwall's inequality applied to \eqref{equ:theta nab theta xi} yields
\begin{align}\label{equ:theta L2}
	\norm{\bff{\theta}(t)}{\bb{L}^2}^2
	+
	\int_0^t \norm{\nabla \bff{\theta}(s)}{\bb{L}^2}^2 \ds 
	+
	\int_0^t \norm{\bff{\xi}(s)}{\bb{L}^2}^2 \ds
	\leq
	Ch^{2(r+1)}.
\end{align}

%Finally, using the inverse estimates (Theorem \ref{the:inverse}),
%\begin{align*}
%	\norm{\nabla(\bff{u}_h(t)-\bff{u}(t))}{\bb{L}^2}
%	&\leq
%	\norm{\nabla \big( \bff{u}_h(t)- \mathcal{I}_h(\bff{u}(t)) \big)}{\bb{L}^2}
%	+
%	\norm{\nabla \big( \mathcal{I}_h(\bff{u}(t))- \bff{u}(t) \big)}{\bb{L}^2}
%	\\
%	&\leq
%	Ch^{-1} \norm{\bff{u}_h(t)- \mathcal{I}_h(\bff{u}(t))}{\bb{L}^2}
%	+
%	\norm{\nabla \big(\mathcal{I}_h(\bff{u}(t))- \bff{u}(t) \big)}{\bb{L}^2}
%	\\
%	&\leq
%	Ch^r,
%\end{align*}
%where we used \eqref{equ:interp approx} in the last step.

Next, differentiating \eqref{equ:xi + eta} with respect to $t$, then taking $\bff{\phi}= \lambda_e \bff{\xi}$ yields
\begin{align}\label{equ:difft xi phi=xi}
	\nonumber
	\frac{\lambda_e}{2} \ddt \norm{\bff{\xi}}{\bb{L}^2}^2
	+
	\lambda_e \inpro{\partial_t \bff{\eta}}{\bff{\xi}}
	&=
	-
	\lambda_e \inpro{\nabla \partial_t \bff{\theta}}{\nabla \bff{\xi}}
	+
	\kappa\mu \lambda_e \inpro{\partial_t \bff{\theta}+ \partial_t\bff{\rho}}{\bff{\xi}}
	\\
	&\quad
	-
	\kappa \lambda_e \inpro{\partial_t(|\bff{u}_h|^2 \bff{u}_h - |\bff{u}|^2 \bff{u})}{\bff{\xi}}
	-
	\beta\lambda_e \inpro{\bff{e}(\bff{e}\cdot\partial_t \bff{\theta})}{\bff{\xi}}.
\end{align}
Taking $\bff{\chi}=\kappa\mu\lambda_e \bff{\xi}$ and rearranging the terms, we have
\begin{align}\label{equ:dt theta xi}
	\nonumber
	\kappa\mu\lambda_e \lambda_r \norm{\bff{\xi}}{\bb{L}^2}^2
	+
	\kappa\mu \lambda_e^2 \norm{\nabla \bff{\xi}}{\bb{L}^2}^2
	&=
	-
	\kappa\mu\lambda_e \inpro{\partial_t \bff{\theta}+\partial_t \bff{\rho}}{\bff{\xi}}
	-
	\kappa\mu\lambda_e \lambda_r \inpro{\bff{\eta}}{\bff{\xi}}
	\\
	&\quad
	+
	\kappa\mu\lambda_e \gamma \inpro{\bff{u}_h \times \bff{H}_h- \bff{u}\times \bff{H}}{\bff{\xi}}.
\end{align}
Adding \eqref{equ:difft xi phi=xi} and \eqref{equ:dt theta xi} yields
\begin{align}\label{equ:add xi dt}
	\nonumber
	&\frac{\lambda_e}{2} \ddt \norm{\bff{\xi}}{\bb{L}^2}^2
	+
	\kappa\mu\lambda_e\lambda_r \norm{\bff{\xi}}{\bb{L}^2}^2
	+
	\kappa\mu\lambda_e^2 \norm{\nabla \bff{\xi}}{\bb{L}^2}^2
	\\
	\nonumber
	&=
	-
	\lambda_e \inpro{\nabla \partial_t \bff{\theta}}{\nabla \bff{\xi}}
	- 
	\lambda_e \inpro{\partial_t \bff{\eta}}{\bff{\xi}}
	-
	\kappa\mu\lambda_e\lambda_r \inpro{\bff{\eta}}{\bff{\xi}}
	-
	\kappa\mu\lambda_e \gamma \inpro{\bff{u}_h \times \bff{H}_h- \bff{u}\times \bff{H}}{\bff{\xi}}
	\\
	&\quad
	-
	\kappa \lambda_e \inpro{\partial_t(|\bff{u}_h|^2 \bff{u}_h - |\bff{u}|^2 \bff{u})}{\bff{\xi}}
	-
	\beta\lambda_e \inpro{\bff{e}(\bff{e}\cdot\partial_t \bff{\theta})}{\bff{\xi}}.
\end{align}
Furthermore, taking $\bff{\chi}= \partial_t \bff{\theta}$ in \eqref{equ:dt theta + dt rho} gives
\begin{align}\label{equ:dt theta chi dt theta}
	\norm{\partial_t \bff{\theta}}{\bb{L}^2}^2
	+
	\inpro{\partial_t \bff{\rho}}{\partial_t \bff{\theta}}
	&=
	\lambda_r \inpro{\bff{\xi}+\bff{\eta}}{\partial_t \bff{\theta}}
	+
	\lambda_e \inpro{\nabla \bff{\xi}}{\nabla \partial_t \bff{\theta}}
	-
	\gamma \inpro{\bff{u}_h \times \bff{H}_h - \bff{u} \times \bff{H}}{\partial_t \bff{\theta}}.
\end{align}
Taking $\bff{\phi}=\lambda_r \partial_t \bff{\theta}$ in \eqref{equ:xi + eta}, we obtain
\begin{align}\label{equ:xi+eta dt theta}
	\nonumber
	\lambda_r \inpro{\bff{\xi}+\bff{\eta}}{\partial_t \bff{\theta}}
	&=
	-\frac{\lambda_r}{2} \ddt \norm{\nabla \bff{\theta}}{\bb{L}^2}^2
	+
	\kappa\mu \lambda_r \inpro{\bff{\theta}}{\partial_t \bff{\theta}}
	+
	\kappa \mu\lambda_r \inpro{\bff{\rho}}{\partial_t \bff{\theta}}
	\\
	&\quad
	-
	\kappa \lambda_r \inpro{|\bff{u}_h|^2 \bff{u}_h - |\bff{u}|^2 \bff{u}}{\partial_t \bff{\theta}}
	-
	\beta\lambda_r \inpro{\bff{e}(\bff{e}\cdot \bff{\theta})}{\partial_t \bff{\theta}}.
\end{align}
Substituting \eqref{equ:xi+eta dt theta} into \eqref{equ:dt theta chi dt theta}, then adding the result to \eqref{equ:add xi dt} give
\begin{align*}
	&\frac{\lambda_r}{2} \ddt \norm{\nabla \bff{\theta}}{\bb{L}^2}^2 
	+
	\frac{\lambda_e}{2} \ddt \norm{\bff{\xi}}{\bb{L}^2}^2
	+
	\norm{\partial_t \bff{\theta}}{\bb{L}^2}^2
	+
	\kappa\mu\lambda_e\lambda_r \norm{\bff{\xi}}{\bb{L}^2}^2
	+
	\kappa\mu\lambda_e^2 \norm{\nabla \bff{\xi}}{\bb{L}^2}^2
	\\
	&=
	-
	\inpro{\partial_t \bff{\rho}}{\partial_t \bff{\theta}}
	-
	\lambda_e \inpro{\partial_t \bff{\eta}}{\bff{\xi}}
	+
	\kappa \mu\lambda_r \inpro{\bff{\theta}}{\partial_t \bff{\theta}}
	+
	\kappa \mu\lambda_r \inpro{\bff{\rho}}{\partial_t \bff{\theta}}
	-
	\kappa\mu\lambda_e\lambda_r \inpro{\bff{\eta}}{\bff{\xi}}
	\\
	&\quad 
	-
	\gamma \inpro{\bff{u}_h \times \bff{H}_h - \bff{u} \times \bff{H}}{\partial_t \bff{\theta}}
	+
	\kappa\mu\lambda_e \gamma \inpro{\bff{u}_h \times \bff{H}_h- \bff{u}\times \bff{H}}{\bff{\xi}}
	\\
	&\quad
	-
	\kappa \lambda_e \inpro{\partial_t(|\bff{u}_h|^2 \bff{u}_h - |\bff{u}|^2 \bff{u})}{\bff{\xi}}
	-
	\kappa \lambda_r \inpro{|\bff{u}_h|^2 \bff{u}_h - |\bff{u}|^2 \bff{u}}{\partial_t \bff{\theta}}
	-
	\beta\lambda_r \inpro{\bff{e}(\bff{e}\cdot \bff{\theta})}{\partial_t \bff{\theta}}.
\end{align*}
Applying H\"{o}lder's and Young's inequalities to bound the terms on the second line, then using \eqref{equ:uh cross Hh dt theta}, \eqref{equ:uh2uh dt theta} and \eqref{equ:dt uh2uh xi} for the remaining terms yield
\begin{align*}
	&\frac{\lambda_r}{2} \ddt \norm{\nabla \bff{\theta}}{\bb{L}^2}^2 
	+
	\frac{\lambda_e}{2} \ddt \norm{\bff{\xi}}{\bb{L}^2}^2
	+
	\norm{\partial_t \bff{\theta}}{\bb{L}^2}^2
	+
	\lambda_r \norm{\bff{\xi}}{\bb{L}^2}^2
	+
	\lambda_e \norm{\nabla \bff{\xi}}{\bb{L}^2}^2
	\\
	&\leq
	\norm{\partial_t \bff{\rho}}{\bb{L}^2}
	\norm{\partial_t \bff{\theta}}{\bb{L}^2}
	+
	\lambda_e \norm{\partial_t \bff{\eta}}{\bb{L}^2} \norm{\bff{\xi}}{\bb{L}^2}
	+
	\kappa\mu \lambda_r \norm{\bff{\theta}}{\bb{L}^2} \norm{\partial_t \bff{\theta}}{\bb{L}^2}
	+
	\kappa\mu \lambda_r \norm{\bff{\rho}}{\bb{L}^2}
	\norm{\partial_t \bff{\theta}}{\bb{L}^2}
	+
	\kappa\mu\lambda_e\lambda_r \norm{\bff{\eta}}{\bb{L}^2}
	\norm{\bff{\xi}}{\bb{L}^2}
	\\
	&\quad 
	+
	\gamma \big| \inpro{\bff{u}_h \times \bff{H}_h - \bff{u} \times \bff{H}}{\partial_t \bff{\theta}} \big|
	+
	\gamma \big| \inpro{\bff{u}_h \times \bff{H}_h- \bff{u}\times \bff{H}}{\bff{\xi}} \big|
	\\
	&\quad
	+
	\kappa \lambda_e \big| \inpro{\partial_t(|\bff{u}_h|^2 \bff{u}_h - |\bff{u}|^2 \bff{u})}{\bff{\xi}} \big|
	+
	\kappa \lambda_r \big| \inpro{|\bff{u}_h|^2 \bff{u}_h - |\bff{u}|^2 \bff{u}}{\partial_t \bff{\theta}} \big|
	+
	\beta\lambda_r \big|\inpro{\bff{e}(\bff{e}\cdot \bff{\theta})}{\partial_t \bff{\theta}} \big|
	\\
	&\lesssim
	\norm{\bff{\theta}}{\bb{L}^2}^2
	+
	\norm{\partial_t \bff{\rho}}{\bb{L}^2}^2
	+
	\norm{\partial_t \bff{\eta}}{\bb{L}^2}^2
	+
	\norm{\bff{\xi}}{\bb{L}^2}^2
	+
	\norm{\bff{\rho}}{\bb{L}^2}^2
	+
	\norm{\bff{\eta}}{\bb{L}^2}^2
	+
	\epsilon \norm{\partial_t \bff{\theta}}{\bb{L}^2}^2
	\\
	&\quad
	+
	h^{2(r+1)}
	+
	\norm{\bff{\xi}}{\bb{L}^2}^2
	+
	\big( \norm{\bff{\theta}}{\bb{L}^2}^2 + \epsilon \norm{\nabla \bff{\theta}}{\bb{L}^2}^2 \big) \norm{\bff{H}_h}{\bb{H}^1}^2
	+
	\epsilon \norm{\partial_t \bff{\theta}}{\bb{L}^2}^2
	\\
	&\quad
	+
	h^{2(r+1)} 
	+ 
	\norm{\bff{\xi}}{\bb{L}^2}^2 
	+ 
	\epsilon \norm{\nabla \bff{\xi}}{\bb{L}^2}^2
	+
	\epsilon \norm{\bff{\theta}}{\bb{L}^2}^2
	\\
	&\quad
	+
	h^{2(r+1)}
	+
	\big( 1 + \norm{\bff{\xi}}{\bb{H}^1}^2 + \norm{\partial_t \bff{u}_h}{\bb{L}^2}^2 \big) \norm{\bff{\theta}}{\bb{H}^1}^2 
	+ 
	\norm{\bff{\xi}}{\bb{L}^2}^2
	+
	\epsilon \norm{\partial_t \bff{\theta}}{\bb{L}^2}^2
	\\
	&\quad
	+
	h^{2(r+1)}
	+
	\big( 1+\norm{\partial_t \bff{\theta}}{\bb{L}^2}^2 \big) \norm{\bff{\theta}}{\bb{H}^1}^2  
	+ 
	\epsilon \norm{\partial_t \bff{\theta}}{\bb{L}^2}^2
	\\
	&\lesssim
	h^{2(r+1)}
	+
	\norm{\bff{\theta}}{\bb{L}^2}^2
	+
	\norm{\bff{\xi}}{\bb{L}^2}^2
	+
	\epsilon \norm{\partial_t \bff{\theta}}{\bb{L}^2}^2
	+
	\big( 1+ \norm{\bff{\xi}}{\bb{H}^1}^2 + \norm{\partial_t \bff{\theta}}{\bb{L}^2}^2
	+ \norm{\partial_t \bff{u}_h}{\bb{L}^2}^2 + \norm{\bff{H}_h}{\bb{H}^1}^2 \big)
	\norm{\bff{\theta}}{\bb{H}^1}^2
\end{align*}
for any $\epsilon>0$, where the constant is independent of $h$, $t$, and $T$. Choosing $\epsilon>0$ sufficiently small, then integrating over $(0,t)$ and using \eqref{equ:theta L2} (and noting \eqref{equ:assum xi}), we obtain
\begin{align*}
	\norm{\nabla \bff{\theta}(t)}{\bb{L}^2}^2
	+
	\norm{\bff{\xi}(t)}{\bb{L}^2}^2
	+
	\int_0^t \norm{\partial_t \bff{\theta}(s)}{\bb{L}^2}^2 \ds
	&+
	\int_0^t \norm{\nabla \bff{\xi}(s)}{\bb{L}^2}^2 \ds
	\\
	&\lesssim
	h^{2(r+1)}
	+
	\int_0^t \norm{\bff{\xi}(s)}{\bb{L}^2}^2 \ds 
	+
	\int_0^t \mathcal{B}(s) \norm{\bff{\theta}(s)}{\bb{H}^1}^2 \ds,
\end{align*}
where $\mathcal{B}(s):= 1+ \norm{\bff{\xi}(s)}{\bb{H}^1}^2 + \norm{\partial_t \bff{\theta}(s)}{\bb{L}^2}^2
+ \norm{\partial_t \bff{u}_h(s)}{\bb{L}^2}^2 + \norm{\bff{H}_h(s)}{\bb{H}^1}^2$. Note that by Proposition \ref{pro:semidisc est1}, Proposition \ref{pro:semidisc est2}, and inequality~\eqref{equ:Ritz stab infty},
\begin{align*}
	\int_0^t \mathcal{B}(s)\, \ds
	\leq
	\int_0^t \big( 1 + \norm{\bff{H}_h(s)}{\bb{H}^1}^2 + \norm{R_h\bff{H}(s)}{\bb{H}^1}^2 
	+ \norm{\partial_t \bff{u}_h(s)}{\bb{L}^2}^2 + \norm{R_h\partial_t \bff{u}(s)}{\bb{L}^2}^2 \big)\, \ds
	\lesssim 1.
\end{align*}
Therefore, by Gronwall's inequality,
\begin{align}\label{equ:nab theta xi L2 int}
	\norm{\nabla \bff{\theta}(t)}{\bb{L}^2}^2
	+
	\norm{\bff{\xi}(t)}{\bb{L}^2}^2
	+
	\int_0^t \norm{\partial_t \bff{\theta}(s)}{\bb{L}^2}^2 \ds
	+
	\int_0^t \norm{\nabla \bff{\xi}(s)}{\bb{L}^2}^2 \ds
	\leq
	Ch^{2(r+1)}.
\end{align}
Inequality \eqref{equ:semidisc theta xi} then follows from \eqref{equ:theta L2} and \eqref{equ:nab theta xi L2 int}. 

Suppose now that the triangulation is globally quasi-uniform. Taking $\bff{\phi}= \Delta_h \bff{\theta}$ in \eqref{equ:xi + eta}, rearranging the terms, then applying Young's inequality and \eqref{equ:uh2uh Delta theta} give
\begin{align*}
	\norm{\Delta_h \bff{\theta}}{\bb{L}^2}^2
	&=
	\inpro{\bff{\xi}+\bff{\eta}}{\Delta_h \bff{\theta}}
	-
	\kappa \mu\inpro{\bff{\theta}+\bff{\rho}}{\Delta_h \bff{\theta}}
	+
	\kappa \inpro{|\bff{u}_h|^2 \bff{u}_h - |\bff{u}|^2 \bff{u}}{\Delta_h \bff{\theta}}
	-
	\beta \inpro{\bff{e}(\bff{e}\cdot\bff{\theta}}{\Delta_h \bff{\theta}}
	\\
	&\lesssim
	\norm{\bff{\xi}}{\bb{L}^2}^2 +\norm{\bff{\theta}}{\bb{L}^2}^2
	+ \epsilon \norm{\Delta_h \bff{\theta}}{\bb{L}^2}^2
	\lesssim
	h^{2(r+1)} + \epsilon \norm{\Delta_h \bff{\theta}}{\bb{L}^2}^2,
\end{align*}
where in the last step we used \eqref{equ:semidisc theta xi} shown above. Choosing $\epsilon>0$ sufficiently small then yields
\begin{align}\label{equ:Delta theta L2}
	\norm{\Delta_h \bff{\theta}}{\bb{L}^2}^2 \lesssim h^{2(r+1)}.
\end{align}
Next, we take $\bff{\chi}=\Delta_h \bff{\xi}$ and apply the same argument (using \eqref{equ:semidisc theta xi} and \eqref{equ:uh cross Hh Delta} this time) to obtain
\begin{align}\label{equ:int Delta xi L2}
	\int_0^t \norm{\Delta_h \bff{\xi}(s)}{\bb{L}^2}^2 \ds \lesssim h^{2(r+1)}.
\end{align}
Inequality \eqref{equ:semidisc theta xi infty} then follows from \eqref{equ:Delta theta L2}, \eqref{equ:int Delta xi L2}, and \eqref{equ:disc lapl L infty}.

Finally, setting $\bff{\chi}=-\Delta_h^2 \bff{\xi}$ and applying \eqref{equ:disc laplacian} as necessary, we have
\begin{align}\label{equ:nab Delta xi}
	\nonumber
	\lambda_r \norm{\Delta_h \bff{\xi}}{\bb{L}^2}^2
	+
	\lambda_e \norm{\nabla \Delta_h \bff{\xi}}{\bb{L}^2}^2
	&=
	-\inpro{\partial_t \nabla \bff{\theta}}{\nabla \Delta_h \bff{\xi}}
	-
	\inpro{\nabla \Pi_h \partial_t \bff{\rho}}{\nabla \Delta_h \bff{\xi}}
	\\
	&\quad
	+
	\gamma \inpro{\nabla \Pi_h(\bff{u}_h\times \bff{H}_h - \bff{u}\times \bff{H})}{\nabla \Delta_h \bff{\xi}}.
\end{align}
Differentiating the second equation in \eqref{equ:weaksemidisc} with respect to $t$, then setting $\bff{\phi}=-\Delta_h \bff{\xi}$, we obtain
\begin{align}\label{equ:ddt nab xi}
	\nonumber
	\frac{1}{2}\ddt \norm{\nabla \bff{\xi}}{\bb{L}^2}^2
	&=
	\inpro{\partial_t \bff{\eta}}{\Delta_h \bff{\xi}}
	+
	\inpro{\partial_t \nabla \bff{\theta}}{\nabla \Delta_h \bff{\xi}}
	-
	\kappa \mu\inpro{\partial_t \bff{\theta}+\partial_t \bff{\rho}}{\Delta_h \bff{\xi}}
	\\
	&\quad
	+
	\kappa \inpro{\partial_t \left(\abs{\bff{u}_h}^2 \bff{u}_h - \abs{\bff{u}}^2 \bff{u}\right)}{\Delta_h \bff{\xi}}
	+
	\beta \inpro{\bff{e}(\bff{e}\cdot \partial_t \bff{\theta})}{\Delta_h \bff{\xi}}
\end{align}
Adding \eqref{equ:nab Delta xi} and \eqref{equ:ddt nab xi}, we obtain
\begin{align*}
	&\frac{1}{2}\ddt \norm{\nabla \bff{\xi}}{\bb{L}^2}^2
	+
	\lambda_r \norm{\Delta_h \bff{\xi}}{\bb{L}^2}^2
	+
	\lambda_e \norm{\nabla \Delta_h \bff{\xi}}{\bb{L}^2}^2
	\\
	&=
	\inpro{\partial_t \bff{\eta}}{\Delta_h \bff{\xi}}
	-
	\inpro{\nabla \Pi_h \partial_t \bff{\rho}}{\nabla \Delta_h \bff{\xi}}
	-
	\kappa \mu\inpro{\partial_t \bff{\theta}+\partial_t \bff{\rho}}{\Delta_h \bff{\xi}}
	+
	\beta \inpro{\bff{e}(\bff{e}\cdot \partial_t \bff{\theta})}{\Delta_h \bff{\xi}}
	\\
	&\quad
	+
	\gamma \inpro{\nabla \Pi_h(\bff{u}_h\times \bff{H}_h - \bff{u}\times \bff{H})}{\nabla \Delta_h \bff{\xi}}
	+
	\kappa \inpro{\partial_t \left(\abs{\bff{u}_h}^2 \bff{u}_h - \abs{\bff{u}}^2 \bff{u}\right)}{\Delta_h \bff{\xi}}.
\end{align*}
The first four terms on the right-hand side can be estimated using Young's inequality (together with \eqref{equ:Ritz ineq}) in a straightforward manner. The last two terms can be bounded using \eqref{equ:nab uh cross Hh} and \eqref{equ:dt uh2uh}. We then have for any $\epsilon>0$,
\begin{align*}
	&\frac{1}{2}\ddt \norm{\nabla \bff{\xi}}{\bb{L}^2}^2
	+
	\lambda_r \norm{\Delta_h \bff{\xi}}{\bb{L}^2}^2
	+
	\lambda_e \norm{\nabla \Delta_h \bff{\xi}}{\bb{L}^2}^2
	\lesssim
	h^{2r}
	+
	\norm{\partial_t \bff{\theta}}{\bb{L}^2}^2
	+
	\norm{\bff{\theta}}{\bb{H}^1}^2
	+
	\norm{\bff{\xi}}{\bb{H}^1}^2
	+
	\epsilon \norm{\Delta_h \bff{\xi}}{\bb{H}^1}^2
\end{align*}
Choosing $\epsilon>0$ sufficiently small, integrating both sides with respect to $t$ (and noting \eqref{equ:semidisc theta xi}), we obtain~\eqref{equ:semidisc xi L2}. This completes the proof of the proposition.
\end{proof}

We are now ready to state the main theorem of this section on the order of convergence for the semi-discrete scheme~\eqref{equ:weaksemidisc}.

\begin{theorem}\label{the:semidisc rate}
Let~$(\bff{u},\bff{H})$ be the solution of \eqref{equ:weakform} as described in Section~\ref{sec:exact sol}, and let~$(\bff{u}_h,\bff{H}_h)\in \bb{V}_h\times \bb{V}_h$ be the solution of \eqref{equ:weaksemidisc} with initial data~$\bff{u}_0$. Then
\begin{align*}
	\norm{\bff{u}_h - \bff{u}}{L^\infty(\bb{L}^2)} 
	+
	h \norm{\nabla \bff{u}_h - \nabla \bff{u}}{L^\infty(\bb{L}^2)}
	&\leq
	Ch^{r+1},
	\\
	\norm{\bff{H}_h - \bff{H}}{L^\infty(\bb{L}^2)} 
	+
	h \norm{\nabla \bff{H}_h - \nabla \bff{H}}{L^2(\bb{L}^2)}
	&\leq
	Ch^{r+1}.
\end{align*}
Moreover, if the triangulation is globally quasi-uniform, then for $s=0$ or $1$,
\begin{align*}
	\norm{\bff{u}_h - \bff{u}}{L^\infty(\bb{H}^s)}
	+
	\norm{\bff{H}_h - \bff{H}}{L^\infty(\bb{H}^s)}
	&\leq
	Ch^{r+1-s},
	\\
	\norm{\bff{u}_h - \bff{u}}{L^\infty(\bb{L}^\infty)}
	+
	\norm{\bff{H}_h - \bff{H}}{L^2(\bb{L}^\infty)}
	&\leq
	Ch^{r+1} \abs{\ln h}^{\frac{1}{2}}.
\end{align*}
The constant $C$ depends on the coefficients of the equation, $|\mathscr{D}|$, $T$, and $K_0$ (as defined in \eqref{equ:ass 2 u}), but is independent of $h$.

\begin{proof}
The result follows from Proposition \ref{pro:semidisc theta xi}, equations \eqref{equ:theta rho} and \eqref{equ:xi eta}, estimates \eqref{equ:Ritz ineq}, \eqref{equ:Ritz ineq L infty}, and the triangle inequality.
\end{proof}
\end{theorem}

\section{A Fully Discrete Scheme Based on the Semi-Implicit Euler Method}\label{sec:Euler}

We propose a time-discrete scheme for the LLBar equation using the semi-implicit Euler method. First, we fix some notations. Let $k$ be the time step and $\bff{u}_h^n$ be the approximation in $\bb{V}_h$ of $\bff{u}(t)$ at time $t=t_n:=nk$, where $n=0,1,2,\ldots, \lfloor T/k \rfloor$. For any function $\bff{v}$, we denote $\bff{v}^n:= \bff{v}(t_n)$, and define
\begin{align*}
	\delta \bff{v}^{n+1} :=
	\frac{\bff{v}^{n+1}-\bff{v}^{n}}{k},\quad
	\text{for } n=0,1,\ldots.
\end{align*}

A fully discrete scheme can now be described as follows. We start with $\bff{u}_h^0= R_h \bff{u}(0) \in \bb{V}_h$ for simplicity. For $t_n\in [0,T]$ where $n\in \bb{N}$, given $\bff{u}_h^n \in \bb{V}_h$, define $\bff{u}_h^{n+1}$ and $\bff{H}_h^{n+1}$ by
\begin{align}\label{equ:euler}
	\left\{
	\begin{alignedat}{1}
		\inpro{\delta \bff{u}_h^{n+1}}{\bff{\chi}}
		&=
		\lambda_r \inpro{\bff{H}_h^{n+1}}{\bff{\chi}}
		+
		\lambda_e \inpro{\nabla \bff{H}_h^{n+1}}{\nabla \bff{\chi}}
		-
		\gamma \inpro{\bff{u}_h^{n} \times \bff{H}_h^{n+1}}{\bff{\chi}}
		\\
		\inpro{\bff{H}_h^{n+1}}{\bff{\phi}} 
		&=
		-
		\inpro{\nabla \bff{u}_h^{n+1}}{\nabla \bff{\phi}}
		+
		\kappa \mu \inpro{\bff{u}_h^n}{\bff{\phi}}
		-
		\kappa \inpro{\abs{\bff{u}_h^{n+1}}^2 \bff{u}_h^{n+1}}{\bff{\phi}}
		-
		\beta \inpro{\bff{e}\big(\bff{e}\cdot \bff{u}_h^{n+1}\big)}{\bff{\phi}},
	\end{alignedat}
	\right.
\end{align}
for all $\bff{\chi}, \bff{\phi}\in \bb{V}_h$.
	
Note that under the assumptions \eqref{equ:ass 2 u}, for $p\in [1,\infty]$ we have
\begin{align}\label{equ:delta un Lp euler}
	\norm{\delta \bff{u}^{n+1}}{\bb{L}^p}
	&=
	\norm{\frac{1}{k} \int_{t_n}^{t_{n+1}} \partial_t \bff{u}(t)\, \dt}{\bb{L}^p}
	\leq 
	\norm{\partial_t \bff{u}(t)}{\bb{L}^p},
	\\
	\label{equ:delta un min u euler}
	\norm{\delta \bff{u}^n-\partial_t \bff{u}^{n+1}}{\bb{L}^p}
	&\leq
	\norm{\delta \bff{u}^n-\partial_t \bff{u}^n}{\bb{L}^p}
	+
	\norm{\int_{t_n}^{t_{n+1}} \partial_{tt}\bff{u}(t)\, \dt}{\bb{L}^p}
	\leq
	Ck,
\end{align}
where Taylor's theorem was used in the last step of \eqref{equ:delta un min u euler}, and $C$ depends on $\norm{\partial_{tt} \bff{u}}{L^\infty(\bb{L}^p)}$.

We now show that the scheme~\eqref{equ:euler} is well-posed for any time step size $k$.

\begin{proposition}\label{pro:wellpos euler}
Let $\bff{u}_h^0 \in \bb{V}_h$ and $k>0$ be given. For each $n\in \bb{N}$, given $\bff{u}_h^n$, there exists unique $\bff{u}_h^{n+1}$ and $\bff{H}_h^{n+1}$ solving~\eqref{equ:euler}.
\end{proposition}

\begin{proof}
Let $\bb{X}:= \bb{V}_h\times \bb{V}_h$ (which is a subspace of $\bb{H}^1\times \bb{H}^1$), equipped with norm $\norm{(\bff{u},\bff{v})}{\bb{X}} := \norm{\bff{u}}{\bb{H}^1}+ \norm{\bff{v}}{\bb{H}^1}$. Define a nonlinear form $\mathcal{A}:\bb{X}\times \bb{X} \to \bb{R}$ by
\begin{align*}
	\nonumber
	\mathcal{A}\big((\bff{u},\bff{v}), (\bff{\phi},\bff{\chi})\big)
	&:=
	\inpro{\bff{v}}{\bff{\phi}}
	-
	\inpro{\bff{u}}{\bff{\chi}}
	+
	k\lambda_r \inpro{\bff{v}}{\bff{\chi}}
	+
	k\lambda_e \inpro{\nabla \bff{v}}{\nabla \bff{\chi}}
	-
	k\gamma \inpro{\bff{u}_h^n\times \bff{v}}{\bff{\chi}}
	\\
	&\;\quad
	+
	\inpro{\nabla \bff{u}}{\nabla \bff{\phi}}
	+
	\kappa \inpro{\abs{\bff{u}}^2 \bff{u}}{\bff{\phi}}
	+
	\beta\inpro{\bff{e}(\bff{e}\cdot\bff{u})}{\bff{\phi}}.
\end{align*}
and a linear form $f:\bb{X}\to \bb{R}$ by 
\[
	f\big((\bff{\phi},\bff{\chi})\big):=
	\kappa \mu \inpro{\bff{u}_h^n}{\bff{\phi}}
	-
	\inpro{\bff{u}_h^n}{\bff{\chi}}.
\]
Multiplying the first equation in~\eqref{equ:euler} by $-k$ and adding it to the second equation, we see that solving~\eqref{equ:euler} is equivalent to solving
\begin{align}\label{equ:A equal f}
\mathcal{A}\left((\bff{u}_h^{n+1}, \bff{H}_h^{n+1}), (\bff{\phi},\bff{\chi})\right)
= f\big((\bff{\phi},\bff{\chi})\big),
\quad
\forall (\bff{\phi},\bff{\chi})\in \bb{X}.
\end{align}
For each fixed $(\bff{u},\bff{v})\in \bb{X}$, the map $(\bff{\phi},\bff{\chi})\mapsto \mathcal{A}\big((\bff{u},\bff{v}), (\bff{\phi},\bff{\chi})\big)$ is a bounded linear functional. Thus, there exists a map $T:\bb{X}\to \bb{X}^\ast$ defined by
\[
(\bff{u},\bff{v})\mapsto T((\bff{u},\bff{v})), \quad \text{such that} \quad \mathcal{A}\big((\bff{u},\bff{v}), (\bff{\phi},\bff{\chi})\big)
= \inpro{T\big((\bff{u},\bff{v})\big)}{(\bff{\phi},\bff{\chi})},
\quad \forall (\bff{\phi},\bff{\chi})\in \bb{X},
\]
where $\inpro{\cdot}{\cdot}$ denotes the duality pairing which extends the usual $\bb{L}^2$-inner product.
We aim to use the Browder--Minty theorem to deduce the existence and uniqueness of~\eqref{equ:A equal f}. To this end, it remains to show $T$ is bounded, continuous, strictly monotone, and coercive.

\medskip
\noindent
\underline{Boundedness}: We have by H\"older's inequality and Sobolev embedding $\bb{H}^1\subset \bb{L}^6$,
\begin{align*}
	\norm{T\big((\bff{u},\bff{v})\big)}{\bb{X}^\ast}
	&:=
	\sup \left\{
	\left| \mathcal{A}\big((\bff{u},\bff{v}), (\bff{\phi},\bff{\chi})\big) \right| :
	(\bff{\phi},\bff{\chi})\in \bb{X} \text{ and } \norm{(\bff{\phi},\bff{\chi})}{\bb{X}}\leq 1 \right\}
	\\
	&\;\leq
	C\norm{\bff{u}}{\bb{H}^1}+ \left(1+k\lambda_r+k\lambda_e+k\gamma \norm{\bff{u}_h^n}{\bb{H}^1}\right) \norm{\bff{v}}{\bb{H}^1} + \norm{\bff{u}}{\bb{H}^1}^3
	\leq
	C \norm{(\bff{u},\bff{v})}{\bb{X}} \left(1+ \norm{(\bff{u},\bff{v})}{\bb{X}}^2\right).
\end{align*}
This shows $T$ maps bounded sets in $\bb{X}$ into bounded sets in $\bb{X}^\ast$.

\medskip
\noindent
\underline{Continuity}: It is clear that $(\bff{u}_n, \bff{v}_n)\to (\bff{u},\bff{v})$ strongly in $\bb{X}$ implies $\norm{T\big((\bff{u}_n,\bff{v}_n)\big) - T\big((\bff{u}, \bff{v})\big)}{\bb{X}^\ast} \to 0$ as $n\to\infty$. The details are omitted.
%
%Suppose that $(\bff{u}_n, \bff{v}_n)\to (\bff{u},\bff{v})$ strongly in $\bb{X}$. Then by H\"older's inequality,
%\begin{align*}
%	\norm{T\big((\bff{u}_n,\bff{v}_n)\big) - T\big((\bff{u}, \bff{v})\big)}{\bb{X}^\ast}
%	&\leq
%	C\norm{\bff{u}_n-\bff{u}}{\bb{H}^1} + C\norm{\bff{v}_n-\bff{v}}{\bb{H}^1}
%	+
%	\kappa \norm{\bff{u}_n}{\bb{L}^6}^2 \norm{\bff{u}_n-\bff{u}}{\bb{L}^6}
%	\\
%	&\quad
%	+
%	\kappa \norm{\bff{u}}{\bb{L}^6} \norm{\bff{u}_n+\bff{u}}{\bb{L}^6} \norm{\bff{u}_n-\bff{u}}{\bb{L}^6}
%	\leq
%	C\norm{(\bff{u}_n,\bff{v}_n)-(\bff{u},\bff{v})}{\bb{X}},
%\end{align*}
%where in the last step we used Sobolev embedding. Thus, $\norm{T\big((\bff{u}_n,\bff{v}_n)\big) - T\big((\bff{u}, \bff{v})\big)}{\bb{X}^\ast} \to 0$ as $n\to\infty$, showing the continuity of $T$.

\medskip
\noindent
\underline{Strict monotonicity}: Let $(\bff{u},\bff{v}), (\bff{u}',\bff{v}')\in \bb{X}$. Then we have
\begin{align*}
	&\inpro{T\big((\bff{u},\bff{v})\big)- T\big((\bff{u}',\bff{v}')\big)}{(\bff{u},\bff{v})-(\bff{u}',\bff{v}')}
	\\
	&=
	\mathcal{A}\big((\bff{u},\bff{v}), (\bff{u}-\bff{u}',\bff{v}-\bff{v}')\big)
	-
	\mathcal{A}\big((\bff{u}',\bff{v}'), (\bff{u}-\bff{u}',\bff{v}-\bff{v}')\big)
	\\
	&=
	k\lambda_r \norm{\bff{v}-\bff{v}'}{\bb{L}^2}^2
	+
	k \lambda_e \norm{\nabla \bff{v}-\nabla \bff{v}'}{\bb{L}^2}^2
	+
	\norm{\nabla \bff{u}-\nabla \bff{u}'}{\bb{L}^2}^2
	+
	\beta \norm{\bff{e}\cdot (\bff{u}-\bff{u}')}{L^2}^2
	\\
	&\quad
	+
	\kappa \norm{\abs{\bff{u}}^2- \abs{\bff{u}'}^2}{\bb{L}^2}^2
	+
	\kappa \norm{\abs{\bff{u}} \abs{\bff{u}-\bff{u}'}}{\bb{L}^2}^2
	+
	\kappa \norm{\abs{\bff{u}'} \abs{\bff{u}-\bff{u}'}}{\bb{L}^2}^2 \geq 0,
\end{align*}
where in the last step we used~\eqref{equ:a2a b2b dot ab}. Moreover, equality holds if and only if $(\bff{u},\bff{v})=(\bff{u}',\bff{v}')$. This shows the strict monotonicity of $T$.

\medskip
\noindent
\underline{Coercivity}: Let $(\bff{u},\bff{v})\in \bb{X}$. We have
\begin{align*}
	\inpro{T\big((\bff{u},\bff{v})\big)}{(\bff{u},\bff{v})}
	&=
	\mathcal{A}\big((\bff{u},\bff{v}), (\bff{u},\bff{v})\big) 
	\\
	&=
	k\lambda_r \norm{\bff{v}}{\bb{L}^2}^2
	+
	k\lambda_e \norm{\nabla \bff{v}}{\bb{L}^2}^2
	+
	\norm{\nabla \bff{u}}{\bb{L}^2}^2
	+
	\kappa \norm{\bff{u}}{\bb{L}^4}^4
	+
	\beta \norm{\bff{e}\cdot\bff{u}}{L^2}^2
	\\
	&\geq
	C_1 \norm{\bff{v}}{\bb{H}^1}^2
	+
	C_2 \norm{\bff{u}}{\bb{H}^1}^2
	-
	\kappa \abs{\mathscr{D}}^{-4}
	\\
	&\geq
	\min\{C_1,C_2\} \norm{(\bff{u},\bff{v})}{\bb{X}}^2 - \kappa\abs{\mathscr{D}}^{-4},
\end{align*}
where $C_1:=\min\{k\lambda_r,k\lambda_e\}$ and $C_2:= \min\{1,\kappa\abs{\mathscr{D}}^{-4}\}$ are both positive constants, and we also used the inequality
\[
\norm{\bff{u}}{\bb{L}^4}^4 \geq \abs{\mathscr{D}}^{-4} \left(\norm{\bff{u}}{\bb{L}^2}^2-1\right),
\]
which follows from H\"older's and Young's inequalities. Therefore,
\begin{align*} 			
	\frac{\inpro{T\big((\bff{u},\bff{v})\big)}{(\bff{u},\bff{v})}}{\norm{(\bff{u},\bff{v})}{\bb{X}}}
	\to \infty
	\quad \text{as } \norm{(\bff{u},\bff{v})}{\bb{X}}\to \infty,
\end{align*}
showing the coercivity of $T$.

\medskip
\noindent
The existence and uniqueness of  $\bff{u}_h^{n+1}$ and $\bff{H}_h^{n+1}$ solving~\eqref{equ:euler} then follows from the Browder--Minty theorem. This completes the proof of the proposition.
\end{proof}

Next, we show some stability results. The following proposition shows that $\bff{u}_h^n$ is stable in $\ell^\infty(\bb{H}^1)$ norm, while $\bff{H}_h^n$ is stable in $\ell^2(\bb{H}^1)$ norm (over an arbitrary number of iterations $n$, even if $T=\infty$). Moreover, the energy dissipation property~\eqref{equ:ene decrease euler} is satisfied unconditionally.

\begin{proposition}\label{pro:ene dec H1 euler}
	Let $\bff{u}_h^0 \in \bb{V}_h$ be given and let $\left(\bff{u}_h^n, \bff{H}_h^n\right)$ be defined by
	\eqref{equ:euler}. Then for any $k>0$ and~$n\in \bb{N}$,
	\begin{align}\label{equ:ene decrease euler}
		\mathcal{E}(\bff{u}_h^{n+1}) \leq \mathcal{E}(\bff{u}_h^n),
	\end{align}
	where $\mathcal{E}$ was defined in \eqref{equ:energy}. Moreover,
	\begin{align}\label{equ:stab L4 H1 euler}
		\norm{\bff{u}_h^n}{\bb{L}^4}^4
		+
		\norm{\bff{u}_h^n}{\bb{H}^1}^2
		+
		k\lambda_r \sum_{m=1}^n \norm{\bff{H}_h^{m}}{\bb{L}^2}^2
		+
		k\lambda_e \sum_{m=1}^n \norm{\nabla \bff{H}_h^{m}}{\bb{L}^2}^2
		\leq
		C \left( \norm{\bff{u}_h^0}{\bb{H}^1}^2 + |\mathscr{D}| \right),
	\end{align}
	where $C$ depends only on $\kappa$ and $\mu$.
\end{proposition}

\begin{proof}
	Setting $\bff{\chi}=\bff{H}_h^{n+1}$ in \eqref{equ:euler} gives
	\begin{align}\label{equ:ene stab Un Hn euler}
		\inpro{\delta \bff{u}_h^{n+1}}{\bff{H}_h^{n+1}}
		&=
		\lambda_r \norm{\bff{H}_h^{n+1}}{\bb{L}^2}^2
		+
		\lambda_e \norm{\nabla \bff{H}_h^{n+1}}{\bb{L}^2}^2,
	\end{align}
	while setting $\bff{\phi}=\delta \bff{u}_h^{n+1}$ gives
	\begin{align}\label{equ:ene stab Hn Un euler}
		\nonumber
		\inpro{\bff{H}_h^{n+1}}{\delta \bff{u}_h^{n+1}}
		&=
		-
		\frac{1}{2k} \big(\norm{\nabla \bff{u}_h^{n+1}}{\bb{L}^2}^2 - \norm{\nabla \bff{u}_h^n}{\bb{L}^2}^2 \big)
		-
		\frac{1}{2k}\norm{\nabla \bff{u}_h^{n+1}- \nabla \bff{u}_h^n}{\bb{L}^2}^2
		\\
		\nonumber
		&\quad
		-
		\frac{\kappa}{4k} \left(\norm{\abs{\bff{u}_h^{n+1}}^2 - \mu}{\bb{L}^2}^2 - \norm{\abs{\bff{u}_h^n}^2 -\mu}{\bb{L}^2}^2 \right)
		-
		\frac{\kappa}{4k} \norm{\abs{\bff{u}_h^{n+1}}^2 - \abs{\bff{u}_h^n}^2}{\bb{L}^2}^2
		\\
		&\quad
		-
		\frac{\kappa k}{2} \norm{\abs{\bff{u}_h^{n+1}} \abs{\delta \bff{u}_h^{n+1}}}{\bb{L}^2}^2
		-
		\frac{\mu k}{2} \norm{\bff{u}_h^{n+1}-\bff{u}_h^n}{\bb{L}^2}^2
		\nonumber \\
		&\quad
		-
		\frac{\beta}{2k} \left(\norm{\bff{e}\cdot \bff{u}_h^{n+1}}{L^2}^2 - \norm{\bff{e}\cdot \bff{u}_h^n}{L^2}^2 \right)
		-
		\frac{\beta}{2k} \norm{\bff{e}\cdot (\bff{u}_h^{n+1}-\bff{u}_h^n)}{L^2}^2.
	\end{align}
	Substituting \eqref{equ:ene stab Hn Un euler} into \eqref{equ:ene stab Un Hn euler}, using identities \eqref{equ:a dot ab} and \eqref{equ:a2a dot ab}, and rearranging the terms yield
	\begin{align*}
		&\frac{1}{2} \big(\norm{\nabla \bff{u}_h^{n+1}}{\bb{L}^2}^2 - \norm{\nabla \bff{u}_h^n}{\bb{L}^2}^2 \big)
		+
		\frac{\kappa}{4} \left(\norm{\abs{\bff{u}_h^{n+1}}^2 - \mu}{\bb{L}^2}^2 - \norm{\abs{\bff{u}_h^n}^2 -\mu}{\bb{L}^2}^2 \right)
		+
		\frac{\beta}{2} \left(\norm{\bff{e}\cdot \bff{u}_h^{n+1}}{L^2}^2 - \norm{\bff{e}\cdot \bff{u}_h^n}{L^2}^2 \right)
		\\
		&\quad
		+
		\frac{1}{2}\norm{\nabla \bff{u}_h^{n+1}- \nabla \bff{u}_h^n}{\bb{L}^2}^2
		+
		\frac{\kappa}{4} \norm{\abs{\bff{u}_h^{n+1}}^2 - \abs{\bff{u}_h^n}^2}{\bb{L}^2}^2
		+
		\frac{\kappa k^2}{2} \norm{\abs{\bff{u}_h^{n+1}} \abs{\delta \bff{u}_h^{n+1}}}{\bb{L}^2}^2
		+
		\frac{\mu k^2}{2} \norm{\bff{u}_h^{n+1}-\bff{u}_h^n}{\bb{L}^2}^2
		\\
		&\quad
		+
		\frac{\beta}{2} \norm{\bff{e}\cdot (\bff{u}_h^{n+1}-\bff{u}_h^n)}{L^2}^2
		+
		k\lambda_r \norm{\bff{H}_h^{n+1}}{\bb{L}^2}^2
		+
		k\lambda_e \norm{\nabla \bff{H}_h^{n+1}}{\bb{L}^2}^2
		= 0,
	\end{align*}
	which implies \eqref{equ:ene decrease euler}. Finally, applying similar arguments as \eqref{equ:nab uh eq} and \eqref{equ:norm uh L4} yield \eqref{equ:stab L4 H1 euler}.
\end{proof}

We also derive the stability of $\bff{H}_h^{n}$ in $\ell^\infty(\bb{L}^2)$ norm. The following identity will be used in the proof:
\begin{align}
	\label{equ:identity psi euler}
	\abs{\bff{u}_h^{n+1}}^2 \bff{u}_h^{n+1} - \abs{\bff{u}_h^n}^2 \bff{u}_h^n
	&=
	\abs{\bff{u}_h^{n+1}}^2 \left(\bff{u}_h^{n+1}-\bff{u}_h^n\right)
	+
	\left(\bff{u}_h^{n+1}-\bff{u}_h^n\right) \cdot \left(\bff{u}_h^{n+1}+\bff{u}_h^n\right) \bff{u}_h^n.
\end{align}

\begin{proposition}
	Let $\bff{u}_h^0 \in \bb{V}_h$ be given and let $\bff{u}_h^n$, $\bff{H}_h^{n}$ be defined by \eqref{equ:euler}. Then for any $k>0$ and~$n\in \bb{N}$,
	\begin{align}\label{equ:Hn L2 stab euler}
		\norm{\bff{H}_h^{n}}{\bb{L}^2}^2
		+
		k \sum_{m=1}^n \norm{\delta \bff{u}_h^m}{\bb{L}^2}^2
		\leq
		C,
	\end{align}
	where $C$ depends on the coefficients of the equation, $K_0$, and $|\mathscr{D}|$, but is independent of $n$, $k$, and $T$.
\end{proposition}

\begin{proof}
Taking $\bff{\chi}=\delta \bff{u}_h^{n+1}$ in \eqref{equ:euler} gives
	\begin{align}\label{equ:delta n half euler}
		\norm{\delta \bff{u}_h^{n+1}}{\bb{L}^2}^2
		&=
		\lambda_r \inpro{\bff{H}_h^{n+1}}{\delta \bff{u}_h^{n+1}}
		+
		\lambda_e \inpro{\nabla \bff{H}_h^{n+1}}{\nabla \delta \bff{u}_h^{n+1}}
		-
		\gamma \inpro{\bff{u}_h^{n} \times \bff{H}_h^{n+1}}{\delta \bff{u}_h^{n+1}}.
	\end{align}
	Subtracting the second equation in \eqref{equ:euler} at time step $n-1$ from the same equation at time step $n$, then setting $\bff{\phi}= \lambda_e \bff{H}_h^{n+1}/k$ gives
	\begin{align}\label{equ:H n half euler}
		\nonumber
		&\frac{\lambda_e}{2k} \left( \norm{\bff{H}_h^{n+1}}{\bb{L}^2}^2 - \norm{\bff{H}_h^n}{\bb{L}^2}^2 \right)
		+
		\frac{\lambda_e}{2k} \norm{\bff{H}_h^{n+1}-\bff{H}_h^n}{\bb{L}^2}^2
		\\
		\nonumber
		&=
		\kappa\mu \lambda_e \inpro{\bff{H}_h^{n+1}}{\delta \bff{u}_h^{n+1}}
		-
		\lambda_e \inpro{\nabla \bff{H}_h^{n+1}}{\nabla \delta \bff{u}_h^{n+1}}
		-
		\beta\lambda_e \inpro{\bff{e}(\bff{e}\cdot \delta \bff{u}_h^{n+1})}{\bff{H}_h^{n+1}}
		\\
		&\quad 
		-
		\kappa \lambda_e \inpro{\abs{\bff{u}_h^{n+1}}^2 \delta\bff{u}_h^{n+1} + \delta\bff{u}_h^{n+1} \cdot (\bff{u}_h^{n+1}+\bff{u}_h^n) \bff{u}_h^n}{\bff{H}_h^{n+1}},
	\end{align}
	where we used \eqref{equ:a dot ab} for the left-hand side, and \eqref{equ:identity psi euler} for the last inner product. Adding \eqref{equ:delta n half euler} and \eqref{equ:H n half euler} gives
	\begin{align*}
		&\frac{\lambda_e}{2k} \left( \norm{\bff{H}_h^{n+1}}{\bb{L}^2}^2 - \norm{\bff{H}_h^n}{\bb{L}^2}^2 \right)
		+
		\frac{\lambda_e}{2k} \norm{\bff{H}_h^{n+1}-\bff{H}_h^n}{\bb{L}^2}^2
		+
		\norm{\delta \bff{u}_h^{n+1}}{\bb{L}^2}^2
		\\
		&=
		(\lambda_r+\kappa\mu \lambda_e) \inpro{\bff{H}_h^{n+1}}{\delta \bff{u}_h^{n+1}}
		-
		\beta\lambda_e \inpro{\bff{e}(\bff{e}\cdot \delta \bff{u}_h^{n+1})}{\bff{H}_h^{n+1}}
		-
		\gamma \inpro{\bff{u}_h^{n} \times \bff{H}_h^{n+1}}{\delta \bff{u}_h^{n+1}}
		\\
		&\quad
		-
		\kappa \lambda_e \inpro{\abs{\bff{u}_h^{n+1}}^2 \delta\bff{u}_h^{n+1} + \delta\bff{u}_h^{n+1} \cdot (\bff{u}_h^{n+1}+\bff{u}_h^n) \bff{u}_h^n}{\bff{H}_h^{n+1}}
		\\
		&=:S_1+S_2+S_3+S_4.
	\end{align*}
	We will estimate each term on the right-hand side by applying H\"older's and Young's inequalities. Firstly, for the terms $S_1$ and $S_2$,
	\begin{align}\label{equ:S1 euler}
		|S_1| + |S_2|
		&\leq 
		C \norm{\bff{H}_h^{n+1}}{\bb{L}^2}^2 
		+ 
		\frac{1}{4} \norm{\delta \bff{u}_h^{n+1}}{\bb{L}^2}^2.
	\end{align}
	For the third term, similarly we have
	\begin{align}\label{equ:S2 euler}
		|S_3|
		&\leq
		\norm{\bff{u}_h^n}{\bb{L}^4} \norm{\bff{H}_h^{n+1}}{\bb{L}^4} \norm{\delta \bff{u}_h^{n+1}}{\bb{L}^2}
		\leq
		C \norm{\bff{H}_h^{n+1}}{\bb{H}^1}^2
		+ 
		\frac{1}{4} \norm{\delta \bff{u}_h^{n+1}}{\bb{L}^2}^2,
	\end{align}
	where in the last step we used \eqref{equ:stab L4 H1 euler} and the Sobolev embedding $\bb{H}^1 \hookrightarrow \bb{L}^4$. Similarly for the last term,
	\begin{align}\label{equ:S3 euler}
		\nonumber
		|S_4|
		&\leq
		C \left( \norm{\bff{u}_h^{n+1}}{\bb{L}^6}^2 \norm{\delta \bff{u}_h^{n+1}}{\bb{L}^2} + \norm{\delta \bff{u}_h^{n+1}}{\bb{L}^2} \norm{\bff{u}_h^{n+1}+\bff{u}_h^n}{\bb{L}^6} \norm{\bff{u}_h^n}{\bb{L}^6} \right) \norm{\bff{H}_h^{n+1}}{\bb{L}^6}
		\\
		&\leq
		C \norm{\bff{H}_h^{n+1}}{\bb{H}^1}^2 
		+
		\frac{1}{4} \norm{\delta \bff{u}_h^{n+1}}{\bb{L}^2}^2.
	\end{align}
	Altogether, upon rearranging and summing the terms, \eqref{equ:S1 euler}, \eqref{equ:S2 euler}, and \eqref{equ:S3 euler} imply
	\begin{align*}
		\lambda_e \norm{\bff{H}_h^n}{\bb{L}^2}^2
		+
		\lambda_e \sum_{m=0}^n \norm{\bff{H}_h^{m+1}-\bff{H}_h^m}{\bb{L}^2}^2
		+
		k \sum_{m=0}^n \norm{\delta \bff{u}_h^n}{\bb{L}^2}^2
		&\leq
		\lambda_e \norm{\bff{H}_h^0}{\bb{L}^2}^2
		+
		Ck \sum_{m=0}^n \norm{\bff{H}_h^{m+1}}{\bb{H}^1}^2
		\leq C,
	\end{align*}
	as required.
\end{proof}

The above proposition implies the stability of $\bff{u}_h^n$ in $\ell^\infty(\bb{L}^\infty)$ norm under an additional assumption that the triangulation is globally quasi-uniform.

\begin{proposition}
	Let $\bff{u}_h^0 \in \bb{V}_h$ be given and let $\left(\bff{u}_h^n, \bff{H}_h^n\right)$ be defined by \eqref{equ:euler}. Then for any $k>0$ and~$n\in \bb{N}$,
	\begin{align}\label{equ:stab Delta uh euler}
		\norm{\Delta_h \bff{u}_h^n}{\bb{L}^2}^2 \leq C.
	\end{align}
	Moreover, if the triangulation $\mathcal{T}_h$ is (globally) quasi-uniform, then
	\begin{align}\label{equ:stab uh L infty euler}
		\norm{\bff{u}_h^n}{\bb{L}^\infty}^2 \leq C.
	\end{align}
	Here, the constant $C$ depends on the coefficients of the equation, $K_0$, and $|\mathscr{D}|$, but is independent of $n$, $k$, and $T$.
\end{proposition}

\begin{proof}
	Taking $\bff{\phi}=\Delta_h \bff{u}_h^{n+1}$ and applying Young's and H\"older's inequalities, we have
	\begin{align*}
		\norm{\Delta_h \bff{u}_h^{n+1}}{\bb{L}^2}^2
		&=
		\inpro{\bff{H}_h^{n+1}}{\Delta_h \bff{u}_h^{n+1}}
		-
		\kappa \mu \inpro{\bff{u}_h^n}{\Delta_h \bff{u}_h^{n+1}}
		+
		\beta \inpro{\bff{e}(\bff{e}\cdot \bff{u}_h^{n+1})}{\Delta_h \bff{u}_h^{n+1}}
		\\
		&\quad
		+
		\kappa \inpro{\abs{\bff{u}_h^{n+1}}^2 \bff{u}_h^{n+1}}{\Delta_h \bff{u}_h^{n+1}}
		\\
		&\leq
		2\norm{\bff{H}_h^{n+1}}{\bb{L}^2}^2
		+
		\frac{1}{8} \norm{\Delta_h \bff{u}_h^{n+1}}{\bb{L}^2}^2
		+
		(\kappa^2 \mu^2+\beta^2) \norm{\bff{u}_h^n}{\bb{L}^2}^2
		+
		\frac{3}{8} \norm{\Delta_h \bff{u}_h^n}{\bb{L}^2}^2
		+
		\kappa^2 \norm{\bff{u}_h^{n+1}}{\bb{L}^6}^6.
	\end{align*}
	Therefore, rearranging the terms, using the Sobolev embedding $\bb{H}^1 \hookrightarrow \bb{L}^6$ (noting \eqref{equ:stab L4 H1 euler} and \eqref{equ:Hn L2 stab euler}), we infer \eqref{equ:stab Delta uh euler}.
	Inequality \eqref{equ:stab uh L infty euler} then follows from \eqref{equ:disc lapl L infty}, completing the proof of the proposition.
\end{proof}

Furthermore, in dimensions 1 and 2 (or $d=3$ with an additional quasi-uniformity assumption on the triangulation), we can derive the stability of $\bff{H}_h^n$ in $\ell^\infty(\bb{H}^1)$.

\begin{proposition}\label{pro:dt un stab euler}
Let $\bff{u}_h^0 \in \bb{V}_h$ be given and let $\bff{u}_h^n$, $\bff{H}_h^{n}$ be defined by \eqref{equ:euler}. Suppose that one of the following holds:
\begin{enumerate}
	\item $d=1$ or $2$,
	\item $d=3$ and the triangulation is globally quasi-uniform.
\end{enumerate}
Then for any $k>0$ and~$n\in \bb{N}$,
\begin{align}\label{equ:dt un stab euler}
	\norm{\delta \bff{u}_h^n}{\bb{L}^2}^2
	+
	k \sum_{m=1}^n \norm{\delta \bff{H}_h^m}{\bb{L}^2}^2
	+
	k \sum_{m=1}^n \norm{\nabla \delta \bff{u}_h^m}{\bb{L}^2}^2
	\leq
	C,
\end{align}
where $C$ depends on the coefficients of the equation, $K_0$, and $|\mathscr{D}|$, but is independent of $n$, $k$, and $T$.
\end{proposition}

\begin{proof}
For $n\in \bb{N}$, subtracting the first equation in \eqref{equ:euler} at time step $n$ from the corresponding equation at time step $n+1$, dividing by $k$, then setting $\bff{\chi}=\delta\bff{u}_h^{n+1}$ give
\begin{align}\label{equ:delta uh n1}
	\nonumber
	\frac{1}{2k} \left(\norm{\delta \bff{u}_h^{n+1}}{\bb{L}^2}^2 - \norm{\delta \bff{u}_h^n}{\bb{L}^2}^2\right)
	+
	\frac{1}{2k} \norm{\delta \bff{u}_h^{n+1}-\delta \bff{u}_h^n}{\bb{L}^2}^2
	&=
	\lambda_r \inpro{\delta \bff{H}_h^{n+1}}{\delta \bff{u}_h^{n+1}}
	+
	\lambda_e \inpro{\nabla \delta \bff{H}_h^{n+1}}{\nabla \delta \bff{u}_h^{n+1}}
	\\
	&\quad
	-
	\gamma \inpro{\delta \bff{u}_h^n\times \bff{H}_h^{n+1} + \bff{u}_h^{n-1} \times \delta \bff{H}_h^{n+1}}{\delta \bff{u}_h^{n+1}}.
\end{align}
Applying the same operations to the second equation in \eqref{equ:euler}, then successively setting $\bff{\phi}=\lambda_e \delta \bff{H}_h^{n+1}$ and $\bff{\phi}=\Lambda \delta \bff{u}_h^{n+1}$ (where $\Lambda:= \lambda_r+\kappa \mu \lambda_e$) yield
\begin{align}
	\label{equ:lambda e delta Hh n1}
	\nonumber
	\lambda_e \norm{\delta \bff{H}_h^{n+1}}{\bb{L}^2}^2
	&=
	-
	\lambda_e \inpro{\nabla \delta \bff{u}_h^{n+1}}{\nabla \delta \bff{H}_h^{n+1}}
	+
	\kappa \mu \lambda_e \inpro{\delta\bff{u}_h^n}{\delta \bff{H}_h^{n+1}}
	-
	\beta\lambda_e \inpro{\bff{e}(\bff{e}\cdot \delta\bff{u}_h^{n+1})}{\delta\bff{H}_h^{n+1}}
	\\
	&\quad
	-
	\kappa \lambda_e \inpro{\abs{\bff{u}_h^{n+1}}^2 \delta \bff{u}_h^{n+1} + \left( (\bff{u}_h^{n+1}+\bff{u}_h^n)\cdot \delta \bff{u}_h^{n+1}\right) \bff{u}_h^n}{\delta \bff{H}_h^{n+1}},
\end{align}
and (noting the identity~\eqref{equ:a2a b2b dot ab}),
\begin{align}
	\label{equ:delta Hh delta uh}
	\Lambda \inpro{\delta \bff{H}_h^{n+1}}{\delta \bff{u}_h^{n+1}}
	&=
	-\Lambda\left(\norm{\nabla \delta \bff{u}_h^{n+1}}{\bb{L}^2}^2 - \kappa \mu \inpro{\delta \bff{u}_h^n}{\delta \bff{u}_h^{n+1}}
	+ P \right).
\end{align}
In the last step, we denoted $\bff{u}_h^{n+\frac{1}{2}}:= \frac{1}{2} (\bff{u}_h^{n+1}+\bff{u}_h^n)$ and
\[
P:= \kappa \norm{\bff{u}_h^{n+\frac{1}{2}}\cdot \delta \bff{u}_h^{n+1}}{\bb{L}^2}^2
+ \frac{\kappa}{2} \norm{\abs{\bff{u}_h^{n+1}} \abs{\delta \bff{u}_h^{n+1}}}{\bb{L}^2}^2
+ \frac{\kappa}{2} \norm{\abs{\bff{u}_h^n} \abs{\delta \bff{u}_h^{n+1}}}{\bb{L}^2}^2
+ \beta\norm{\bff{e}\cdot \delta\bff{u}_h^{n+1}}{L^2}^2.
\]
Adding \eqref{equ:delta uh n1}, \eqref{equ:lambda e delta Hh n1}, and \eqref{equ:delta Hh delta uh}, then applying Young's inequality, we obtain for any $\epsilon >0$,
\begin{align*}
	E&:= \frac{1}{2k} \left(\norm{\delta \bff{u}_h^{n+1}}{\bb{L}^2}^2 - \norm{\delta \bff{u}_h^n}{\bb{L}^2}^2\right)
	+
	\frac{1}{2k} \norm{\delta \bff{u}_h^{n+1}-\delta \bff{u}_h^n}{\bb{L}^2}^2
	+
	\lambda_e \norm{\delta \bff{H}_h^{n+1}}{\bb{L}^2}^2
	+
	\Lambda \norm{\nabla \delta \bff{u}_h^{n+1}}{\bb{L}^2}^2
	+
	P
	\\
	&\;=
	\Lambda \kappa \mu \inpro{\delta \bff{u}_h^n}{\delta \bff{u}_h^{n+1}}
	-
	\kappa \lambda_e \inpro{\abs{\bff{u}_h^{n+1}}^2 \delta \bff{u}_h^{n+1} + \left( (\bff{u}_h^{n+1}+\bff{u}_h^n)\cdot \delta \bff{u}_h^{n+1}\right) \bff{u}_h^n}{\delta \bff{H}_h^{n+1}}
	\\
	&\quad
	-
	\beta\lambda_e \inpro{\bff{e}(\bff{e}\cdot \delta\bff{u}_h^{n+1})}{\delta\bff{H}_h^{n+1}}
	-
	\gamma \inpro{\delta \bff{u}_h^n\times \bff{H}_h^{n+1} + \bff{u}_h^{n-1} \times \delta \bff{H}_h^{n+1}}{\delta \bff{u}_h^{n+1}}
	\\
	&\;\leq
	C \norm{\delta \bff{u}_h^n}{\bb{L}^2} \norm{\delta \bff{u}_h^{n+1}}{\bb{L}^2}
	+
	C \left(\norm{\bff{u}_h^{n+1}}{\bb{L}^8}^2 + \norm{\bff{u}_h^n}{\bb{L}^8}^2\right) \norm{\delta \bff{u}_h^{n+1}}{\bb{L}^4} \norm{\delta \bff{H}_h^{n+1}}{\bb{L}^2}
	\\
	&\quad
	+
	C \norm{\delta \bff{u}_h^{n+1}}{\bb{L}^2} \norm{\delta \bff{H}_h^{n+1}}{\bb{L}^2}
	+
	C \norm{\delta \bff{u}_h^n}{\bb{L}^2} \norm{\bff{H}_h^{n+1}}{\bb{L}^4} \norm{\delta \bff{u}_h^{n+1}}{\bb{L}^4}
	\\
	&\quad
	+
	C \norm{\bff{u}_h^{n-1}}{\bb{L}^4} \norm{\delta \bff{H}_h^{n+1}}{\bb{L}^2} \norm{\delta \bff{u}_h^{n+1}}{\bb{L}^4}.
\end{align*}
Note that $P\geq 0$. The expression in the last step can be estimated as follows.

\medskip
\noindent
\underline{Case 1: $d=1$ or $2$}. By Young's inequality, Sobolev embedding $\bb{H}^1\hookrightarrow \bb{L}^8$, and~\eqref{equ:L4 gal nir}, we have
\begin{align}\label{equ:E dun}
	E&\leq
	C \norm{\delta \bff{u}_h^n}{\bb{L}^2}^2
	+
	C\norm{\delta \bff{u}_h^{n+1}}{\bb{L}^2}^2
	+
	\epsilon \norm{\nabla \delta \bff{u}_h^{n+1}}{\bb{L}^2}^2
	+
	\epsilon \norm{\delta \bff{H}_h^{n+1}}{\bb{L}^2}^2
	+
	C \norm{\bff{H}_h^{n+1}}{\bb{L}^4}^2 \norm{\delta \bff{u}_h^n}{\bb{L}^2}^2.
\end{align}
Choosing $\epsilon>0$ sufficiently small, multiplyting by $2k$, rearranging the terms, then summing over $m\in \{1,2,\ldots,n-1\}$, we infer that
\begin{align*}
	\norm{\delta \bff{u}_h^n}{\bb{L}^2}^2
	+
	k \sum_{m=1}^{n-1} \left(
	\norm{\delta \bff{H}_h^{m+1}}{\bb{L}^2}^2
	+
	\norm{\nabla \delta \bff{u}_h^{m+1}}{\bb{L}^2}^2 \right)
	\leq
	\norm{\delta \bff{u}_h^1}{\bb{L}^2}^2
	+
	Ck\sum_{m=1}^{n-1} \left(1+ C \norm{\bff{H}_h^{m+1}}{\bb{L}^4}^2\right) \norm{\delta \bff{u}_h^m}{\bb{L}^2}^2.
\end{align*}
Note that by taking $n=0$ and $\bff{\chi}=\delta \bff{u}_h^1$ in the first equation of \eqref{equ:euler}, we have
\[
	\norm{\delta \bff{u}_h^1}{\bb{L}^2}^2 
	\lesssim
	\norm{\bff{H}_h^1}{\bb{H}^1}^2 + \norm{\Delta_h \bff{H}_h^1}{\bb{L}^2}^2 \lesssim 1.
\]
Therefore, the discrete Gronwall inequality and~\eqref{equ:stab L4 H1 euler} then imply~\eqref{equ:dt un stab euler}.

\medskip
\noindent
\underline{Case 2: $d=3$ and the triangulation is quasi-uniform}. In this case, we apply Young's inequality, \eqref{equ:stab uh L infty euler}, and~\eqref{equ:L4 gal nir} instead to infer~\eqref{equ:E dun}. 
The same conclusion is then attained.
\end{proof}

\begin{proposition}\label{pro:Hn H1 stab euler}
Suppose that the hypotheses of Proposition~\ref{pro:dt un stab euler} hold. Then for any $k>0$ and~$n\in \bb{N}$,
\begin{align}\label{equ:Hn H1 stab euler}
	\norm{\nabla \bff{H}_h^n}{\bb{L}^2}^2
	\leq
	C,
\end{align}
where $C$ depends on the coefficients of the equation, $K_0$, and $|\mathscr{D}|$, but is independent of $n$, $k$, and $T$.
\end{proposition}

\begin{proof}
Setting $\bff{\chi}=\bff{H}_h^{n+1}$ in \eqref{equ:euler} gives
\begin{align*}
	\lambda_e \norm{\nabla \bff{H}_h^{n+1}}{\bb{L}^2}^2
	=
	\inpro{\delta \bff{u}_h^{n+1}}{\bff{H}_h^{n+1}}
	-
	\lambda_r \norm{\bff{H}_h^{n+1}}{\bb{L}^2}^2
	\leq
	C\norm{\bff{H}_h^{n+1}}{\bb{L}^2}^2
	+
	C\norm{\delta \bff{u}_h^{n+1}}{\bb{L}^2}^2
	\leq C,
\end{align*}
where in the last step we used \eqref{equ:Hn L2 stab euler} and \eqref{equ:dt un stab euler}.
\end{proof}

\begin{remark}
Proposition~\ref{pro:dt un stab euler} and~\ref{pro:Hn H1 stab euler} also hold under a technical smallness assumption on the initial data similar to Proposition~\ref{pro:dt uh L2 assum}. We will not elaborate them further for brevity.
\end{remark}

To derive optimal error estimates for the fully discrete scheme, we begin with some preparatory results. As done previously, we write
\begin{align}
	\label{equ:Un utn}
	\bff{u}_h^n-\bff{u}(t_n)
	&=
	\big(\bff{u}_h^n- R_h\bff{u}(t_n) \big)
	+
	\big(R_h \bff{u}(t_n)-\bff{u}(t_n) \big)
	=:
	\bff{\theta}^n + \bff{\rho}^n,
	\\
	\label{equ:Hn Htn}
	\bff{H}_h^n-\bff{H}(t_n)
	&=
	\big(\bff{H}_h^n- R_h\bff{H}(t_n) \big)
	+
	\big(R_h \bff{H}(t_n)-\bff{H}(t_n) \big)
	=:
	\bff{\xi}^n + \bff{\eta}^n.
\end{align}
Recall that by the definition of Ritz projection,
\begin{align}\label{equ:nab Ritz disc}
	\inpro{\nabla \bff{\rho}^n}{\nabla \bff{\chi}}
	=
	\inpro{\nabla \bff{\eta}^n}{\nabla \bff{\chi}}
	=
	0,
	\quad\forall \bff{\chi} \in \bb{V}_h.
\end{align}

The following lemmas are needed to bound the nonlinear terms in the main theorem.

\begin{lemma}\label{lem:cross disc euler}
	Let $\epsilon>0$ be arbitrary. The following inequality holds:
	\begin{align}
		\label{equ:cross theta est euler}
		\nonumber
		\left| \inpro{\bff{u}_h^{n} \times \bff{H}_h^{n+1} - \bff{u}^{n+1} \times \bff{H}^{n+1}}{\bff{\theta}^{n+1}} \right| 
		&\lesssim
		h^{2(r+1)} + k^2+ \norm{\bff{\theta}^{n+1}}{\bb{L}^2}^2
		\\
		&\quad
		+ \epsilon \norm{\bff{\theta}^n}{\bb{L}^2}^2
		+ \epsilon \norm{\nabla \bff{\theta}^{n+1}}{\bb{L}^2}^2
		+ \epsilon \norm{\bff{\xi}^{n+1}}{\bb{L}^2}^2,
		\\
		\nonumber
		\label{equ:cross xi euler}
		\big| \inpro{\bff{u}_h^n \times \bff{H}_h^{n+1}- \bff{u}^{n+1} \times \bff{H}^{n+1}}{\bff{\xi}^{n+1}} \big| 
		&\lesssim
		h^{2(r+1)} 
		+ 
		k^2
		+
		\norm{\bff{\xi}^{n+1}}{\bb{L}^2}^2 
		\\
		&\quad
		+ 
		\epsilon \norm{\nabla \bff{\xi}^{n+1}}{\bb{L}^2}^2
		+
		\epsilon \norm{\bff{\theta}^n}{\bb{L}^2}^2,
		\\
		\label{equ:cross dt theta euler}
		\nonumber
		\big| \inpro{\bff{u}_h^n \times \bff{H}_h^{n+1}- \bff{u}^{n+1} \times \bff{H}^{n+1}}{\delta \bff{\theta}^{n+1}} \big| 
		&\lesssim
		h^{2(r+1)} + k^2
		+
		\big( \norm{\bff{\theta}^n}{\bb{L}^2}^2 + \epsilon \norm{\nabla \bff{\theta}^n}{\bb{L}^2}^2 \big) \norm{\bff{H}_h^{n+1}}{\bb{H}^1}^2
		\\
		&\quad
		+
		\norm{\bff{\xi}^{n+1}}{\bb{L}^2}^2
		+
		\epsilon \norm{\delta \bff{\theta}^{n+1}}{\bb{L}^2}^2.
	\end{align}
	Moreover, if the triangulation $\mathcal{T}_h$ is globally quasi-uniform, then for any $\bff{\zeta} \in \bb{V}_h$,
	\begin{align}\label{equ:cross est euler}
		\left| \inpro{\bff{u}_h^{n} \times \bff{H}_h^{n+1} - \bff{u}^{n+1} \times \bff{H}^{n+1}}{\bff{\zeta}} \right| 
		&\lesssim
		h^{2(r+1)} + k^2+ \norm{\bff{\xi}^{n+1}}{\bb{L}^2}^2
		+ \epsilon \norm{\bff{\theta}^n}{\bb{L}^2}^2
		+ \epsilon \norm{\bff{\zeta}}{\bb{L}^2}^2
		\\
		\label{equ:cross nab est euler}
		\big|\inpro{\nabla\Pi_h\big(\bff{u}_h^n \times \bff{H}_h^{n+1}- \bff{u}^{n+1} \times \bff{H}^{n+1}\big)}{\bff{\zeta}} \big|
		&\lesssim
		h^{2r} + k^2
		+
		\norm{\bff{\theta}^n}{\bb{H}^1}^2
		+
		\norm{\bff{\xi}^{n+1}}{\bb{H}^1}^2
		+
		\epsilon \norm{\bff{\zeta}}{\bb{L}^2}^2.
	\end{align}
\end{lemma}

\begin{proof}
	We write
	\begin{align}\label{equ:cross un Hn}
		\bff{u}_h^{n} \times \bff{H}_h^{n+1} - \bff{u}^{n+1} \times \bff{H}^{n+1}
		&=
		\bff{u}_h^{n} \times \left(\bff{\xi}^{n+1}+\bff{\eta}^{n+1}\right)
		+
		\left(\bff{\theta}^{n} + \bff{\rho}^{n} - k\cdot \delta \bff{u}^{n+1}\right) \times \bff{H}^{n+1}.
	\end{align}
	The proof of \eqref{equ:cross theta est euler} then follows by arguments similar to that in \eqref{equ:uh cross Hh theta} (noting \eqref{equ:delta un Lp euler}), without assuming $\mathcal{T}_h$ is globally quasi-uniform. Similarly, the proof of~\eqref{equ:cross xi euler} follows that in~\eqref{equ:uh cross Hh dt theta}.
	
	Next, following \eqref{equ:nonlinear cross}, we write
	\begin{align*}
		\bff{u}_h^{n} \times \bff{H}_h^{n+1} - \bff{u}^{n+1} \times \bff{H}^{n+1}
		&=
		\bff{\theta}^n \times \bff{H}_h^{n+1}
		+
		R_h\bff{u}^n \times (\bff{\xi}^{n+1}+\bff{\eta}^{n+1})
		+
		(\bff{\rho}^n- k \cdot \delta \bff{u}^{n+1}) \times \bff{H}^{n+1}.
	\end{align*}
	The proof of \eqref{equ:cross dt theta euler} then follows along the line of~\eqref{equ:uh cross Hh dt theta} with obvious modifications.
	
	Now, suppose the triangulation is globally quasi-uniform (and thus \eqref{equ:stab uh L infty euler} holds in this case). Noting \eqref{equ:cross un Hn}, by H\"older's inequality and assumptions \eqref{equ:ass 2 u} on the exact solution, we have
	\begin{align*}
		\nonumber
		&\norm{\bff{u}_h^{n} \times \bff{H}_h^{n+1} - \bff{u}^{n+1} \times \bff{H}^{n+1}}{\bb{L}^2}
		\\
		\nonumber
		&\leq
		\norm{\bff{u}_h^{n}}{\bb{L}^\infty} \norm{\bff{\xi}^{n+1}+\bff{\eta}^{n+1}}{\bb{L}^2}
		+
		\left( \norm{\bff{\theta}^{n} + \bff{\rho}^{n}}{\bb{L}^2} + k \norm{\delta \bff{u}^{n+1}}{\bb{L}^2} \right) \norm{\bff{H}^{n+1}}{\bb{L}^\infty}
		\\
		&\lesssim
		h^{r+1} + \norm{\bff{\xi}^{n+1}}{\bb{L}^2}
		+ \norm{\bff{\theta}^n}{\bb{L}^2} + k,
	\end{align*}
	where in the last step we used \eqref{equ:stab uh L infty euler} and \eqref{equ:Ritz ineq}. Estimate \eqref{equ:cross est euler} follows by Young's inequality.
	Similarly, noting \eqref{equ:H1 stab proj}, we obtain \eqref{equ:cross nab est euler} following the proof of~\eqref{equ:nab uh cross Hh}. This completes the proof of the lemma.
\end{proof}

\begin{lemma}\label{lem:cub disc euler}
	Let $\epsilon>0$ be arbitrary. The following inequalities hold:
	\begin{align}\label{equ:uh2uh theta disc euler}
		\left| \inpro{\abs{\bff{u}_h^{n+1}}^2 \bff{u}_h^{n+1}  - \abs{\bff{u}^{n+1}}^2 \bff{u}^{n+1}}{\bff{\theta}^{n+1}} \right| 
		&\lesssim
		h^{2(r+1)} + \norm{\bff{\theta}^{n+1}}{\bb{L}^2}^2 
		+ \epsilon \norm{\nabla \bff{\theta}^{n+1}}{\bb{L}^2}^2,
		\\
		\nonumber
		\label{equ:uh2uh xi disc euler}
		\left| \inpro{\abs{\bff{u}_h^{n+1}}^2 \bff{u}_h^{n+1}  - \abs{\bff{u}^{n+1}}^2 \bff{u}^{n+1}}{\bff{\xi}^{n+1}} \right| 
		&\lesssim
		h^{2(r+1)}  
		+ \left( 1+\norm{\bff{\xi}^{n+1}}{\bb{H}^1}^2 \right) \norm{\bff{\theta}^{n+1}}{\bb{L}^2}^2 
		\\
		&\quad
		+ \epsilon \norm{\nabla \bff{\theta}^{n+1}}{\bb{L}^2}^2
		+ \epsilon \norm{\bff{\xi}^{n+1}}{\bb{L}^2}^2,
		\\
		\label{equ:uh2uh dt theta euler}
		\nonumber
		\big| \inpro{\abs{\bff{u}_h^{n+1}}^2 \bff{u}_h^{n+1}  - \abs{\bff{u}^{n+1}}^2 \bff{u}^{n+1}}{\delta \bff{\theta}^{n+1}} \big|
		&\lesssim
		h^{2(r+1)}
		+
		\big( 1+\norm{\delta \bff{\theta}^{n+1}}{\bb{L}^2}^2 \big) \norm{\bff{\theta}^{n+1}}{\bb{H}^1}^2  
		\\
		&\quad
		+ 
		\epsilon \norm{\delta \bff{\theta}^{n+1}}{\bb{L}^2}^2.
	\end{align}
	Suppose now that the triangulation $\mathcal{T}_h$ is globally quasi-uniform. Then for any $\bff{\zeta}\in \bb{V}_h$,
	\begin{align}\label{equ:psi xi disc euler}
		\left| \inpro{\abs{\bff{u}_h^{n+1}}^2 \bff{u}_h^{n+1}  - \abs{\bff{u}^{n+1}}^2 \bff{u}^{n+1}}{\bff{\zeta}} \right| 
		&\lesssim
		h^{2(r+1)}  + \norm{\bff{\theta}^{n+1}}{\bb{L}^2}^2 
		+ \epsilon \norm{\bff{\zeta}}{\bb{L}^2}^2,
		\\
		\label{equ:psi xi Delta euler}
		\left| \inpro{\abs{\bff{u}_h^{n+1}}^2 \bff{u}_h^{n+1}  - \abs{\bff{u}^{n+1}}^2 \bff{u}^{n+1}}{\Delta_h \bff{\zeta}} \right|
		&\lesssim
		h^{2r} + \norm{\bff{\theta}^{n+1}}{\bb{H}^1}^2
		+
		\epsilon \norm{\nabla \bff{\zeta}}{\bb{L}^2}^2.
	\end{align}
\end{lemma}

\begin{proof}
	First, we write
	\begin{align}\label{equ:psi u2u euler}
		\nonumber
		&\abs{\bff{u}_h^{n+1}}^2 \bff{u}_h^{n+1}  - \abs{\bff{u}^{n+1}}^2 \bff{u}^{n+1}
		\\
		&=
		\abs{\bff{u}_h^{n+1}}^2  \left(\bff{\theta}^{n+1} + \bff{\rho}^{n+1}\right) 
		+
		\left(\bff{\theta}^{n+1} + \bff{\rho}^{n+1}\right) \cdot \left( \bff{u}_h^{n+1}+ \bff{u}^{n+1} \right)  \bff{u}^{n+1}.
	\end{align}
	The proof of \eqref{equ:uh2uh theta disc euler} then follows the same argument as that of \eqref{equ:uh2uh theta}.

	Next, noting $\bff{u}_h^{n+1}=\bff{\theta}^{n+1}+ R_h \bff{u}^{n+1}$ and following \eqref{equ:nonlinear Rh u}, we write
	\begin{align}\label{equ:uhn2 minus un2}
		\nonumber
		\abs{\bff{u}_h^{n+1}}^2 \bff{u}_h^{n+1}  - \abs{\bff{u}^{n+1}}^2 \bff{u}^{n+1}
		&=
		\big(\bff{\theta}^{n+1} \cdot (\bff{\theta}^{n+1}+ 2 R_h \bff{u}^{n+1}) \big) \bff{u}_h^{n+1}
		+
		\abs{R_h \bff{u}^{n+1}}^2 \big(\bff{\theta}^{n+1}+\bff{\rho}^{n+1}\big)
		\\
		&\quad
		+
		\big( \bff{\rho}^{n+1} \cdot (R_h \bff{u}^{n+1} + \bff{u}^{n+1}) \big) \bff{u}^{n+1}.
	\end{align}
	Applying the same argument as in the proof of \eqref{equ:uh2uh xi}, we then obtain \eqref{equ:uh2uh xi disc euler}. The proof of \eqref{equ:uh2uh dt theta euler} follows along the line of \eqref{equ:uh2uh dt theta} (by replacing $\bff{\theta}$, $\bff{\rho}$, $\bff{u}$, and $\partial_t \bff{\theta}$ with $\bff{\theta}^{n+1}$, $\bff{\rho}^{n+1}$, $\bff{u}^{n+1}$, and $\delta \bff{\theta}^{n+1}$ respectively).

	Next, if the triangulation $\mathcal{T}_h$ is quasi-uniform, then \eqref{equ:stab uh L infty euler} holds. As such, we can bound the $\bb{L}^\infty$ norms of $\bff{u}_h^n$ and $\bff{u}_h^{n+1}$ appearing in \eqref{equ:psi u2u euler} (uniformly in $n$ and $h$), giving
	\begin{align*}
		\norm{\abs{\bff{u}_h^{n+1}}^2 \bff{u}_h^{n+1}  - \abs{\bff{u}^{n+1}}^2 \bff{u}^{n+1}}{\bb{L}^2}
		\lesssim
		\norm{\bff{\theta}^{n+1}}{\bb{L}^2}
		+ \norm{\bff{\rho}^{n+1}}{\bb{L}^2}.
	\end{align*} 
	The estimate \eqref{equ:psi xi disc euler} then follows by Young's inequality.
	
	It remains to prove \eqref{equ:psi xi Delta euler}. Let $S$ be the right-hand side of \eqref{equ:psi u2u euler}. Then we have
	\begin{align}\label{equ:inpro S Delta euler}
		\inpro{S}{\Delta_h \bff{\zeta}}
		=
		\inpro{\nabla \Pi_h S}{\nabla \bff{\zeta}}.
	\end{align}
	We will proceed by estimating $\norm{\nabla S}{\bb{L}^2}$ using the expression \eqref{equ:uhn2 minus un2}. By the product rule for gradient,
	\begin{align*}
		\nabla S
		&=
		2\left(\bff{u}_h^{n+1}\cdot \nabla \bff{u}_h^{n+1}\right) \left(\bff{\theta}^{n+1}+\bff{\rho}^{n+1}\right)
		+
		\abs{\bff{u}_h^{n+1}}^2 \left(\nabla \bff{\theta}^{n+1}+\nabla \bff{\rho}^{n+1}\right)
		\\
		&\quad
		+
		\left(\nabla \bff{\theta}^{n+1}+\nabla \bff{\rho}^{n+1}\right) \cdot \left(\bff{u}_h^{n+1}+\bff{u}^{n+1}\right) \bff{u}^{n+1}
		+
		\left(\bff{\theta}^{n+1}+\bff{\rho}^{n+1}\right)\cdot \left(\nabla \bff{u}_h^{n+1}+\nabla \bff{u}^{n+1}\right) \bff{u}^{n+1}
		\\
		&\quad
		+
		\left(\bff{\theta}^{n+1}+\bff{\rho}^{n+1}\right) \cdot \left(\bff{u}_h^{n+1}+\bff{u}^{n+1}\right) \nabla \bff{u}^{n+1}.
	\end{align*}
	Therefore, by H\"older's inequalities and Sobolev embedding $\bb{H}^1\hookrightarrow \bb{L}^6$ (and noting \eqref{equ:stab uh L infty euler}), it is straightforward to see that
	\begin{align*}
		\norm{\nabla S}{\bb{L}^2} \lesssim \norm{\bff{\theta}^{n+1}}{\bb{H}^1} + \norm{\bff{\rho}^{n+1}}{\bb{H}^1}
		\lesssim
		\norm{\bff{\theta}^{n+1}}{\bb{H}^1} + h^r.
	\end{align*}
	Inequality~\eqref{equ:psi xi Delta euler} then follows by applying Young's inequality and \eqref{equ:H1 stab proj} to \eqref{equ:inpro S Delta euler}.
\end{proof}

\begin{lemma}\label{lem:psi psi dt un euler}
	Let $\epsilon>0$ be arbitrary.
	\begin{align}\label{equ:psi psi dt un euler}
		\nonumber
		&\left| \inpro{\frac{\abs{\bff{u}_h^{n+1}}^2 \bff{u}_h^{n+1}- \abs{\bff{u}_h^n}^2 \bff{u}_h^n}{k} - \partial_t \left(\abs{\bff{u}^{n+1}}^2 \bff{u}^{n+1}\right)}{\bff{\xi}^{n+1}} \right| 
		\\
		\nonumber
		&\quad
		\lesssim
		h^{2(r+1)} + k^2
		+
		\left(1+ \norm{\bff{\xi}^{n+1}}{\bb{H}^1}^2 + \norm{\bff{\xi}^n}{\bb{H}^1}^2\right) \left(\norm{\bff{\theta}^{n+1}}{\bb{H}^1}^2 + \norm{\bff{\theta}^n}{\bb{H}^1}^2 \right)
		\\
		&\qquad
		+
		\norm{\bff{\xi}^n}{\bb{L}^2}^2
		+
		\epsilon \norm{\bff{\xi}^{n+1}}{\bb{L}^2}^2
		+
		\epsilon \norm{\delta \bff{\theta}^{n+1}}{\bb{L}^2}^2.
	\end{align}
	Moreover, if the triangulation $\mathcal{T}_h$ is globally quasi-uniform, then for any $\bff{\zeta}\in \bb{V}_h$,
	\begin{align}\label{equ:dt psi zeta euler}
		\nonumber
		\left| \inpro{\frac{\abs{\bff{u}_h^{n+1}}^2 \bff{u}_h^{n+1}- \abs{\bff{u}_h^n}^2 \bff{u}_h^n}{k} - \partial_t \left(\abs{\bff{u}^{n+1}}^2 \bff{u}^{n+1}\right)}{\bff{\zeta}} \right| 
		&\lesssim
		h^{2(r+1)}
		+ 
		\norm{\bff{\theta}^n}{\bb{L}^2}^2 
		+ 
		\norm{\bff{\theta}^{n+1}}{\bb{L}^2}^2 
		\\
		&\quad
		+
		\norm{\delta \bff{\theta}^{n+1}}{\bb{L}^2}^2 
		+
		\epsilon \norm{\bff{\zeta}}{\bb{H}^1}^2.
	\end{align}
\end{lemma}

\begin{proof}
	After some tedious algebra, we can write the first component in the inner product on the left-hand side as
	\begin{align}\label{equ:E id}
		\nonumber
		E 
		&:= 
		\frac{\abs{\bff{u}_h^{n+1}}^2 \bff{u}_h^{n+1}- \abs{\bff{u}_h^n}^2 \bff{u}_h^n}{k} - \partial_t \left(\abs{\bff{u}^{n+1}}^2 \bff{u}^{n+1}\right)
		\\
		\nonumber
		&\;=
		\abs{\bff{\theta}^{n+1}+ R_h \bff{u}^{n+1}}^2 \left(\delta \bff{u}_h^{n+1} - \partial_t \bff{u}^{n+1} \right)
		+
		\left[\big(\bff{\theta}^{n+1}+\bff{\rho}^{n+1}\big)\cdot \big(\bff{\theta}^{n+1}+R_h\bff{u}^{n+1}+\bff{u}^{n+1}\big)\right] \partial_t \bff{u}^{n+1}
		\\
		\nonumber
		&\;\quad
		+
		\left[\left(\delta \bff{u}_h^{n+1}- \partial_t \bff{u}^{n+1}\right)\cdot \left(\bff{\theta}^{n+1}+\bff{\theta}^n+ R_h\bff{u}^{n+1}+R_h\bff{u}^n\right) \right]
		\left(\bff{\theta}^n+R_h \bff{u}^n\right)
		\\
		\nonumber
		&\;\quad
		+
		\left[\partial_t \bff{u}^{n+1} \cdot \left(\bff{u}_h^{n+1}+\bff{u}_h^{n}-2\bff{u}^{n+1}\right)\right] \left(\bff{\theta}^n+R_h \bff{u}^n\right)
		+
		2 \left(\bff{u}^{n+1}\cdot \partial_t \bff{u}^{n+1}\right) \left(\bff{u}_h^n- \bff{u}^{n+1}\right)
		\\
		&=: E_1+E_2+E_3+E_4+E_5.
	\end{align}
	We want to obtain a bound for $\left|\inpro{E_i}{\bff{\xi}^{n+1}}\right|$ for $i=1,2,\cdots,5$. First, we note the frequently used inequality
	\begin{align}\label{equ:D est}
		\nonumber
		\norm{\mathcal{D}}{\bb{L}^2}
		:=
		\norm{\delta \bff{u}_h^{n+1} - \partial_t \bff{u}^{n+1}}{\bb{L}^2}
		&=
		\norm{\delta \bff{\theta}^{n+1}+\delta \bff{\rho}^{n+1}+\delta \bff{u}^{n+1}-\partial_t \bff{u}^{n+1}}{\bb{L}^2}
		\\
		&\lesssim
		\norm{\delta \bff{\theta}^{n+1}}{\bb{L}^2} + h^{r+1} + k,
	\end{align}
	where in the last step \eqref{equ:delta un Lp euler} and~\eqref{equ:delta un min u euler} were used.
	Now, we begin to estimate the term with $E_1$ by applying H\"older's and Young's inequalities (noting \eqref{equ:stab L4 H1 euler} and \eqref{equ:Ritz stab infty}), giving
	\begin{align}\label{equ:E1 euler}
		\nonumber
		\left|\inpro{E_1}{\bff{\xi}^{n+1}}\right|
		&\lesssim
		\norm{\bff{\theta}^{n+1}}{\bb{L}^6}^2 \norm{\mathcal{D}}{\bb{L}^2} \norm{\bff{\xi}^{n+1}}{\bb{L}^6}
		+
		\norm{R_h \bff{u}^{n+1}}{\bb{L}^\infty}^2 \norm{\mathcal{D}}{\bb{L}^2} \norm{\bff{\xi}^{n+1}}{\bb{L}^2}
		\\
		\nonumber
		&\lesssim
		\norm{\bff{\xi}^{n+1}}{\bb{H}^1}^2 \norm{\bff{\theta}^{n+1}}{\bb{H}^1}^2
		+
		\epsilon \norm{\mathcal{D}}{\bb{L}^2}^2
		+
		\norm{\bff{\xi}^{n+1}}{\bb{L}^2}^2
		\\
		&\lesssim
		h^{2(r+1)}+k^2
		+
		\left(1+ \norm{\bff{\xi}^{n+1}}{\bb{H}^1}^2\right) \norm{\bff{\theta}^{n+1}}{\bb{H}^1}^2
		+
		\epsilon \norm{\delta \bff{\theta}^{n+1}}{\bb{L}^2}^2.
	\end{align}
	Similarly, for the next term,
	\begin{align}\label{equ:E2 euler}
		\nonumber
		\left|\inpro{E_2}{\bff{\xi}^{n+1}}\right|
		&\lesssim
		\norm{\bff{\theta}^{n+1}+\bff{\rho}^{n+1}}{\bb{L}^2}
		\norm{\bff{\theta}^{n+1}}{\bb{L}^6}
		\norm{\partial_t \bff{u}^{n+1}}{\bb{L}^6}
		\norm{\bff{\xi}^{n+1}}{\bb{L}^6}
		\\
		\nonumber
		&\quad
		+
		\norm{\bff{\theta}^{n+1}+\bff{\rho}^{n+1}}{\bb{L}^2}
		\norm{R_h \bff{u}^{n+1}+\bff{u}^{n+1}}{\bb{L}^\infty}
		\norm{\partial_t \bff{u}^{n+1}}{\bb{L}^\infty}
		\norm{\bff{\xi}^{n+1}}{\bb{L}^2}
		\\
		&\lesssim
		h^{2(r+1)}
		+
		\left(1+ \norm{\bff{\xi}^{n+1}}{\bb{H}^1}^2\right) \norm{\bff{\theta}^{n+1}}{\bb{H}^1}^2
		+
		\epsilon \norm{\bff{\xi}^{n+1}}{\bb{L}^2}^2.
	\end{align}
	For the term with $E_3$, by similar argument we obtain
	\begin{align}\label{equ:E3 euler}
		\nonumber
		\left|\inpro{E_3}{\bff{\xi}^{n+1}}\right|
		&\lesssim
		\norm{\mathcal{D}}{\bb{L}^2}
		\norm{\bff{\theta}^{n+1}+\bff{\theta}^n}{\bb{L}^6}
		\norm{\bff{\theta}^n}{\bb{L}^6}
		\norm{\bff{\xi}^{n+1}}{\bb{L}^6}
		+
		\norm{\mathcal{D}}{\bb{L}^2}
		\norm{\bff{\theta}^{n+1}+\bff{\theta}^n}{\bb{L}^6}
		\norm{R_h \bff{u}^n}{\bb{L}^6}
		\norm{\bff{\xi}^{n+1}}{\bb{L}^6}
		\\
		\nonumber
		&\quad
		+
		\norm{\mathcal{D}}{\bb{L}^2}
		\norm{R_h \bff{u}^{n+1}+R_h\bff{u}^n}{\bb{L}^6}
		\norm{\bff{\theta}^n}{\bb{L}^6}
		\norm{\bff{\xi}^{n+1}}{\bb{L}^6}
		\\
		\nonumber
		&\quad
		+
		\norm{\mathcal{D}}{\bb{L}^2}
		\norm{R_h \bff{u}^{n+1}+R_h\bff{u}^n}{\bb{L}^\infty}
		\norm{R_h \bff{u}^n}{\bb{L}^\infty}
		\norm{\bff{\xi}^{n+1}}{\bb{L}^2}
		\\
		&\lesssim
		h^{2(r+1)} + k^2
		+
		\left(1+ \norm{\bff{\xi}^{n+1}}{\bb{H}^1}^2\right) \left(\norm{\bff{\theta}^{n+1}}{\bb{H}^1}^2 + \norm{\bff{\theta}^n}{\bb{H}^1}^2 \right)
		+
		\epsilon \norm{\bff{\xi}^{n+1}}{\bb{L}^2}^2.
	\end{align}
	By the same argument, we also have
	\begin{align}\label{equ:E4 euler}
		\left|\inpro{E_4}{\bff{\xi}^{n+1}}\right|
		&\lesssim
		h^{2(r+1)} + k^2
		+
		\left(1+ \norm{\bff{\xi}^{n+1}}{\bb{H}^1}^2 + \norm{\bff{\xi}^n}{\bb{H}^1}^2\right) \left(\norm{\bff{\theta}^{n+1}}{\bb{H}^1}^2 + \norm{\bff{\theta}^n}{\bb{H}^1}^2 \right)
		+
		\epsilon \norm{\bff{\xi}^{n+1}}{\bb{L}^2}^2
	\end{align}
	and
	\begin{align}\label{equ:E5 euler}
		\left|\inpro{E_5}{\bff{\xi}^{n+1}}\right|
		&\lesssim
		h^{2(r+1)} + k^2
		+
		\norm{\bff{\theta}^n}{\bb{L}^2}^2
		+
		\norm{\bff{\xi}^n}{\bb{L}^2}^2
		+
		\epsilon\norm{\bff{\xi}^{n+1}}{\bb{L}^2}^2.
	\end{align}
	Altogether, \eqref{equ:E1 euler}, \eqref{equ:E2 euler}, \eqref{equ:E3 euler}, \eqref{equ:E4 euler}, and \eqref{equ:E5 euler} yield~\eqref{equ:psi psi dt un euler}.
	
	Next, with $E$ as defined in \eqref{equ:E id}, we have the identity
	\begin{align*}
		E&=
		\abs{\bff{u}_h^{n+1}}^2 \left(\delta \bff{u}_h^{n+1}-\partial_t \bff{u}^{n+1}\right)
		+
		\left[ \big(\bff{\theta}^n+\bff{\rho}^n\big)\cdot \big(\bff{u}_h^{n+1}+\bff{u}^{n+1}\big)\right] \partial_t \bff{u}^{n+1}
		\\
		&\quad
		+
		\left[ \big(\delta \bff{u}_h^{n+1}-\partial_t \bff{u}^{n+1}\big)\cdot \big(\bff{u}_h^{n+1}+\bff{u}_h^n\big)\right] \bff{u}_h^n
		\\
		&\quad
		+
		\left[\partial_t \bff{u}^{n+1} \cdot \big(\bff{u}_h^{n+1}+\bff{u}_h^n-2\bff{u}^n\big) \right] \bff{u}_h^n
		+
		2 \left(\bff{u}^{n+1}\cdot \partial_t \bff{u}^n\right) \left(\bff{u}_h^n-\bff{u}^{n+1}\right).
	\end{align*}
	Therefore using the assumption~\eqref{equ:ass 2 u}, applying Young's inequality, \eqref{equ:D est}, and Taylor's theorem (noting that~\eqref{equ:stab uh L infty euler} can be applied by the quasi-uniformity of $\mathcal{T}_h$), we can follow the proof of~\eqref{equ:dt uh2uh} to obtain the estimate \eqref{equ:dt psi zeta euler}.
\end{proof}

We also have the following superconvergence estimates on $\bff{\theta}^n$ and $\bff{\xi}^n$ analogous to Proposition~\ref{pro:semidisc theta xi}, which is an essential step in the proof of the main theorem.

\begin{proposition}\label{pro:theta xi n half euler}
	Suppose $\bff{u}$ and $\bff{H}$ satisfy \eqref{equ:ass 2 u}. Then for $h, k>0$ and~$n\in \{1,2,\ldots,\lfloor T/k \rfloor\}$,
	\begin{align}\label{equ:theta n L2 euler}
		\norm{\bff{\theta}^n}{\bb{L}^2}^2
		+
		k \sum_{m=0}^{n-1} \norm{\nabla \bff{\theta}^{m+1}}{\bb{L}^2}^2
		+
		k \sum_{m=0}^{n-1} \norm{\bff{\xi}^{m+1}}{\bb{L}^2}^2
		&\leq
		C \big(h^{2(r+1)}+k^2 \big).
	\end{align}
	Assume further that one of the following holds:
	\begin{enumerate}
		\item $d=1$ or $2$,
		\item $d=3$ and the triangulation is globally quasi-uniform.
	\end{enumerate}
	Then for $h, k>0$ and~$n\in \{1,2,\ldots,\lfloor T/k \rfloor\}$,
	\begin{align}
		\label{equ:xi nab theta n L2 euler}
		\norm{\bff{\xi}^n}{\bb{L}^2}^2
		+
		\norm{\nabla \bff{\theta}^{n}}{\bb{L}^2}^2
		+
		k \sum_{m=0}^{n-1} \norm{\delta \bff{\theta}^{m+1}}{\bb{L}^2}^2
		+
		k \sum_{m=0}^{n-1} \norm{\nabla \bff{\xi}^{m+1}}{\bb{L}^2}^2
		&\leq
		C \big(h^{2(r+1)}+k^4 \big).
	\end{align}
	Finally, for $d=1$, $2$, $3$, if the triangulation is globally quasi-uniform, then
	\begin{align}\label{equ:Delta theta n L2 euler}
		\norm{\Delta_h \bff{\theta}^{n}}{\bb{L}^2}^2
		+
		\norm{\bff{\theta}^{n}}{\bb{L}^\infty}^2
		+
		k\sum_{m=0}^{n-1} \norm{\Delta_h \bff{\xi}^{m+1}}{\bb{L}^2}^2
		+
		k\sum_{m=0}^{n-1} \norm{\bff{\xi}^{m+1}}{\bb{L}^\infty}^2
		&\leq
		C \big(h^{2(r+1)}+k^2 \big),
		\\
		\label{equ:nab xi n L2 euler}
		\norm{\nabla \bff{\xi}^n}{\bb{L}^2}^2
		+
		k \sum_{m=0}^{n-1} \norm{\nabla \Delta_h \bff{\xi}^{m+1}}{\bb{L}^2}^2
		&\leq
		C \big(h^{2r}+k^2 \big).
	\end{align}
	The constant $C$ depends on the coefficients of the equation, $|\mathscr{D}|$, $T$, and $K_0$ (as defined in \eqref{equ:ass 2 u}), but is independent of $n$, $h$, and $k$.
\end{proposition}

\begin{proof}
	Subtracting \eqref{equ:weakform} from \eqref{equ:euler} at time step $n+1$, using \eqref{equ:Un utn}, \eqref{equ:Hn Htn} (and noting the definition of Ritz projection), we obtain for all $\bff{\chi}$, $\bff{\phi}\in \bb{V}_h$,
	\begin{align}\label{equ:dt theta disc euler}
		\nonumber
		\inpro{\delta \bff{\theta}^{n+1} + \delta \bff{\rho}^{n+1} + \delta \bff{u}^{n+1}- \partial_t \bff{u}^{n+1}}{\bff{\chi}}
		&=
		\lambda_r \inpro{\bff{\xi}^{n+1}+\bff{\eta}^{n+1}}{\bff{\chi}}
		+
		\lambda_e \inpro{\nabla \bff{\xi}^{n+1}}{\nabla \bff{\chi}}
		\\
		&\quad 
		-
		\gamma \inpro{\bff{u}_h^n \times \bff{H}_h^{n+1} - \bff{u}^{n+1} \times \bff{H}^{n+1}}{\bff{\chi}}
	\end{align}
	and
	\begin{align}
		\nonumber
		\label{equ:xi eta disc euler}
		\inpro{\bff{\xi}^{n+1}+\bff{\eta}^{n+1}}{\bff{\phi}}
		&=
		-\inpro{\nabla \bff{\theta}^{n+1}}{\nabla \bff{\phi}}
		+
		\kappa \inpro{\bff{\theta}^n+\bff{\rho}^n+\bff{u}^n-\bff{u}^{n+1}}{\bff{\phi}}
		\\
		&\quad 
		-
		\kappa \inpro{\abs{\bff{u}_h^{n+1}}^2 \bff{u}_h^{n+1} - \abs{\bff{u}^{n+1}}^2 \bff{u}^{n+1}}{\bff{\phi}}
		-
		\beta \inpro{\bff{e}\big(\bff{e}\cdot (\bff{\theta}^{n+1}+\bff{\rho}^{n+1})\big)}{\bff{\phi}}.
	\end{align}
	Taking $\bff{\chi}= \bff{\theta}^{n+1}$ in \eqref{equ:dt theta disc euler} and $\bff{\phi}=\lambda_r \bff{\theta}^{n+1}$ in \eqref{equ:xi eta disc euler}, then adding the resulting equations, we obtain
	\begin{align}\label{equ:dt theta chi disc euler}
		\nonumber
		&\frac{1}{2k} \left( \norm{\bff{\theta}^{n+1}}{\bb{L}^2}^2 - \norm{\bff{\theta}^n}{\bb{L}^2}^2 \right)
		+
		\frac{1}{2k} \norm{\bff{\theta}^{n+1}-\bff{\theta}^n}{\bb{L}^2}^2
		+
		\inpro{\delta \bff{\rho}^{n+1} + \delta \bff{u}^{n+1} -\partial_t \bff{u}^{n+1}}{\bff{\theta}^{n+1}}
		+
		\lambda_r \norm{\nabla \bff{\theta}^{n+\frac{1}{2}}}{\bb{L}^2}^2
		\\
		\nonumber
		&=
		\kappa \lambda_r \inpro{\bff{\theta}^n}{\bff{\theta}^{n+1}}
		+
		\kappa\lambda_r \inpro{\bff{\rho}^{n}}{\bff{\theta}^{n+1}}
		+
		\kappa\lambda_r \inpro{\bff{u}^n-\bff{u}^{n+1}}{\bff{\theta}^{n+1}}
		+
		\lambda_e \inpro{\nabla \bff{\xi}^{n+1}}{\nabla \bff{\theta}^{n+1}}
		\\
		&\quad
		-
		\gamma \inpro{\bff{u}_h^n \times \bff{H}_h^{n+1} - \bff{u}^{n+1} \times \bff{H}^{n+1}}{\bff{\theta}^{n+1}}
		-
		\kappa \lambda_r \inpro{\abs{\bff{u}_h^{n+1}}^2 \bff{u}_h^{n+1} - \abs{\bff{u}^{n+1}}^2 \bff{u}^{n+1}}{\bff{\theta}^{n+1}}
		\nonumber \\
		&\quad 
		-
		\beta \lambda_r \norm{\bff{e}\cdot \bff{\theta}^{n+1}}{L^2}^2
		-
		\beta \lambda_r \inpro{\bff{e}(\bff{e}\cdot \bff{\rho}^{n+1})}{\bff{\theta}^{n+1}}.
	\end{align}
	Next, taking $\bff{\phi}= \lambda_e \bff{\xi}^{n+1}$, we obtain
	\begin{align}\label{equ:xi phi xi disc euler}
		\nonumber
		\lambda_e \norm{\bff{\xi}^{n+1}}{\bb{L}^2}^2
		+
		\lambda_e \inpro{\bff{\eta}^{n+1}}{\bff{\xi}^{n+1}}
		&=
		-
		\lambda_e \inpro{\nabla \bff{\theta}^{n+1}}{\nabla \bff{\xi}^{n+1}}
		+
		\kappa \lambda_e \inpro{\bff{\theta}^n+\bff{\rho}^n+\bff{u}^n-\bff{u}^{n+1}}{\bff{\xi}^{n+1}}
		\nonumber \\
		&\quad 
		-
		\kappa \lambda_e \inpro{\abs{\bff{u}_h^{n+1}}^2 \bff{u}_h^{n+1} - \abs{\bff{u}^{n+1}}^2 \bff{u}^{n+1}}{\bff{\xi}^{n+1}}
		\nonumber \\
		&\quad
		-
		\beta\lambda_e \inpro{\bff{e}(\bff{e}\cdot \bff{\theta}^{n+1})}{\bff{\xi}^{n+1}}
		-
		\beta \lambda_e \inpro{\bff{e}(\bff{e}\cdot \bff{\rho}^{n+1})}{\bff{\xi}^{n+1}}.
 	\end{align}
	Adding \eqref{equ:dt theta chi disc euler} and \eqref{equ:xi phi xi disc euler} gives
	\begin{align*}
		&\frac{1}{2k} \left( \norm{\bff{\theta}^{n+1}}{\bb{L}^2}^2 - \norm{\bff{\theta}^n}{\bb{L}^2}^2 \right)
		+
		\frac{1}{2k} \norm{\bff{\theta}^{n+1}-\bff{\theta}^n}{\bb{L}^2}^2
		+
		\lambda_r \norm{\nabla \bff{\theta}^{n+1}}{\bb{L}^2}^2
		+
		\lambda_e \norm{\bff{\xi}^{n+1}}{\bb{L}^2}^2
		+
		\beta \lambda_r \norm{\bff{e}\cdot \bff{\theta}^{n+1}}{L^2}^2
		\\
		&=
		-\inpro{\delta \bff{\rho}^{n+1} + \delta \bff{u}^{n+1} -\partial_t \bff{u}^{n+1}}{\bff{\theta}^{n+1}}
		-
		\lambda_e \inpro{\bff{\eta}^{n+1}}{\bff{\xi}^{n+1}}
		+
		\kappa \lambda_e \inpro{\bff{\theta}^n+\bff{\rho}^n+\bff{u}^n-\bff{u}^{n+1}}{\bff{\xi}^{n+1}}
		\\
		&\quad
		+
		\kappa \lambda_r \inpro{\bff{\theta}^n+\bff{\rho}^n}{\bff{\theta}^{n+1}}
		+
		\kappa\lambda_r \inpro{\bff{u}^n-\bff{u}^{n+1}}{\bff{\theta}^{n+1}}
		-
		\beta\lambda_e \inpro{\bff{e}(\bff{e}\cdot \bff{\theta}^{n+1})}{\bff{\xi}^{n+1}}
		\\
		&\quad
		-
		\beta\lambda_e \inpro{\bff{e}(\bff{e}\cdot \bff{\theta}^{n+1})}{\bff{\xi}^{n+1}}
		-
		\beta \lambda_r \inpro{\bff{e}(\bff{e}\cdot \bff{\rho}^{n+1})}{\bff{\theta}^{n+1}}
		\\
		&\quad 
		-
		\gamma \inpro{\bff{u}_h^n \times \bff{H}_h^{n+1} - \bff{u}^{n+1} \times \bff{H}^{n+1}}{\bff{\theta}^{n+1}}
		-
		\kappa \lambda_r \inpro{\abs{\bff{u}_h^{n+1}}^2 \bff{u}_h^{n+1} - \abs{\bff{u}^{n+1}}^2 \bff{u}^{n+1}}{\bff{\theta}^{n+1}}
		\\
		&\quad
		-
		\kappa \lambda_e \inpro{\abs{\bff{u}_h^{n+1}}^2 \bff{u}_h^{n+1} - \abs{\bff{u}^{n+1}}^2 \bff{u}^{n+1}}{\bff{\xi}^{n+1}}.
	\end{align*}
	It remains to bound each term on the right-hand side. We apply \eqref{equ:delta un min u euler}, Young's inequality, and Taylor's theorem as necessary for the first eight terms. The last three terms can be estimated by applying Lemma~\ref{lem:cross disc euler} and Lemma~\ref{lem:cub disc euler}. Hence, we infer that for any $\epsilon>0$,
	\begin{align*}
		&\frac{1}{2k} \left( \norm{\bff{\theta}^{n+1}}{\bb{L}^2}^2 - \norm{\bff{\theta}^n}{\bb{L}^2}^2 \right)
		+
		\frac{1}{2k} \norm{\bff{\theta}^{n+1}-\bff{\theta}^n}{\bb{L}^2}^2
		+
		\lambda_r \norm{\nabla \bff{\theta}^{n+1}}{\bb{L}^2}^2
		+
		\lambda_e \norm{\bff{\xi}^{n+1}}{\bb{L}^2}^2
		\\
		&\lesssim
		h^{2(r+1)}+k^2
		+
		\norm{\bff{\theta}^{n+1}}{\bb{L}^2}^2
		+
		\epsilon \norm{\bff{\theta}^n}{\bb{L}^2}^2
		+
		\epsilon \norm{\nabla \bff{\theta}^{n+1}}{\bb{L}^2}^2
		+
		\epsilon \norm{\bff{\xi}^{n+1}}{\bb{L}^2}^2.
	\end{align*}
	Choosing $\epsilon>0$ sufficiently small, rearranging the terms, and summing over $m\in \{0,1, \ldots, n-1\}$, we have
	\begin{align*}
		\norm{\bff{\theta}^n}{\bb{L}^2}^2
		+
		k \sum_{m=0}^{n-1} \norm{\nabla \bff{\theta}^{m+1}}{\bb{L}^2}^2
		+
		k \sum_{m=0}^{n-1} \norm{\bff{\xi}^{m+1}}{\bb{L}^2}^2
		\leq
		C \big(h^{2(r+1)} +k^2 \big)
		+
		Ck \sum_{m=0}^{n-1} \norm{\bff{\theta}^m}{\bb{L}^2}^2,
	\end{align*}
	where $C$ is a constant depending on the coefficients of the equation and $T$ (but is independent of $n$, $h$, and $k$). The discrete Gronwall inequality then yields \eqref{equ:theta n L2 euler}.
	
	We aim to prove \eqref{equ:xi nab theta n L2 euler} next. First, consider the difference of the second equation in \eqref{equ:euler} at time steps $n+1$ and $n$. After dividing the result by $k$, subtracting it from the corresponding equation in \eqref{equ:weakform} and setting $\bff{\phi}=\lambda_e \bff{\xi}^{n+1}$, we obtain
	\begin{align}\label{equ:dt xi n euler}
		\nonumber
		&\frac{\lambda_e}{2k} \left(\norm{\bff{\xi}^{n+1}}{\bb{L}^2}^2-\norm{\bff{\xi}^n}{\bb{L}^2}^2\right)
		+
		\frac{\lambda_e}{2k} \norm{\bff{\xi}^{n+1}-\bff{\xi}^n}{\bb{L}^2}^2
		\\
		\nonumber
		&=
		-\lambda_e \inpro{\delta \bff{\eta}^{n+1}}{\bff{\xi}^{n+1}}
		-
		\lambda_e \inpro{\delta\bff{H}^{n+1}-\partial_t \bff{H}^{n+1}}{\bff{\xi}^{n+1}}
		\\
		\nonumber
		&\quad
		-\lambda_e \inpro{\nabla \delta \bff{\theta}^{n+1}}{\nabla \bff{\xi}^{n+1}}
		-
		\lambda_e \inpro{\nabla \delta \bff{u}^{n+1}-\nabla \partial_t \bff{u}^{n+1}}{\nabla \bff{\xi}^{n+1}}
		+
		\kappa \mu \lambda_e \inpro{\delta \bff{u}_h^n -\partial_t \bff{u}^{n+1}}{\bff{\xi}^{n+1}}
		\\
		&\quad
		-
		\kappa \lambda_e
		\inpro{\frac{\abs{\bff{u}_h^{n+1}}^2 \bff{u}_h^{n+1} - \abs{\bff{u}_h^n}^2 \bff{u}_h^n}{k} - \partial_t \left(\abs{\bff{u}^{n+1}}^2 \bff{u}^{n+1}\right)}{\bff{\xi}^{n+1}}
		\nonumber \\
		&\quad
		-
		\beta\lambda_e \inpro{\bff{e}\big(\bff{e}\cdot (\delta \bff{u}_h^{n+1}-\partial_t \bff{u}^{n+1})\big)}{\bff{\xi}^{n+1}}.
	\end{align}
	Next, setting $\bff{\chi}=\bff{\xi}^{n+1}$ gives
	\begin{align}\label{equ:xi nab xi euler}
		\nonumber
		\lambda_r \norm{\bff{\xi}^{n+1}}{\bb{L}^2}^2
		+
		\lambda_e \norm{\nabla \bff{\xi}^{n+1}}{\bb{L}^2}^2
		&=
		\inpro{\delta \bff{\theta}^{n+1}}{\bff{\xi}^{n+1}}
		+
		\inpro{\delta \bff{\rho}^{n+1}+\delta \bff{u}^{n+1}-\partial_t \bff{u}^{n+1}}{\bff{\xi}^{n+1}}
		\\
		&\quad
		-
		\lambda_r \inpro{\bff{\eta}^{n+1}}{\bff{\xi}^{n+1}}
		+
		\gamma \inpro{\bff{u}_h^n\times \bff{H}_h^{n+1}- \bff{u}^{n+1}\times \bff{H}^{n+1}}{\bff{\xi}^{n+1}},
	\end{align}
	while taking $\bff{\chi}=\delta \bff{\theta}^{n+1}$ yields
	\begin{align}\label{equ:d theta n euler}
		\nonumber
		\norm{\delta \bff{\theta}^{n+1}}{\bb{L}^2}^2
		&=
		-
		\inpro{\delta \bff{\rho}^{n+1}+\delta \bff{u}^{n+1}-\partial_t \bff{u}^{n+1}}{\delta \bff{\theta}^{n+1}}
		+
		\lambda_r \inpro{\bff{\xi}^{n+1}+\bff{\eta}^{n+1}}{\delta \bff{\theta}^{n+1}}
		\\
		&\quad
		+
		\lambda_e \inpro{\nabla \bff{\xi}^{n+1}}{\nabla \delta \bff{\theta}^{n+1}}
		-
		\gamma \inpro{\bff{u}_h^n\times \bff{H}_h^{n+1}- \bff{u}^{n+1}\times \bff{H}^{n+1}}{\delta\bff{\theta}^{n+1}}.
	\end{align}
	Furthermore, setting $\bff{\phi}=\lambda_r \delta \bff{\theta}^{n+1}$ in \eqref{equ:xi eta disc euler} gives
	\begin{align}\label{equ:xi d theta n euler}
		\nonumber
		&\frac{\lambda_r}{2k} \left(\norm{\nabla \bff{\theta}^{n+1}}{\bb{L}^2}^2 -\norm{\nabla \bff{\theta}^n}{\bb{L}^2}^2\right)
		+
		\frac{\lambda_r}{2k} \norm{\nabla \bff{\theta}^{n+1}-\nabla \bff{\theta}^n}{\bb{L}^2}^2
		\\
		\nonumber
		&=
		-\lambda_r \inpro{\bff{\xi}^{n+1}+\bff{\eta}^{n+1}}{\delta \bff{\theta}^{n+1}}
		+
		\kappa\mu\lambda_r \inpro{\bff{\theta}^{n+1}+\bff{\rho}^{n+1}}{\delta \bff{\theta}^{n+1}}
		\\
		&\quad
		-
		\kappa\lambda_r \inpro{\abs{\bff{u}_h^{n+1}}^2 \bff{u}_h^{n+1}- \abs{\bff{u}^{n+1}}^2 \bff{u}^{n+1}}{\delta \bff{\theta}^{n+1}}
		-
		\beta\lambda_r \inpro{\bff{e}\big(\bff{e}\cdot (\bff{\theta}^{n+1}+\bff{\rho}^{n+1})\big)}{\delta \bff{\theta}^{n+1}}.
	\end{align}
	We now add~\eqref{equ:dt xi n euler}, \eqref{equ:xi nab xi euler}, \eqref{equ:d theta n euler}, and \eqref{equ:xi d theta n euler} to obtain
	\begin{align}\label{equ:I1 to I13 euler}
		\nonumber
		&\frac{\lambda_r}{2k} \left(\norm{\nabla \bff{\theta}^{n+1}}{\bb{L}^2}^2 -\norm{\nabla \bff{\theta}^n}{\bb{L}^2}^2\right)
		+
		\frac{\lambda_r}{2k} \norm{\nabla \bff{\theta}^{n+1}-\nabla \bff{\theta}^n}{\bb{L}^2}^2
		+
		\frac{\lambda_e}{2k} \left(\norm{\bff{\xi}^{n+1}}{\bb{L}^2}^2-\norm{\bff{\xi}^n}{\bb{L}^2}^2\right)
		+
		\frac{\lambda_e}{2k} \norm{\bff{\xi}^{n+1}-\bff{\xi}^n}{\bb{L}^2}^2
		\\
		\nonumber
		&\quad
		+
		\lambda_r \norm{\bff{\xi}^{n+1}}{\bb{L}^2}^2
		+
		\lambda_e \norm{\nabla \bff{\xi}^{n+1}}{\bb{L}^2}^2
		+
		\norm{\delta \bff{\theta}^{n+1}}{\bb{L}^2}^2
		\\
		\nonumber
		&=
		-\lambda_e \inpro{\delta \bff{\eta}^{n+1}}{\bff{\xi}^{n+1}}
		-
		\lambda_e \inpro{\delta\bff{H}^{n+1}-\partial_t \bff{H}^{n+1}}{\bff{\xi}^{n+1}}
		\\
		\nonumber
		&\quad
		-
		\lambda_e \inpro{\nabla \delta \bff{u}^{n+1}-\nabla \partial_t \bff{u}^{n+1}}{\nabla \bff{\xi}^{n+1}}
		+
		\kappa \mu \lambda_e \inpro{\delta \bff{u}_h^n -\partial_t \bff{u}^{n+1}}{\bff{\xi}^{n+1}}
		\\
		\nonumber
		&\quad
		-
		\kappa \lambda_e
		\inpro{\frac{\abs{\bff{u}_h^{n+1}}^2 \bff{u}_h^{n+1} - \abs{\bff{u}_h^n}^2 \bff{u}_h^n}{k} - \partial_t \left(\abs{\bff{u}^{n+1}}^2 \bff{u}^{n+1}\right)}{\bff{\xi}^{n+1}}
		\\
		\nonumber
		&\quad
		+
		\inpro{\delta \bff{\theta}^{n+1}}{\bff{\xi}^{n+1}}
		+
		\inpro{\delta \bff{\rho}^{n+1}+\delta \bff{u}^{n+1}-\partial_t \bff{u}^{n+1}}{\bff{\xi}^{n+1}}
		\\
		\nonumber
		&\quad
		-
		\lambda_r \inpro{\bff{\eta}^{n+1}}{\bff{\xi}^{n+1}}
		+
		\gamma \inpro{\bff{u}_h^n\times \bff{H}_h^{n+1}- \bff{u}^{n+1}\times \bff{H}^{n+1}}{\bff{\xi}^{n+1}}
		\\
		\nonumber
		&\quad
		-
		\inpro{\delta \bff{\rho}^{n+1}+\delta \bff{u}^{n+1}-\partial_t \bff{u}^{n+1}}{\delta \bff{\theta}^{n+1}}
		-
		\gamma \inpro{\bff{u}_h^n\times \bff{H}_h^{n+1}- \bff{u}^{n+1}\times \bff{H}^{n+1}}{\delta\bff{\theta}^{n+1}}
		\\
		\nonumber
		&\quad
		+
		\kappa\mu\lambda_r \inpro{\bff{\theta}^{n+1}+\bff{\rho}^{n+1}}{\delta \bff{\theta}^{n+1}}
		-
		\kappa\lambda_r \inpro{\abs{\bff{u}_h^{n+1}}^2 \bff{u}_h^{n+1}- \abs{\bff{u}^{n+1}}^2 \bff{u}^{n+1}}{\delta \bff{\theta}^{n+1}}
		\\
		&\quad
		-
		\beta\lambda_e \inpro{\bff{e}\big(\bff{e}\cdot (\delta \bff{u}_h^{n+1}-\partial_t \bff{u}^{n+1})\big)}{\bff{\xi}^{n+1}}
		-
		\beta\lambda_r \inpro{\bff{e}\big(\bff{e}\cdot (\bff{\theta}^{n+1}+\bff{\rho}^{n+1})\big)}{\delta \bff{\theta}^{n+1}}
		\nonumber \\
		&=:
		I_1+I_2+\cdots+I_{15}.
	\end{align}
	There are fifteen terms involving inner products on the right-hand side of~\eqref{equ:I1 to I13 euler}, which we will estimate in the following. Let $\epsilon>0$ be a number. Firstly, by Young's inequality, \eqref{equ:Ritz ineq}, and \eqref{equ:delta un Lp euler}, 
	\begin{align}\label{equ:I1 euler}
		\abs{I_1}
		\lesssim
		\norm{\delta \bff{\eta}^{n+1}}{\bb{L}^2}^2
		+
		\epsilon \norm{\bff{\xi}^{n+1}}{\bb{L}^2}^2
		\lesssim
		h^{2(r+1)} 
		+
		\epsilon \norm{\bff{\xi}^{n+1}}{\bb{L}^2}^2.
	\end{align}
	Secondly, by Young's inequality and Taylor's theorem (noting the assumption on $\bff{H}$ in~\eqref{equ:ass 2 u}),
	\begin{align*}
		\abs{I_2}
		\lesssim
		\norm{\delta\bff{H}^{n+1}-\partial_t \bff{H}^{n+1}}{\bb{L}^2}^2
		+
		\epsilon \norm{\bff{\xi}^{n+1}}{\bb{L}^2}^2
		\lesssim
		k^2
		+
		\epsilon \norm{\bff{\xi}^{n+1}}{\bb{L}^2}^2.
	\end{align*}
	By similar argument (noting~\eqref{equ:ass 2 u}), the third term can be estimated as
	\begin{align*}
		\abs{I_3}
		\lesssim
		\norm{\nabla \delta\bff{u}^{n+1}-\nabla \partial_t \bff{u}^{n+1}}{\bb{L}^2}^2
		+
		\epsilon \norm{\nabla \bff{\xi}^{n+1}}{\bb{L}^2}^2
		\lesssim
		k^2
		+
		\epsilon \norm{\nabla \bff{\xi}^{n+1}}{\bb{L}^2}^2.
	\end{align*}
	For the fourth term, similarly by Young's inequality, \eqref{equ:Ritz ineq}, \eqref{equ:delta un Lp euler}, and \eqref{equ:delta un min u euler}, we have
	\begin{align}\label{equ:I4 euler}
		\abs{I_4}
		\lesssim
		\norm{\bff{\xi}^{n+1}}{\bb{L}^2}^2
		+
		\epsilon \norm{\delta \bff{u}_h^n-\partial_t \bff{u}^{n+1}}{\bb{L}^2}^2
		&\lesssim
		\norm{\bff{\xi}^{n+1}}{\bb{L}^2}^2
		+ h^{2(r+1)} + k^2
		+
		\epsilon \norm{\delta \bff{\theta}^{n+1}}{\bb{L}^2}^2.
	\end{align}
	For the fifth term, we use Lemma~\ref{lem:psi psi dt un euler}. Noting that by assumptions \eqref{equ:Hn H1 stab euler} holds, so that
	\[
		\norm{\bff{\xi}^n}{\bb{H}^1}^2 
		\leq
		\norm{\bff{H}_h^n}{\bb{H}^1}^2
		+
		\norm{R_h \bff{H}_h^n}{\bb{H}^1}^2 \leq C,
	\]
	we then obtain (noting \eqref{equ:theta n L2 euler})
	\begin{align*}
		\nonumber
		\abs{I_5}
		&\lesssim
		h^{2(r+1)} + k^2
		+
		\left(1+ \norm{\bff{\xi}^{n+1}}{\bb{H}^1}^2 + \norm{\bff{\xi}^n}{\bb{H}^1}^2\right) \left(\norm{\bff{\theta}^{n+1}}{\bb{H}^1}^2 + \norm{\bff{\theta}^n}{\bb{H}^1}^2 \right)
		\\
		\nonumber
		&\quad
		+
		\norm{\bff{\xi}^n}{\bb{L}^2}^2
		+
		\epsilon \norm{\bff{\xi}^{n+1}}{\bb{L}^2}^2
		+
		\epsilon \norm{\delta \bff{\theta}^{n+1}}{\bb{L}^2}^2
		\\
		&\lesssim
		h^{2(r+1)} + k^2
		+
		\norm{\nabla \bff{\theta}^{n+1}}{\bb{L}^2}^2 + \norm{\nabla \bff{\theta}^n}{\bb{L}^2}^2
		+
		\norm{\bff{\xi}^n}{\bb{L}^2}^2
		+
		\epsilon \norm{\bff{\xi}^{n+1}}{\bb{L}^2}^2
		+
		\epsilon \norm{\delta \bff{\theta}^{n+1}}{\bb{L}^2}^2.
	\end{align*}
	For the terms $I_6$ and $I_7$, the same argument as in \eqref{equ:I4 euler} yields
	\begin{align*}
		\abs{I_6}+\abs{I_7}
		\lesssim
		\norm{\bff{\xi}^{n+1}}{\bb{L}^2}^2
		+
		\epsilon \norm{\delta \bff{\theta}^{n+1}}{\bb{L}^2}^2
		+
		h^{2(r+1)} + k^2.
	\end{align*}
	Similarly, for the terms $I_8$ and $I_{10}$, we have
	\begin{align*}
		\abs{I_8}+\abs{I_{10}}
		\lesssim
		\norm{\bff{\xi}^{n+1}}{\bb{L}^2}^2
		+
		\epsilon \norm{\delta \bff{\theta}^{n+1}}{\bb{L}^2}^2
		+
		h^{2(r+1)} + k^2.
	\end{align*}
	Next, we estimate the terms $I_9$ and $I_{11}$. 
	By \eqref{equ:cross xi euler} and \eqref{equ:cross dt theta euler}, noting \eqref{equ:Hn H1 stab euler} and \eqref{equ:theta n L2 euler}, we have
	\begin{align*}
		\abs{I_9}+\abs{I_{11}}
		\lesssim
		h^{2(r+1)} 
		+ 
		k^2
		+
		\norm{\bff{\xi}^{n+1}}{\bb{L}^2}^2 
		+ 
		\epsilon \norm{\nabla \bff{\xi}^{n+1}}{\bb{L}^2}^2
		+
		\epsilon \norm{\nabla \bff{\theta}^n}{\bb{L}^2}^2
		+
		\epsilon \norm{\delta \bff{\theta}^{n+1}}{\bb{L}^2}^2.
	\end{align*}
	The term $I_{12}$ can be estimated using Young's inequality and \eqref{equ:theta n L2 euler} as
	\begin{align}\label{equ:I12 euler}
		\abs{I_{12}}
		\lesssim
		h^{2(r+1)}+k^2 + \epsilon \norm{\delta \bff{\theta}^{n+1}}{\bb{L}^2}^2.
	\end{align}
	To estimate the term $I_{13}$, note that by \eqref{equ:Un utn} and the triangle inequality \eqref{equ:dt un stab euler},
	\[
		\norm{\delta \bff{\theta}^{n+1}}{\bb{L}^2}^2
		\leq
		2\norm{\delta \bff{u}_h^{n+1}}{\bb{L}^2}^2
		+
		2\norm{\delta \bff{u}^{n+1}}{\bb{L}^2}^2
		+
		2\norm{\delta \bff{\rho}^{n+1}}{\bb{L}^2}^2 \leq C,
	\]
	where in the last step we used \eqref{equ:dt un stab euler}, \eqref{equ:ass 2 u}, \eqref{equ:delta un Lp euler}, and \eqref{equ:Ritz ineq}. Thus, applying \eqref{equ:uh2uh dt theta euler}, we obtain
	\begin{align*}
		\abs{I_{13}}
		\lesssim
		h^{2(r+1)}
		+
		\norm{\nabla \bff{\theta}^{n+1}}{\bb{L}^2}^2  
		+ 
		\epsilon \norm{\delta \bff{\theta}^{n+1}}{\bb{L}^2}^2.
	\end{align*}
	Finally, the last two terms can be estimated following~\eqref{equ:I4 euler} and \eqref{equ:I12 euler} to obtain
	\begin{align}\label{equ:I14 euler}
		|I_{14}|+|I_{15}|
		\lesssim
		\norm{\bff{\xi}^{n+1}}{\bb{L}^2}^2
		+ h^{2(r+1)} + k^2
		+
		\epsilon \norm{\delta \bff{\theta}^{n+1}}{\bb{L}^2}^2.
	\end{align}
	Altogether, continuing from \eqref{equ:I1 to I13 euler}, using estimates \eqref{equ:I1 euler}--\eqref{equ:I14 euler}, choosing $\epsilon>0$ sufficiently small, and rearranging the terms, we obtain
	\begin{align}\label{equ:lambda r 2k}
		\nonumber
		&\frac{\lambda_r}{2k} \left(\norm{\nabla \bff{\theta}^{n+1}}{\bb{L}^2}^2 -\norm{\nabla \bff{\theta}^n}{\bb{L}^2}^2\right)
		+
		\frac{\lambda_e}{2k} \left(\norm{\bff{\xi}^{n+1}}{\bb{L}^2}^2-\norm{\bff{\xi}^n}{\bb{L}^2}^2\right)
		+
		\frac{\lambda_r}{2} \norm{\bff{\xi}^{n+1}}{\bb{L}^2}^2
		+
		\frac{\lambda_e}{2} \norm{\nabla \bff{\xi}^{n+1}}{\bb{L}^2}^2
		+
		\frac{1}{2} \norm{\delta \bff{\theta}^{n+1}}{\bb{L}^2}^2
		\\
		&\leq
		C\big(h^{2(r+1)}+k^2\big)
		+
		C\norm{\nabla \bff{\theta}^{n+1}}{\bb{L}^2}^2
		+
		C\norm{\nabla \bff{\theta}^n}{\bb{L}^2}^2
		+
		C\norm{\bff{\xi}^{n+1}}{\bb{L}^2}^2
		+
		C\norm{\bff{\xi}^n}{\bb{L}^2}^2.
	\end{align}
	Multiplying both sides of \eqref{equ:lambda r 2k} by $2k$, summing over $m\in \{0,1,\ldots,n-1\}$, and applying the discrete Gronwall inequality, we obtain \eqref{equ:xi nab theta n L2 euler}.
	
	Finally, suppose that the triangulation is globally quasi-uniform. Setting $\bff{\phi}=\Delta_h \bff{\theta}^{n+1}$ in \eqref{equ:xi eta disc euler}, we then apply Young's inequality and~\eqref{equ:psi xi Delta euler} with $\bff{\zeta}=\Delta_h \bff{\theta}^{n+1}$ to infer
	\begin{align*}
		\norm{\Delta_h \bff{\theta}^{n+1}}{\bb{L}^2}^2
		&=
		\inpro{\bff{\xi}^{n+1}+\bff{\eta}^{n+1}}{\Delta_h \bff{\theta}^{n+1}}
		-
		\kappa \inpro{\bff{\theta}^n+\bff{\rho}^n+\bff{u}^n-\bff{u}^{n+1}}{\Delta_h \bff{\theta}^{n+1}}
		\\
		&\quad
		-
		\kappa \inpro{\abs{\bff{u}_h^{n+1}}^2 \bff{u}_h^{n+1} - \abs{\bff{u}^{n+1}}^2 \bff{u}^{n+1}}{\Delta_h \bff{\theta}^{n+1}}
		-
		\beta \inpro{\bff{e}\big(\bff{e}\cdot (\bff{\theta}^{n+1}+\bff{\rho}^{n+1})\big)}{\Delta_h \bff{\theta}^{n+1}}
		\\
		&\lesssim
		\norm{\bff{\xi}^{n+1}+\bff{\eta}^{n+1}}{\bb{L}^2}^2
		+
		\epsilon \norm{\Delta_h \bff{\theta}^{n+1}}{\bb{L}^2}^2
		+
		\norm{\bff{\theta}^n+\bff{\rho}^n}{\bb{L}^2}^2
		+
		\norm{\bff{u}^n-\bff{u}^{n+1}}{\bb{L}^2}^2
		+
		\epsilon \norm{\Delta_h \bff{\theta}^{n+1}}{\bb{L}^2}^2
		\\
		&\quad
		+
		h^{2(r+1)}+ \norm{\bff{\theta}^{n+1}}{\bb{L}^2}^2
		+
		\epsilon \norm{\Delta_h \bff{\theta}^{n+1}}{\bb{L}^2}^2
		\\
		&\lesssim
		h^{2(r+1)}+k^2 +\epsilon \norm{\Delta_h \bff{\theta}^{n+1}}{\bb{L}^2}^2,
	\end{align*}
	where in the last step we used inequalities \eqref{equ:theta n L2 euler} and \eqref{equ:xi nab theta n L2 euler}, \eqref{equ:Ritz ineq}, and Taylor's theorem. Choosing $\epsilon>0$ sufficiently small and rearranging the terms, we obtain
	\begin{equation}\label{equ:Delta theta leq}
		\norm{\Delta_h \bff{\theta}^{n+1}}{\bb{L}^2}^2
		\lesssim
		h^{2(r+1)}+k^2.
	\end{equation}
	Next, taking $\bff{\chi}=\Delta_h \bff{\xi}^{n+1}$ and applying similar argument (using~\eqref{equ:cross est euler} with $\bff{\zeta}=\Delta_h \bff{\xi}^{n+1}$), we obtain
	\begin{equation}\label{equ:k sum Delta xi}
		k\sum_{m=0}^{n-1} \norm{\Delta_h \bff{\xi}^{m+1}}{\bb{L}^2}^2
		\lesssim
		h^{2(r+1)}+k^2.
	\end{equation}
	Inequality \eqref{equ:Delta theta n L2 euler} then follows from \eqref{equ:Delta theta leq}, \eqref{equ:k sum Delta xi}, and \eqref{equ:disc lapl L infty}.
	
	Finally, setting $\bff{\chi}=-\Delta_h^2 \bff{\xi}^{n+1}$ in \eqref{equ:dt theta disc euler} and applying \eqref{equ:disc laplacian} as necessary, we have
	\begin{align}\label{equ:nab Delta xi euler}
		\nonumber
		&\lambda_r \norm{\Delta_h \bff{\xi}^{n+1}}{\bb{L}^2}^2
		+
		\lambda_e \norm{\nabla \Delta_h \bff{\xi}}{\bb{L}^2}^2
		\\
		\nonumber
		&=
		-
		\inpro{\delta\nabla \bff{\theta}^{n+1}}{\nabla \Delta_h \bff{\xi}^{n+1}}
		-
		\inpro{\nabla \Pi_h \left(\delta \bff{u}^{n+1}-\partial_t \bff{u}^{n+1}\right)}{\nabla \Delta_h \bff{\xi}^{n+1}}
		-
		\inpro{\nabla \Pi_h \delta \bff{\rho}^{n+1}}{\nabla \Delta_h \bff{\xi}^{n+1}}
		\\
		&\quad
		-
		\lambda_r \inpro{\nabla \Pi_h \bff{\eta}^{n+1}}{\nabla \Delta_h \bff{\xi}^{n+1}}
		+
		\gamma \inpro{\nabla \Pi_h(\bff{u}_h^n \times \bff{H}_h^{n+1} - \bff{u}^{n+1} \times \bff{H}^{n+1})}{\nabla \Delta_h \bff{\xi}}.
	\end{align}
	Next, consider the difference of the second equation in \eqref{equ:euler} at time steps $n+1$ and $n$. After dividing the result by $k$, subtracting it from the corresponding equation in \eqref{equ:weakform} and taking $\bff{\phi}=-\Delta_h \bff{\xi}^{n+1}$, we have
	\begin{align}\label{equ:ddt nab xi euler}
		\nonumber
		&\frac{1}{2k} \left(\norm{\nabla \bff{\xi}^{n+1}}{\bb{L}^2}^2 - \norm{\nabla \bff{\xi}^n}{\bb{L}^2}^2 \right)
		+
		\frac{1}{2k} \norm{\nabla \bff{\xi}^{n+1}-\nabla \bff{\xi}^n}{\bb{L}^2}^2
		\\
		\nonumber
		&=
		\inpro{\nabla \Pi_h \delta \bff{\eta}^{n+1}}{\Delta_h \bff{\xi}^{n+1}}
		+
		\inpro{\delta \nabla \bff{\theta}^{n+1}}{\nabla \Delta_h \bff{\xi}^{n+1}}
		-
		\kappa \mu\inpro{\delta \bff{\theta}^n+\delta \bff{\rho}^n+\delta \bff{u}^n-\partial_t \bff{u}^{n+1}}{\Delta_h \bff{\xi}^{n+1}}
		\\
		&\quad
		+
		\kappa \inpro{\frac{\abs{\bff{u}_h^{n+1}}^2 \bff{u}_h^{n+1} - \abs{\bff{u}_h^n}^2 \bff{u}_h^n}{k} - \partial_t \left(\abs{\bff{u}^{n+1}}^2 \bff{u}^{n+1}\right)}{\Delta_h \bff{\xi}^{n+1}}
		\nonumber \\
		&\quad
		-
		\beta\inpro{\bff{e}\big( \bff{e}\cdot(\bff{\theta}^{n+1}+\bff{\rho}^{n+1})\big)}{\Delta_h \bff{\xi}^{n+1}}.
	\end{align}
	Adding \eqref{equ:nab Delta xi euler} and \eqref{equ:ddt nab xi euler} yields
	\begin{align*}
		&\frac{1}{2k} \left(\norm{\nabla \bff{\xi}^{n+1}}{\bb{L}^2}^2 - \norm{\nabla \bff{\xi}^n}{\bb{L}^2}^2 \right)
		+
		\frac{1}{2k} \norm{\nabla \bff{\xi}^{n+1}-\nabla \bff{\xi}^n}{\bb{L}^2}^2
		+
		\lambda_r \norm{\Delta_h \bff{\xi}^{n+1}}{\bb{L}^2}^2
		+
		\lambda_e \norm{\nabla \Delta_h \bff{\xi}}{\bb{L}^2}^2
		\\
		&=
		-
		\inpro{\nabla \Pi_h \left(\delta \bff{u}^{n+1}-\partial_t \bff{u}^{n+1}\right)}{\nabla \Delta_h \bff{\xi}^{n+1}}
		-
		\inpro{\nabla \Pi_h \delta \bff{\rho}^{n+1}}{\nabla \Delta_h \bff{\xi}^{n+1}}
		\\
		&\quad
		-
		\lambda_r \inpro{\nabla \Pi_h \bff{\eta}^{n+1}}{\nabla \Delta_h \bff{\xi}^{n+1}}
		+
		\gamma \inpro{\nabla \Pi_h(\bff{u}_h^n \times \bff{H}_h^{n+1} - \bff{u}^{n+1} \times \bff{H}^{n+1})}{\nabla \Delta_h \bff{\xi}}
		\\
		&\quad
		-
		\inpro{\nabla \Pi_h \delta \bff{\eta}^{n+1}}{\Delta_h \bff{\xi}^{n+1}}
		-
		\kappa \mu\inpro{\delta \bff{\theta}^n+\delta \bff{\rho}^n+\delta \bff{u}^n-\partial_t \bff{u}^{n+1}}{\Delta_h \bff{\xi}^{n+1}}
		\\
		&\quad
		+
		\kappa \inpro{\frac{\abs{\bff{u}_h^{n+1}}^2 \bff{u}_h^{n+1} - \abs{\bff{u}_h^n}^2 \bff{u}_h^n}{k} - \partial_t \left(\abs{\bff{u}^{n+1}}^2 \bff{u}^{n+1}\right)}{\Delta_h \bff{\xi}^{n+1}}
		\\
		&\quad
		-
		\beta\inpro{\bff{e}\big( \bff{e}\cdot(\bff{\theta}^{n+1}+\bff{\rho}^{n+1})\big)}{\Delta_h \bff{\xi}^{n+1}}.
	\end{align*}
	The first three terms on the right-hand side can be estimated using Young's inequality (together with \eqref{equ:Ritz ineq} and \eqref{equ:delta un min u euler}) in a straightforward manner. The fifth and the sixth terms can also be bounded similarly. To estimate the remaining terms, we use \eqref{equ:cross nab est euler} and \eqref{equ:dt psi zeta euler}. 
	We then have for any $\epsilon>0$,
	\begin{align*}
		&\frac{1}{2k} \left(\norm{\nabla \bff{\xi}^{n+1}}{\bb{L}^2}^2 - \norm{\nabla \bff{\xi}^n}{\bb{L}^2}^2 \right)
		+
		\frac{1}{2k} \norm{\nabla \bff{\xi}^{n+1}-\nabla \bff{\xi}^n}{\bb{L}^2}^2
		+
		\lambda_r \norm{\Delta_h \bff{\xi}^{n+1}}{\bb{L}^2}^2
		+
		\lambda_e \norm{\nabla \Delta_h \bff{\xi}}{\bb{L}^2}^2
		\\
		&\lesssim
		h^{2r}
		+
		\norm{\delta \bff{\theta}^{n+1}}{\bb{L}^2}^2
		+
		\norm{\bff{\theta}^n}{\bb{H}^1}^2
		+
		\norm{\bff{\theta}^{n+1}}{\bb{H}^1}^2
		+
		\norm{\bff{\xi}^{n+1}}{\bb{H}^1}^2
		+
		\epsilon \norm{\Delta_h \bff{\xi}^{n+1}}{\bb{L}^2}^2
		+
		\epsilon \norm{\nabla\Delta_h \bff{\xi}^{n+1}}{\bb{L}^2}^2.
	\end{align*}
	Choosing $\epsilon>0$ sufficiently small, multiplying both sides by $2k$, then using \eqref{equ:theta n L2 euler} and \eqref{equ:xi nab theta n L2 euler}), we obtain~\eqref{equ:nab xi n L2 euler}. This completes the proof of the proposition.
\end{proof}

We can now state the main result on the rate of convergence of the scheme \eqref{equ:euler}.

\begin{theorem}\label{the:euler rate}
	Let~$(\bff{u},\bff{H})$ be the solution of \eqref{equ:weakform} as described in Section~\ref{sec:exact sol}, and let~$(\bff{u}_h^n,\bff{H}_h^n)\in \bb{V}_h\times \bb{V}_h$ be the solution of \eqref{equ:euler} with initial data~$\bff{u}_h^0$. Then for $h, k>0$ and~$n\in \{1,2,\ldots,\lfloor T/k \rfloor\}$,
	\begin{align*}
		\norm{\bff{u}_h^n - \bff{u}^n}{\bb{L}^2}^2
		+
		k \sum_{m=0}^{n-1} \norm{\bff{H}_h^{m+1} - \bff{H}^{m+1}}{\bb{L}^2}^2
		&\leq
		C \big(h^{2(r+1)}+k^2 \big),
		\\
		k \sum_{m=0}^{n-1} \norm{\nabla \bff{u}_h^{m+1} - \nabla \bff{u}^{m+1}}{\bb{L}^2}^2
		&\leq
		C \big(h^{2r}+k^2 \big).
	\end{align*}
	Assume further that one of the following holds:
	\begin{enumerate}
		\item $d=1$ or $2$,
		\item $d=3$ and the triangulation is globally quasi-uniform.
	\end{enumerate}
	Then for $h, k>0$ and~$n\in \{1,2,\ldots,\lfloor T/k \rfloor\}$,
	\begin{align*}
		\norm{\bff{H}_h^n-\bff{H}^n}{\bb{L}^2}^2
		&\leq
		C \big(h^{2(r+1)}+k \big),
		\\
		\norm{\nabla \bff{u}_h^n-\nabla \bff{u}^n}{\bb{L}^2}^2
		+
		k \sum_{m=0}^{n-1} \norm{\nabla\bff{H}_h^{m+1} - \nabla\bff{H}^{m+1}}{\bb{L}^2}^2
		&\leq
		C \big(h^{2r}+k \big).
	\end{align*}
	Finally, for $d=1$, $2$, $3$, if the triangulation is globally quasi-uniform, then
	\begin{align*}
		\norm{\nabla \bff{H}_h^n- \nabla \bff{H}^n}{\bb{L}^2}^2
		&\leq
		C \big(h^{2r}+k^2 \big),
		\\
		\norm{\bff{u}_h^n-\bff{u}^n}{\bb{L}^\infty}^2
		+
		k\sum_{m=0}^{n-1} \norm{\bff{H}_h^m- \bff{H}^m}{\bb{L}^\infty}^2
		&\leq
		C \big(h^{2(r+1)} \abs{\ln h} +k^2 \big).
	\end{align*}
	The constant $C$ depends on the coefficients of the equation, $|\mathscr{D}|$, $T$, and $K_0$ (as defined in \eqref{equ:ass 2 u}), but is independent of $n$, $h$, and $k$.
\end{theorem}

\begin{proof}
	The results follow from Proposition \ref{pro:semidisc theta xi}, equations \eqref{equ:theta rho} and \eqref{equ:xi eta}, estimates \eqref{equ:Ritz ineq}, \eqref{equ:Ritz ineq L infty}, and the triangle inequality.
\end{proof}

\section{A Fully Discrete Scheme Based on the Crank--Nicolson Method}\label{sec:Crank}

We now propose a second-order in time numerical scheme for the LLBar equation. Let $k$ be the time step and $\bff{u}_h^n$ be the approximation in $\bb{V}_h$ of $\bff{u}(t)$ at time $t=t_n:=nk$, where $n=0,1,2,\ldots, \lfloor T/k \rfloor$. 
For any function $\bff{v}$, we denote $\bff{v}^n:= \bff{v}(t_n)$. In this section, we define
\begin{align*}
\delta \bff{v}^{n+\frac{1}{2}} :=
\frac{\bff{v}^{n+1}-\bff{v}^{n}}{k}, 
\quad 
\bff{v}^{n+\frac{1}{2}}:=
\frac{\bff{v}^{n+1}+\bff{v}^n}{2},
\quad 
\text{for } n=0,1,\ldots.
\end{align*}
Moreover, define
\begin{align}\label{equ:v n min half}
\delta \bff{v}^n :=
\frac{\bff{v}^{n+\frac{1}{2}} - \bff{v}^{n-\frac{1}{2}}}{k},
\quad
\widehat{\bff{v}}^{n-\frac{1}{2}}
:=
\frac{3\bff{v}^n-\bff{v}^{n-1}}{2},
\quad
\text{for } n=1,2,\ldots.
\end{align}
%Under the assumptions \eqref{equ:ass 2 u} and \eqref{equ:ass 2 H}, for $p\in [1,\infty]$, we have
%\begin{align}
%	\norm{\delta \bff{u}^{n+\frac{1}{2}}}{\bb{L}^p}
%	=
%	\norm{\frac{1}{k} \int_{t_n}^{t_{n+1}} \partial_t \bff{u}(t)\, \dt}{\bb{L}^p}
%	\leq 
%	\norm{\partial_t \bff{u}(t)}{\bb{L}^p}
%	\leq
%	C.
%\end{align}
A fully discrete scheme based on the Crank--Nicolson time-stepping method can be described as follows. We start with $\bff{u}_h^0= R_h \bff{u}(0) \in \bb{V}_h$. For $t_n\in [0,T]$, $n\in \bb{N}$, given $\bff{u}_h^{n-1}, \bff{u}_h^n \in \bb{V}_h$, define $\bff{u}_h^{n+1}$ and $\bff{H}_h^{n+\frac{1}{2}}$ by
\begin{align}\label{equ:cranknic}
	\left\{
	\begin{alignedat}{1}
		\inpro{\delta \bff{u}_h^{n+\frac{1}{2}}}{\bff{\chi}}
		&=
		\lambda_r \inpro{\bff{H}_h^{n+\frac{1}{2}}}{\bff{\chi}}
		+
		\lambda_e \inpro{\nabla \bff{H}_h^{n+\frac{1}{2}}}{\nabla \bff{\chi}}
		-
		\gamma \inpro{\widehat{\bff{u}}_h^{n-\frac{1}{2}} \times \bff{H}_h^{n+\frac{1}{2}}}{\bff{\chi}},
		\\
		\inpro{\bff{H}_h^{n+\frac{1}{2}}}{\bff{\phi}} 
		&=
		-
		\inpro{\nabla \bff{u}_h^{n+\frac{1}{2}}}{\nabla \bff{\phi}}
		+
		\kappa \mu \inpro{\bff{u}_h^{n+\frac{1}{2}}}{\bff{\phi}}
		-
		\kappa \inpro{\psi \big(\bff{u}_h^n, \bff{u}_h^{n+1}\big)}{\bff{\phi}}
		-
		\beta \inpro{\bff{e}(\bff{e}\cdot \bff{u}_h^{n+\frac12})}{\bff{\phi}},
	\end{alignedat}
	\right.
\end{align}
for all $\bff{\chi}, \bff{\phi}\in \bb{H}^1$, where
\begin{align}\label{equ:psi un}
	\psi \big(\bff{u}_h^n, \bff{u}_h^{n+1}\big)
	:=
	\frac{1}{2} \left(\abs{\bff{u}_h^{n+1}}^2 + \abs{\bff{u}_h^n}^2\right) \bff{u}_h^{n+\frac{1}{2}}.
\end{align} 
Due to the presence of $\widehat{\bff{u}}_h^{n-\frac{1}{2}}$ term in~\eqref{equ:cranknic}, this is a two-step scheme and we need to prescribe $\bff{u}_h^1$ and $\bff{H}_h^{\frac{1}{2}}$ separately. This can be done, for instance, by solving with sufficiently small time step size,
\begin{align}\label{equ:cranknic first}
	\left\{
	\begin{alignedat}{1}
		\inpro{\delta \bff{u}_h^{\frac{1}{2}}}{\bff{\chi}}
		&=
		\lambda_r \inpro{\bff{H}_h^{\frac{1}{2}}}{\bff{\chi}}
		+
		\lambda_e \inpro{\nabla \bff{H}_h^{\frac{1}{2}}}{\nabla \bff{\chi}}
		-
		\gamma \inpro{{\bff{u}}_h^{\frac{1}{2}} \times \bff{H}_h^{\frac{1}{2}}}{\bff{\chi}},
		\\
		\inpro{\bff{H}_h^{\frac{1}{2}}}{\bff{\phi}} 
		&=
		-
		\inpro{\nabla \bff{u}_h^{\frac{1}{2}}}{\nabla \bff{\phi}}
		+
		\kappa \mu \inpro{\bff{u}_h^{\frac{1}{2}}}{\bff{\phi}}
		-
		\kappa \inpro{\psi \big(\bff{u}_h^0, \bff{u}_h^1\big)}{\bff{\phi}}.
	\end{alignedat}
	\right.
\end{align}
One could then show by arguments similar to those in this section that 
\begin{align*}
	\norm{\bff{u}_h^1- \bff{u}(t_1)}{\bb{H}^1} + \norm{\bff{H}_h^{\frac{1}{2}}- \bff{H}\big(t_{\frac{1}{2}}\big)}{\bb{H}^1} &\leq C \big(h^{r}+k^2 \big),
\end{align*}
where $C$ depends on the coefficients of the equation, $K_0$, and $|\mathscr{D}|$. This scheme is well-posed by similar argument as in Proposition~\ref{pro:wellpos euler}. The details are omitted for brevity. Subsequently, the proof in this section will focus on the general scheme \eqref{equ:cranknic} where $n\geq 1$, to avoid repetitive and lengthy arguments.

Next, we show some stability results analogous to Proposition~\ref{pro:ene dec H1 euler}. The energy dissipation property is also satisfied unconditionally.

\begin{proposition}\label{pro:ene dec H1 crank}
Let $\bff{u}_h^0 \in \bb{V}_h$ be given and let $\left(\bff{u}_h^n, \bff{H}_h^n\right)$ be defined by
\eqref{equ:cranknic}. Then for any $k>0$ and~$n\in \bb{N}$,
\begin{align}\label{equ:ene decrease}
		\mathcal{E}(\bff{u}_h^{n+1}) \leq \mathcal{E}(\bff{u}_h^n),
\end{align}
where $\mathcal{E}$ was defined in \eqref{equ:energy}. Moreover,
\begin{align}\label{equ:stab L4 H1}
	\norm{\bff{u}_h^n}{\bb{L}^4}^4
	+
	\norm{\bff{u}_h^n}{\bb{H}^1}^2
	+
	k\lambda_r \sum_{m=0}^n \norm{\bff{H}_h^{m+\frac{1}{2}}}{\bb{L}^2}^2
	+
	k\lambda_e \sum_{m=0}^n \norm{\nabla \bff{H}_h^{m+\frac{1}{2}}}{\bb{L}^2}^2
	\leq
	C \left( \norm{\bff{u}_h^0}{\bb{H}^1}^2 + |\mathscr{D}| \right),
\end{align}
where $C$ depends only on $\kappa$ and $\mu$.
\end{proposition}

\begin{proof}
Firstly, taking $\bff{\chi}=\bff{H}_h^{\frac{1}{2}}$ and $\bff{\phi}=\delta \bff{u}_h^{\frac{1}{2}}$ in \eqref{equ:cranknic first}, then subtracting the results give 
\begin{align*}
	\frac{1}{2} \big(\norm{\nabla \bff{u}_h^{1}}{\bb{L}^2}^2 - \norm{\nabla \bff{u}_h^0}{\bb{L}^2}^2 \big)
	&+
	\frac{\kappa}{4} \big( \norm{\bff{u}_h^{1}}{\bb{L}^4}^4 - \norm{\bff{u}_h^0}{\bb{L}^4}^4 \big)
	-
	\frac{\kappa\mu}{2} \big(\norm{\bff{u}_h^{1}}{\bb{L}^2}^2 - \norm{\bff{u}_h^0}{\bb{L}^2}^2\big)
	\\
	&
	+
	\frac{\beta}{2} \left(\norm{\bff{e}\cdot\bff{u}_h^1}{L^2}^2 - \norm{\bff{e}\cdot \bff{u}_h^0}{L^2}^2\right)
	+
	k\lambda_r \norm{\bff{H}_h^{\frac{1}{2}}}{\bb{L}^2}^2
	+
	k\lambda_e \norm{\nabla \bff{H}_h^{\frac{1}{2}}}{\bb{L}^2}^2
	= 0,
\end{align*}
which implies \eqref{equ:ene decrease} for $n=0$.

Similarly, for $n\geq 1$, setting $\bff{\chi}=\bff{H}_h^{n+\frac{1}{2}}$ in \eqref{equ:cranknic} gives
\begin{align}\label{equ:ene stab Un Hn}
	\inpro{\delta \bff{u}_h^{n+\frac{1}{2}}}{\bff{H}_h^{n+\frac{1}{2}}}
	&=
	\lambda_r \norm{\bff{H}_h^{n+\frac{1}{2}}}{\bb{L}^2}^2
	+
	\lambda_e \norm{\nabla \bff{H}_h^{n+\frac{1}{2}}}{\bb{L}^2}^2,
\end{align}
while setting $\bff{\phi}=\delta \bff{u}_h^{n+\frac{1}{2}}$ gives
\begin{align}\label{equ:ene stab Hn Un}
	\nonumber
	\inpro{\bff{H}_h^{n+\frac{1}{2}}}{\delta \bff{u}_h^{n+\frac{1}{2}}}
	&=
	-\frac{1}{2k} \big(\norm{\nabla \bff{u}_h^{n+1}}{\bb{L}^2}^2 - \norm{\nabla \bff{u}_h^n}{\bb{L}^2}^2 \big)
	+
	\frac{\kappa\mu}{2k} \big(\norm{\bff{u}_h^{n+1}}{\bb{L}^2}^2 - \norm{\bff{u}_h^n}{\bb{L}^2}^2\big)
	\\
	&\quad
	-
	\frac{\kappa}{4k} \big( \norm{\bff{u}_h^{n+1}}{\bb{L}^4}^4 - \norm{\bff{u}_h^n}{\bb{L}^4}^4 \big)
	-
	\frac{\beta}{2k} \left(\norm{\bff{e}\cdot\bff{u}_h^1}{L^2}^2 - \norm{\bff{e}\cdot \bff{u}_h^0}{L^2}^2\right).
\end{align}
Substituting \eqref{equ:ene stab Hn Un} into \eqref{equ:ene stab Un Hn} and rearranging yield the identity
\begin{align*}
	\frac{1}{2} \big(\norm{\nabla \bff{u}_h^{n+1}}{\bb{L}^2}^2 - \norm{\nabla \bff{u}_h^n}{\bb{L}^2}^2 \big)
	&+
	\frac{\kappa}{4} \big( \norm{\bff{u}_h^{n+1}}{\bb{L}^4}^4 - \norm{\bff{u}_h^n}{\bb{L}^4}^4 \big)
	-
	\frac{\kappa\mu}{2} \big(\norm{\bff{u}_h^{n+1}}{\bb{L}^2}^2 - \norm{\bff{u}_h^n}{\bb{L}^2}^2\big)
	\\
	&
	+
	\frac{\beta}{2} \left(\norm{\bff{e}\cdot\bff{u}_h^{n+1}}{L^2}^2 - \norm{\bff{e}\cdot \bff{u}_h^n}{L^2}^2\right)
	+
	k\lambda_r \norm{\bff{H}_h^{n+\frac{1}{2}}}{\bb{L}^2}^2
	+
	k\lambda_e \norm{\nabla \bff{H}_h^{n+\frac{1}{2}}}{\bb{L}^2}^2
	= 0,
\end{align*}
which implies \eqref{equ:ene decrease} and~\eqref{equ:stab L4 H1}.
\end{proof}

We will also derive the stability of $\bff{H}_h^{n+\frac{1}{2}}$ in $\ell^\infty(\bb{L}^2)$ norm. The following identity will be used in the proof (where $\psi$ was defined in~\eqref{equ:psi un}):
\begin{align}
	\nonumber
	\label{equ:identity psi}
	2 \big(\psi(\bff{u}_h^n, \bff{u}_h^{n+1})-\psi(\bff{u}_h^{n-1},\bff{u}_h^n) \big)
	&=
	\big( |\bff{u}_h^{n+1}|^2 + |\bff{u}_h^n|^2 \big) \big(\bff{u}_h^{n+\frac{1}{2}} - \bff{u}_h^{n-\frac{1}{2}} \big)
	\\
	&\quad 
	+
	\big(\bff{u}_h^{n+1}- \bff{u}_h^{n-1}\big) \cdot \big(\bff{u}_h^{n+1}+\bff{u}_h^{n-1}\big) \bff{u}_h^{n-\frac{1}{2}}.
\end{align}

\begin{proposition}
Let $\bff{u}_h^0 \in \bb{V}_h$ be given and let $\bff{u}_h^n$, $\bff{H}_h^{n+\frac{1}{2}}$ be defined by \eqref{equ:cranknic}. Then for any $k>0$ and~$n\in \bb{N}$,
\begin{align}\label{equ:Hn L2 stab}
	\norm{\bff{H}_h^{n+\frac{1}{2}}}{\bb{L}^2}^2
	+
	k \sum_{m=1}^n \norm{\delta \bff{u}_h^m}{\bb{L}^2}^2
	\leq
	C,
\end{align}
where $C$ depends only on the coefficients of the equation, $K_0$, and $|\mathscr{D}|$.
\end{proposition}

\begin{proof}
Adding the first equation in \eqref{equ:cranknic} at time steps $n+\frac{1}{2}$ and $n-\frac{1}{2}$ gives
\begin{align}\label{equ:delta n half}
	\nonumber
	\inpro{\delta \bff{u}_h^{n+\frac{1}{2}} + \delta \bff{u}_h^{n-\frac{1}{2}}}{\bff{\chi}}
	&=
	\lambda_r \inpro{\bff{H}_h^{n+\frac{1}{2}}+\bff{H}_h^{n-\frac{1}{2}}}{\bff{\chi}}
	+
	\lambda_e \inpro{\nabla \bff{H}_h^{n+\frac{1}{2}}+ \nabla \bff{H}_h^{n-\frac{1}{2}}}{\nabla \bff{\chi}}
	\\
	&\quad 
	-
	\gamma \inpro{\widehat{\bff{u}}_h^{n-\frac{1}{2}} \times \bff{H}_h^{n+\frac{1}{2}} + \widehat{\bff{u}}_h^{n-\frac{3}{2}} \times \bff{H}_h^{n-\frac{1}{2}}}{\bff{\chi}}.
\end{align}
Subtracting the second equation in \eqref{equ:cranknic} at time step $n-\frac{1}{2}$ from the same equation at time step $n+\frac{1}{2}$ gives
\begin{align}\label{equ:H n half}
	\nonumber
	\inpro{\bff{H}_h^{n+\frac{1}{2}}-\bff{H}_h^{n-\frac{1}{2}}}{\bff{\phi}}
	&=
	-\inpro{\nabla \bff{u}_h^{n+\frac{1}{2}}- \nabla \bff{u}_h^{n-\frac{1}{2}}}{\nabla \bff{\phi}}
	+
	\kappa\mu \inpro{\bff{u}_h^{n+\frac{1}{2}}-\bff{u}_h^{n-\frac{1}{2}}}{\bff{\phi}}
	\\
	&\quad 
	-
	\kappa \inpro{\psi(\bff{u}_h^n, \bff{u}_h^{n+1})-\psi(\bff{u}_h^{n-1},\bff{u}_h^n)}{\bff{\phi}}
	-
	\beta \inpro{\bff{e}\big(\bff{e}\cdot (\bff{u}_h^{n+\frac12}-\bff{u}_h^{n-\frac12})\big)}{\bff{\phi}}.
\end{align}
Taking $\bff{\chi}=\bff{u}_h^{n+\frac{1}{2}}-\bff{u}_h^{n-\frac{1}{2}}$ in \eqref{equ:delta n half} and $\bff{\phi}= \lambda_e \big(\bff{H}_h^{n+\frac{1}{2}}+ \bff{H}_h^{n-\frac{1}{2}} \big)$ in \eqref{equ:H n half}, then adding the resulting equations, and noting the identity
\begin{align*}
	2 \big(\bff{u}_h^{n+\frac{1}{2}}-\bff{u}_h^{n-\frac{1}{2}} \big)
	&=
	\bff{u}_h^{n+1}-\bff{u}_h^{n-1},
\end{align*}
we obtain
\begin{align*}
	&\lambda_e \left(\norm{\bff{H}_h^{n+\frac{1}{2}}}{\bb{L}^2}^2 - \norm{\bff{H}_h^{n-\frac{1}{2}}}{\bb{L}^2}^2 \right)
	+
	\frac{4}{k} \norm{\bff{u}_h^{n+\frac{1}{2}}- \bff{u}_h^{n-\frac{1}{2}}}{\bb{L}^2}^2
	\\
	&=
	(\lambda_r + \kappa\mu \lambda_e) \inpro{\bff{H}_h^{n+\frac{1}{2}}+\bff{H}_h^{n-\frac{1}{2}}}{\bff{u}_h^{n+\frac{1}{2}}- \bff{u}_h^{n-\frac{1}{2}}}
	-
	\gamma \inpro{\widehat{\bff{u}}_h^{n-\frac{1}{2}} \times \bff{H}_h^{n+\frac{1}{2}} + \widehat{\bff{u}}_h^{n-\frac{3}{2}} \times \bff{H}_h^{n-\frac{1}{2}}}{\bff{u}_h^{n+\frac{1}{2}}-\bff{u}_h^{n-\frac{1}{2}}}
	\\
	&\quad 
	-
	\kappa \lambda_e \inpro{\psi(\bff{u}_h^n, \bff{u}_h^{n+1})-\psi(\bff{u}_h^{n-1},\bff{u}_h^n)}{\bff{H}_h^{n+\frac{1}{2}}+ \bff{H}_h^{n-\frac{1}{2}}}
	-
	\beta \lambda_e \inpro{\bff{e}\big(\bff{e}\cdot (\bff{u}_h^{n+\frac12}-\bff{u}_h^{n-\frac12})\big)}{\bff{H}_h^{n+\frac{1}{2}}+\bff{H}_h^{n-\frac{1}{2}}}
	\\
	&=: S_1 + S_2 + S_3+ S_4.
\end{align*}
We will estimate each term on the right-hand side by applying H\"older's and Young's inequalities. Firstly,
\begin{align}\label{equ:S1}
	|S_1| + |S_4|
	&\leq 
	Ck \norm{\bff{H}_h^{n+\frac{1}{2}}}{\bb{L}^2}^2 
	+ 
	Ck \norm{\bff{H}_h^{n-\frac{1}{2}}}{\bb{L}^2}^2
	+
	\frac{1}{k} \norm{\bff{u}_h^{n+\frac{1}{2}}- \bff{u}_h^{n-\frac{1}{2}}}{\bb{L}^2}^2.
\end{align}
For the second term, similarly we have
\begin{align}\label{equ:S2}
	\nonumber
	|S_2|
	&\leq
	\gamma \left( \norm{\widehat{\bff{u}}_h^{n-\frac{1}{2}}}{\bb{L}^4} \norm{\bff{H}_h^{n+\frac{1}{2}}}{\bb{L}^4}
	+
	\norm{\widehat{\bff{u}}_h^{n-\frac{3}{2}}}{\bb{L}^4} \norm{\bff{H}_h^{n-\frac{1}{2}}}{\bb{L}^4} \right)
	\norm{\bff{u}_h^{n+\frac{1}{2}}- \bff{u}_h^{n-\frac{1}{2}}}{\bb{L}^2}
	\\
	&\leq
	Ck \norm{\bff{H}_h^{n+\frac{1}{2}}}{\bb{H}^1}^2
	+
	Ck \norm{\bff{H}_h^{n-\frac{1}{2}}}{\bb{H}^1}^2
	+
	\frac{1}{k} \norm{\bff{u}_h^{n+\frac{1}{2}}- \bff{u}_h^{n-\frac{1}{2}}}{\bb{L}^2}^2,
\end{align}
where in the last step we used \eqref{equ:stab L4 H1} and the Sobolev embedding $\bb{H}^1 \hookrightarrow \bb{L}^4$. For the term in $S_3$, noting the identities \eqref{equ:identity psi} and \eqref{equ:H n half}, we have
\begin{align}\label{equ:S3}
	\nonumber
	|S_3|
	&\leq
	C \left( \norm{\bff{u}_h^{n+1}}{\bb{L}^6}^2 + \norm{\bff{u}_h^n}{\bb{L}^6}^2
	+ \norm{\bff{u}_h^{n+1}}{\bb{L}^6} \norm{\bff{u}_h^{n-\frac{1}{2}}}{\bb{L}^6}
	+ \norm{\bff{u}_h^{n-1}}{\bb{L}^6} \norm{\bff{u}_h^{n-\frac{1}{2}}}{\bb{L}^6} \right)
	\\
	\nonumber
	&\quad 
	\norm{\bff{u}_h^{n+\frac{1}{2}} - \bff{u}_h^{n-\frac{1}{2}}}{\bb{L}^2}
	\norm{\bff{H}_h^{n+\frac{1}{2}}+ \bff{H}_h^{n-\frac{1}{2}}}{\bb{L}^6}
	\\
	&\leq
	Ck \norm{\bff{H}_h^{n+\frac{1}{2}}}{\bb{H}^1}^2 
	+ 
	Ck \norm{\bff{H}_h^{n-\frac{1}{2}}}{\bb{H}^1}^2
	+
	\frac{1}{k} \norm{\bff{u}_h^{n+\frac{1}{2}}- \bff{u}_h^{n-\frac{1}{2}}}{\bb{L}^2}^2,
\end{align}
where in the final step we used \eqref{equ:stab L4 H1}, Young's inequality, and the Sobolev embedding. Altogether, upon rearranging and summing the terms, \eqref{equ:S1}, \eqref{equ:S2}, and \eqref{equ:S3} imply
\begin{align*}
	\lambda_e \norm{\bff{H}_h^{n+\frac{1}{2}}}{\bb{L}^2}^2 
	+
	\frac{1}{k} \sum_{m=1}^n \norm{\bff{u}_h^{m+\frac{1}{2}}- \bff{u}_h^{m-\frac{1}{2}}}{\bb{L}^2}^2
	&\leq
	\lambda_e \norm{\bff{H}_h^{\frac{1}{2}}}{\bb{L}^2}^2
	+
	Ck \sum_{m=1}^n \left( \norm{\bff{H}_h^{m+\frac{1}{2}}}{\bb{H}^1}^2 + \norm{\bff{H}_h^{m-\frac{1}{2}}}{\bb{H}^1}^2 \right)
	\leq C,
\end{align*}
where in the last step we used \eqref{equ:stab L4 H1}. This implies the required inequality.
\end{proof}

The above proposition implies the stability of $\bff{u}_h^n$ in $\ell^\infty(\bb{L}^\infty)$ norm under an additional assumption that the triangulation is quasi-uniform.

\begin{proposition}
Let $\bff{u}_h^0 \in \bb{V}_h$ be given and let $\left(\bff{u}_h^n, \bff{H}_h^n\right)$ be defined by
\eqref{equ:cranknic}. Then for any $k>0$ and~$n\in \bb{N}$,
\begin{align}\label{equ:stab Delta uh}
	\norm{\Delta_h \bff{u}_h^n}{\bb{L}^2}^2 \leq C.
\end{align}
Moreover, if the triangulation $\mathcal{T}_h$ is (globally) quasi-uniform, then
\begin{align}\label{equ:stab uh L infty}
	\norm{\bff{u}_h^n}{\bb{L}^\infty}^2 \leq C.
\end{align}
Here, the constant $C$ depends only on the coefficients of the equation, $\norm{\bff{u}_h^0}{\bb{H}^1}$, $\norm{\bff{H}_h^0}{\bb{L}^2}$, and $|\mathscr{D}|$.
\end{proposition}

\begin{proof}
Taking $\bff{\phi}=\Delta_h \bff{u}_h^{n+\frac12}$ and applying Young's and H\"older's inequalities, we have
\begin{align*}
	\norm{\Delta_h \bff{u}_h^{n+\frac12}}{\bb{L}^2}^2
	&=
	\kappa \norm{\nabla \bff{u}_h^{n+\frac12}}{\bb{L}^2}^2
	+
	\inpro{\bff{H}_h^{n+\frac{1}{2}}}{\Delta_h \bff{u}_h^{n+\frac12}}
	+
	\kappa \inpro{\psi \big(\bff{u}_h^n, \bff{u}_h^{n+1}\big)}{\Delta_h \bff{u}_h^{n+\frac12}}
	\\
	&\quad
	+
	\beta\inpro{\bff{e}\big(\bff{e}\cdot \bff{u}_h^{n+\frac12}\big)}{\Delta_h \bff{u}_h^{n+\frac12}}
	\\
	&\leq
	C\norm{\bff{u}_h^{n+\frac12}}{\bb{H}^1}^2
	+
	\frac{1}{2} \norm{\Delta_h \bff{u}_h^{n+\frac12}}{\bb{L}^2}^2
	+
	4 \norm{\bff{H}_h^{n+\frac{1}{2}}}{\bb{L}^2}^2
	+
	C \left( \norm{\bff{u}_h^{n+1}}{\bb{L}^6}^2 + \norm{\bff{u}_h^n}{\bb{L}^6}^2 \right) \norm{\bff{u}_h^{n+\frac{1}{2}}}{\bb{L}^6}.
\end{align*}
Therefore, rearranging the terms, using the Sobolev embedding $\bb{H}^1 \hookrightarrow \bb{L}^6$ (noting \eqref{equ:stab L4 H1} and \eqref{equ:Hn L2 stab}), we inequality \eqref{equ:stab Delta uh}.
Inequality \eqref{equ:stab uh L infty} then follows from \eqref{equ:disc lapl L infty}, completing the proof of the proposition.
\end{proof}

Before proceeding to prove various estimates leading to the main theorem, we note the following inequalities. For $s=0$ or $1$, by Taylor's theorem,
\begin{align}\label{equ:u hat u half}
	\norm{\widehat{\bff{u}}^{n-\frac{1}{2}}- \bff{u}\big(t_{n+\frac{1}{2}}\big)}{\bb{H}^s}
	\leq
	\norm{\bff{u}^n- \frac{\bff{u}^{n+1}+\bff{u}^{n-1}}{2}}{\bb{H}^s}
	+
	\norm{\frac{\bff{u}^n+\bff{u}^{n+1}}{2}- \bff{u}\big(t_{n+\frac{1}{2}}\big)}{\bb{H}^s}
	&\lesssim
	k^2,
\end{align}
and
\begin{align}\label{equ:d rho dt u}
	\norm{\delta \bff{\rho}^{n+\frac{1}{2}} + \delta \bff{u}^{n+\frac{1}{2}}-\partial_t \bff{u}\big(t_{n+\frac{1}{2}}\big)}{\bb{H}^s}
	\leq
	\norm{\delta \bff{\rho}^{n+\frac{1}{2}}}{\bb{H}^s} + \norm{\delta \bff{u}^{n+\frac{1}{2}}-\partial_t \bff{u}\big(t_{n+\frac{1}{2}}\big)}{\bb{H}^s}
	&\lesssim
	h^{r+1-s}+k^2,
\end{align}
where the assumptions on $\bff{u}$ in \eqref{equ:ass 2 u} and \eqref{equ:Ritz ineq} are used in the last step.

The following lemmas are needed to bound the nonlinear terms in the main theorem.

\begin{lemma}\label{lem:cross disc}
Let $\epsilon>0$ be arbitrary. The following inequality holds:
\begin{align}
	\label{equ:cross theta est}
	\nonumber
	\left| \inpro{\widehat{\bff{u}}_h^{n-\frac{1}{2}} \times \bff{H}_h^{n+\frac{1}{2}} - \bff{u}\big(t_{n+\frac{1}{2}}\big) \times \bff{H}\big(t_{n+\frac{1}{2}}\big)}{\bff{\theta}^{n+\frac{1}{2}}} \right| 
	&\lesssim
	h^{2(r+1)} + k^4+ \norm{\bff{\theta}^{n+\frac{1}{2}}}{\bb{L}^2}^2 + \epsilon \norm{\nabla \bff{\theta}^{n+\frac{1}{2}}}{\bb{L}^2}^2
	\\
	&\quad
	+ \epsilon \norm{\bff{\theta}^n}{\bb{L}^2}^2 + \epsilon \norm{\bff{\theta}^{n-1}}{\bb{L}^2}^2
	+ \epsilon \norm{\bff{\xi}^{n+\frac{1}{2}}}{\bb{L}^2}^2.
\end{align}
Moreover, if the triangulation $\mathcal{T}_h$ is quasi-uniform, then for any $\bff{\zeta} \in \bb{V}_h$,
\begin{align}\label{equ:cross est}
	\nonumber
	\left| \inpro{\widehat{\bff{u}}_h^{n-\frac{1}{2}} \times \bff{H}_h^{n+\frac{1}{2}} - \bff{u}\big(t_{n+\frac{1}{2}}\big) \times \bff{H}\big(t_{n+\frac{1}{2}}\big)}{\bff{\zeta}} \right| 
	&\lesssim
	h^{2(r+1)} + k^4+ \norm{\bff{\xi}^{n+\frac{1}{2}}}{\bb{L}^2}^2
	\\
	&\quad
	+ \epsilon \norm{\bff{\theta}^n}{\bb{L}^2}^2 + \epsilon \norm{\bff{\theta}^{n-1}}{\bb{L}^2}^2
	+ \epsilon \norm{\bff{\zeta}}{\bb{L}^2}^2,
	\\
	\label{equ:cross nab est crank}
	\left|\inpro{\nabla\Pi_h\left(\widehat{\bff{u}}_h^{n-\frac12} \times \bff{H}_h^{n+\frac12}- \bff{u}\big(t_{n+\frac{1}{2}}\big) \times \bff{H}\big(t_{n+\frac{1}{2}}\big)\right)}{\bff{\zeta}} \right|
	&\lesssim
	h^{2r}
	+
	k^4
	+
	\norm{\bff{\theta}^n}{\bb{H}^1}^2
	+
	\norm{\bff{\theta}^{n-1}}{\bb{H}^1}^2
	\nonumber \\
	&\quad
	+
	\norm{\bff{\xi}^{n+\frac12}}{\bb{H}^1}^2
	+
	\epsilon \norm{\bff{\zeta}}{\bb{L}^2}^2.
\end{align}
\end{lemma}

\begin{proof}
We write
\begin{align*}
	\nonumber
	\widehat{\bff{u}}_h^{n-\frac{1}{2}} \times \bff{H}_h^{n+\frac{1}{2}} - \bff{u}\big(t_{n+\frac{1}{2}}\big) \times \bff{H}\big(t_{n+\frac{1}{2}}\big)
	&=
	\widehat{\bff{u}}_h^{n-\frac{1}{2}} \times \left(\bff{\xi}^{n+\frac{1}{2}}+\bff{\eta}^{n+\frac{1}{2}}\right).
	\\
	&\quad
	+
	\left(\widehat{\bff{\theta}}^{n-\frac{1}{2}} + \widehat{\bff{\rho}}^{n-\frac{1}{2}} + \widehat{\bff{u}}^{n-\frac{1}{2}} - \bff{u}\big(t_{n+\frac{1}{2}}\big)\right) \times \bff{H}\big(t_{n+\frac{1}{2}}\big).
\end{align*}
The proof of \eqref{equ:cross theta est} then follows by arguments similar to that in \eqref{equ:uh cross Hh theta} (and noting \eqref{equ:u hat u half}), without assuming $\mathcal{T}_h$ is quasi-uniform.

Next, suppose the triangulation is quasi-uniform (and thus \eqref{equ:stab uh L infty} holds in this case). By H\"older's inequality and assumptions \eqref{equ:ass 2 u} on the exact solution, we have
\begin{align*}
	\nonumber
	&\norm{\widehat{\bff{u}}_h^{n-\frac{1}{2}} \times \bff{H}_h^{n+\frac{1}{2}} - \bff{u}\big(t_{n+\frac{1}{2}}\big) \times \bff{H}\big(t_{n+\frac{1}{2}}\big)}{\bb{L}^2}
	\\
	\nonumber
	&\leq
	\norm{\widehat{\bff{u}}_h^{n-\frac{1}{2}}}{\bb{L}^\infty} \norm{\bff{\xi}^{n+\frac{1}{2}}+\bff{\eta}^{n+\frac{1}{2}}}{\bb{L}^2}
	+
	\left( \norm{\widehat{\bff{\theta}}^{n-\frac{1}{2}} + \widehat{\bff{\rho}}^{n-\frac{1}{2}}}{\bb{L}^2} + \norm{\widehat{\bff{u}}^{n-\frac{1}{2}} - \bff{u}\big(t_{n+\frac{1}{2}}\big)}{\bb{L}^2} \right) \norm{\bff{H}\big(t_{n+\frac{1}{2}}\big)}{\bb{L}^\infty}
	\\
	&\lesssim
	h^{r+1} + \norm{\bff{\xi}^{n+\frac{1}{2}}}{\bb{L}^2}
	+ \norm{\bff{\theta}^n}{\bb{L}^2} + \norm{\bff{\theta}^{n-1}}{\bb{L}^2} + k^2,
\end{align*}
where in the last step we used \eqref{equ:stab uh L infty}, \eqref{equ:Ritz ineq}, and \eqref{equ:u hat u half}. Inequality \eqref{equ:cross est} then follows by Young's inequality. Finally, the proof of~\eqref{equ:cross nab est crank} is similar to that of \eqref{equ:cross nab est euler}.
\end{proof}

\begin{lemma}\label{lem:cub disc}
Let $\epsilon>0$ be arbitrary. The following inequality holds:
\begin{align}\label{equ:psi theta disc}
	\left| \inpro{\psi\big(\bff{u}_h^n, \bff{u}_h^{n+1}\big) - \big|\bff{u}\big(t_{n+\frac{1}{2}}\big)\big|^2 \bff{u}\big(t_{n+\frac{1}{2}}\big)}{\bff{\theta}^{n+\frac{1}{2}}} \right| 
	&\lesssim
	h^{2(r+1)} + k^4 + \norm{\bff{\theta}^n}{\bb{L}^2}^2 + \norm{\bff{\theta}^{n+1}}{\bb{L}^2}^2 
	+ \epsilon \norm{\nabla \bff{\theta}^{n+\frac{1}{2}}}{\bb{L}^2}^2.
\end{align}
Moreover, if the triangulation $\mathcal{T}_h$ is quasi-uniform, then for any $\bff{\zeta}\in \bb{V}_h$,
\begin{align}\label{equ:psi xi disc}
	\left| \inpro{\psi\big(\bff{u}_h^n, \bff{u}_h^{n+1}\big) - \big|\bff{u}\big(t_{n+\frac{1}{2}}\big)\big|^2 \bff{u}\big(t_{n+\frac{1}{2}}\big)}{\bff{\zeta}} \right| 
	&\lesssim
	h^{2(r+1)} + k^4 + \norm{\bff{\theta}^n}{\bb{L}^2}^2 + \norm{\bff{\theta}^{n+1}}{\bb{L}^2}^2 
	+ \epsilon \norm{\bff{\zeta}}{\bb{L}^2}^2,
	\\
	\label{equ:psi xi Delta}
	\left| \inpro{\psi\big(\bff{u}_h^n,\bff{u}_h^{n+1}\big) - \abs{\bff{u}\big(t_{n+\frac{1}{2}}\big)}^2 \bff{u}\big(t_{n+\frac{1}{2}}\big)}{\Delta_h \bff{\zeta}} \right|
	&\lesssim
	h^{2r}+k^4+\norm{\bff{\theta}^n}{\bb{H}^1}^2 + \norm{\bff{\theta}^{n+1}}{\bb{H}^1}^2
	+
	\epsilon \norm{\nabla \bff{\zeta}}{\bb{L}^2}^2.
\end{align}
\end{lemma}

\begin{proof}
Note that we have the identity
\begin{align}\label{equ:psi u2u}
	\nonumber
	&\psi\big(\bff{u}_h^n, \bff{u}_h^{n+1}\big) - \big|\bff{u}\big(t_{n+\frac{1}{2}}\big)\big|^2 \bff{u}\big(t_{n+\frac{1}{2}}\big)
	\\
	\nonumber
	&=
	\frac{1}{2} \left(|\bff{u}_h^{n+1}|^2 + |\bff{u}_h^n|^2 \right)  \bff{\theta}^{n+\frac{1}{2}} + \frac{1}{2} \left(|\bff{u}_h^{n+1}|^2 + |\bff{u}_h^n|^2 \right) \bff{\rho}^{n+\frac{1}{2}} 
	\\
	\nonumber
	&\quad
	+
	\frac{1}{2} \bff{u}\big(t_{n+\frac{1}{2}}\big) \left( \bff{u}_h^{n+1}+ \bff{u}^{n+1} \right) \cdot \bff{\theta}^{n+1}
	+ \frac{1}{2} \bff{u}\big(t_{n+\frac{1}{2}}\big) \left( \bff{u}_h^{n+1}+ \bff{u}^{n+1} \right) \cdot \bff{\rho}^{n+1}
	\\
	\nonumber
	&\quad
	+
	\frac{1}{2} \bff{u}\big(t_{n+\frac{1}{2}}\big) \left( \bff{u}_h^n+ \bff{u}^n \right) \cdot \bff{\theta}^n
	+ \frac{1}{2} \bff{u}\big(t_{n+\frac{1}{2}}\big) \left( \bff{u}_h^n+ \bff{u}^n \right) \cdot \bff{\rho}^n
	\\
	&\quad
	+
	\bff{u}\big(t_{n+\frac{1}{2}}\big) \left(\frac{\abs{\bff{u}^{n+1}}^2 + \abs{\bff{u} \big(t_n\big)}^2}{2} - \abs{\bff{u}\big(t_{n+\frac{1}{2}}\big)}^2 \right)
	=: S_1+S_2+\cdots+S_7.
\end{align}
We will proceed by bounding each term above. Terms containing $\bff{\theta}$ will be bounded in the $\bb{L}^{4/3}$ norm, while those containing $\bff{\rho}$ will be bounded in the $\bb{L}^{6/5}$ norm. For the first term, by H\"older's inequality,
\begin{align*}
	\norm{S_1}{\bb{L}^{4/3}}
	&\lesssim
	\left(\norm{\bff{u}_h^{n+1}}{\bb{L}^4}^2 + \norm{\bff{u}_h^n}{\bb{L}^4}^2 \right) \norm{\bff{\theta}^{n+\frac{1}{2}}}{\bb{L}^4}.
\end{align*}
Similarly for the next term,
\begin{align*}
	\norm{S_2}{\bb{L}^{6/5}}
	&\lesssim
	\left(\norm{\bff{u}_h^{n+1}}{\bb{L}^6}^2 + \norm{\bff{u}_h^n}{\bb{L}^6}^2 \right) \norm{\bff{\rho}^{n+\frac{1}{2}}}{\bb{L}^2}.
\end{align*}
For the third and the fourth terms,
\begin{align*}
	\norm{S_3}{\bb{L}^{4/3}}
	&\lesssim
	\norm{\bff{u}\big(t_{n+\frac{1}{2}}\big)}{\bb{L}^\infty} \norm{\bff{u}_h^{n+1}+ \bff{u}^{n+1}}{\bb{L}^4} \norm{\bff{\theta}^{n+1}}{\bb{L}^2},
	\\
	\norm{S_4}{\bb{L}^{6/5}}
	&\lesssim
	\norm{\bff{u}\big(t_{n+\frac{1}{2}}\big)}{\bb{L}^6} \norm{\bff{u}_h^{n+1}+ \bff{u}^{n+1}}{\bb{L}^6} \norm{\bff{\rho}^{n+1}}{\bb{L}^2}.
\end{align*}
Similarly, we also have
\begin{align*}
	\norm{S_5}{\bb{L}^{4/3}}
	&\lesssim
	\norm{\bff{u}\big(t_{n+\frac{1}{2}}\big)}{\bb{L}^\infty} \norm{\bff{u}_h^n+ \bff{u}^n}{\bb{L}^4} \norm{\bff{\theta}^n}{\bb{L}^2},
	\\
	\norm{S_6}{\bb{L}^{6/5}}
	&\lesssim
	\norm{\bff{u}\big(t_{n+\frac{1}{2}}\big)}{\bb{L}^6} \norm{\bff{u}_h^n+ \bff{u}^{n+1}}{\bb{L}^6} \norm{\bff{\rho}^n}{\bb{L}^2}.
\end{align*}
For the last term, by H\"older's inequality and Taylor's theorem with integral remainder,
\begin{align}\label{equ:psi S7}
	\nonumber
	\norm{S_7}{\bb{L}^2}
	&\lesssim
	k^{3/2} \norm{\bff{u}\big(t_{n+\frac{1}{2}}\big)}{\bb{L}^\infty} \left(\int_{t_n}^{t_{n+1}} \norm{\partial_{tt} \abs{\bff{u}(t)}^2}{\bb{L}^2}^2 \dt \right)^{1/2}
	\\
	&\lesssim
	k^2 \norm{\bff{u}\big(t_{n+\frac{1}{2}}\big)}{\bb{L}^\infty} \norm{\partial_{tt} \abs{\bff{u}}^2}{L^\infty(\bb{L}^2)}.
\end{align}
Altogether, using the assumptions on $\bff{u}$ in \eqref{equ:ass 2 u} and inequality \eqref{equ:stab L4 H1}, by H\"older's inequality we infer that
\begin{align*}
	&\left| \inpro{\psi\big(\bff{u}_h^n, \bff{u}_h^{n+1}\big) - \big|\bff{u}\big(t_{n+\frac{1}{2}}\big)\big|^2 \bff{u}\big(t_{n+\frac{1}{2}}\big)}{\bff{\theta}^{n+\frac{1}{2}}} \right| 
	\\
	&\leq
	\big(\norm{S_1}{\bb{L}^{4/3}} +
	\norm{S_3}{\bb{L}^{4/3}} + \norm{S_5}{\bb{L}^{4/3}}\big) \norm{\bff{\theta}^{n+\frac{1}{2}}}{\bb{L}^4}
	+
	\big( \norm{S_2}{\bb{L}^{6/5}} + \norm{S_4}{\bb{L}^{6/5}} + \norm{S_6}{\bb{L}^{6/5}}\big) \norm{\bff{\theta}^{n+\frac{1}{2}}}{\bb{L}^6}
	\\
	&\quad 
	+
	\norm{S_7}{\bb{L}^2} \norm{\bff{\theta}^{n+\frac{1}{2}}}{\bb{L}^2}
	\\
	&\lesssim
	\norm{\bff{\theta}^{n+\frac{1}{2}}}{\bb{L}^4}^2
	+
	\norm{\bff{\theta}^{n+1}}{\bb{L}^2} \norm{\bff{\theta}^{n+\frac{1}{2}}}{\bb{L}^4}
	+
	\norm{\bff{\theta}^n}{\bb{L}^2} \norm{\bff{\theta}^{n+\frac{1}{2}}}{\bb{L}^4}
	+
	\big( \norm{\bff{\rho}^n}{\bb{L}^2} + \norm{\bff{\rho}^{n+1}}{\bb{L}^2} \big) \norm{\bff{\theta}^{n+\frac{1}{2}}}{\bb{L}^4}
	\\
	&\quad
	+
	k^2 \norm{\bff{\theta}^{n+\frac{1}{2}}}{\bb{L}^2}
	\\
	&\lesssim
	\norm{\bff{\theta}^n}{\bb{L}^2}^2 + \norm{\bff{\theta}^{n+1}}{\bb{L}^2}^2 
	+
	\epsilon \norm{\nabla \bff{\theta}^{n+\frac{1}{2}}}{\bb{L}^2}^2
	+
	h^{2(r+1)} + k^4,
\end{align*}
where in the last step we used Young's inequality, \eqref{equ:L4 young}, and \eqref{equ:Ritz ineq}. This proves \eqref{equ:psi theta disc}.

Next, if the triangulation $\mathcal{T}_h$ is quasi-uniform, then \eqref{equ:stab uh L infty} holds. As such, we can bound the $\bb{L}^\infty$ norm of $\bff{u}_h^n$ and $\bff{u}_h^{n+1}$ appearing in \eqref{equ:psi u2u} (uniformly in $n$ and $h$). Noting \eqref{equ:psi S7}, we then obtain
\begin{align}\label{equ:psi u}
	\norm{\psi\big(\bff{u}_h^n, \bff{u}_h^{n+1}\big) - \big|\bff{u}\big(t_{n+\frac{1}{2}}\big)\big|^2 \bff{u}\big(t_{n+\frac{1}{2}}\big)}{\bb{L}^2}
	\lesssim
	\norm{\bff{\theta}^n}{\bb{L}^2} + \norm{\bff{\theta}^{n+1}}{\bb{L}^2}
	+ \norm{\bff{\rho}^n}{\bb{L}^2} + \norm{\bff{\rho}^{n+1}}{\bb{L}^2} + k^2.
\end{align} 
By Young's inequality, we have for any $\epsilon >0$,
\begin{align*}
	&\left| \inpro{\psi\big(\bff{u}_h^n, \bff{u}_h^{n+1}\big) - \big|\bff{u}\big(t_{n+\frac{1}{2}}\big)\big|^2 \bff{u}\big(t_{n+\frac{1}{2}}\big)}{\bff{\zeta}} \right| 
	\lesssim
	\norm{\psi\big(\bff{u}_h^n, \bff{u}_h^{n+1}\big) - \big|\bff{u}\big(t_{n+\frac{1}{2}}\big)\big|^2 \bff{u}\big(t_{n+\frac{1}{2}}\big)}{\bb{L}^2}^2
	+
	\epsilon \norm{\bff{\zeta}}{\bb{L}^2}^2,
\end{align*}
and thus inequality \eqref{equ:psi xi disc} follows from \eqref{equ:Ritz ineq} and \eqref{equ:psi u}.

It remains to prove \eqref{equ:psi xi Delta}. Let $S:=S_1+S_2+\cdots+S_7$ defined in \eqref{equ:psi u2u}. Then we have
\[
	\inpro{S}{\Delta_h \bff{\zeta}}
	=
	\inpro{\nabla \Pi_h S}{\nabla \bff{\zeta}}.
\]
We will proceed by estimating $\norm{\nabla S}{\bb{L}^2}$ using the expression \eqref{equ:psi u2u} and the product rule for gradient. Firstly, by H\"older's and Young's inequality,
\begin{align*}
	\nonumber
	\norm{\nabla S_1}{\bb{L}^2}
	&\leq
	\left(\norm{\bff{u}_h^{n+1}}{\bb{L}^6} \norm{\nabla \bff{u}_h^{n+1}}{\bb{L}^6} + \norm{\bff{u}_h^n}{\bb{L}^6} \norm{\nabla \bff{u}_h^n}{\bb{L}^6} \right) \norm{\bff{\theta}^{n+\frac{1}{2}}}{\bb{L}^6}
	+
	\left(\norm{\bff{u}_h^{n+1}}{\bb{L}^\infty}^2 + \norm{\bff{u}_h^n}{\bb{L}^\infty}^2 \right) \norm{\nabla \bff{\theta}^{n+\frac{1}{2}}}{\bb{L}^2}
	\\
	&\lesssim
	\norm{\bff{\theta}^n}{\bb{H}^1} + \norm{\bff{\theta}^{n+1}}{\bb{H}^1}.
\end{align*}
Similarly, for the second term,
\begin{align*}
	\norm{\nabla S_2}{\bb{L}^2}
	&\lesssim
	\norm{\bff{\rho}^n}{\bb{H}^1} + \norm{\bff{\rho}^{n+1}}{\bb{H}^1}
	\lesssim
	h^r.
\end{align*}
The terms $\nabla S_3$ up to $\nabla S_6$ are estimated in a similar manner. The details are omitted for brevity. Lastly, the term $\nabla S_7$ can be bounded as in \eqref{equ:psi S7}, giving
\begin{align*}
	\norm{\nabla S_7}{\bb{L}^2} 
	\lesssim
	k^2 \norm{\bff{u}\big(t_{n+\frac{1}{2}}\big)}{\bb{W}^{1,\infty}} \norm{\partial_{tt} \abs{\bff{u}}^2}{L^\infty(\bb{H}^1)}.
\end{align*}
Altogether, by Young's inequality and stability of the projection operator, we have
\begin{align*}
	\left| \inpro{S}{\Delta_h \bff{\zeta}} \right|
	=
	\left|\inpro{\nabla \Pi_h S}{\nabla \bff{\zeta}}\right|
	&\lesssim
	\norm{\nabla S}{\bb{L}^2}^2 + \epsilon \norm{\nabla \bff{\zeta}}{\bb{L}^2}^2
	\lesssim
	h^{2r}+k^4+\norm{\bff{\theta}^n}{\bb{H}^1}^2 + \norm{\bff{\theta}^{n+1}}{\bb{H}^1}^2
	+
	\epsilon \norm{\nabla \bff{\zeta}}{\bb{L}^2}^2,
\end{align*}
thus proving \eqref{equ:psi xi Delta}.
\end{proof}

\begin{lemma}
Let $\epsilon>0$ be arbitrary. If the triangulation $\mathcal{T}_h$ is quasi-uniform, then for any $\bff{\zeta}\in \bb{V}_h$,
\begin{align}\label{equ:psi psi dt un}
	\nonumber
	\left| \inpro{\frac{\psi\big(\bff{u}_h^n,\bff{u}_h^{n+1}\big) - \psi\big(\bff{u}_h^{n-1}, \bff{u}_h^n\big)}{k} - \partial_t \left(\abs{\bff{u}^n}^2 \bff{u}^n\right)}{\bff{\zeta}} \right| 
	&\lesssim
	h^{2(r+1)}+k^4 + \norm{\bff{\zeta}}{\bb{L}^2}^2 +
	\norm{\bff{\theta}^{n+1}}{\bb{L}^2}^2 + \norm{\bff{\rho}^{n+1}}{\bb{L}^2}^2 
	\\
	&\quad
	+ \norm{\bff{\theta}^n}{\bb{L}^2}^2 + \norm{\bff{\rho}^n}{\bb{L}^2}^2
	+ \epsilon \norm{\delta \bff{\theta}^n}{\bb{L}^2}^2.
\end{align}
\end{lemma}

\begin{proof}
After some tedious algebra, we can write the first component in the inner product on the left-hand side as
\begin{align*}
	E 
	&:= 
	\frac{\psi\big(\bff{u}_h^n,\bff{u}_h^{n+1}\big) - \psi\big(\bff{u}_h^{n-1}, \bff{u}_h^n\big)}{k} - \partial_t \left(\abs{\bff{u}^n}^2 \bff{u}^n\right)
	\\
	&=
	\frac{1}{2} \left(\abs{\bff{u}_h^{n+1}}^2 + \abs{\bff{u}_h^n}^2 \right)  \left(\delta \bff{u}_h^n - \partial_t \bff{u}^n \right)
	+
	\left( \frac{\abs{\bff{u}_h^{n+1}}^2+\abs{\bff{u}_h^n}^2}{2} - \abs{\bff{u}^n}^2\right) \partial_t \bff{u}^n
	\\
	&\quad
	+
	\frac{1}{2} \left[\left(\delta \bff{u}_h^n- \partial_t \bff{u}^n\right)\cdot \left(\bff{u}_h^{n+1}+\bff{u}_h^{n-1}\right) \right] \left(\bff{u}_h^n+\bff{u}_h^{n-1}\right)
	+
	\left[\partial_t \bff{u}^n \cdot \left(\frac{\bff{u}_h^{n+1}+\bff{u}_h^{n-1}}{2}-\bff{u}^n\right)\right] \left(\bff{u}_h^n+\bff{u}_h^{n-1}\right)
	\\
	&\quad
	+
	2 \left(\bff{u}^n\cdot \partial_t \bff{u}^n\right) \left(\frac{\bff{u}_h^n+\bff{u}_h^{n-1}}{2}- \bff{u}^n\right)
	\\
	&=: E_1+E_2+E_3+E_4+E_5.
\end{align*}
We want to obtain a bound for $\norm{E}{\bb{L}^2}$. To this end, we will estimate the $\bb{L}^2$ norm of each term above. Note that we have \eqref{equ:stab uh L infty}, which we will use without further mention. Firstly, by H\"older's inequality,
\begin{align}\label{equ:L1}
	\norm{E_1}{\bb{L}^2}
	&\lesssim
	\left( \norm{\bff{u}_h^{n+1}}{\bb{L}^\infty}^2 + \norm{\bff{u}_h^n}{\bb{L}^\infty}^2 \right)
	\norm{\delta \bff{\theta}^n+\delta \bff{\rho}^n+\delta \bff{u}^n-\partial_t \bff{u}^n}{\bb{L}^2}
	\lesssim
	\norm{\delta \bff{\theta}^n}{\bb{L}^2} + h^{r+1} + k^2,
\end{align}
where we used \eqref{equ:d rho dt u} and the triangle inequality. For the second term, note that by subtracting and adding $\frac{1}{2} \left(\abs{\bff{u}^{n+1}}^2 + \abs{\bff{u}^n}^2\right)$, we have by H\"older's inequality,
\begin{align*}
	\norm{\frac{\abs{\bff{u}_h^{n+1}}^2+\abs{\bff{u}_h^n}^2}{2} - \abs{\bff{u}^n}^2}{\bb{L}^2} 
	&\lesssim
	\norm{\left(\bff{\theta}^{n+1}+\bff{\rho}^{n+1}\right) \cdot \left(\bff{u}_h^{n+1}+\bff{u}^{n+1}\right)}{\bb{L}^2}
	+
	\norm{\left(\bff{\theta}^n+\bff{\rho}^n\right) \cdot \left(\bff{u}_h^n+\bff{u}^n\right)}{\bb{L}^2}
	\\
	&\quad
	+
	\norm{\frac{\abs{\bff{u}^{n+1}}^2+\abs{\bff{u}^n}^2}{2} - \abs{\bff{u}^n}^2}{\bb{L}^2}
	\\
	&\lesssim
	\norm{\bff{\theta}^{n+1}}{\bb{L}^2} + \norm{\bff{\rho}^{n+1}}{\bb{L}^2} + \norm{\bff{\theta}^n}{\bb{L}^2} + \norm{\bff{\rho}^n}{\bb{L}^2} + k^2,
\end{align*}
where in the last step we also used Taylor's theorem with integral remainder as in \eqref{equ:psi S7}. Therefore,
\begin{align}\label{equ:L2}
	\nonumber
	\norm{E_2}{\bb{L}^2}
	&\lesssim
	\norm{\frac{\abs{\bff{u}_h^{n+1}}^2+\abs{\bff{u}_h^n}^2}{2} - \abs{\bff{u}^n}^2}{\bb{L}^2} 
	\norm{\partial_t \bff{u}^n}{\bb{L}^\infty}
	\\
	&\lesssim
	\norm{\bff{\theta}^{n+1}}{\bb{L}^2} + \norm{\bff{\rho}^{n+1}}{\bb{L}^2} + \norm{\bff{\theta}^n}{\bb{L}^2} + \norm{\bff{\rho}^n}{\bb{L}^2} + k^2.
\end{align}
The third term can be estimated in the same way as $E_1$, giving
\begin{align}\label{equ:L3}
	\norm{E_3}{\bb{L}^2}
	&\lesssim
	\norm{\delta \bff{\theta}^n}{\bb{L}^2} + h^{r+1} + k^2,
\end{align}
while the terms $E_4$ and $E_5$ can be estimated in a similar manner as $E_2$, giving
\begin{align}\label{equ:L4 L5}
	\norm{E_4}{\bb{L}^2} + \norm{E_5}{\bb{L}^2}
	\lesssim
	\norm{\bff{\theta}^{n+1}}{\bb{L}^2} + \norm{\bff{\rho}^{n+1}}{\bb{L}^2} + \norm{\bff{\theta}^n}{\bb{L}^2} + \norm{\bff{\rho}^n}{\bb{L}^2} + k^2.
\end{align}
Altogether, \eqref{equ:L1}, \eqref{equ:L2}, \eqref{equ:L3}, \eqref{equ:L4 L5}, and Young's inequality yield the required result.
\end{proof}

We have the following superconvergence estimates on $\bff{\theta}^n$ and $\bff{\xi}^n$, which form an essential step in the proof of the main theorem, analogous to Proposition~\ref{pro:semidisc theta xi}.

\begin{proposition}\label{pro:theta xi n half}
Assume that $\bff{u}$ and $\bff{H}$ satisfisfy \eqref{equ:ass 2 u}. Then for $h, k>0$ and~$n\in \{1,2,\ldots,\lfloor T/k \rfloor\}$,
\begin{align}\label{equ:theta n L2}
	\norm{\bff{\theta}^n}{\bb{L}^2}^2
	+
	k \sum_{m=0}^{n} \norm{\nabla \bff{\theta}^{m+\frac{1}{2}}}{\bb{L}^2}^2
	+
	k \sum_{m=0}^{n} \norm{\bff{\xi}^{m+\frac{1}{2}}}{\bb{L}^2}^2
	&\leq
	C \big(h^{2(r+1)}+k^4 \big).
\end{align} 
Moreover, if the triangulation $\mathcal{T}_h$ is globally quasi-uniform, then
\begin{align}
	\label{equ:xi nab theta n L2}
	\norm{\bff{\xi}^{n+\frac{1}{2}}}{\bb{L}^2}^2
	+
	\norm{\nabla \bff{\theta}^{n+\frac{1}{2}}}{\bb{L}^2}^2
	+
	k \sum_{m=1}^{n-1} \norm{\delta \bff{\theta}^m}{\bb{L}^2}^2
	+
	k \sum_{m=1}^{n-1} \norm{\nabla \bff{\xi}^{m+\frac{1}{2}} +\nabla \bff{\xi}^{m-\frac{1}{2}}}{\bb{L}^2}^2
	&\leq
	C \big(h^{2(r+1)}+k^4 \big),
	\\
	\label{equ:Delta theta n L2}
	\norm{\Delta_h \bff{\theta}^{n+\frac{1}{2}}}{\bb{L}^2}^2
	+
	\norm{\bff{\theta}^{n+\frac{1}{2}}}{\bb{L}^\infty}^2
	&\leq
	C \big(h^{2(r+1)}+k^4 \big),
\end{align}
where $C$ depends on the coefficients of the equation, $|\mathscr{D}|$, $T$, and $K_0$ (as defined in \eqref{equ:ass 2 u}), but is independent of $n$, $h$, and $k$.
\end{proposition}

\begin{proof}
Subtracting \eqref{equ:weakform} from \eqref{equ:cranknic} at time step $n+\frac{1}{2}$, using \eqref{equ:Un utn}, \eqref{equ:Hn Htn} (and noting the definition of Ritz projection), we obtain for all $\bff{\chi}$, $\bff{\phi}\in \bb{V}_h$,
\begin{align}\label{equ:dt theta disc}
	\nonumber
	\inpro{\delta \bff{\theta}^{n+\frac{1}{2}} + \delta \bff{\rho}^{n+\frac{1}{2}} + \delta \bff{u}^{n+\frac{1}{2}}- \partial_t \bff{u}\big(t_{n+\frac{1}{2}}\big)}{\bff{\chi}}
	&=
	\lambda_r \inpro{\bff{\xi}^{n+\frac{1}{2}}+\bff{\eta}^{n+\frac{1}{2}}}{\bff{\chi}}
	+
	\lambda_e \inpro{\nabla \bff{\xi}^{n+\frac{1}{2}}}{\nabla \bff{\chi}}
	\\
	&\quad 
	-
	\gamma \inpro{\widehat{\bff{u}}_h^{n-\frac{1}{2}} \times \bff{H}_h^{n+\frac{1}{2}} - \bff{u}\big(t_{n+\frac{1}{2}}\big) \times \bff{H}\big(t_{n+\frac{1}{2}}\big)}{\bff{\chi}}
\end{align}
and
\begin{align}
	\nonumber
	\label{equ:xi eta disc}
	\inpro{\bff{\xi}^{n+\frac{1}{2}}+\bff{\eta}^{n+\frac{1}{2}}}{\bff{\phi}}
	&=
	-\inpro{\nabla \bff{\theta}^{n+\frac{1}{2}}}{\nabla \bff{\phi}}
	+
	\kappa\mu \inpro{\bff{\theta}^{n+\frac{1}{2}}+\bff{\rho}^{n+\frac{1}{2}}}{\bff{\phi}}
		-
	\beta \inpro{\bff{e}\big(\bff{e}\cdot (\bff{\theta}^{n+\frac12}+\bff{\rho}^{n+\frac12})\big)}{\bff{\phi}}
	\\
	&\quad 
	-
	\kappa \inpro{\psi\big(\bff{u}_h^n, \bff{u}_h^{n+1}\big) - \big|\bff{u}\big(t_{n+\frac{1}{2}}\big)\big|^2 \bff{u}\big(t_{n+\frac{1}{2}}\big)}{\bff{\phi}}.
\end{align}
Taking $\bff{\chi}= \bff{\theta}^{n+\frac{1}{2}}$ in \eqref{equ:dt theta disc} and $\bff{\phi}=\lambda_r \bff{\theta}^{n+\frac{1}{2}}$ in \eqref{equ:xi eta disc}, then adding the resulting expressions, we obtain
\begin{align}\label{equ:dt theta chi disc}
	\nonumber
	&\frac{1}{2k} \left( \norm{\bff{\theta}^{n+1}}{\bb{L}^2}^2 - \norm{\bff{\theta}^n}{\bb{L}^2}^2 \right)
	+
	\inpro{\delta \bff{\rho}^{n+\frac{1}{2}} + \delta \bff{u}^{n+\frac{1}{2}} -\partial_t \bff{u}\big(t_{n+\frac{1}{2}}\big)}{\bff{\theta}^{n+\frac{1}{2}}}
	+
	\lambda_r \norm{\nabla \bff{\theta}^{n+\frac{1}{2}}}{\bb{L}^2}^2
	\\
	\nonumber
	&=
	\kappa\mu \lambda_r \norm{\bff{\theta}^{n+\frac{1}{2}}}{\bb{L}^2}^2
	+
	\kappa\mu \lambda_r \inpro{\bff{\rho}^{n+\frac{1}{2}}}{\bff{\theta}^{n+\frac{1}{2}}}
	+
	\lambda_e \inpro{\nabla \bff{\xi}^{n+\frac{1}{2}}}{\nabla \bff{\theta}^{n+\frac{1}{2}}}
	\\
	\nonumber
	&\quad
	-
	\gamma \inpro{\widehat{\bff{u}}_h^{n-\frac{1}{2}} \times \bff{H}_h^{n+\frac{1}{2}} - \bff{u}\big(t_{n+\frac{1}{2}}\big) \times \bff{H}\big(t_{n+\frac{1}{2}}\big)}{\bff{\theta}^{n+\frac{1}{2}}}
	\\
	&\quad 
	-
	\kappa \lambda_r \inpro{\psi\big(\bff{u}_h^n, \bff{u}_h^{n+1}\big) - \big|\bff{u}\big(t_{n+\frac{1}{2}}\big)\big|^2 \bff{u}\big(t_{n+\frac{1}{2}}\big)}{\bff{\theta}^{n+\frac{1}{2}}}
	-
	\beta\lambda_r \inpro{\bff{e}\big(\bff{e}\cdot (\bff{\theta}^{n+\frac12}+\bff{\rho}^{n+\frac12})\big)}{\bff{\theta}^{n+\frac12}}.
\end{align}
Next, taking $\bff{\phi}= \lambda_e \bff{\xi}^{n+\frac{1}{2}}$, we obtain
\begin{align}\label{equ:xi phi xi disc}
	\nonumber
	&\lambda_e \norm{\bff{\xi}^{n+\frac{1}{2}}}{\bb{L}^2}^2
	+
	\lambda_e \inpro{\bff{\eta}^{n+\frac{1}{2}}}{\bff{\xi}^{n+\frac{1}{2}}}
	\\
	&=
	-
	\lambda_e \inpro{\nabla \bff{\theta}^{n+\frac{1}{2}}}{\nabla \bff{\xi}^{n+\frac{1}{2}}}
	+
	\kappa\mu \lambda_e \inpro{\bff{\theta}^{n+\frac{1}{2}}+\bff{\rho}^{n+\frac{1}{2}}}{\bff{\xi}^{n+\frac{1}{2}}}
	\nonumber \\
	&\quad 
	-
	\kappa \lambda_e \inpro{\psi\big(\bff{u}_h^n, \bff{u}_h^{n+1}\big) - \big|\bff{u}\big(t_{n+\frac{1}{2}}\big)\big|^2 \bff{u}\big(t_{n+\frac{1}{2}}\big)}{\bff{\xi}^{n+\frac{1}{2}}}
	-
	\beta\lambda_e \inpro{\bff{e}\big(\bff{e}\cdot (\bff{\theta}^{n+\frac12}+\bff{\rho}^{n+\frac12})\big)}{\bff{\xi}^{n+\frac12}}.
\end{align}
Adding \eqref{equ:dt theta chi disc} and \eqref{equ:xi phi xi disc} gives
\begin{align*}
	&\frac{1}{2k} \left( \norm{\bff{\theta}^{n+1}}{\bb{L}^2}^2 - \norm{\bff{\theta}^n}{\bb{L}^2}^2 \right)
	+
	\lambda_r \norm{\nabla \bff{\theta}^{n+\frac{1}{2}}}{\bb{L}^2}^2
	+
	\lambda_e \norm{\bff{\xi}^{n+\frac{1}{2}}}{\bb{L}^2}^2
	\\
	&=
	-\inpro{\delta \bff{\rho}^{n+\frac{1}{2}} + \delta \bff{u}^{n+\frac{1}{2}} -\partial_t \bff{u}\big(t_{n+\frac{1}{2}}\big)}{\bff{\theta}^{n+\frac{1}{2}}}
	-
	\lambda_e \inpro{\bff{\eta}^{n+\frac{1}{2}}}{\bff{\xi}^{n+\frac{1}{2}}}
	+
	\kappa\mu \lambda_e \inpro{\bff{\theta}^{n+\frac{1}{2}}+\bff{\rho}^{n+\frac{1}{2}}}{\bff{\xi}^{n+\frac{1}{2}}}
	\\
	&\quad
	+
	\kappa\mu \lambda_r \norm{\bff{\theta}^{n+\frac{1}{2}}}{\bb{L}^2}^2
	+
	\kappa\mu\lambda_r \inpro{\bff{\rho}^{n+\frac{1}{2}}}{\bff{\theta}^{n+\frac{1}{2}}}
	-
	\gamma \inpro{\widehat{\bff{u}}_h^{n-\frac{1}{2}} \times \bff{H}_h^{n+\frac{1}{2}} - \bff{u}\big(t_{n+\frac{1}{2}}\big) \times \bff{H}\big(t_{n+\frac{1}{2}}\big)}{\bff{\theta}^{n+\frac{1}{2}}}
	\\
	&\quad 
	-
	\kappa \lambda_r \inpro{\psi\big(\bff{u}_h^n, \bff{u}_h^{n+1}\big) - \big|\bff{u}\big(t_{n+\frac{1}{2}}\big)\big|^2 \bff{u}\big(t_{n+\frac{1}{2}}\big)}{\bff{\theta}^{n+\frac{1}{2}}}
	-
	\beta\lambda_r \inpro{\bff{e}\big(\bff{e}\cdot (\bff{\theta}^{n+\frac12}+\bff{\rho}^{n+\frac12})\big)}{\bff{\theta}^{n+\frac12}}\\
	&\quad
	-
	\kappa \lambda_e \inpro{\psi\big(\bff{u}_h^n, \bff{u}_h^{n+1}\big) - \big|\bff{u}\big(t_{n+\frac{1}{2}}\big)\big|^2 \bff{u}\big(t_{n+\frac{1}{2}}\big)}{\bff{\xi}^{n+\frac{1}{2}}}
	-
	\beta\lambda_e \inpro{\bff{e}\big(\bff{e}\cdot (\bff{\theta}^{n+\frac12}+\bff{\rho}^{n+\frac12})\big)}{\bff{\xi}^{n+\frac12}}.
\end{align*}
It remains to bound each term on the right-hand side. We apply \eqref{equ:d rho dt u} and Young's inequality as necessary for the first five terms. The last four terms can be estimated by applying Lemma~\ref{lem:cross disc} and Lemma~\ref{lem:cub disc}. We then infer that for any $\epsilon>0$,
\begin{align*}
	&\frac{1}{2k} \left( \norm{\bff{\theta}^{n+1}}{\bb{L}^2}^2 - \norm{\bff{\theta}^n}{\bb{L}^2}^2 \right)
	+
	\lambda_r \norm{\nabla \bff{\theta}^{n+\frac{1}{2}}}{\bb{L}^2}^2
	+
	\lambda_e \norm{\bff{\xi}^{n+\frac{1}{2}}}{\bb{L}^2}^2
	\\
	&\lesssim
	h^{2(r+1)}+k^4
	+
	\norm{\bff{\theta}^n}{\bb{L}^2}^2
	+
	\norm{\bff{\theta}^{n+1}}{\bb{L}^2}^2
	+
	\epsilon \norm{\nabla \bff{\theta}^{n+\frac{1}{2}}}{\bb{L}^2}^2
	+
	\epsilon \norm{\bff{\xi}^{n+\frac{1}{2}}}{\bb{L}^2}^2.
\end{align*}
Choosing $\epsilon>0$ sufficiently small, rearranging the terms, and summing over $m\in \{0,1, \ldots, n-1\}$, we have
\begin{align}\label{equ:sum theta n}
	\norm{\bff{\theta}^n}{\bb{L}^2}^2
	+
	k \sum_{m=0}^{n-1} \norm{\nabla \bff{\theta}^{m+\frac{1}{2}}}{\bb{L}^2}^2
	+
	k \sum_{m=0}^{n-1} \norm{\bff{\xi}^{m+\frac{1}{2}}}{\bb{L}^2}^2
	\leq
	C \big(h^{2(r+1)} +k^4 \big)
	+
	Ck \sum_{m=0}^{n-1} \norm{\bff{\theta}^m}{\bb{L}^2}^2,
\end{align}
where $C$ is a constant depending on the coefficients of the equation and $T$ (but is independent of $n$, $h$, and $k$). The discrete Gronwall inequality then yields \eqref{equ:theta n L2}.

We aim to prove \eqref{equ:xi nab theta n L2} next.
First, we consider the difference of the second equation in \eqref{equ:cranknic} at time steps $n+\frac{1}{2}$ and $n-\frac{1}{2}$. After dividing the result by $k$, then subtracting it from the corresponding equation in \eqref{equ:weakform} and taking $\bff{\phi}= \lambda_e \big(\bff{\xi}^{n+\frac{1}{2}}+\bff{\xi}^{n-\frac{1}{2}}\big)$, we obtain (noting \eqref{equ:nab Ritz disc} and notations defined in \eqref{equ:v n min half}),
\begin{align}\label{equ:dt xi n half}
	\nonumber
	\lambda_e \left(\frac{\norm{\bff{\xi}^{n+\frac{1}{2}}}{\bb{L}^2}^2 - \norm{\bff{\xi}^{n-\frac{1}{2}}}{\bb{L}^2}^2}{k}\right) 
	&=
	-
	\kappa \lambda_e \inpro{\delta \bff{\eta}^n}{\bff{\xi}^{n+\frac{1}{2}}+\bff{\xi}^{n-\frac{1}{2}}}
	-
	\lambda_e \inpro{\delta \bff{H}(t_n) - \partial_t \bff{H}^n}{\bff{\xi}^{n+\frac{1}{2}}+\bff{\xi}^{n-\frac{1}{2}}}
	\\
	\nonumber
	&\quad
	-
	\lambda_e \inpro{\nabla \delta \bff{\theta}^n}{\nabla \bff{\xi}^{n+\frac{1}{2}}+\nabla \bff{\xi}^{n-\frac{1}{2}}}
	+
	\kappa\mu \lambda_e \inpro{\delta \bff{u}_h^n-\partial_t \bff{u}^n}{\bff{\xi}^{n+\frac{1}{2}}+ \bff{\xi}^{n-\frac{1}{2}}}
	\\
	\nonumber
	&\quad
	-
	\lambda_e \inpro{\nabla \delta \bff{u}(t_n) - \nabla\partial_t \bff{u}^n}{\nabla \bff{\xi}^{n+\frac{1}{2}}+ \nabla \bff{\xi}^{n-\frac{1}{2}}}
	\\
	&\quad
	-
	\kappa \lambda_e
	\inpro{\frac{\psi\big(\bff{u}_h^n,\bff{u}_h^{n+1}\big) - \psi\big(\bff{u}_h^{n-1}, \bff{u}_h^n\big)}{k} - \partial_t \left(\abs{\bff{u}^n}^2 \bff{u}^n\right)}{\bff{\xi}^{n+\frac{1}{2}}+ \bff{\xi}^{n-\frac{1}{2}}}
	\nonumber \\
	&\quad
	-
	\beta\lambda_e \inpro{\bff{e}\big(\bff{e}\cdot (\delta\bff{u}_h^n-\partial_t \bff{u}^n)\big)}{\bff{\xi}^{n+\frac12}+\bff{\xi}^{n-\frac12}}.
\end{align}
By similar arguments, taking $\bff{\chi}=\bff{\xi}^{n+\frac{1}{2}}+\bff{\xi}^{n-\frac{1}{2}}$, we have
\begin{align}\label{equ:xi nab xi disc}
	\nonumber
	&\lambda_r \norm{\bff{\xi}^{n+\frac{1}{2}}+\bff{\xi}^{n-\frac{1}{2}}}{\bb{L}^2}^2 
	+
	\lambda_e \norm{\nabla \bff{\xi}^{n+\frac{1}{2}}+\nabla \bff{\xi}^{n-\frac{1}{2}}}{\bb{L}^2}^2
	\\
	\nonumber
	&=
	2\inpro{\delta \bff{\theta}^n}{\bff{\xi}^{n+\frac{1}{2}}+\bff{\xi}^{n-\frac{1}{2}}}
	+
	\inpro{\delta \bff{\rho}^{n+\frac{1}{2}}+ \delta \bff{u}^{n+\frac{1}{2}} - \partial_t \bff{u}\big(t_{n+\frac{1}{2}}\big)}{\bff{\xi}^{n+\frac{1}{2}}+\bff{\xi}^{n-\frac{1}{2}}}
	\\
	\nonumber
	&\quad
	+
	\inpro{\delta \bff{\rho}^{n-\frac{1}{2}} + \delta \bff{u}^{n-\frac{1}{2}} - \partial_t \bff{u}\big(t_{n-\frac{1}{2}}\big)}{\bff{\xi}^{n+\frac{1}{2}}+\bff{\xi}^{n-\frac{1}{2}}}
	-
	\lambda_r \inpro{\bff{\eta}^{n+\frac{1}{2}}+\bff{\eta}^{n-\frac{1}{2}}}{\bff{\xi}^{n+\frac{1}{2}}+\bff{\xi}^{n-\frac{1}{2}}}
	\\
	&\;
	+
	\gamma \inpro{\widehat{\bff{u}}_h^{n-\frac{1}{2}} \times \bff{H}_h^{n+\frac{1}{2}} - \bff{u}\big(t_{n+\frac{1}{2}}\big)\times \bff{H}\big(t_{n+\frac{1}{2}}\big)
	+ \widehat{\bff{u}}_h^{n-\frac{3}{2}} \times \bff{H}_h^{n-\frac{1}{2}} - \bff{u}\big(t_{n-\frac{1}{2}}\big)\times \bff{H}\big(t_{n-\frac{1}{2}}\big)}{\bff{\xi}^{n+\frac{1}{2}}+\bff{\xi}^{n-\frac{1}{2}}}.
\end{align}
Next, we add \eqref{equ:dt theta disc} at time step $n+\frac{1}{2}$ and $n-\frac{1}{2}$. Taking $\bff{\chi}= \delta \bff{\theta}^n= \big(\bff{\theta}^{n+\frac{1}{2}}- \bff{\theta}^{n-\frac{1}{2}}\big)/k$, we have
\begin{align}\label{equ:d theta n disc}
	\nonumber
	2 \norm{\delta \bff{\theta}^n}{\bb{L}^2}^2
	&=
	- \inpro{\delta \bff{\rho}^{n+\frac{1}{2}}+ \delta \bff{u}^{n+\frac{1}{2}}- \partial_t \bff{u}\big(t_{n+\frac{1}{2}}\big) + \delta \bff{\rho}^{n-\frac{1}{2}}+ \delta \bff{u}^{n-\frac{1}{2}}- \partial_t \bff{u}\big(t_{n-\frac{1}{2}}\big)}{\delta \bff{\theta}^n}
	\\
	\nonumber
	&\quad
	+
	\lambda_r \inpro{\bff{\xi}^{n+\frac{1}{2}}+ \bff{\xi}^{n-\frac{1}{2}} + \bff{\eta}^{n+\frac{1}{2}} + \bff{\eta}^{n-\frac{1}{2}}}{\delta \bff{\theta}^n}
	+
	\lambda_e \inpro{\nabla \bff{\xi}^{n+\frac{1}{2}} + \nabla \bff{\xi}^{n-\frac{1}{2}}}{\nabla \delta \bff{\theta}^n}
	\\
	&\;
	-
	\gamma \inpro{\widehat{\bff{u}}_h^{n-\frac{1}{2}} \times \bff{H}_h^{n+\frac{1}{2}} - \bff{u}\big(t_{n+\frac{1}{2}}\big)\times \bff{H}\big(t_{n+\frac{1}{2}}\big)
	+ \widehat{\bff{u}}_h^{n-\frac{3}{2}} \times \bff{H}_h^{n-\frac{1}{2}} - \bff{u}\big(t_{n-\frac{1}{2}}\big)\times \bff{H}\big(t_{n-\frac{1}{2}}\big)}{\delta \bff{\theta}^n}.
\end{align}
Furthermore, we add \eqref{equ:xi eta disc} at time steps $n+\frac{1}{2}$ and $n-\frac{1}{2}$. Taking $\bff{\phi}=\lambda_r \delta \bff{\theta}^n$ and rearranging the terms, we have
\begin{align}\label{equ:nab d theta n half disc}
	\nonumber
	&\lambda_r \left(\frac{\norm{\nabla \bff{\theta}^{n+\frac{1}{2}}}{\bb{L}^2}^2 - \norm{\nabla \bff{\theta}^{n-\frac{1}{2}}}{\bb{L}^2}^2}{k}\right)
	+
	\lambda_r \inpro{\bff{\xi}^{n+\frac{1}{2}}+ \bff{\xi}^{n-\frac{1}{2}} + \bff{\eta}^{n+\frac{1}{2}} + \bff{\eta}^{n-\frac{1}{2}}}{\delta \bff{\theta}^n}
	\\
	\nonumber
	&=
	\kappa\mu \lambda_r \inpro{\bff{\theta}^{n+\frac{1}{2}}+ \bff{\rho}^{n+\frac{1}{2}} + \bff{\theta}^{n-\frac{1}{2}} + \bff{\rho}^{n-\frac{1}{2}}}{\delta \bff{\theta}^n}
	-
	\beta\lambda_r \inpro{\bff{e}\big(\bff{e}\cdot (\bff{\theta}^{n+\frac{1}{2}}+ \bff{\rho}^{n+\frac{1}{2}} + \bff{\theta}^{n-\frac{1}{2}} + \bff{\rho}^{n-\frac{1}{2}})\big)}{\delta \bff{\theta}^n}
	\\
	&\quad
	-\kappa \lambda_r \inpro{\psi\big(\bff{u}_h^n,\bff{u}_h^{n+1}\big) - \abs{\bff{u}\big(t_{n+\frac{1}{2}}\big)}^2 \bff{u}\big(t_{n+\frac{1}{2}}\big) + \psi\big(\bff{u}_h^{n-1}, \bff{u}_h^n\big) - \abs{\bff{u}\big(t_{n-\frac{1}{2}}\big)}^2 \bff{u}\big(t_{n-\frac{1}{2}}\big)}{\delta \bff{\theta}^n}.
\end{align}
Adding \eqref{equ:dt xi n half}, \eqref{equ:xi nab xi disc}, \eqref{equ:d theta n disc}, and \eqref{equ:nab d theta n half disc}, we then obtain, for any $n\in \bb{N}$,
\begin{align}\label{equ:xi nab theta}
	\nonumber
	&\lambda_e \left(\frac{\norm{\bff{\xi}^{n+\frac{1}{2}}}{\bb{L}^2}^2 - \norm{\bff{\xi}^{n-\frac{1}{2}}}{\bb{L}^2}^2}{k}\right) 
	+
	\lambda_r \left(\frac{\norm{\nabla \bff{\theta}^{n+\frac{1}{2}}}{\bb{L}^2}^2 - \norm{\nabla \bff{\theta}^{n-\frac{1}{2}}}{\bb{L}^2}^2}{k}\right)
	\\
	\nonumber
	&\quad
	+
	\norm{\delta \bff{\theta}^n}{\bb{L}^2}^2
	+
	\lambda_r \norm{\bff{\xi}^{n+\frac{1}{2}}+\bff{\xi}^{n-\frac{1}{2}}}{\bb{L}^2}^2 
	+
	\lambda_e \norm{\nabla \bff{\xi}^{n+\frac{1}{2}}+\nabla \bff{\xi}^{n-\frac{1}{2}}}{\bb{L}^2}^2
	\\
	\nonumber
	&=
	-
	\kappa \lambda_e \inpro{\delta \bff{\eta}^n}{\bff{\xi}^{n+\frac{1}{2}}+\bff{\xi}^{n-\frac{1}{2}}}
	-
	\lambda_e \inpro{\delta \bff{H}(t_n) - \partial_t \bff{H}^n}{\bff{\xi}^{n+\frac{1}{2}}+\bff{\xi}^{n-\frac{1}{2}}}
	\\
	\nonumber
	&\quad
	+
	\kappa\mu \lambda_e \inpro{\delta \bff{u}_h^n-\partial_t \bff{u}^n}{\bff{\xi}^{n+\frac{1}{2}}+ \bff{\xi}^{n-\frac{1}{2}}}
	-
	\lambda_e \inpro{\nabla \delta \bff{u}(t_n) - \nabla\partial_t \bff{u}^n}{\nabla \bff{\xi}^{n+\frac{1}{2}}+ \nabla \bff{\xi}^{n-\frac{1}{2}}}
	\\
	\nonumber
	&\quad
	-
	\kappa \lambda_e
	\inpro{\frac{\psi\big(\bff{u}_h^n,\bff{u}_h^{n+1}\big) - \psi\big(\bff{u}_h^{n-1}, \bff{u}_h^n\big)}{k} - \partial_t \left(\abs{\bff{u}^n}^2 \bff{u}^n\right)}{\bff{\xi}^{n+\frac{1}{2}}+ \bff{\xi}^{n-\frac{1}{2}}}
	\\
	\nonumber
	&\quad
	-
	\beta\lambda_e \inpro{\bff{e}\big(\bff{e}\cdot (\delta\bff{u}_h^n-\partial_t \bff{u}^n)\big)}{\bff{\xi}^{n+\frac12}+\bff{\xi}^{n-\frac12}}
	\\
	\nonumber
	&\quad
	+
	2\inpro{\delta \bff{\theta}^n}{\bff{\xi}^{n+\frac{1}{2}}+\bff{\xi}^{n-\frac{1}{2}}}
	+
	\inpro{\delta \bff{\rho}^{n+\frac{1}{2}}+ \delta \bff{u}^{n+\frac{1}{2}} - \partial_t \bff{u}\big(t_{n+\frac{1}{2}}\big)}{\bff{\xi}^{n+\frac{1}{2}}+\bff{\xi}^{n-\frac{1}{2}}}
	\\
	\nonumber
	&\quad
	+
	\inpro{\delta \bff{\rho}^{n-\frac{1}{2}} + \delta \bff{u}^{n-\frac{1}{2}} - \partial_t \bff{u}\big(t_{n-\frac{1}{2}}\big)}{\bff{\xi}^{n+\frac{1}{2}}+\bff{\xi}^{n-\frac{1}{2}}}
	\\
	\nonumber
	&\quad
	-
	\lambda_r \inpro{\bff{\eta}^{n+\frac{1}{2}}+\bff{\eta}^{n-\frac{1}{2}}}{\bff{\xi}^{n+\frac{1}{2}}+\bff{\xi}^{n-\frac{1}{2}}}
	\\
	\nonumber
	&\quad
	+
	\gamma \inpro{\widehat{\bff{u}}_h^{n-\frac{1}{2}} \times \bff{H}_h^{n+\frac{1}{2}} - \bff{u}\big(t_{n+\frac{1}{2}}\big)\times \bff{H}\big(t_{n+\frac{1}{2}}\big)
	+ \widehat{\bff{u}}_h^{n-\frac{3}{2}} \times \bff{H}_h^{n-\frac{1}{2}} - \bff{u}\big(t_{n-\frac{1}{2}}\big)\times \bff{H}\big(t_{n-\frac{1}{2}}\big)}{\bff{\xi}^{n+\frac{1}{2}}+\bff{\xi}^{n-\frac{1}{2}}}
	\\
	\nonumber
	&\quad 
	- 
	\inpro{\delta \bff{\rho}^{n+\frac{1}{2}}+ \delta \bff{u}^{n+\frac{1}{2}}- \partial_t \bff{u}\big(t_{n+\frac{1}{2}}\big) + \delta \bff{\rho}^{n-\frac{1}{2}}+ \delta \bff{u}^{n-\frac{1}{2}}- \partial_t \bff{u}\big(t_{n-\frac{1}{2}}\big)}{\delta \bff{\theta}^n}
	\\
	\nonumber
	&\quad
	-
	\gamma \inpro{\widehat{\bff{u}}_h^{n-\frac{1}{2}} \times \bff{H}_h^{n+\frac{1}{2}} - \bff{u}\big(t_{n+\frac{1}{2}}\big)\times \bff{H}\big(t_{n+\frac{1}{2}}\big)
	+ 
	\widehat{\bff{u}}_h^{n-\frac{3}{2}} \times \bff{H}_h^{n-\frac{1}{2}} - \bff{u}\big(t_{n-\frac{1}{2}}\big)\times \bff{H}\big(t_{n-\frac{1}{2}}\big)}{\delta \bff{\theta}^n}
	\\
	\nonumber
	&\quad
	+
	\kappa\mu \lambda_r \inpro{\bff{\theta}^{n+\frac{1}{2}}+ \bff{\rho}^{n+\frac{1}{2}} + \bff{\theta}^{n-\frac{1}{2}} + \bff{\rho}^{n-\frac{1}{2}}}{\delta \bff{\theta}^n}
	-
	\beta\lambda_r \inpro{\bff{e}\big(\bff{e}\cdot (\bff{\theta}^{n+\frac{1}{2}}+ \bff{\rho}^{n+\frac{1}{2}} + \bff{\theta}^{n-\frac{1}{2}} + \bff{\rho}^{n-\frac{1}{2}})\big)}{\delta \bff{\theta}^n}
	\\
	\nonumber
	&\quad
	-\kappa \lambda_r \inpro{\psi\big(\bff{u}_h^n,\bff{u}_h^{n+1}\big) - \abs{\bff{u}\big(t_{n+\frac{1}{2}}\big)}^2 \bff{u}\big(t_{n+\frac{1}{2}}\big) + \psi\big(\bff{u}_h^{n-1}, \bff{u}_h^n\big) - \abs{\bff{u}\big(t_{n-\frac{1}{2}}\big)}^2 \bff{u}\big(t_{n-\frac{1}{2}}\big)}{\delta \bff{\theta}^n}
	\\
	&=: I_1+I_2+ \cdots + I_{16}.
\end{align}
There are sixteen terms involving inner products on the right-hand side of \eqref{equ:nab d theta n half disc} which will be estimated in the following. Firstly, by Young's inequality and \eqref{equ:Ritz ineq},
\begin{align}\label{equ:I1}
	\abs{I_1} 
	&\lesssim 
	\norm{\delta \bff{\eta}^n}{\bb{L}^2}^2
	+
	\epsilon \norm{\bff{\xi}^{n+\frac{1}{2}}+\bff{\xi}^{n-\frac{1}{2}}}{\bb{L}^2}^2
	\lesssim
	h^{2(r+1)}+ \epsilon\norm{\bff{\xi}^{n+\frac{1}{2}}+\bff{\xi}^{n-\frac{1}{2}}}{\bb{L}^2}^2.
\end{align}
Secondly, by Young's inequality and Taylor's theorem (noting \eqref{equ:ass 2 u}),
\begin{align*}
	\abs{I_2}
	&\lesssim
	\norm{\delta \bff{H}(t_n) - \partial_t \bff{H}^n}{\bb{L}^2}^2
	+
	\epsilon \norm{\bff{\xi}^{n+\frac{1}{2}}+\bff{\xi}^{n-\frac{1}{2}}}{\bb{L}^2}^2
	\lesssim
	k^4 + \epsilon\norm{\bff{\xi}^{n+\frac{1}{2}}+\bff{\xi}^{n-\frac{1}{2}}}{\bb{L}^2}^2.
\end{align*}
For the terms $I_3$ and $I_6$, by writing $\bff{u}_h^n= \bff{\theta}^n+\bff{\rho}^n+\bff{u}^n$ then applying Young's inequality, we have
\begin{align*}
	\nonumber
	\abs{I_3}+\abs{I_6}
	&\lesssim
	\norm{\bff{\xi}^{n+\frac{1}{2}}+\bff{\xi}^{n-\frac{1}{2}}}{\bb{L}^2}^2
	+
	\epsilon \norm{\delta \bff{\theta}^n}{\bb{L}^2}^2
	+
	\epsilon \norm{\delta \bff{\rho}^n+\delta \bff{u}(t_n)-\partial_t \bff{u}^n}{\bb{L}^2}^2
	\\
	&\lesssim
	\norm{\bff{\xi}^{n+\frac{1}{2}}}{\bb{L}^2}^2 + \norm{\bff{\xi}^{n-\frac{1}{2}}}{\bb{L}^2}^2
	+
	\epsilon \norm{\delta \bff{\theta}^n}{\bb{L}^2}^2
	+
	h^{2(r+1)} + k^4.
\end{align*}
The term $I_4$ can be estimated in the same way as the term $I_2$ (noting \eqref{equ:ass 2 u}), giving
\begin{align*}
	\abs{I_4} \lesssim k^4 + \epsilon\norm{\bff{\xi}^{n+\frac{1}{2}}+\bff{\xi}^{n-\frac{1}{2}}}{\bb{L}^2}^2.
\end{align*}
The term $I_5$ can be estimated by using \eqref{equ:psi psi dt un} with $\bff{\zeta}= \bff{\xi}^{n+\frac{1}{2}}+\bff{\xi}^{n-\frac{1}{2}}$, giving
\begin{align}\label{equ:I5}
	\nonumber
	\abs{I_5}
	&\lesssim
	h^{2(r+1)}+k^4 + \norm{\bff{\xi}^{n+\frac{1}{2}} + \bff{\xi}^{n-\frac{1}{2}}}{\bb{L}^2}^2 +
	\norm{\bff{\theta}^{n+1}}{\bb{L}^2}^2 + \norm{\bff{\rho}^{n+1}}{\bb{L}^2}^2 
	+ \norm{\bff{\theta}^n}{\bb{L}^2}^2 + \norm{\bff{\rho}^n}{\bb{L}^2}^2
	+ \epsilon \norm{\delta \bff{\theta}^n}{\bb{L}^2}^2
	\\
	&\lesssim
	h^{2(r+1)}+k^4 + \norm{\bff{\xi}^{n+\frac{1}{2}} + \bff{\xi}^{n-\frac{1}{2}}}{\bb{L}^2}^2
	+ \epsilon \norm{\delta \bff{\theta}^n}{\bb{L}^2}^2,
\end{align}
where in the last step we used \eqref{equ:Ritz ineq} and \eqref{equ:theta n L2}.
The terms $I_7$ can be estimated directly by applying Young's inequality to obtain
\begin{align*}
	\abs{I_7}
	\lesssim
	\epsilon \norm{\delta \bff{\theta}^n}{\bb{L}^2}^2 
	+
	\norm{\bff{\xi}^{n+\frac{1}{2}}}{\bb{L}^2}^2 + \norm{\bff{\xi}^{n-\frac{1}{2}}}{\bb{L}^2}^2.
\end{align*}
The next two terms can be estimated using Young's inequality and \eqref{equ:d rho dt u}, yielding
\begin{align*}
	\abs{I_8} + \abs{I_9}
	\lesssim
	h^{2(r+1)}+k^4 + \epsilon \norm{\bff{\xi}^{n+\frac{1}{2}}+\bff{\xi}^{n-\frac{1}{2}}}{\bb{L}^2}^2.
\end{align*}
Next, by Young's inequality and \eqref{equ:Ritz ineq},
\begin{align*}
	\abs{I_{10}}
	\lesssim
	\norm{\bff{\eta}^{n+\frac{1}{2}}+\bff{\eta}^{n-\frac{1}{2}}}{\bb{L}^2}^2
	+
	\epsilon \norm{\bff{\xi}^{n+\frac{1}{2}}+\bff{\xi}^{n-\frac{1}{2}}}{\bb{L}^2}^2
	\lesssim
	h^{2(r+1)} + \epsilon \norm{\bff{\xi}^{n+\frac{1}{2}}+\bff{\xi}^{n-\frac{1}{2}}}{\bb{L}^2}^2.
\end{align*}
Now, the term $I_{12}$ can be bounded in a similar way as the terms $I_8$ and $I_9$ to give
\begin{align*}
	\abs{I_{12}}
	\lesssim
	h^{2(r+1)}+k^4 + \epsilon \norm{\delta \bff{\theta}^n}{\bb{L}^2}^2,
\end{align*}
while the terms $I_{11}$ and $I_{13}$ can be estimated by applying \eqref{equ:cross est} and \eqref{equ:theta n L2} to obtain
\begin{align*}
	\abs{I_{11}}+\abs{I_{13}}
	\lesssim
	h^{2(r+1)} + \norm{\bff{\xi}^{n+\frac{1}{2}}}{\bb{L}^2}^2
	+ \norm{\bff{\xi}^{n-\frac{1}{2}}}{\bb{L}^2}^2 + \epsilon \norm{\delta \bff{\theta}^n}{\bb{L}^2}^2.
\end{align*}
The next two terms are straightforward to estimate:
\begin{align*}
	\abs{I_{14}} + \abs{I_{15}}
	\lesssim
	h^{2(r+1)}+ \norm{\bff{\theta}^{n+\frac{1}{2}}}{\bb{L}^2}^2 + \norm{\bff{\theta}^{n-\frac{1}{2}}}{\bb{L}^2}^2
	+ \epsilon \norm{\delta \bff{\theta}^n}{\bb{L}^2}^2
	\lesssim
	h^{2(r+1)}+k^4+ \epsilon \norm{\delta \bff{\theta}^n}{\bb{L}^2}^2.
\end{align*}
Finally, the last term can be estimated by using \eqref{equ:psi xi disc} to give
\begin{align*}
	\abs{I_{16}}
	\lesssim
	h^{2(r+1)}+k^4+ \norm{\bff{\theta}^{n+1}}{\bb{L}^2}^2 + \norm{\bff{\theta}^n}{\bb{L}^2}^2
	+\norm{\bff{\theta}^{n-1}}{\bb{L}^2}^2 + \epsilon \norm{\delta \bff{\theta}^n}{\bb{L}^2}^2
	\lesssim
	h^{2(r+1)}+ k^4+ \epsilon \norm{\delta \bff{\theta}^n}{\bb{L}^2}^2.
\end{align*}
Altogether, applying the above estimates for $I_j$, for $j=1,2,\ldots, 16$, to \eqref{equ:xi nab theta}, choosing $\epsilon>0$ sufficiently small, and rearranging the terms yield
\begin{align*}
	&\lambda_e \left(\frac{\norm{\bff{\xi}^{n+\frac{1}{2}}}{\bb{L}^2}^2 - \norm{\bff{\xi}^{n-\frac{1}{2}}}{\bb{L}^2}^2}{k}\right) 
	+
	\lambda_r \left(\frac{\norm{\nabla \bff{\theta}^{n+\frac{1}{2}}}{\bb{L}^2}^2 - \norm{\nabla \bff{\theta}^{n-\frac{1}{2}}}{\bb{L}^2}^2}{k}\right)
	\\
	&\quad
	+
	\norm{\delta \bff{\theta}^n}{\bb{L}^2}^2
	+
	\lambda_r \norm{\bff{\xi}^{n+\frac{1}{2}}+\bff{\xi}^{n-\frac{1}{2}}}{\bb{L}^2}^2 
	+
	\lambda_e \norm{\nabla \bff{\xi}^{n+\frac{1}{2}}+\nabla \bff{\xi}^{n-\frac{1}{2}}}{\bb{L}^2}^2
\lesssim
	h^{2(r+1)} + k^4 + \norm{\bff{\xi}^{n+\frac{1}{2}}}{\bb{L}^2}^2 + \norm{\bff{\xi}^{n-\frac{1}{2}}}{\bb{L}^2}^2.
\end{align*}
Summing over $m\in \{1,2,\ldots,n-1\}$, and noting the fact that $\bff{u}_h^0=R_h\bff{u}(0)$, so that 
\[
\norm{\bff{\xi}^{\frac{1}{2}}}{\bb{L}^2}^2 + \norm{\nabla \bff{\theta}^{\frac{1}{2}}}{\bb{L}^2}^2 \lesssim h^{2(r+1)}+k^4,
\]
we obtain
\begin{align*}
	\lambda_e \norm{\bff{\xi}^{n-\frac{1}{2}}}{\bb{L}^2}^2
	+
	\lambda_r \norm{\nabla \bff{\theta}^{n-\frac{1}{2}}}{\bb{L}^2}^2
	&+
	k \sum_{m=1}^{n-1} \left( \norm{\delta \bff{\theta}^m}{\bb{L}^2}^2 + \lambda_r \norm{\bff{\xi}^{m+\frac{1}{2}}+\bff{\xi}^{m-\frac{1}{2}}}{\bb{L}^2}^2 
	+
	\lambda_e \norm{\nabla \bff{\xi}^{m+\frac{1}{2}}+\nabla \bff{\xi}^{m-\frac{1}{2}}}{\bb{L}^2}^2 \right)
	\\
	&\lesssim
	h^{2(r+1)}+k^4+ \norm{\bff{\xi}^{\frac{1}{2}}}{\bb{L}^2}^2 + \norm{\nabla \bff{\theta}^{\frac{1}{2}}}{\bb{L}^2}^2
	+
	k \sum_{m=1}^{n-1} \left(\norm{\bff{\xi}^{m+\frac{1}{2}}}{\bb{L}^2}^2 + \norm{\bff{\xi}^{m-\frac{1}{2}}}{\bb{L}^2}^2\right)
	\\
	&\lesssim
	h^{2(r+1)}+k^4,
\end{align*}
where in the last step we also used \eqref{equ:theta n L2}. This implies inequality \eqref{equ:xi nab theta n L2}.

Finally, we will show \eqref{equ:Delta theta n L2}. Taking $\bff{\phi}=\Delta_h \bff{\theta}^{n+\frac{1}{2}}$, rearranging the terms, and applying Young's inequality, we obtain
\begin{align*}
	\norm{\Delta_h \bff{\theta}^{n+\frac{1}{2}}}{\bb{L}^2}^2
	&=
	\kappa\mu \norm{\nabla \bff{\theta}^{n+\frac{1}{2}}}{\bb{L}^2}^2
	+
	\inpro{\bff{\xi}^{n+\frac{1}{2}}}{\Delta_h \bff{\theta}^{n+\frac{1}{2}}}
	+
	\inpro{\bff{\eta}^{n+\frac{1}{2}}}{\Delta_h \bff{\theta}^{n+\frac{1}{2}}}
	-
	\kappa \inpro{\bff{\rho}^{n+\frac{1}{2}}}{\Delta_h \bff{\theta}^{n+\frac{1}{2}}}
	\\
	&\quad
	+
	\kappa \inpro{\psi(\bff{u}_h^n,\bff{u}_h^{n+1})- \big|\bff{u}(t_{n+\frac12})\big|^2 \bff{u}(t_{n+\frac12})}{\Delta_h \bff{\theta}^{n+\frac12}}
	+
	\beta \inpro{\bff{e}\big(\bff{e}\cdot (\bff{\theta}^{n+\frac12}+\bff{\rho}^{n+\frac12})\big)}{\Delta_h \bff{\theta}^{n+\frac12}}
	\\
	&\leq
	C \norm{\bff{\theta}^{n+\frac{1}{2}}}{\bb{H}^1}^2
	+
	C\norm{\bff{\xi}^{n+\frac{1}{2}}}{\bb{L}^2}^2
	+
	C\norm{\bff{\eta}^{n+\frac{1}{2}}}{\bb{L}^2}^2
	+
	C\norm{\bff{\rho}^{n+\frac{1}{2}}}{\bb{L}^2}^2
	+
	\frac{1}{2} \norm{\Delta_h \bff{\theta}^{n+\frac{1}{2}}}{\bb{L}^2}^2
	\\
	&\leq
	C \big(h^{2(r+1)}+ k^4\big) +\frac{1}{2} \norm{\Delta_h \bff{\theta}^{n+\frac{1}{2}}}{\bb{L}^2}^2,
\end{align*}
where in the last step we used \eqref{equ:Ritz ineq}, \eqref{equ:psi xi disc}, and inequality \eqref{equ:xi nab theta n L2} (which has been shown previously).
This, together with inequalities \eqref{equ:theta n L2} and \eqref{equ:disc lapl L infty}, yields \eqref{equ:Delta theta n L2}. The proof is now complete.
\end{proof}

\begin{proposition}\label{pro:nab theta n xi n}
Suppose that the triangulation $\mathcal{T}_h$ is globally quasi-uniform. Then for $h, k>0$ and~$n\in \{1,2,\ldots,\lfloor T/k \rfloor\}$,
\begin{align*}
	\norm{\nabla \bff{\theta}^n}{\bb{L}^2}^2
	+
	k \sum_{m=0}^n \norm{\nabla \bff{\xi}^{m+\frac{1}{2}}}{\bb{L}^2}^2
	\leq
	C \big(h^{2r}+k^4 \big),
\end{align*}
where $C$ depends on the coefficients of the equation, $|\mathscr{D}|$, $T$, and $K_0$, but is independent of $n$, $h$, and $k$.
\end{proposition}

\begin{proof}
Successively taking $\bff{\chi}=\Delta_h \bff{\theta}^{n+\frac{1}{2}}$ and $\bff{\phi}= \lambda_r \Delta_h \bff{\theta}^{n+\frac{1}{2}}$ in \eqref{equ:xi eta disc}, then taking $\bff{\phi}=\lambda_e \Delta_h \bff{\xi}^{n+\frac{1}{2}}$ in \eqref{equ:dt theta disc}, and adding the resulting equations give
\begin{align*}
	\nonumber
	&\frac{1}{2k} \left(\norm{\nabla \bff{\theta}^{n+1}}{\bb{L}^2}^2 - \norm{\nabla \bff{\theta}^n}{\bb{L}^2}^2 \right) 
	+
	\lambda_r \norm{\Delta_h \bff{\theta}^{n+\frac{1}{2}}}{\bb{L}^2}^2
	+
	\lambda_e \norm{\nabla \bff{\xi}^{n+\frac{1}{2}}}{\bb{L}^2}^2
	\\
	\nonumber
	&=
	\inpro{\delta \bff{\rho}^{n+\frac{1}{2}}+\delta \bff{u}^{n+\frac{1}{2}}- \partial_t \bff{u}\big(t_{n+\frac{1}{2}}\big)}{\Delta_h \bff{\theta}^{n+\frac{1}{2}}}
	+
	\lambda_r \norm{\nabla \bff{\theta}^{n+\frac{1}{2}}}{\bb{L}^2}^2
	\\
	\nonumber
	&\quad 
	-
	\kappa\mu \lambda_r \inpro{\bff{\rho}^{n+\frac{1}{2}}}{\Delta_h \bff{\theta}^{n+\frac{1}{2}}}
	+
	\gamma \inpro{\widehat{\bff{u}}_h^{n-\frac{1}{2}}\times \bff{H}_h^{n+\frac{1}{2}}- \bff{u}\big(t_{n+\frac{1}{2}}\big)\times \bff{H}\big(t_{n+\frac{1}{2}}\big)}{\Delta_h \bff{\theta}^{n+\frac{1}{2}}}
	\\
	\nonumber
	&\quad
	+
	\kappa \lambda_r \inpro{\psi\big(\bff{u}_h^n,\bff{u}_h^{n+1}\big) - \abs{\bff{u}\big(t_{n+\frac{1}{2}}\big)}^2 \bff{u}\big(t_{n+\frac{1}{2}}\big)}{\Delta_h \bff{\theta}^{n+\frac{1}{2}}}
	-
	\beta\lambda_r \inpro{\bff{e}\big(\bff{e}\cdot (\bff{\theta}^{n+\frac12}+\bff{\rho}^{n+\frac12})\big)}{\Delta_h \bff{\theta}^{n+\frac12}}
	\\
	\nonumber
	&\quad 
	+
	\lambda_e \inpro{\bff{\eta}^{n+\frac{1}{2}}}{\Delta_h \bff{\xi}^{n+\frac{1}{2}}}
	-
	\kappa \mu\lambda_e \inpro{\bff{\theta}^{n+\frac{1}{2}}}{\Delta_h \bff{\xi}^{n+\frac{1}{2}}}
	-
	\kappa\mu \lambda_e \inpro{\bff{\rho}^{n+\frac{1}{2}}}{\Delta_h \bff{\xi}^{n+\frac{1}{2}}}
	\\
	\nonumber
	&\quad 
	+
	\kappa \lambda_e \inpro{\psi\big(\bff{u}_h^n,\bff{u}_h^{n+1}\big) - \abs{\bff{u}\big(t_{n+\frac{1}{2}}\big)}^2 \bff{u}\big(t_{n+\frac{1}{2}}\big)}{\Delta_h \bff{\xi}^{n+\frac{1}{2}}}
	-
	\beta\lambda_e \inpro{\bff{e}\big(\bff{e}\cdot (\bff{\theta}^{n+\frac12}+\bff{\rho}^{n+\frac12})\big)}{\Delta_h \bff{\xi}^{n+\frac12}}
	\\
	&=: J_1+J_2+\cdots + J_{11}.
\end{align*}
We will estimate each term on the last line. For the first term, by Young's inequality and \eqref{equ:d rho dt u}, we have
\begin{align*}
	\abs{J_1} \lesssim h^{2(r+1)}+k^4+ \epsilon \norm{\Delta_h \bff{\theta}^{n+\frac{1}{2}}}{\bb{L}^2}^2.
\end{align*}
The term $J_2$ will be left as is. For the third term, by Young's inequality
\begin{align*}
	\abs{J_3} \lesssim \norm{\bff{\rho}^{n+\frac{1}{2}}}{\bb{L}^2}^2 + \epsilon \norm{\Delta_h \bff{\theta}^{n+\frac{1}{2}}}{\bb{L}^2}^2
	\lesssim
	h^{2(r+1)}+ \epsilon \norm{\Delta_h \bff{\theta}^{n+\frac{1}{2}}}{\bb{L}^2}^2.
\end{align*}
The next two terms will be bounded by using \eqref{equ:cross est} and \eqref{equ:psi xi disc} respectively, giving
\begin{align*}
	\abs{J_4}
	\lesssim
	h^{2(r+1)} + \norm{\bff{\theta}^{n+\frac{1}{2}}}{\bb{L}^2}^2 + \norm{\bff{\xi}^{n+\frac{1}{2}}}{\bb{L}^2}^2
	+ \epsilon \norm{\Delta_h \bff{\theta}^{n+\frac{1}{2}}}{\bb{L}^2}^2,
	\\
	\abs{J_5}
	\lesssim
	h^{2(r+1)} + k^4 + \norm{\bff{\theta}^n}{\bb{L}^2}^2 + \norm{\bff{\theta}^{n+1}}{\bb{L}^2}^2 
	+ \epsilon \norm{\Delta_h \bff{\theta}^{n+\frac{1}{2}}}{\bb{L}^2}^2.
\end{align*}
The term $J_6$ can be estimated using Young's inequality, giving
\begin{align*}
	\abs{J_6}
	\lesssim
	\norm{\bff{\theta}^{n+\frac12}+\bff{\rho}^{n+\frac12}}{\bb{L}^2}^2
	+
	\epsilon \norm{\Delta_h \bff{\theta}^{n+\frac{1}{2}}}{\bb{L}^2}^2
	\lesssim
	h^{2(r+1)} 
	+
	\epsilon \norm{\Delta_h \bff{\theta}^{n+\frac{1}{2}}}{\bb{L}^2}^2.
\end{align*}
Next, by using \eqref{equ:orth proj} and \eqref{equ:disc laplacian} (noting \eqref{equ:nab Ritz disc}), the term $J_7$ can be written as
\begin{align*}
	J_7= \lambda_e \inpro{\Pi_h \bff{\eta}^{n+\frac{1}{2}}}{\Delta_h \bff{\xi}^{n+\frac{1}{2}}}
	=
	-\lambda_e \inpro{\nabla \Pi_h \bff{\eta}^{n+\frac{1}{2}}-\nabla \bff{\eta}^{n+\frac{1}{2}}}{\nabla \bff{\xi}^{n+\frac{1}{2}}},
\end{align*}
and similarly for the term $J_9$.
Therefore, by Young's inequality and \eqref{equ:proj approx},
\begin{align*}
	\abs{J_7}+\abs{J_9} 
	\lesssim 
	h^{2r}+ \epsilon \norm{\nabla \bff{\xi}^{n+\frac{1}{2}}}{\bb{L}^2}^2.
\end{align*}
Next, the terms $J_8$ and $J_{11}$ can be estimated by using \eqref{equ:disc laplacian} and Young's inequality as
\begin{align*}
	\abs{J_8}+\abs{J_{11}}
	\lesssim
	h^{2r}+
	\norm{\nabla \bff{\theta}^{n+\frac{1}{2}}}{\bb{L}^2}^2 + \epsilon \norm{\nabla \bff{\xi}^{n+\frac{1}{2}}}{\bb{L}^2}^2.
\end{align*}
Finally, the term $J_{10}$ can be bounded by using \eqref{equ:psi xi Delta} with $\bff{\zeta}=\bff{\xi}^{n+\frac{1}{2}}$, giving
\begin{align*}
	\abs{J_9} 
	\lesssim
	h^{2r}+k^4+ \norm{\bff{\theta}^n}{\bb{H}^1}^2 + \norm{\bff{\theta}^{n+1}}{\bb{H}^1}^2
	+ \epsilon \norm{\nabla \bff{\xi}^{n+\frac{1}{2}}}{\bb{L}^2}^2.
\end{align*}
Altogether, we obtain
\begin{align*}
	&\frac{1}{2k} \left(\norm{\nabla \bff{\theta}^{n+1}}{\bb{L}^2}^2 - \norm{\nabla \bff{\theta}^n}{\bb{L}^2}^2 \right) 
	+
	\lambda_r \norm{\Delta_h \bff{\theta}^{n+\frac{1}{2}}}{\bb{L}^2}^2
	+
	\lambda_e \norm{\nabla \bff{\xi}^{n+\frac{1}{2}}}{\bb{L}^2}^2
	\\
	&\lesssim
	h^{2r}+k^4 + \norm{\bff{\theta}^n}{\bb{H}^1}^2 + \norm{\bff{\theta}^{n+1}}{\bb{H}^1}^2
	+
	\norm{\bff{\xi}^{n+\frac{1}{2}}}{\bb{L}^2}^2
	+ 
	\epsilon \norm{\Delta_h \bff{\theta}^{n+\frac{1}{2}}}{\bb{L}^2}^2
	+
	\epsilon \norm{\nabla \bff{\xi}^{n+\frac{1}{2}}}{\bb{L}^2}^2.
\end{align*}
We can now proceed as in \eqref{equ:sum theta n}. Choosing $\epsilon>0$ sufficiently small, rearranging the terms, using \eqref{equ:xi nab theta n L2}, and summing over~$m\in \{0,1,\ldots,n-1\}$, we obtain
\begin{align*}
	\norm{\nabla \bff{\theta}^n}{\bb{L}^2}^2
	+
	k \sum_{m=0}^{n-1} \norm{\Delta_h \bff{\theta}^{n+\frac{1}{2}}}{\bb{L}^2}^2
	+
	k \sum_{m=0}^{n-1} \norm{\nabla \bff{\xi}^{n+\frac{1}{2}}}{\bb{L}^2}^2
	\leq
	C \big(h^{2r}+k^4\big)
	+
	Ck \sum_{m=0}^{n-1} \norm{\nabla \bff{\theta}^m}{\bb{L}^2}^2.
\end{align*}
The required estimate then follows from the discrete Gronwall inequality.
\end{proof}

\begin{proposition}\label{pro:nab xi n half crank}
Suppose that the triangulation $\mathcal{T}_h$ is globally quasi-uniform. Then for $h, k>0$ and~$n\in \{1,2,\ldots,\lfloor T/k \rfloor\}$,
\begin{align*}
	\norm{\nabla \bff{\xi}^{n+\frac{1}{2}}}{\bb{L}^2}^2
	\leq
	C \big(h^{2r}+k^4 \big),
\end{align*}
where $C$ depends on the coefficients of the equation, $|\mathscr{D}|$, $T$, and $K_0$, but is independent of $n$, $h$, and $k$.
\end{proposition}

\begin{proof}
Firstly, we follow the same argument leading to \eqref{equ:dt xi n half}, but we now take $\bff{\phi}=\Delta_h \left(\bff{\xi}^{n+\frac{1}{2}}+ \bff{\xi}^{n-\frac{1}{2}}\right)$. Secondly, we follow the argument in \eqref{equ:xi nab xi disc}, but with $\bff{\chi}=\frac{1}{2} \Delta_h^2 \left(\bff{\xi}^{n+\frac{1}{2}}+ \bff{\xi}^{n-\frac{1}{2}}\right)$. Adding the resulting equations and applying \eqref{equ:orth proj} and \eqref{equ:disc laplacian} as necessary, we obtain
\begin{align}\label{equ:dt nab xi n half}
	\nonumber
	&\frac{\norm{\nabla \bff{\xi}^{n+\frac{1}{2}}}{\bb{L}^2}^2 - \norm{\nabla \bff{\xi}^{n-\frac{1}{2}}}{\bb{L}^2}^2}{k}
	+\frac{\lambda_r}{2} \norm{\Delta_h\left(\bff{\xi}^{n+\frac{1}{2}}+\bff{\xi}^{n-\frac{1}{2}}\right)}{\bb{L}^2}^2
	+
	\frac{\lambda_e}{2} \norm{\nabla \Delta_h\left(\bff{\xi}^{n+\frac{1}{2}}+\bff{\xi}^{n-\frac{1}{2}}\right)}{\bb{L}^2}^2
	\\
	\nonumber
	&=
	\kappa \inpro{\delta \bff{\eta}^n}{\Delta_h \left(\bff{\xi}^{n+\frac{1}{2}}+ \bff{\xi}^{n-\frac{1}{2}}\right)}
	+
	\inpro{\delta \bff{H}(t_n) - \partial_t \bff{H}^n}{\Delta_h \left(\bff{\xi}^{n+\frac{1}{2}}+ \bff{\xi}^{n-\frac{1}{2}}\right)}
	\\
	\nonumber
	&\quad
	-
	\kappa\mu \inpro{\delta \bff{u}_h^n-\partial_t \bff{u}^n}{\Delta_h \left(\bff{\xi}^{n+\frac{1}{2}}+ \bff{\xi}^{n-\frac{1}{2}}\right)}
	+
	\inpro{\nabla \Pi_h \left(\delta \bff{u}(t_n) - \nabla\partial_t \bff{u}^n\right)}{\nabla \Delta_h \left(\bff{\xi}^{n+\frac{1}{2}}+ \bff{\xi}^{n-\frac{1}{2}}\right)}
	\\
	\nonumber
	&\quad
	+
	\kappa
	\inpro{\frac{\psi\big(\bff{u}_h^n,\bff{u}_h^{n+1}\big) - \psi\big(\bff{u}_h^{n-1}, \bff{u}_h^n\big)}{k} - \partial_t \left(\abs{\bff{u}^n}^2 \bff{u}^n\right)}{\Delta_h \left(\bff{\xi}^{n+\frac{1}{2}}+ \bff{\xi}^{n-\frac{1}{2}}\right)}
	\\
	\nonumber
	&\quad
	+
	\beta
	\inpro{\bff{e}\big(\bff{e}\cdot(\delta\bff{u}_h^n-\partial_t \bff{u}^n)\big)}{\Delta_h \left(\bff{\xi}^{n+\frac{1}{2}}+ \bff{\xi}^{n-\frac{1}{2}}\right)}
	\\
	\nonumber
	&\quad
	+
	\frac{1}{2} \inpro{\nabla \Pi_h \left(\delta \bff{\rho}^{n+\frac{1}{2}}+ \delta \bff{u}^{n+\frac{1}{2}} - \partial_t \bff{u}\big(t_{n+\frac{1}{2}}\big)\right)}{\nabla \Delta_h\left(\bff{\xi}^{n+\frac{1}{2}}+\bff{\xi}^{n-\frac{1}{2}}\right)}
	\\
	\nonumber
	&\quad
	+
	\frac{1}{2} \inpro{\nabla \Pi_h \left(\delta \bff{\rho}^{n-\frac{1}{2}} + \delta \bff{u}^{n-\frac{1}{2}} - \partial_t \bff{u}\big(t_{n-\frac{1}{2}}\big)\right)}{\nabla \Delta_h\left(\bff{\xi}^{n+\frac{1}{2}}+\bff{\xi}^{n-\frac{1}{2}}\right)}
	\\
	\nonumber
	&\quad
	-
	\frac{\lambda_r}{2} \inpro{\nabla \Pi_h \left(\bff{\eta}^{n+\frac{1}{2}}+\bff{\eta}^{n-\frac{1}{2}}\right)}{\nabla \Delta_h\left(\bff{\xi}^{n+\frac{1}{2}}+\bff{\xi}^{n-\frac{1}{2}}\right)}
	\\
	\nonumber
	&\quad
	+
	\frac{\gamma}{2} \inpro{\nabla \Pi_h\left(\widehat{\bff{u}}_h^{n-\frac{1}{2}} \times \bff{H}_h^{n+\frac{1}{2}} - \bff{u}\big(t_{n+\frac{1}{2}}\big)\times \bff{H}\big(t_{n+\frac{1}{2}}\big)\right)}{\nabla \Delta_h\left(\bff{\xi}^{n+\frac{1}{2}}+\bff{\xi}^{n-\frac{1}{2}}\right)}
	\\
	\nonumber
	&\quad
	+
	\frac{\gamma}{2} \inpro{\nabla \Pi_h\left(\widehat{\bff{u}}_h^{n-\frac{3}{2}} \times \bff{H}_h^{n-\frac{1}{2}} - \bff{u}\big(t_{n-\frac{1}{2}}\big)\times \bff{H}\big(t_{n-\frac{1}{2}}\big)\right)}{\nabla \Delta_h\left(\bff{\xi}^{n+\frac{1}{2}}+\bff{\xi}^{n-\frac{1}{2}}\right)}
	\\
	&=:
	M_1+M_2+\cdots+M_{11}.
\end{align}
It remains to estimate each of the eleven terms involving inner products above. For $i=1,2,\ldots, 6$, the terms $M_i$ can be estimated in a similar way as the corresponding terms $I_i$ in \eqref{equ:I1}--\eqref{equ:I5}, giving
\begin{align*}
	\sum_{i=1}^6 \abs{M_i} \lesssim 
	h^{2(r+1)}+ k^4+ \norm{\delta \bff{\theta}^n}{\bb{L}^2}^2
	+
	\epsilon \norm{\Delta_h\left(\bff{\xi}^{n+\frac{1}{2}}+\bff{\xi}^{n-\frac{1}{2}}\right)}{\bb{L}^2}^2
	+
	\epsilon \norm{\nabla \Delta_h\left(\bff{\xi}^{n+\frac{1}{2}}+\bff{\xi}^{n-\frac{1}{2}}\right)}{\bb{L}^2}^2.
\end{align*}
Moreover, by Young's inequality, stability of the projection operator, and \eqref{equ:d rho dt u}, we have
\begin{align*}
	\abs{M_7}+\abs{M_8}+\abs{M_9} 
	\lesssim
	h^{2r}+k^4 + \epsilon \norm{\nabla \Delta_h\left(\bff{\xi}^{n+\frac{1}{2}}+\bff{\xi}^{n-\frac{1}{2}}\right)}{\bb{L}^2}^2.
\end{align*}
Finally, by~\eqref{equ:cross nab est crank} (and noting Proposition~\ref{pro:nab theta n xi n}) we obtain
\begin{align*}
	\abs{M_{10}} + \abs{M_{11}}
	&\lesssim
	h^{2r}
	+
	k^4
	+
	\norm{\bff{\theta}^n}{\bb{H}^1}^2
	+
	\norm{\bff{\theta}^{n-1}}{\bb{H}^1}^2
	+
	\norm{\bff{\theta}^{n-2}}{\bb{H}^1}^2
	+
	\norm{\bff{\xi}^{n+\frac12}}{\bb{H}^1}^2
	+
	\norm{\bff{\xi}^{n-\frac12}}{\bb{H}^1}^2
	\\
	&\quad
	+
	\epsilon \norm{\nabla \Delta_h\left(\bff{\xi}^{n+\frac{1}{2}}+\bff{\xi}^{n-\frac{1}{2}}\right)}{\bb{L}^2}^2
	\\
	&\lesssim
	h^{2r}
	+
	k^4
	+
	\norm{\bff{\xi}^{n+\frac12}}{\bb{H}^1}^2
	+
	\norm{\bff{\xi}^{n-\frac12}}{\bb{H}^1}^2
	+
	\epsilon \norm{\nabla \Delta_h\left(\bff{\xi}^{n+\frac{1}{2}}+\bff{\xi}^{n-\frac{1}{2}}\right)}{\bb{L}^2}^2.
\end{align*}
Substituting these into~\eqref{equ:dt nab xi n half}, choosing $\epsilon>0$ sufficiently small, and summing over~$m\in \{0,1,\ldots,n-1\}$ (noting Proposition~\ref{pro:nab theta n xi n} and inequality~\eqref{equ:xi nab theta n L2}), we obtain the required estimate.
\end{proof}

The following theorem on the rate of convergence of the numerical scheme~\eqref{equ:cranknic} is now an immediate consequence of the previous propositions.

\begin{theorem}\label{the:crank rate}
Let $(\bff{u},\bff{H})$ be the solution of \eqref{equ:weakform} which satisfies \eqref{equ:ass 2 u}, and let~$(\bff{u}_h^n,\bff{H}_h^n)\in \bb{V}_h\times \bb{V}_h$ be the solution of \eqref{equ:cranknic}. Then for $h, k>0$ and~$n\in \{1,2,\ldots,\lfloor T/k \rfloor\}$,
\begin{align*}
	\norm{\bff{u}_h^n - \bff{u}(t_n)}{\bb{L}^2}^2 
	+
	k\sum_{m=0}^{n-1} \norm{\bff{H}_h^{m+\frac12}- \bff{H}\big(t_{m+\frac12}\big)}{\bb{L}^2}^2 
	&\leq
	C \big(h^{2(r+1)} + k^4\big),
	\\
	k \sum_{m=0}^{n-1} \norm{\nabla \bff{u}_h^m- \nabla \bff{u}(t_m)}{\bb{L}^2}^2
	&\leq
	C \big(h^{2r} + k^4\big).
\end{align*}
Moreover, if the triangulation is globally quasi-uniform, then for $h, k>0$ and~$n\in \{1,2,\ldots,\lfloor T/k \rfloor\}$,
\begin{align*}
	\norm{\bff{H}_h^{n+\frac{1}{2}} - \bff{H}\big(t_{n+\frac{1}{2}}\big)}{\bb{L}^2}^2
	&\leq
	C \big(h^{2(r+1)}+k^4 \big),
	\\
	\norm{\nabla \bff{u}_h^n-\nabla \bff{u}(t_n)}{\bb{L}^2}^2
	+
	\norm{\nabla \bff{H}_h^{n+\frac{1}{2}} - \nabla \bff{H}\big(t_{n+\frac{1}{2}}\big)}{\bb{L}^2}^2 
	&\leq
	C \big(h^{2r}+k^4 \big),
	\\
	\norm{\bff{u}_h^{n+\frac12}-\bff{u}\big(t_{n+\frac12}\big)}{\bb{L}^\infty}^2
	+
	k\sum_{m=0}^{n-1} \norm{\bff{H}_h^{m+\frac12}- \bff{H}\big(t_{m+\frac12}\big)}{\bb{L}^\infty}^2
	&\leq
	C \big(h^{2(r+1)} \abs{\ln h} +k^2 \big).
\end{align*}
The constant $C$ depends on the coefficients of the equation, $|\mathscr{D}|$, $T$, and $K_0$ (as defined in \eqref{equ:ass 2 u}), but is independent of $n$, $h$, and $k$.
\end{theorem}

\begin{proof}
Note that we have \eqref{equ:Un utn} and \eqref{equ:Hn Htn}. The results then follow immediately by the estimates in Proposition~\ref{pro:theta xi n half}, Proposition~\ref{pro:nab theta n xi n}, Proposition~\ref{pro:nab xi n half crank}, inequalities \eqref{equ:Ritz ineq}, \eqref{equ:Ritz ineq L infty}, and the triangle inequality.
\end{proof}

\begin{remark}
More careful estimates would allow us to obtain the stability of $\bff{H}_h^n$ in $\ell^\infty(\bb{H}^1)$ (as in Proposition~\ref{pro:dt un stab euler}) and to remove the global quasi-uniformity assumption for $d=1$ and $2$ (analogous to Proposition~\ref{pro:theta xi n half euler} and Theorem~\ref{the:euler rate}). Further details are omitted for brevity.
\end{remark}

\section{Approximation of the Regularised LLBloch Equation}\label{sec:bloch}

In this section, we assume $\mu<0$. The aim here is to show that as $\lambda_e \to 0^+$, the solution of the LLBar equation converges to that of the LLBloch equation at a certain rate, thus we can treat the LLBar equation as a \emph{regularised} LLBloch equation. We also outline some modifications needed in the schemes to ensure energy dissipativity at the discrete level for this regularised LLBloch equation (with~$\lambda_e \Delta \bff{H}$ as the regularisation term). Once these are done, we can then conclude that the numerical schemes proposed in Sections~\ref{sec:semidiscrete}, \ref{sec:Euler}, and~\ref{sec:Crank} are also suitable to approximate the solution of the LLBloch equation (by choosing small $\lambda_e$ and sufficiently small $h$ and $k$).

First, we recall some known results about the LLBloch equation above $T_\mathrm{c}$ (equation~\eqref{equ:llbar} with $\lambda_e=0$ and $\mu<0$).
Given $T>0$ and~$\bff{u}_0\in \bb{H}^1 \cap \bb{L}^\infty$, a global weak solution $\bff{u}\in L^\infty(\bb{L}^\infty)\cap L^\infty(\bb{H}^1) \cap L^2(\bb{H}^2)$ exists~\cite{Le16, LeSoeTra24} for $d=1,2,3$. This weak solution satisfies
\begin{equation*}
	\norm{\bff{u}}{L^\infty(\bb{L}^\infty)} + 
	\norm{\bff{u}}{L^\infty(\bb{H}^1)} +
	\norm{\bff{u}}{L^2(\bb{H}^2)} +
	\norm{\partial_t \bff{u}}{L^2(\bb{L}^2)} + 
	\norm{\bff{H}}{L^2(\bb{L}^2)} \leq K_1.
\end{equation*}
Furthermore, if $\bff{u}_0\in \bb{H}^2$, then a strong solution $\bff{u}\in L^\infty(\bb{H}^2)\cap L^2(\bb{H}^3)$ exists, possibly only locally in time for $d=3$ (cf.~\cite{LeSoeTra24}). This strong solution satisfies
\begin{align}\label{equ:u H bloch H2}
	\norm{\bff{u}}{L^\infty(\bb{H}^2)} + \norm{\bff{u}}{L^2(\bb{H}^3)} +
	\norm{\partial_t \bff{u}}{L^\infty(\bb{L}^2)} + \norm{\bff{H}}{L^2(\bb{H}^1)} \leq K_2,
\end{align}
where $K_r$, for $r=1,2$, are positive constants depending on the coefficients of the equation, $T$, and $\norm{\bff{u}_0}{\bb{H}^r}$.

For the LLBar equation with $\lambda_e=\varepsilon$ and initial data $\bff{u}_0^\varepsilon\in \bb{H}^1$, a weak solution~$\bff{u}^\varepsilon\in L^\infty(\bb{H}^1) \cap L^2(\bb{H}^3)$ with a corresponding magnetic field $\bff{H}^\varepsilon\in L^2(\bb{H}^1)$ exist~\cite{SoeTra23}. As implied by Proposition~\ref{pro:semidisc est1}, this weak solution enjoys the estimate
\begin{align}\label{equ:ue est}
	\norm{\bff{u}^\varepsilon}{L^\infty(\bb{H}^1)} 
	+
	\norm{\Delta \bff{u}^\varepsilon}{L^2(\bb{L}^2)}
	+
	\norm{\bff{H}^\varepsilon}{L^2(\bb{L}^2)}
	+
	\sqrt{\varepsilon} \norm{\nabla \bff{H}^\varepsilon}{L^2(\bb{L}^2)}
	\leq
	K_3.
\end{align}
where $K_3$ depends on $T$ and $\norm{\bff{u}_0^\varepsilon}{\bb{H}^1}$, but is independent of $\varepsilon$. Furthermore, if $\bff{u}_0^\varepsilon\in \bb{H}^2$, then we have a strong solution~$\bff{u}^\varepsilon\in L^\infty(\bb{H}^2)\cap L^2(\bb{H}^4)$.

We now show the convergence of the weak solution of the LLBar equation to that of the LLBloch equation as $\varepsilon \to 0^+$.

\begin{theorem}\label{the:bloch to bar}
Let $\bff{u}^\varepsilon$ be a strong solution of the LLBar equation with $\lambda_e=\varepsilon$ and initial data $\bff{u}_0^\varepsilon\in \bb{H}^2$. Let $\bff{u}$ be a strong solution of the LLBloch equation ($\lambda_e=0$) with initial data~$\bff{u}_0\in \bb{H}^2$. Then
\begin{align}\label{equ:est ue u H1}
	\norm{\bff{u}^\varepsilon-\bff{u}}{L^\infty(\bb{H}^1)}
	+
	\norm{\bff{u}^\varepsilon-\bff{u}}{L^2(\bb{H}^2)}
	+
	\norm{\bff{H}^\varepsilon-\bff{H}}{L^2(\bb{L}^2)}
	&\leq
	C \sqrt{\varepsilon}.
\end{align}
\end{theorem}

\begin{proof}
We write $\bff{v}:= \bff{u}^\varepsilon-\bff{u}$ and $\bff{B}:= \bff{H}^\varepsilon-\bff{H}$. Let $\bff{v}_0:= \bff{u}_0^\varepsilon-\bff{u}_0$, so that $\bff{v}(0)=\bff{v}_0$. Since $\bff{u}^\varepsilon$ and $\bff{u}$ are strong solutions of the corresponding equations, we have
\begin{equation}\label{equ:diff v}
	\begin{aligned}
		\partial_t \bff{v}
		&=
		\lambda_r \bff{B}
		- \varepsilon \Delta \bff{H}^\varepsilon 
		- \gamma \bff{u}^\varepsilon \times \bff{B}
		- \gamma \bff{v}\times \bff{H},
		\\
		\bff{B}
		&=
		\Delta \bff{v}
		+ \kappa\mu \bff{v}
		-
		\kappa |\bff{u}^\varepsilon|^2 \bff{v}
		-
		\kappa \big( (\bff{u}^\varepsilon+\bff{u})\cdot \bff{v} \big) \bff{u}
		-
		\beta \bff{e}(\bff{e}\cdot \bff{v}).
	\end{aligned}
\end{equation}
Successively taking the inner product of the first equation in~\eqref{equ:diff v} with $\bff{B}$, then taking the inner product of the second equation in~\eqref{equ:diff v} with $-\partial_t \bff{v}$, we obtain
\begin{align*}
	\inpro{\partial_t \bff{v}}{\bff{B}}
	&=
	\lambda_r \norm{\bff{B}}{\bb{L}^2}^2
	+
	\varepsilon \norm{\nabla \bff{H}^\varepsilon}{\bb{L}^2}^2
	-
	\varepsilon \inpro{\nabla \bff{H}^\varepsilon}{\nabla \bff{H}}
	-
	\gamma \inpro{\bff{v}\times \bff{H}}{\bff{B}}
	\\
	-\inpro{\bff{B}}{\partial_t \bff{v}}
	&=
	\frac12 \ddt \norm{\nabla \bff{v}}{\bb{L}^2}^2
	-
	\frac{\kappa \mu}{2} \ddt \norm{\bff{v}}{\bb{L}^2}^2
	+
	\kappa \inpro{|\bff{u}^\varepsilon|^2 \bff{v}}{\partial_t \bff{v}}
	+
	\kappa \inpro{\big( (\bff{u}^\varepsilon+\bff{u})\cdot \bff{v} \big) \bff{u}}{\partial_t \bff{v}}
	+
	\frac{\beta}{2} \norm{\bff{e}\cdot \bff{v}}{L^2}^2,
\end{align*}
where on the first equation we integrated by parts and used $\inpro{\nabla \bff{H}^\varepsilon}{\nabla \bff{B}}= \norm{\nabla \bff{H}^\varepsilon}{\bb{L}^2}^2 - \inpro{\nabla \bff{H}^\varepsilon}{\nabla \bff{H}}$.
Adding the above equations (noting $\mu<0$) gives
\begin{align}\label{equ:add dtv B}
	&\frac12 \ddt \norm{\nabla \bff{v}}{\bb{L}^2}^2
	-
	\frac{\kappa\mu}{2} \ddt \norm{\bff{v}}{\bb{L}^2}^2
	+
	\frac{\beta}{2} \norm{\bff{e}\cdot \bff{v}}{L^2}^2
	+
	\lambda_r \norm{\bff{B}}{\bb{L}^2}^2
	+
	\varepsilon \norm{\nabla \bff{H}^\varepsilon}{\bb{L}^2}^2
	\nonumber \\
	&=
	\varepsilon \inpro{\nabla \bff{H}^\varepsilon}{\nabla \bff{H}}
	+
	\gamma \inpro{\bff{v}\times \bff{H}}{\bff{B}}
	-
	\kappa \inpro{|\bff{u}^\varepsilon|^2 \bff{v}}{\partial_t \bff{v}}
	-
	\kappa \inpro{\big( (\bff{u}^\varepsilon+\bff{u})\cdot \bff{v} \big) \bff{u}}{\partial_t \bff{v}}
	\nonumber \\
	&=: J_1+J_2+J_3+J_4.
\end{align}
We will estimate each term on the last line. In the following, the constant $C$ is independent of $\varepsilon$. Firstly, by Young's inequality,
\begin{align}\label{equ:J1}
	\abs{J_1}
	&\leq
	\frac{\varepsilon}{4} \norm{\nabla \bff{H}^\varepsilon}{\bb{L}^2}^2
	+
	4\varepsilon \norm{\nabla \bff{H}}{\bb{L}^2}^2.
\end{align}
Next, by Young's inequality and Sobolev embedding,
\begin{align}\label{equ:J2}
	\abs{J_2}
	&\leq
	\frac{\lambda_r}{8} \norm{\bff{B}}{\bb{L}^2}^2
	+
	C \norm{\bff{H}}{\bb{L}^4}^2 \norm{\bff{v}}{\bb{L}^4}^2
	\leq
	\frac{\lambda_r}{8} \norm{\bff{B}}{\bb{L}^2}^2
	+
	C \norm{\bff{H}}{\bb{H}^1}^2 \norm{\bff{v}}{\bb{H}^1}^2.
\end{align}
We now aim to estimate $J_3$. To this end, substituting $\partial_t \bff{v}$ by the first equation in~\eqref{equ:diff v} and integrating by parts as necessary, we have
\begin{align}\label{equ:L123}
	 J_3
	 =
	 -\kappa\lambda_r \inpro{|\bff{u}^\varepsilon|^2 \bff{v}}{\bff{B}}
	 +
	 \kappa \varepsilon \inpro{\nabla\big(|\bff{u}^\varepsilon|^2 \bff{v}\big)}{\nabla \bff{H}^\varepsilon}
	 +
	 \kappa \gamma \inpro{|\bff{u}^\varepsilon|^2 \bff{v}}{\bff{u}^\varepsilon\times \bff{B}}
	 =: J_{3a}+J_{3b}+J_{3c}.
\end{align}
For the terms $J_{3a}$ and $J_{3b}$, by Young's inequality and the Sobolev embedding (noting~\eqref{equ:ue est}) we have
\begin{align*}
	\abs{J_{3a}} 
	\leq
	\kappa \lambda_r \norm{\bff{u}^\varepsilon}{\bb{L}^6}^2 \norm{\bff{v}}{\bb{L}^6} \norm{\bff{B}}{\bb{L}^2}
	\leq
	\frac{\lambda_r}{8} \norm{\bff{B}}{\bb{L}^2}
	+
	C \norm{\bff{v}}{\bb{H}^1}^2.
\end{align*}
For the terms $J_{3b}$ and $J_{3c}$, similarly we obtain
\begin{align*}
	\abs{J_{3b}}
	&\leq
	\frac{\varepsilon}{4} \norm{\nabla \bff{H}^\varepsilon}{\bb{L}^2}^2
	+
	4\varepsilon \norm{\bff{u}^\varepsilon}{\bb{L}^6}^2 \norm{\nabla \bff{u}^\varepsilon}{\bb{L}^6}^2 \norm{\bff{v}}{\bb{L}^6}^2
	+
	4 \varepsilon \norm{\bff{u}^\varepsilon}{\bb{L}^6}^4 \norm{\nabla \bff{v}}{\bb{L}^6}^2
	\\
	&\leq
	\frac{\varepsilon}{4} \norm{\nabla \bff{H}^\varepsilon}{\bb{L}^2}^2
	+
	C\varepsilon \norm{\Delta \bff{u}^\varepsilon}{\bb{L}^2}^2 \norm{\bff{v}}{\bb{H}^1}^2
	+
	C\varepsilon \norm{\Delta \bff{v}}{\bb{L}^2}^2,
	\\
	\abs{J_{3c}}
	&\leq
	\kappa \gamma \norm{\bff{u}^\varepsilon}{\bb{L}^6}^2 \norm{\bff{v}}{\bb{L}^6}
	\norm{\bff{u}^\varepsilon}{\bb{L}^\infty} \norm{\bff{B}}{\bb{L}^2}
	\leq
	\frac{\lambda_r}{8} \norm{\bff{B}}{\bb{L}^2}
	+
	\norm{\bff{u}^\varepsilon}{\bb{H}^2}^2 \norm{\bff{v}}{\bb{H}^1}^2.
\end{align*}
Substituting these estimates into~\eqref{equ:L123}, we obtain
\begin{align}\label{equ:J3}
	\abs{J_3}
	\leq
	\frac{\lambda_r}{4} \norm{\bff{B}}{\bb{L}^2}^2
	+
	\frac{\varepsilon}{4} \norm{\nabla \bff{H}^\varepsilon}{\bb{L}^2}^2
	+
	C\varepsilon \norm{\Delta \bff{v}}{\bb{L}^2}^2
	+
	C\left(1+\norm{\bff{u}^\varepsilon}{\bb{H}^2}^2 \right)\norm{\bff{v}}{\bb{H}^1}^2,
\end{align}
where $C$ is independent of $\varepsilon$.
The term $J_4$ can be estimated in a similar manner as $J_3$, resulting in
\begin{align}\label{equ:J4}
	\abs{J_4}
	\leq
	\frac{\lambda_r}{4} \norm{\bff{B}}{\bb{L}^2}^2
	+
	\frac{\varepsilon}{4} \norm{\nabla \bff{H}^\varepsilon}{\bb{L}^2}^2
	+
	C\varepsilon \norm{\Delta \bff{v}}{\bb{L}^2}^2
	+
	C\left(1+\norm{\bff{u}^\varepsilon}{\bb{H}^2}^2 + \norm{\bff{u}}{\bb{H}^2}^2
	+ \norm{\bff{H}}{\bb{L}^2}^2 \right)\norm{\bff{v}}{\bb{H}^1}^2.
\end{align}
Altogether, substituting the estimates~\eqref{equ:J1}, \eqref{equ:J2}, \eqref{equ:J3}, and~\eqref{equ:J4} into~\eqref{equ:add dtv B}, taking care to absorb relevant terms to the left-hand side, we obtain
\begin{align*}
	\ddt \norm{\bff{v}}{\bb{H}^1}^2
	+
	\norm{\bff{B}}{\bb{L}^2}^2
	\leq
	C\varepsilon \norm{\nabla \bff{H}}{\bb{L}^2}^2
	+
	C\varepsilon \norm{\Delta \bff{v}}{\bb{L}^2}^2
	+
	C\left(1+\norm{\bff{u}^\varepsilon}{\bb{H}^2}^2 + \norm{\bff{u}}{\bb{H}^2}^2
	+ \norm{\bff{H}}{\bb{H}^1}^2 \right)\norm{\bff{v}}{\bb{H}^1}^2.
\end{align*}
We now integrate both sides with respect to $t$. Note that by~\eqref{equ:u H bloch H2} and~\eqref{equ:ue est}, we have
\begin{align*}
	\int_0^t \left(1+ \norm{\bff{u}^\varepsilon(s)}{\bb{H}^2}^2 + \norm{\bff{u}(s)}{\bb{H}^2}^2
	+ \norm{\bff{H}(s)}{\bb{H}^1}^2 \right) \ds \leq 
	T+K_2+K_3.
\end{align*}
Invoking the Gronwall inequality (noting~\eqref{equ:u H bloch H2} and~\eqref{equ:ue est} again), we obtain~\eqref{equ:est ue u H1}. This completes the proof of the theorem.
\end{proof}

Now that we have shown the convergence of strong solution of the LLBar equation to that of the LLBloch equation, the finite element schemes proposed in Section~\ref{sec:Euler} and Section~\ref{sec:Crank} (with small $\lambda_e$) would also be applicable to approximate the LLBloch equation with a small modification to ensure energy dissipativity (since now $\mu<0$), namely:
\begin{enumerate}
	\item For the scheme~\eqref{equ:euler}, the term $\kappa \mu \inpro{\bff{u}_h^n}{\bff{\phi}}$ is replaced by $\kappa \mu \inpro{\bff{u}_h^{n+1}}{\bff{\phi}}$.
	\item The scheme~\eqref{equ:cranknic} can be kept as is.
\end{enumerate}
In this case, for $\mu<0$, one could check that Proposition~\ref{pro:wellpos euler}, Proposition~\ref{pro:ene dec H1 euler}, and Proposition~\ref{pro:ene dec H1 crank} still hold with the same argument. We remark that the schemes remain unconditionally energy-stable even as $\lambda_e\to 0^+$ as the constant $C$ in~\eqref{equ:stab L4 H1 euler} and~\eqref{equ:stab L4 H1} does not depend on $\lambda_e$. The rest of the results in Section~\ref{sec:Euler} and~\ref{sec:Crank} continue to hold almost verbatim, since the sign of $\mu$ is not used in an essential way in the proofs.

\section{Numerical Simulations}\label{sec:simulation}

Numerical simulations for the scheme~\eqref{equ:euler} are performed using the open-source package~\textsc{FEniCS}~\cite{AlnaesEtal15}. Since the exact solution of the equation is not known, we use extrapolation to verify the spatial order of convergence experimentally. To this end, let $\bff{u}_h^{(n)}$ be the finite element solution with spatial step size $h$ and time-step size $k=\lfloor T/n\rfloor$. For $s=0$ or $1$, define the extrapolated order of convergence
\begin{equation*}
	\text{rate}_s :=  \log_2 \left[\frac{\max_n \norm{\bff{e}_{2h}}{\bb{H}^s}}{\max_n \norm{\bff{e}_{h}}{\bb{H}^s}}\right],
\end{equation*}
where $\bff{e}_h(\bff{u}) := \bff{u}_{h}^{(n)}-\bff{u}_{h/2}^{(n)}$ and $\bff{e}_h(\bff{H}) := \bff{H}_{h}^{(n)}-\bff{H}_{h/2}^{(n)}$. We expect that for scheme~\eqref{equ:euler}, when $k$ is sufficiently small,~$\text{rate}_s \approx h^{r+1-s}$. In these simulations, we take the domain $\mathscr{D}= [0,1]^2\subset \bb{R}^2$ and $r=1$, i.e. piecewise linear polynomials.

\subsection{Simulation 1 (LLBar with $\mu>0$)}
We take $k=2.5\times 10^{-3}$. The coefficients in~\eqref{equ:llbar} are taken to be $\lambda_e=1.0, \lambda_r=4.0, \gamma=10.0, \kappa=2.0, \mu=1.0$, and $\beta=-0.1$. The unit vector $\bff{e}=(0,0,1)^\top$. The initial data $\bff{u}_0$ is given by
\begin{equation*}
	\bff{u}_0(x,y)= \big(\cos(2\pi x),\, \sin(2\pi y),\, 2\cos(2\pi x) \sin(2\pi y) \big).
\end{equation*}
Snapshots of the magnetic spin field $\bff{u}$ and the effective magnetic field $\bff{H}$ at selected times are shown in Figure~\ref{fig:snapshots field 2d} and Figure~\ref{fig:snapshots eff 2d}, respectively. The presence of a Bloch wall can be seen in the simulation around time $t=0.075$. Plots of $\bff{e}_h(\bff{u})$ and $\bff{e}_h(\bff{H})$ against $1/h$ are shown in Figure~\ref{fig:order u 1} and Figure~\ref{fig:order H 1}.

\begin{figure}[!htb]
	\centering
	\begin{subfigure}[b]{0.3\textwidth}
		\centering
		\includegraphics[width=\textwidth]{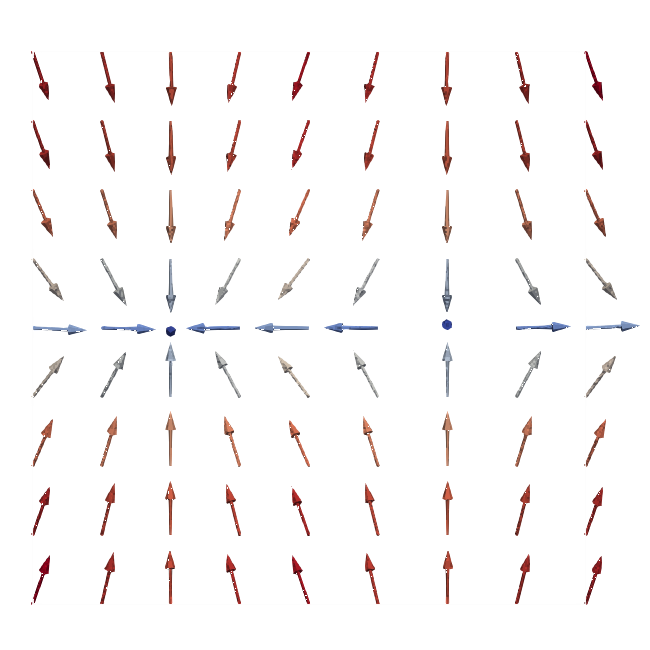}
		\caption{$t=0$}
	\end{subfigure}
	\begin{subfigure}[b]{0.3\textwidth}
		\centering
		\includegraphics[width=\textwidth]{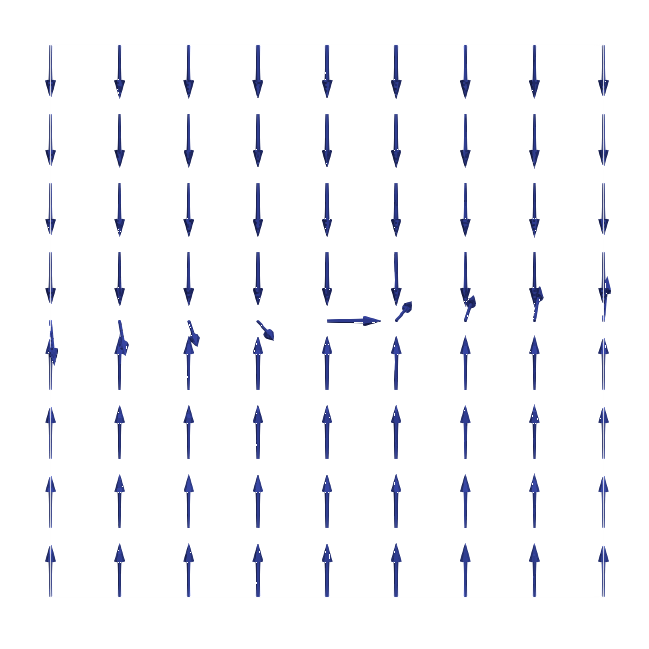}
		\caption{$t=0.05$}
	\end{subfigure}
	\begin{subfigure}[b]{0.3\textwidth}
		\centering
		\includegraphics[width=\textwidth]{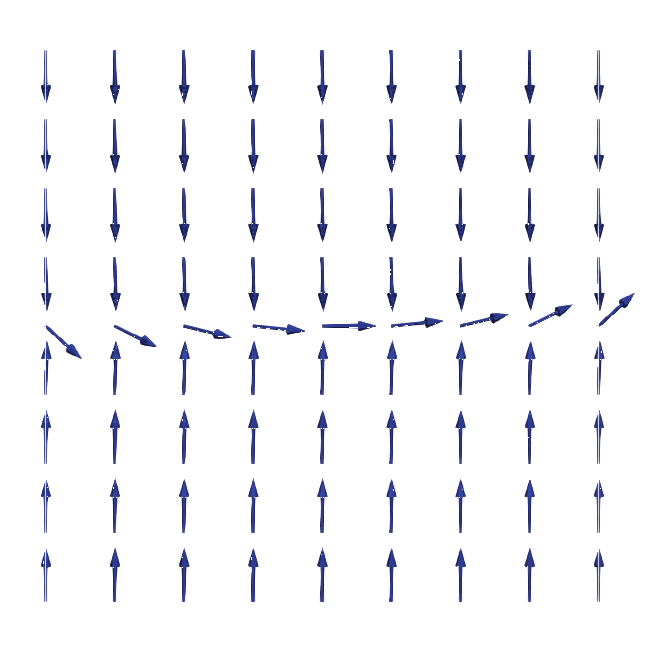}
		\caption{$t=0.075$}
	\end{subfigure}
	\begin{subfigure}[b]{0.3\textwidth}
		\centering
		\includegraphics[width=\textwidth]{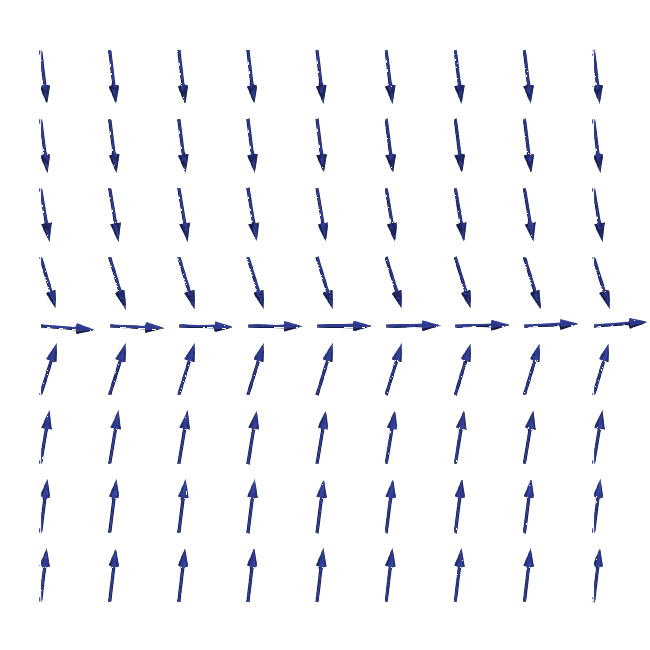}
		\caption{$t=0.1$}
	\end{subfigure}
	\begin{subfigure}[b]{0.3\textwidth}
		\centering
		\includegraphics[width=\textwidth]{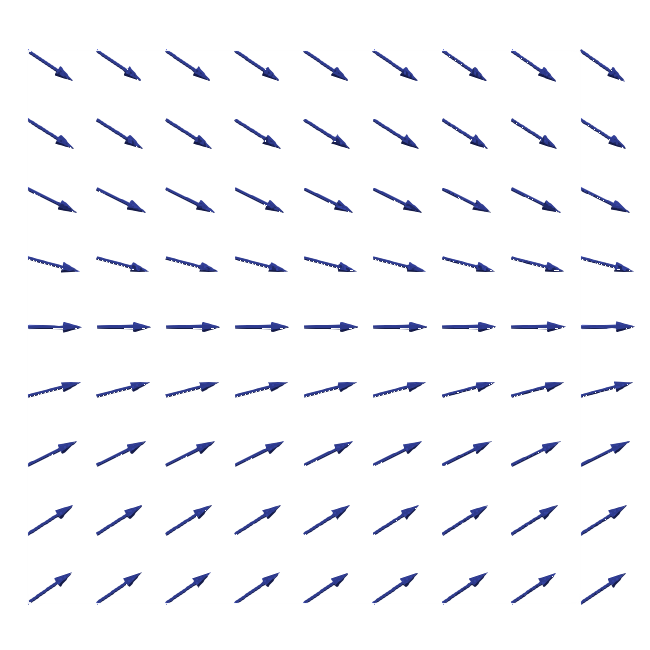}
		\caption{$t=0.125$}
	\end{subfigure}
	\begin{subfigure}[b]{0.3\textwidth}
		\centering
		\includegraphics[width=\textwidth]{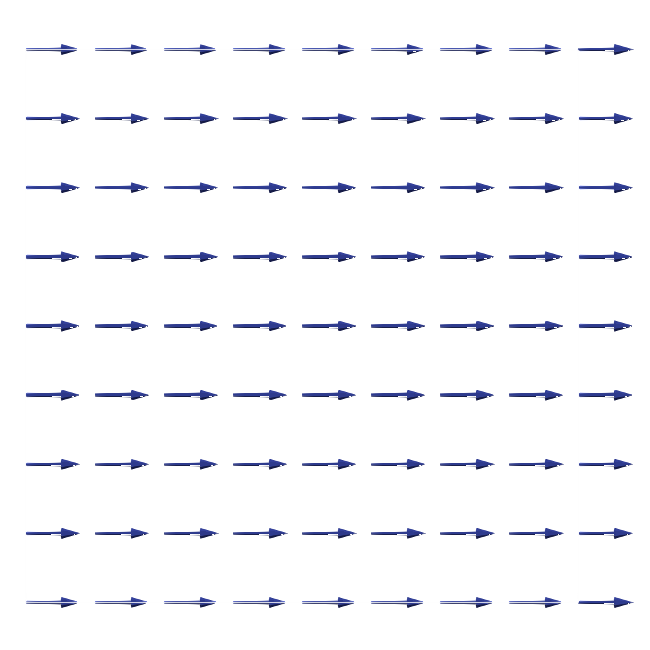}
		\caption{$t=0.5$}
	\end{subfigure}
	\caption{Snapshots of the spin field $\bff{u}$ (projected onto $\bb{R}^2$) for simulation 1.}
	\label{fig:snapshots field 2d}
\end{figure}

\begin{figure}[!htb]
	\centering
	\begin{subfigure}[b]{0.3\textwidth}
		\centering
		\includegraphics[width=\textwidth]{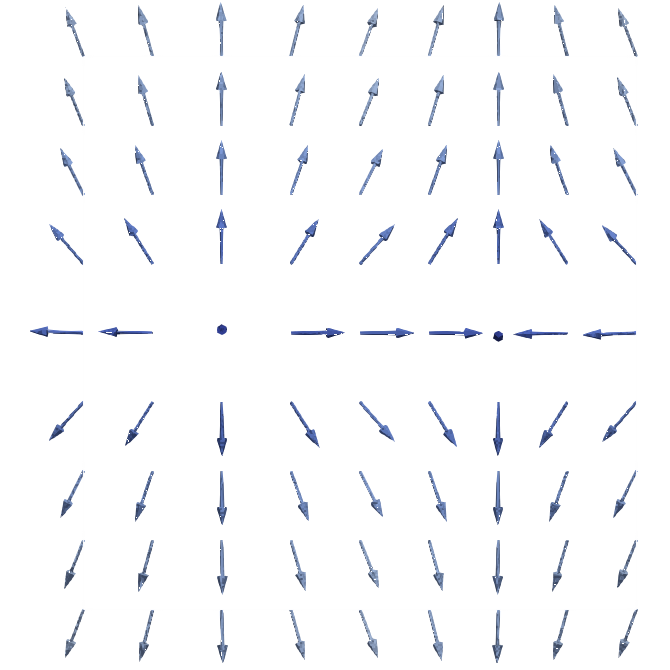}
		\caption{$t=0.0025$}
	\end{subfigure}
	\begin{subfigure}[b]{0.3\textwidth}
		\centering
		\includegraphics[width=\textwidth]{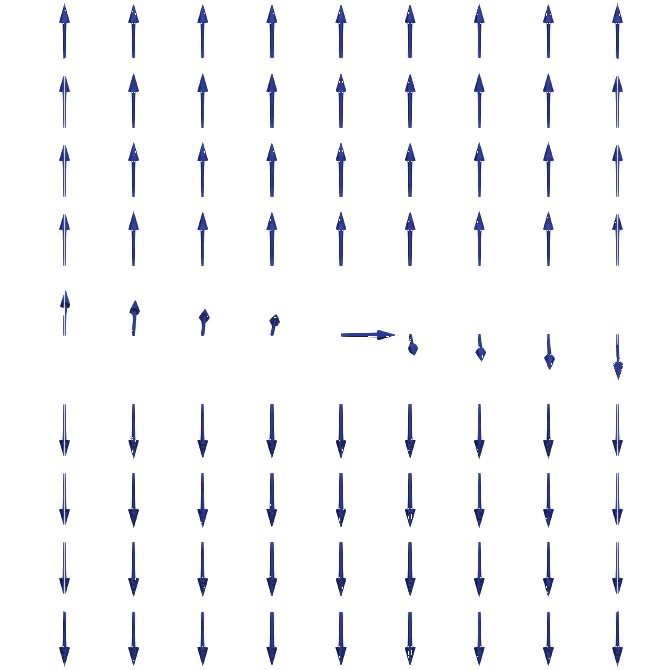}
		\caption{$t=0.05$}
	\end{subfigure}
	\begin{subfigure}[b]{0.3\textwidth}
		\centering
		\includegraphics[width=\textwidth]{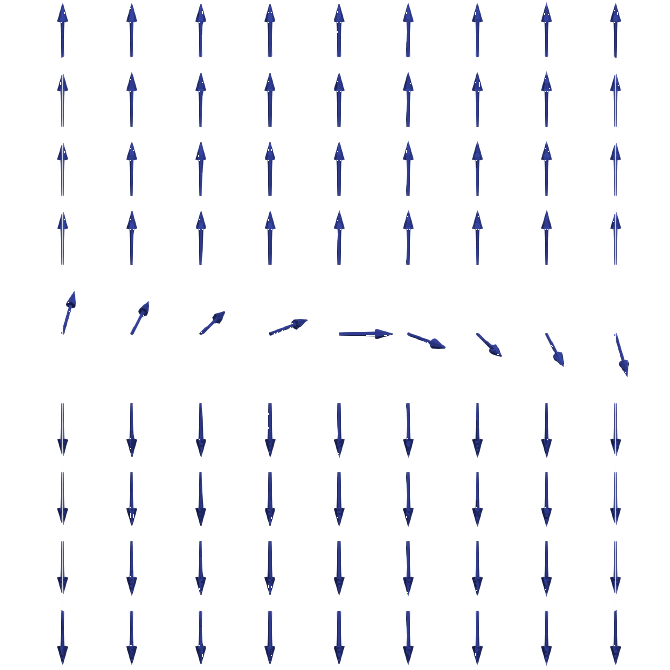}
		\caption{$t=0.075$}
	\end{subfigure}
	\begin{subfigure}[b]{0.3\textwidth}
		\centering
		\includegraphics[width=\textwidth]{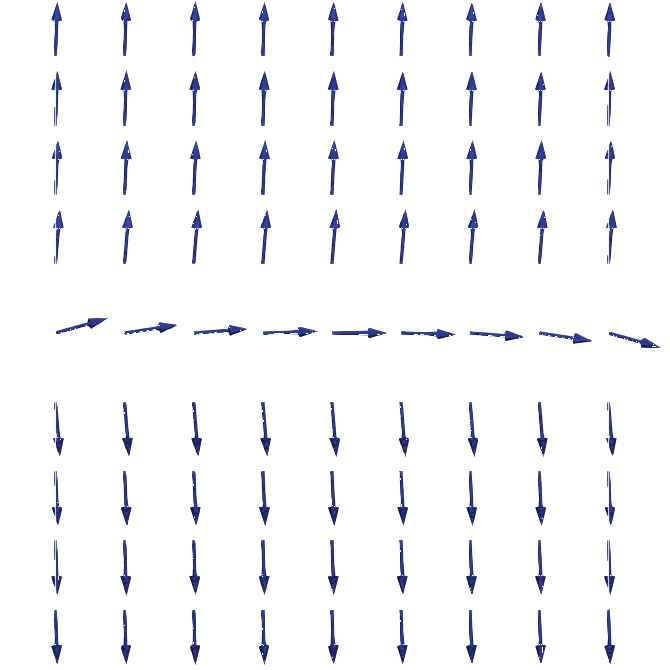}
		\caption{$t=0.1$}
	\end{subfigure}
	\begin{subfigure}[b]{0.3\textwidth}
		\centering
		\includegraphics[width=\textwidth]{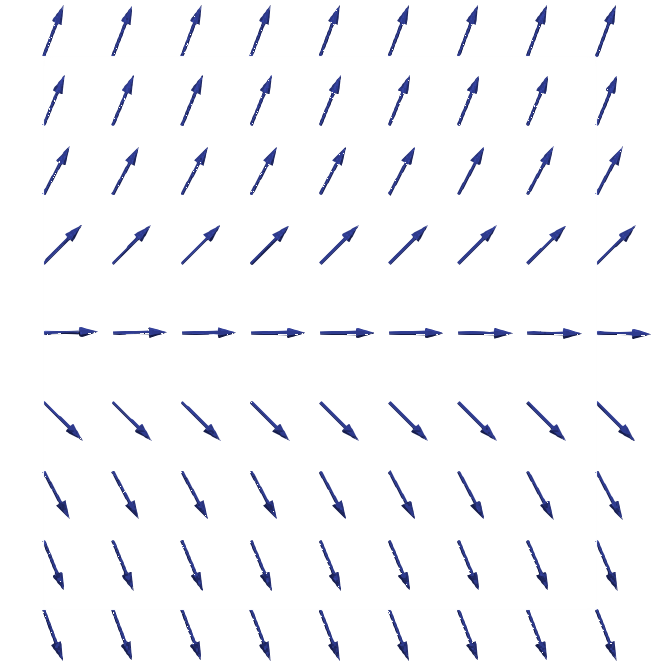}
		\caption{$t=0.125$}
	\end{subfigure}
	\begin{subfigure}[b]{0.3\textwidth}
		\centering
		\includegraphics[width=\textwidth]{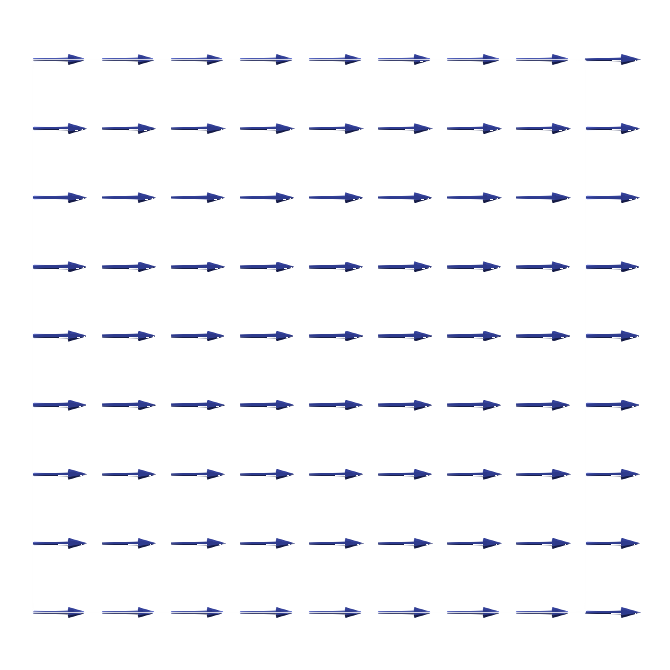}
		\caption{$t=0.5$}
	\end{subfigure}
	\caption{Snapshots of the effective field $\bff{H}$ (projected onto $\bb{R}^2$) for simulation 1.}
	\label{fig:snapshots eff 2d}
\end{figure}

\begin{figure}[!htb]
	\begin{subfigure}[b]{0.45\textwidth}
		\centering
		\begin{tikzpicture}
			\begin{axis}[
				title=Plot of $\bff{e}_h(\bff{u})$ against $1/h$,
				height=1.1\textwidth,
				width=1\textwidth,
				xlabel= $1/h$,
				ylabel= $\bff{e}_h$,
				xmode=log,
				ymode=log,
				legend pos=south west,
				legend cell align=left,
				]
				\addplot+[mark=*,red] coordinates {(4,0.205)(8,0.086)(16,0.041)(32,0.0202)(64,0.0101)};
				\addplot+[mark=*,green] coordinates {(4,0.0756)(8,0.021)(16,0.00574)(32,0.00155)(64,0.000416)};
				\addplot+[mark=*,blue] coordinates {(4,0.0548)(8,0.014)(16,0.00353)(32,0.000886)(64,0.000222)};
				\addplot+[dashed,no marks,red,domain=20:65]{1.25/x};
				\addplot+[dashed,no marks,blue,domain=20:65]{0.45/x^2};
				\legend{\small{$\max_n \norm{\bff{e}_h}{\bb{H}^1}$}, \small{$\max_n \norm{\bff{e}_h}{\bb{L}^\infty}$}, \small{$\max_n \norm{\bff{e}_h}{\bb{L}^2}$}, \small{order 1 line}, \small{order 2 line}}
			\end{axis}
		\end{tikzpicture}
		\caption{Error order of $\bff{u}$ for simulation 1.}
		\label{fig:order u 1}
	\end{subfigure}
	\begin{subfigure}[b]{0.45\textwidth}
		\centering
		\begin{tikzpicture}
			\begin{axis}[
				title=Plot of $\bff{e}_h(\bff{H})$ against $1/h$,
				height=1.1\textwidth,
				width=1\textwidth,
				xlabel= $1/h$,
				ylabel= $\bff{e}_h$,
				xmode=log,
				ymode=log,
				legend pos=south west,
				legend cell align=left,
				]
				\addplot+[mark=*,red] coordinates {(4,1.63)(8,0.75)(16,0.35)(32,0.176)(64,0.0876)};
				\addplot+[mark=*,green] coordinates {(4,0.5)(8,0.135)(16,0.037)(32,0.0101)(64,0.00273)};
				\addplot+[mark=*,blue] coordinates {(4,0.381)(8,0.102)(16,0.0267)(32,0.00677)(64,0.0017)};
				\addplot+[dashed,no marks,red,domain=20:65]{9.5/x};
				\addplot+[dashed,no marks,blue,domain=20:65]{3.2/x^2};
				\legend{\small{$\max_n \norm{\bff{e}_h}{\bb{H}^1}$}, \small{$\max_n \norm{\bff{e}_h}{\bb{L}^\infty}$}, \small{$\max_n \norm{\bff{e}_h}{\bb{L}^2}$}, \small{order 1 line}, \small{order 2 line}}
			\end{axis}
		\end{tikzpicture}
		\caption{Error order of $\bff{H}$ for simulation 1.}
		\label{fig:order H 1}
	\end{subfigure}
\end{figure}

\subsection{Simulation 2 (Regularised LLBloch with $\mu<0$ and small $\lambda_e$)}
We take $k=2.5\times 10^{-3}$. The coefficients in~\eqref{equ:llbar} are taken to be $\lambda_e=0.001, \lambda_r=4.0, \gamma=5.0, \kappa=3.0, \mu=-1.0$, and $\beta=0.2$. The unit vector $\bff{e}=(0,1,0)^\top$. The initial data $\bff{u}_0$ is given by
\begin{equation*}
	\bff{u}_0(x,y)= \big(-2y\cos(2\pi x),\, 4x^2\sin(2\pi y),\, 2\cos(2\pi x) \sin(2\pi y) \big).
\end{equation*}
Snapshots of the magnetic spin field $\bff{u}$ and the effective magnetic field $\bff{H}$ at selected times are shown in Figure~\ref{fig:snapshots field 2d 2} and Figure~\ref{fig:snapshots eff 2d 2}. Plot of $\bff{e}_h(\bff{u})$ and $\bff{e}_h(\bff{H})$ against $1/h$ are shown in Figure~\ref{fig:order u 2} and Figure~\ref{fig:order H 2}. Qualitatively, the magnetisation vectors align and tend to $\bff{0}$ as $t\to\infty$, as predicted by the theory (cf.~\cite{LeSoeTra24}).

\begin{figure}[!htb]
	\centering
	\begin{subfigure}[b]{0.3\textwidth}
		\centering
		\includegraphics[width=\textwidth]{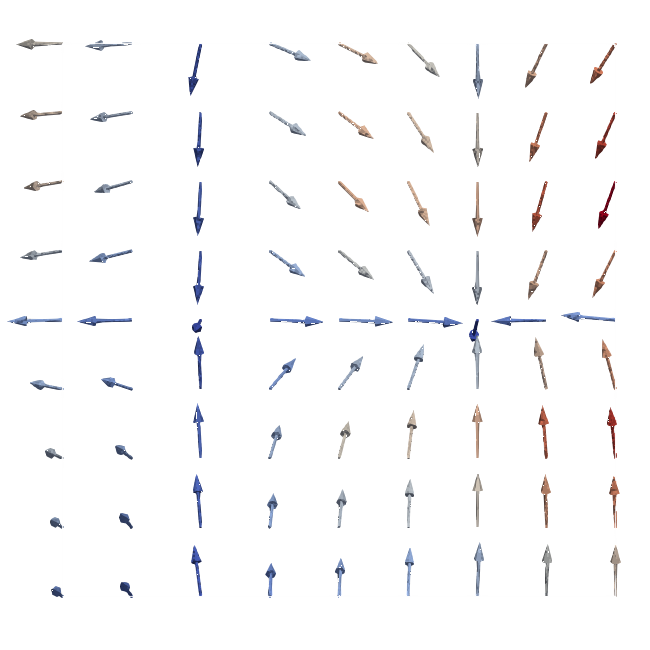}
		\caption{$t=0$}
	\end{subfigure}
	\begin{subfigure}[b]{0.3\textwidth}
		\centering
		\includegraphics[width=\textwidth]{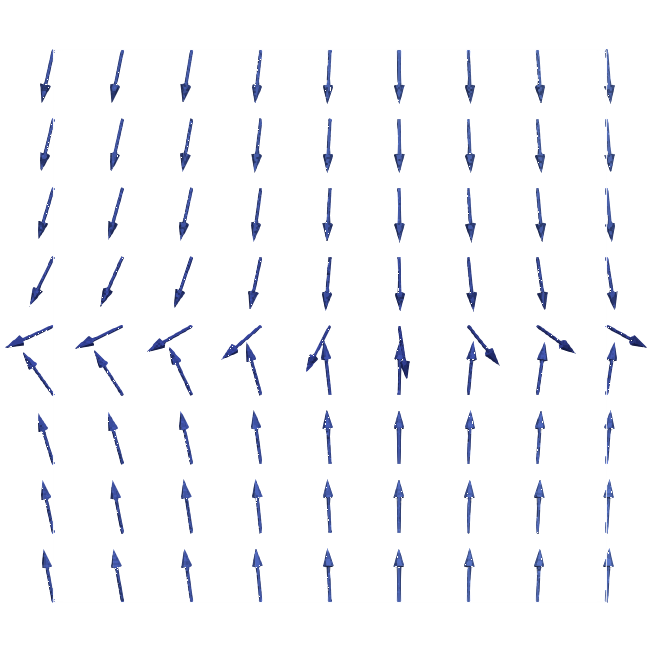}
		\caption{$t=0.025$}
	\end{subfigure}
	\begin{subfigure}[b]{0.3\textwidth}
		\centering
		\includegraphics[width=\textwidth]{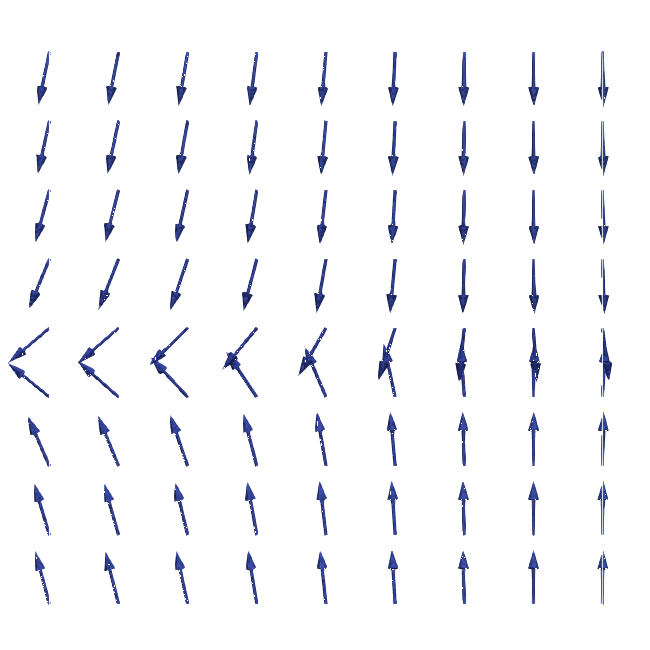}
		\caption{$t=0.05$}
	\end{subfigure}
	\begin{subfigure}[b]{0.3\textwidth}
		\centering
		\includegraphics[width=\textwidth]{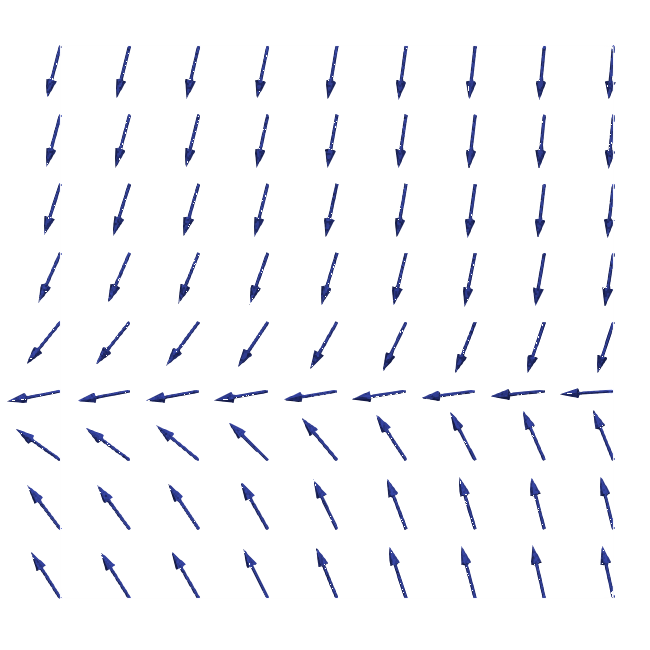}
		\caption{$t=0.075$}
	\end{subfigure}
	\begin{subfigure}[b]{0.3\textwidth}
		\centering
		\includegraphics[width=\textwidth]{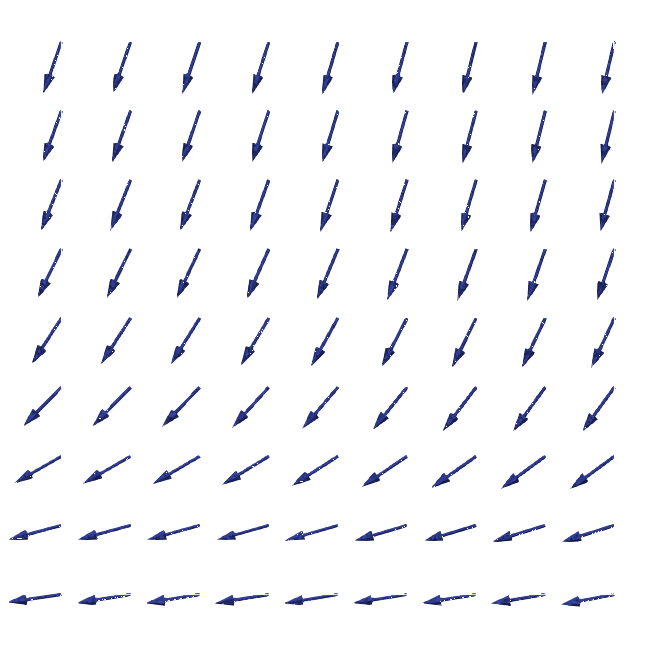}
		\caption{$t=0.1$}
	\end{subfigure}
	\begin{subfigure}[b]{0.3\textwidth}
		\centering
		\includegraphics[width=\textwidth]{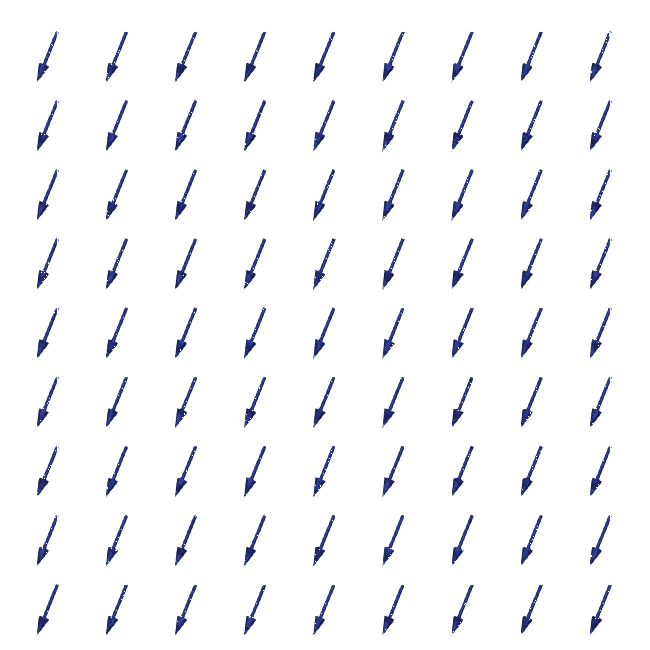}
		\caption{$t=0.5$}
	\end{subfigure}
	\caption{Snapshots of the spin field $\bff{u}$ (projected onto $\bb{R}^2$) for simulation 2.}
	\label{fig:snapshots field 2d 2}
\end{figure}

\begin{figure}[!htb]
	\centering
	\begin{subfigure}[b]{0.3\textwidth}
		\centering
		\includegraphics[width=\textwidth]{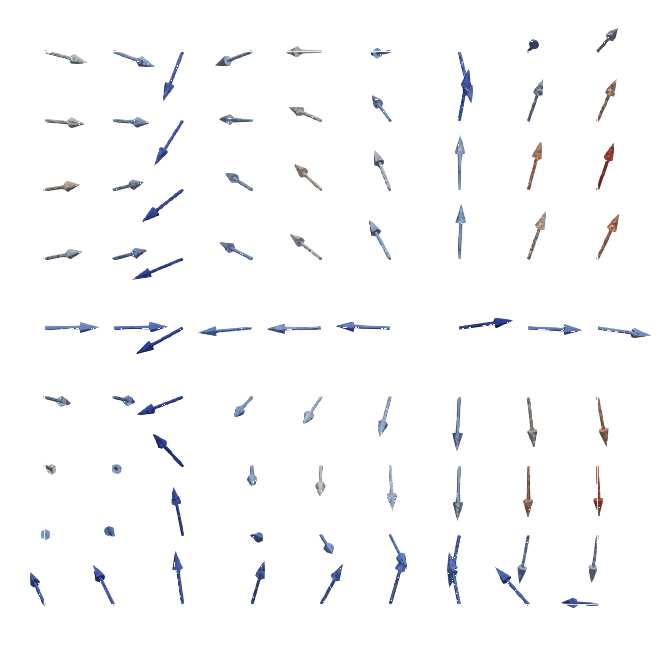}
		\caption{$t=0.0025$}
	\end{subfigure}
	\begin{subfigure}[b]{0.3\textwidth}
		\centering
		\includegraphics[width=\textwidth]{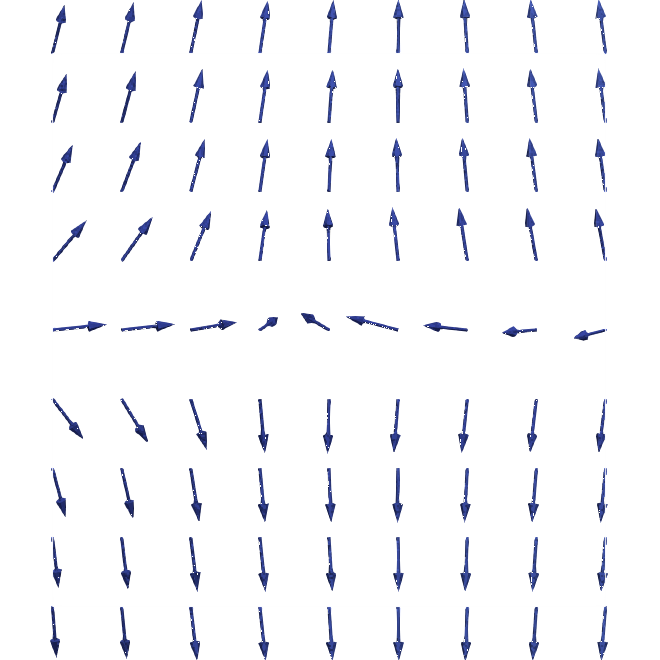}
		\caption{$t=0.025$}
	\end{subfigure}
	\begin{subfigure}[b]{0.3\textwidth}
		\centering
		\includegraphics[width=\textwidth]{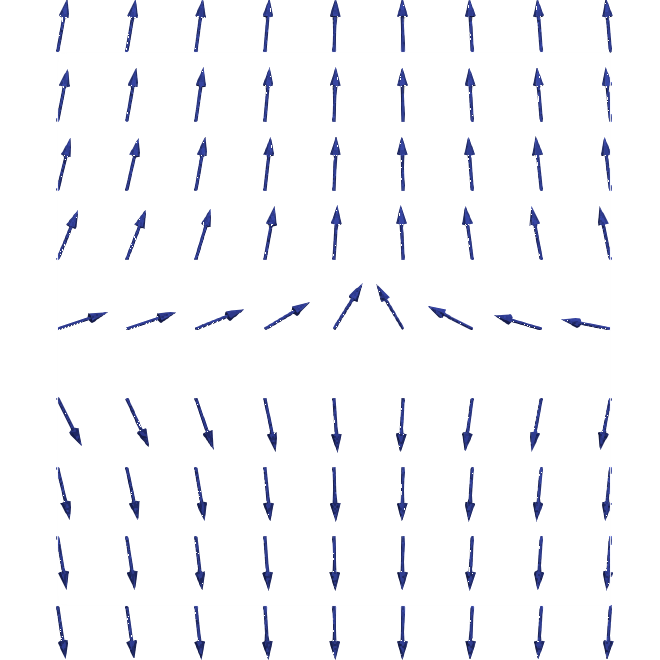}
		\caption{$t=0.05$}
	\end{subfigure}
	\begin{subfigure}[b]{0.3\textwidth}
		\centering
		\includegraphics[width=\textwidth]{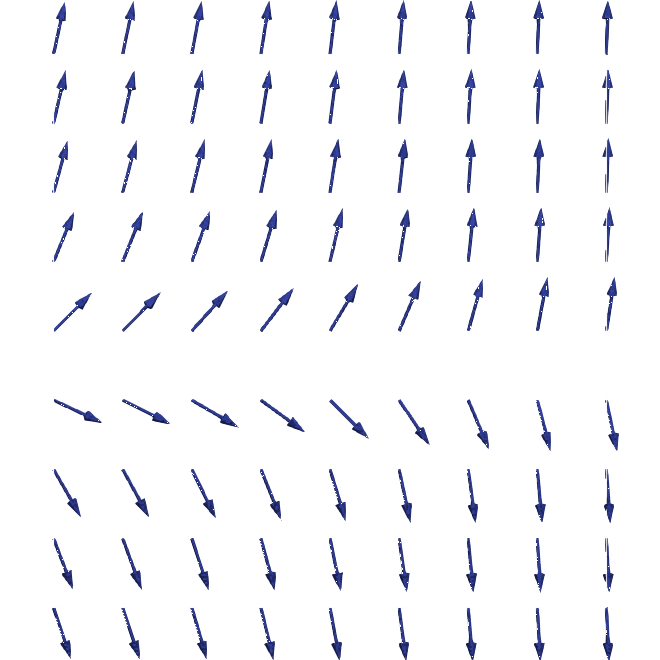}
		\caption{$t=0.1$}
	\end{subfigure}
	\begin{subfigure}[b]{0.3\textwidth}
		\centering
		\includegraphics[width=\textwidth]{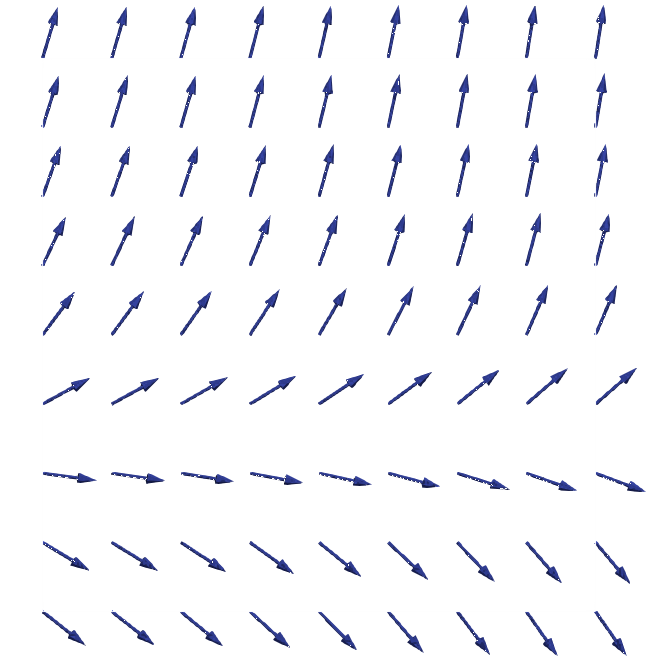}
		\caption{$t=0.125$}
	\end{subfigure}
	\begin{subfigure}[b]{0.3\textwidth}
		\centering
		\includegraphics[width=\textwidth]{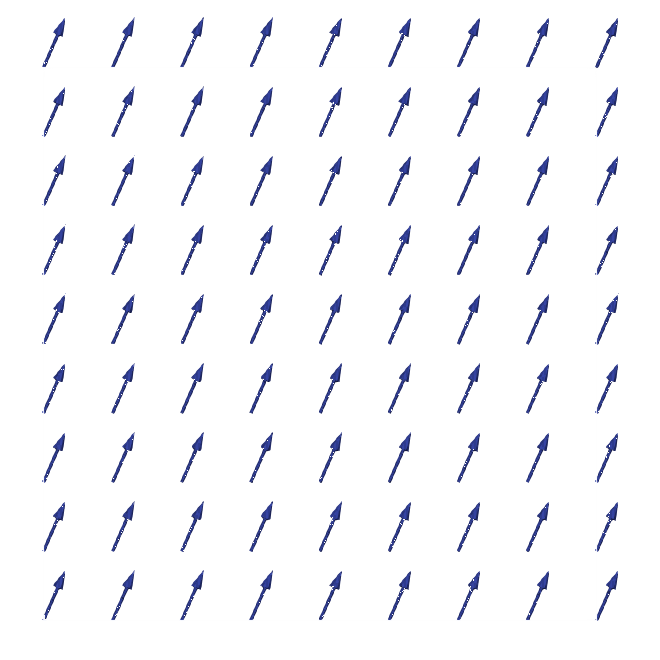}
		\caption{$t=0.5$}
	\end{subfigure}
	\caption{Snapshots of the effective field $\bff{H}$ (projected onto $\bb{R}^2$) for simulation 2.}
	\label{fig:snapshots eff 2d 2}
\end{figure}

\begin{figure}[!htb]
	\begin{subfigure}[b]{0.45\textwidth}
		\centering
		\begin{tikzpicture}
			\begin{axis}[
				title=Plot of $\bff{e}_h(\bff{u})$ against $1/h$,
				height=1.1\textwidth,
				width=1\textwidth,
				xlabel= $1/h$,
				ylabel= $\bff{e}_h$,
				xmode=log,
				ymode=log,
				legend pos=south west,
				legend cell align=left,
				]
				\addplot+[mark=*,red] coordinates {(4,0.958)(8,0.521)(16,0.265)(32,0.133)(64,0.066)};
				\addplot+[mark=*,green] coordinates {(4,0.175)(8,0.041)(16,0.0124)(32,0.00363)(64,0.001)};
				\addplot+[mark=*,blue] coordinates {(4,0.143)(8,0.0425)(16,0.0114)(32,0.00291)(64,0.000727)};
				\addplot+[dashed,no marks,red,domain=20:65]{10/x};
				\addplot+[dashed,no marks,blue,domain=20:65]{1.2/x^2};
				\legend{\small{$\max_n \norm{\bff{e}_h}{\bb{H}^1}$}, \small{$\max_n \norm{\bff{e}_h}{\bb{L}^\infty}$}, \small{$\max_n \norm{\bff{e}_h}{\bb{L}^2}$}, \small{order 1 line}, \small{order 2 line}}
			\end{axis}
		\end{tikzpicture}
		\caption{Error order of $\bff{u}$ for simulation 2.}
		\label{fig:order u 2}
	\end{subfigure}
	\begin{subfigure}[b]{0.45\textwidth}
		\centering
		\begin{tikzpicture}
			\begin{axis}[
				title=Plot of $\bff{e}_h(\bff{H})$ against $1/h$,
				height=1.1\textwidth,
				width=1\textwidth,
				xlabel= $1/h$,
				ylabel= $\bff{e}_h$,
				xmode=log,
				ymode=log,
				legend pos=south west,
				legend cell align=left,
				]
				\addplot+[mark=*,red] coordinates {(4,59.8)(8,36.9)(16,19.1)(32,9.64)(64,4.82)};
				\addplot+[mark=*,green] coordinates {(4,6.03)(8,1.53)(16,0.454)(32,0.164)(64,0.054)};
				\addplot+[mark=*,blue] coordinates {(4,5.18)(8,1.67)(16,0.443)(32,0.112)(64,0.028)};
				\addplot+[dashed,no marks,red,domain=20:65]{650/x};
				\addplot+[dashed,no marks,blue,domain=20:65]{40/x^2};
				\legend{\small{$\max_n \norm{\bff{e}_h}{\bb{H}^1}$}, \small{$\max_n \norm{\bff{e}_h}{\bb{L}^\infty}$}, \small{$\max_n \norm{\bff{e}_h}{\bb{L}^2}$}, \small{order 1 line}, \small{order 2 line}}
			\end{axis}
		\end{tikzpicture}
		\caption{Error order of $\bff{H}$ for simulation 2.}
		\label{fig:order H 2}
	\end{subfigure}
\end{figure}

\subsection{Simulation 3 (Energy dissipativity)}
We take $k=2.5\times 10^{-3}$. The coefficients in~\eqref{equ:llbar} are taken to be $\lambda_e=0.001, \lambda_r=4.0, \gamma=5.0, \kappa=3.0$, and $\beta=0.2$. The unit vector $\bff{e}=(0,1,0)^\top$. The initial data $\bff{u}_0$ is given by
\begin{equation*}
	\bff{u}_0(x,y)= \big(-2y\cos(2\pi x),\, 4x^2\sin(2\pi y),\, 2\cos(2\pi x) \sin(2\pi y) \big).
\end{equation*}
Recall that the energy of the system was defined in~\eqref{equ:energy}. Several plots of energy against time are shown in Figures~\ref{fig:energy mu pos fix k}, \ref{fig:energy mu neg fix k}, \ref{fig:energy mu pos fix h}, and \ref{fig:energy mu neg fix h} for positive or negative $\mu$ and various values of $h$ and $k$. In all cases, the energy is seen to decrease monotonically at the discrete level.

\begin{figure}[!htb]
	\centering
	\begin{subfigure}[b]{0.49\textwidth}
		\centering
		\includegraphics[width=\textwidth]{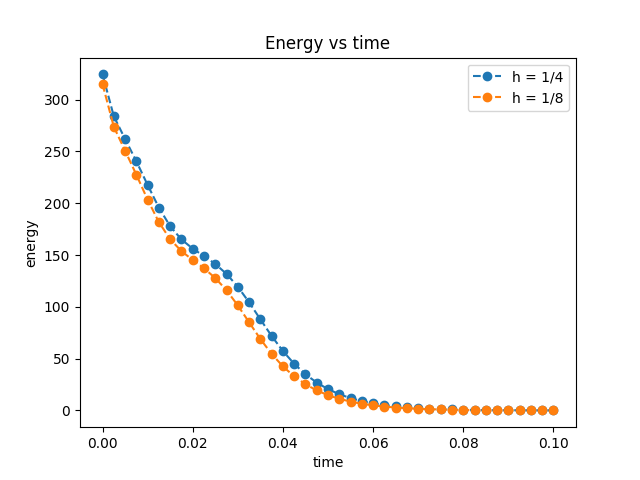}
		\caption{Energy vs time for $\mu=20.0$, $k=0.0025$}
		\label{fig:energy mu pos fix k}
	\end{subfigure}
	\begin{subfigure}[b]{0.49\textwidth}
		\centering
		\includegraphics[width=\textwidth]{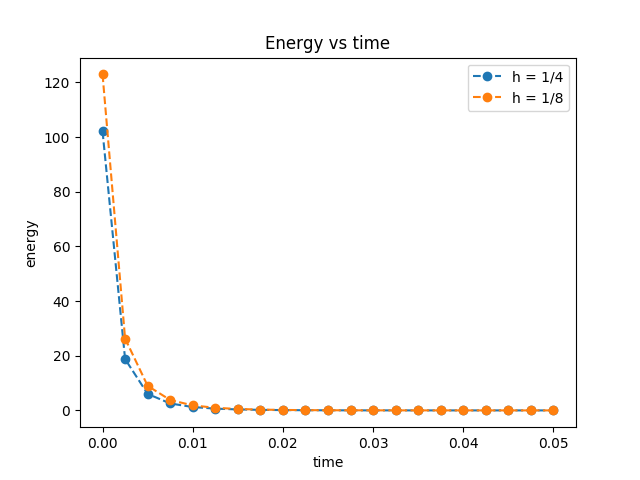}
		\caption{Energy vs time for $\mu=-5.0$, $k=0.0025$}
		\label{fig:energy mu neg fix k}
	\end{subfigure}
	\begin{subfigure}[b]{0.49\textwidth}
	\centering
	\includegraphics[width=\textwidth]{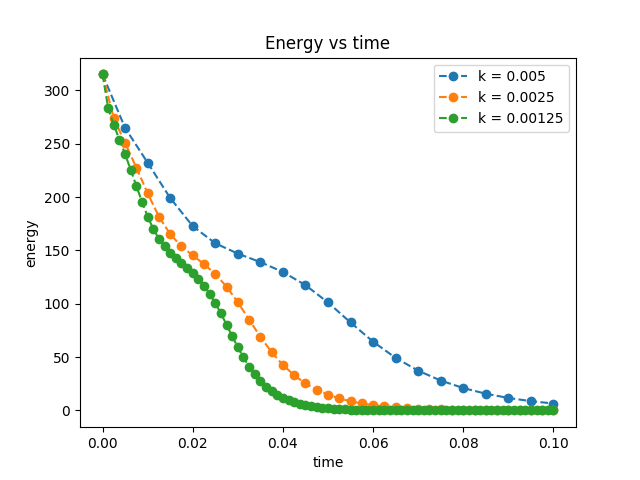}
	\caption{Energy vs time for $\mu=20.0$, $h=1/8$}
	\label{fig:energy mu pos fix h}
	\end{subfigure}
	\begin{subfigure}[b]{0.49\textwidth}
	\centering
	\includegraphics[width=\textwidth]{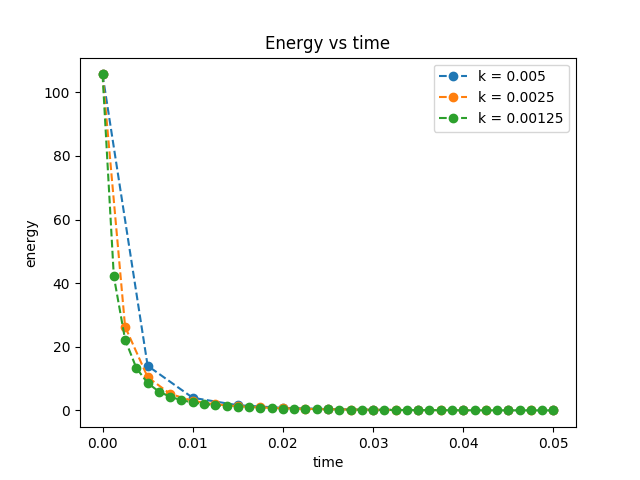}
	\caption{Energy vs time for $\mu=-1.0$, $h=1/8$}
	\label{fig:energy mu neg fix h}
	\end{subfigure}
\end{figure}

\section*{Acknowledgements}
The author is supported by the Australian Government Research Training
Program (RTP) Scholarship awarded at the University of New South Wales, Sydney. Some results from this manuscript have been presented at the WONAPDE~2024 conference in Concepci\'on, Chile.

The author also gratefully acknowledges financial support from the Australian Research Council under grant number DP200101866 (awarded to Prof. Thanh Tran).

\bibliographystyle{myabbrv}
\bibliography{mybib.bib}

\newcommand{\noopsort}[1]{}\def\cprime{$'$}
  \def\soft#1{\leavevmode\setbox0=\hbox{h}\dimen7=\ht0\advance \dimen7
  by-1ex\relax\if t#1\relax\rlap{\raise.6\dimen7
  \hbox{\kern.3ex\char'47}}#1\relax\else\if T#1\relax
  \rlap{\raise.5\dimen7\hbox{\kern1.3ex\char'47}}#1\relax \else\if
  d#1\relax\rlap{\raise.5\dimen7\hbox{\kern.9ex \char'47}}#1\relax\else\if
  D#1\relax\rlap{\raise.5\dimen7 \hbox{\kern1.4ex\char'47}}#1\relax\else\if
  l#1\relax \rlap{\raise.5\dimen7\hbox{\kern.4ex\char'47}}#1\relax \else\if
  L#1\relax\rlap{\raise.5\dimen7\hbox{\kern.7ex
  \char'47}}#1\relax\else\message{accent \string\soft \space #1 not
  defined!}#1\relax\fi\fi\fi\fi\fi\fi}
\begin{thebibliography}{10}

\bibitem{AlnaesEtal15}
M.~S. Alnaes, J.~Blechta, J.~Hake, A.~Johansson, B.~Kehlet, A.~Logg, C.~N.
  Richardson, J.~Ring, M.~E. Rognes, and G.~N. Wells.
\newblock The {FEniCS} project version 1.5.
\newblock {\em Archive of Numerical Software},  {\bf 3} (2015).

\bibitem{AloSoy92}
F.~Alouges and A.~Soyeur.
\newblock On global weak solutions for {L}andau-{L}ifshitz equations: existence
  and nonuniqueness.
\newblock {\em Nonlinear Anal.},  {\bf 18} (1992), 1071--1084.

\bibitem{BarBar98}
I.~Baryakhtar and V.~Baryakhtar.
\newblock A motion equation for magnetization: dynamics and relaxation.
\newblock {\em Ukr. Fiz. Zh.},  {\bf 43} (1998), 1433--1448.

\bibitem{Bar84}
V.~Baryakhtar.
\newblock Phenomenological description of relaxation processes in magnets.
\newblock {\em Zh. Eksp. Teor. Fiz.},  {\bf 87} (1984).

\bibitem{BevGalSimRio13}
L.~Bevilacqua, A.~Gale\~ao, J.~Simas, and A.~Doce.
\newblock A new theory for anomalous diffusion with a bimodal flux
  distribution.
\newblock {\em J. Braz. Soc. Mech. Sci. Eng.},  {\bf 35} (2013), 431--440.

\bibitem{BraPasSte02}
J.~H. Bramble, J.~E. Pasciak, and O.~Steinbach.
\newblock On the stability of the {$L^2$} projection in {$H^1(\Omega)$}.
\newblock {\em Math. Comp.},  {\bf 71} (2002), 147--156.

\bibitem{BrzGolLe20}
Z.~Brze\'{z}niak, B.~Goldys, and K.~N. Le.
\newblock Existence of a unique solution and invariant measures for the
  stochastic {L}andau-{L}ifshitz-{B}loch equation.
\newblock {\em J. Differential Equations},  {\bf 269} (2020), 9471--9507.

\bibitem{CarFab01}
G.~Carbou and P.~Fabrie.
\newblock Regular solutions for {L}andau-{L}ifschitz equation in a bounded
  domain.
\newblock {\em Differential Integral Equations},  {\bf 14} (2001), 213--229.

\bibitem{ChuNie20}
O.~Chubykalo-Fesenko and P.~Nieves.
\newblock {\em Landau-Lifshitz-Bloch approach for magnetization dynamics close
  to phase transition}, pages 867--893.
\newblock Springer International Publishing, Cham, 2020.

\bibitem{ChuNowChaGar06}
O.~Chubykalo-Fesenko, U.~Nowak, R.~W. Chantrell, and D.~Garanin.
\newblock Dynamic approach for micromagnetics close to the {C}urie temperature.
\newblock {\em Phys. Rev. B},  {\bf 74} (2006), 094436.

\bibitem{Cim08}
I.~Cimr{\'a}k.
\newblock A survey on the numerics and computations for the {L}andau-{L}ifshitz
  equation of micromagnetism.
\newblock {\em Arch. Comput. Methods Eng.},  {\bf 15} (2008), 277--309.

\bibitem{CocDi23}
G.~M. Coclite and L.~di~Ruvo.
\newblock On a diffusion model for growth and dispersal in a population.
\newblock {\em Commun. Pure Appl. Anal.},  {\bf 22} (2023), 1194--1225.

\bibitem{CohMur81}
D.~S. Cohen and J.~D. Murray.
\newblock A generalized diffusion model for growth and dispersal in a
  population.
\newblock {\em J. Math. Biol.},  {\bf 12} (1981), 237--249.

\bibitem{DemLeySchWah12}
A.~Demlow, D.~Leykekhman, A.~H. Schatz, and L.~B. Wahlbin.
\newblock Best approximation property in the {$W^{1}_{\infty}$} norm for finite
  element methods on graded meshes.
\newblock {\em Math. Comp.},  {\bf 81} (2012), 743--764.

\bibitem{DvoVanVan13}
M.~Dvornik, A.~Vansteenkiste, and B.~Van~Waeyenberge.
\newblock Micromagnetic modeling of anisotropic damping in magnetic
  nanoelements.
\newblock {\em Phys. Rev. B},  {\bf 88} (2013), 054427.

\bibitem{Gar91}
D.~Garanin.
\newblock Generalized equation of motion for a ferromagnet.
\newblock {\em Physica A: Statistical Mechanics and its Applications},  {\bf
  172} (1991), 470 -- 491.

\bibitem{Gar97}
D.~A. Garanin.
\newblock {F}okker-{P}lanck and {L}andau-{L}ifshitz-{B}loch equations for
  classical ferromagnets.
\newblock {\em Phys. Rev. B},  {\bf 55} (1997), 3050--3057.

\bibitem{GolJiaLe22}
B.~Goldys, C.~Jiao, and K.-N. Le.
\newblock Numerical method and error estimate for stochastic
  {L}andau--{L}ifshitz--{B}loch equation.
\newblock arXiv:2212.10833, 2022.

\bibitem{GolSoeTra24a}
B.~Goldys, A.~L. Soenjaya, and T.~Tran.
\newblock Regularity and asymptotic behaviour of the solution to the
  {L}andau--{L}ifshitz--{B}aryakhtar equation, 2024.
\newblock Preprint.

\bibitem{GolSoeTra24b}
B.~Goldys, A.~L. Soenjaya, and T.~Tran.
\newblock The stochastic {L}andau--{L}ifshitz--{B}aryakhtar equation: Global
  solution and invariant measure.
\newblock arXiv:2405.14112, 2024.

\bibitem{Gri11}
P.~Grisvard.
\newblock {\em Elliptic problems in nonsmooth domains}, volume~69 of {\em
  Classics in Applied Mathematics}.
\newblock Society for Industrial and Applied Mathematics (SIAM), Philadelphia,
  PA, 2011.

\bibitem{GuiLiWan22}
X.~Gui, B.~Li, and J.~Wang.
\newblock Convergence of renormalized finite element methods for heat flow of
  harmonic maps.
\newblock {\em SIAM J. Numer. Anal.},  {\bf 60} (2022), 312--338.

\bibitem{GuoDing08}
B.~Guo and S.~Ding.
\newblock {\em Landau--{L}ifshitz {E}quations}, volume~1 of {\em Frontiers of
  Research with the Chinese Academy of Sciences}.
\newblock World Scientific Publishing Co. Pte. Ltd., Hackensack, NJ, 2008.

\bibitem{JiaBevZhu20}
M.~Jiang, L.~Bevilacqua, J.~Zhu, and X.~Yu.
\newblock Nonlinear {G}alerkin finite element methods for fourth-order bi-flux
  diffusion model with nonlinear reaction term.
\newblock {\em Comput. Appl. Math.},  {\bf 39} (2020), Paper No. 143, 16.

\bibitem{Lak11}
M.~Lakshmanan.
\newblock The fascinating world of the {L}andau-{L}ifshitz-{G}ilbert equation:
  an overview.
\newblock {\em Philos. Trans. R. Soc. Lond. Ser. A Math. Phys. Eng. Sci.},
  {\bf 369} (2011), 1280--1300.

\bibitem{LL35}
L.~Landau and E.~Lifschitz.
\newblock On the theory of the dispersion of magnetic permeability in
  ferromagnetic bodies.
\newblock {\em Phys Z Sowjetunion},  {\bf 8} (1935), 153--168.

\bibitem{Le16}
K.~N. Le.
\newblock Weak solutions of the {L}andau-{L}ifshitz-{B}loch equation.
\newblock {\em J. Differential Equations},  {\bf 261} (2016), 6699--6717.

\bibitem{LeSoeTra24}
K.-N. Le, A.~L. Soenjaya, and T.~Tran.
\newblock The {L}andau--{L}ifshitz--{B}loch equation: Unique existence and
  finite element approximation.
\newblock arXiv:2406.05808, 2024.

\bibitem{Li22}
B.~Li.
\newblock Maximum-norm stability of the finite element method for the {N}eumann
  problem in nonconvex polygons with locally refined mesh.
\newblock {\em Math. Comp.},  {\bf 91} (2022), 1533--1585.

\bibitem{Li17}
H.~Li.
\newblock The {$W^1_p$} stability of the {R}itz projection on graded meshes.
\newblock {\em Math. Comp.},  {\bf 86} (2017), 49--74.

\bibitem{MeoPan20}
A.~Meo, W.~Pantasri, W.~Daeng-am, S.~E. Rannala, S.~I. Ruta, R.~W. Chantrell,
  P.~Chureemart, and J.~Chureemart.
\newblock Magnetization dynamics of granular heat-assisted magnetic recording
  media by means of a multiscale model.
\newblock {\em Phys. Rev. B},  {\bf 102} (2020), 174419.

\bibitem{Mur02}
J.~D. Murray.
\newblock {\em Mathematical biology. {I}}, volume~17 of {\em Interdisciplinary
  Applied Mathematics}.
\newblock Springer-Verlag, New York, third edition, 2002.

\bibitem{NieChu16}
P.~Nieves and O.~Chubykalo-Fesenko.
\newblock Modeling of ultrafast heat- and field-assisted magnetization dynamics
  in {F}e{P}t.
\newblock {\em Phys. Rev. Appl.},  {\bf 5} (2016), 014006.

\bibitem{Och84}
F.~L. Ochoa.
\newblock A generalized reaction diffusion model for spatial structure formed
  by motile cells.
\newblock {\em Biosystems},  {\bf 17} (1984), 35--50.

\bibitem{Sco76}
R.~Scott.
\newblock Optimal {$L\sp{\infty }$} estimates for the finite element method on
  irregular meshes.
\newblock {\em Math. Comp.},  {\bf 30} (1976), 681--697.

\bibitem{SoeTra23}
A.~L. Soenjaya and T.~Tran.
\newblock Global solutions of the {L}andau-{L}ifshitz-{B}aryakhtar equation.
\newblock {\em J. Differential Equations},  {\bf 371} (2023), 191--230.

\bibitem{SoeTra23b}
A.~L. Soenjaya and T.~Tran.
\newblock Stable {$C^1$}-conforming finite element methods for the
  {L}andau-{L}ifshitz-{B}aryakhtar equation.
\newblock arXiv:2309.05530, 2023.

\bibitem{Tho06}
V.~Thom\'{e}e.
\newblock {\em Galerkin finite element methods for parabolic problems},
  volume~25 of {\em Springer Series in Computational Mathematics}.
\newblock Springer-Verlag, Berlin, second edition, 2006.

\bibitem{Vis85}
A.~Visintin.
\newblock On {L}andau-{L}ifshitz' equations for ferromagnetism.
\newblock {\em Japan J. Appl. Math.},  {\bf 2} (1985), 69--84.

\bibitem{WanDvo15}
W.~Wang, M.~Dvornik, M.-A. Bisotti, D.~Chernyshenko, M.~Beg, M.~Albert,
  A.~Vansteenkiste, B.~V. Waeyenberge, A.~N. Kuchko, V.~V. Kruglyak, and
  H.~Fangohr.
\newblock Phenomenological description of the nonlocal magnetization relaxation
  in magnonics, spintronics, and domain-wall dynamics.
\newblock {\em Phys. Rev. B},  {\bf 92} (2015), 054430.

\bibitem{ZhaLi04}
S.~Zhang and Z.~Li.
\newblock Roles of nonequilibrium conduction electrons on the magnetization
  dynamics of ferromagnets.
\newblock {\em Phys. Rev. Lett.},  {\bf 93} (2004), 127204.

\end{thebibliography}

\end{document}